\documentclass{article}
 \PassOptionsToPackage{numbers, compress}{natbib}
 \usepackage[final]{neurips_2024}

\usepackage{algorithm}
\usepackage{algorithmic}
\usepackage{caption}

\usepackage{hyperref}
\usepackage{url}

\newcommand{\myeqc}[1]{ {\tiny\textcircled{#1}} }

\usepackage[utf8]{inputenc} % allow utf-8 input
\usepackage[T1]{fontenc}    % use 8-bit T1 fonts
\usepackage{hyperref}       % hyperlinks
\usepackage{url}            % simple URL typesetting
\usepackage{booktabs}       % professional-quality tables
\usepackage{amsfonts}       % blackboard math symbols
\usepackage{nicefrac}       % compact symbols for 1/2, etc.
\usepackage{microtype}      % microtypography
\usepackage{xcolor}         % colors

\usepackage{pifont}
\usepackage{epsfig}
\usepackage{amssymb}
\usepackage{amsmath}
\usepackage{amsthm}
\usepackage{amsfonts}
\usepackage{bbding}
\usepackage{array}
\usepackage{paralist}
\usepackage{caption}
\hypersetup{
  colorlinks   = true, %Colours links instead of ugly boxes
  urlcolor     = blue, %Colour for external hyperlinks
  linkcolor    = blue, %Colour of internal links
  citecolor   = blue %Colour of citations
}

\newtheorem{theorem}{Theorem}
\newtheorem{lemma}{Lemma}

\newtheorem{corollary}{Corollary}

\newtheorem{remark}{Remark}
\newtheorem{assumption}{Assumption}

\renewcommand{\top}{T}
\newcommand{\set}[1]{\left\{#1\right\}}
\newcommand{\sets}[1]{\{#1\}}

\newcommand{\norms}[1]{\Vert#1\Vert}
\newcommand{\dom}[1]{\mathrm{dom}\left(#1\right)}
\newcommand{\iprods}[1]{\langle#1\rangle}

\usepackage{xcolor}

\newcommand{\R}{\mathbb{R}}
\newcommand{\Rext}{\R\cup\{+\infty\}}

\newcommand{\prox}{\mathrm{prox}}

\newcommand{\argmin}{\mathrm{arg}\!\displaystyle\min}
\newcommand{\argmax}{\mathrm{arg}\!\displaystyle\max}

\newcommand{\zero}[1]{{\boldsymbol{0}}}

\newcommand{\Expb}[1]{\mathbb{E}\big[#1\big]}

\newcommand{\E}{\mathbb{E}}
\newcommand{\Fc}{\mathcal{F}}
\newcommand{\Eb}{\mathbb{E}} 
\newcommand{\Hc}{\mathcal{H}} 
\newcommand{\Lc}{\mathcal{L}} 
 
\newcommand{\Vc}{\mathcal{V}} 
\newcommand{\Tc}{\mathcal{T}} 
\newcommand{\Gc}{\mathcal{G}}
\newcommand{\Oc}{\mathcal{O}} 
\newcommand{\BigO}[1]{\mathcal{O}\big(#1\big)} 
\newcommand{\BigOs}[1]{\mathcal{O}(#1)}

\newcommand{\beforesec}{\vspace{-2ex}}
\newcommand{\aftersec}{\vspace{-2ex}}
\newcommand{\beforesubsec}{\vspace{-1.5ex}}
\newcommand{\aftersubsec}{\vspace{-1ex}}
\newcommand{\beforesubsubsec}{\vspace{-2ex}}
\newcommand{\aftersubsubsec}{\vspace{-1ex}}

\title{Shuffling Gradient-Based Methods for Nonconvex-Concave Minimax Optimization}

\author{%
Quoc Tran-Dinh\\
Department of Statistics and Operations Research \\  
The University of North Carolina at Chapel Hill\\
\texttt{quoctd@email.unc.edu} 
\And
Trang H. Tran\\
School of OR and Information Engineering \\  
 Cornell University, Ithaca, NY \\  
 \texttt{htt27@cornell.edu}  
\And
Lam M. Nguyen \\
IBM Research, Thomas J. Watson Research Center\\ 
Yorktown Heights, NY \\
\texttt{LamNguyen.MLTD@ibm.com}
}

\begin{document}
\maketitle

\begin{abstract}
This paper aims at developing novel shuffling gradient-based methods for tackling two classes of minimax problems: \textit{nonconvex-linear} and \textit{nonconvex-strongly concave} settings.
The first algorithm addresses the nonconvex-linear minimax model and achieves the state-of-the-art oracle complexity typically observed in nonconvex optimization. 
It also employs a new shuffling estimator for the ``hyper-gradient'', departing from standard shuffling techniques in optimization.
The second method consists of two variants: \textit{semi-shuffling} and \textit{full-shuffling} schemes. 
These variants tackle the nonconvex-strongly concave minimax setting. 
We establish their oracle complexity bounds under standard assumptions, which, to our best knowledge, are the best-known for this specific setting.
Numerical examples demonstrate the performance of our algorithms and compare them with two other methods. 
Our results show that the new methods achieve comparable performance with SGD, supporting the potential of incorporating shuffling strategies into minimax algorithms.
\end{abstract}

%%%%%
\section{Introduction}\label{sec_intro}
\aftersec
Minimax problems arise in various applications across generative machine learning, game theory, robust optimization, online learning, and reinforcement learning (e.g., \cite{arjovsky2017wasserstein,azar2017minimax,Ben-Tal2009,bhatia2020online,gidel2018variational,goodfellow2014generative,ho2022game,jabbar2021survey,lin2022distributionally,madry2018towards,wang2021adversarial,zhang2021multi}). 
These models often involve stochastic settings or large finite-sum objective functions.
To tackle these problems, existing methods frequently adapt stochastic gradient descent (SGD) principles to develop algorithms for solving the underlying minimax problems \cite{beznosikov2023stochastic,goodfellow2014generative}. 
For instance, in generative adversarial networks (GANs), early algorithms employed stochastic gradient descent-ascent methods where two routines, each using an SGD loop, ran iteratively \cite{goodfellow2014generative}.
However, practical implementations of SGD often incorporate shuffling strategies, as seen in popular deep learning libraries like TensorFlow and PyTorch. 
This has motivated recent research on developing shuffling techniques specifically for optimization algorithms \cite{mishchenko2020random,mishchenko2022proximal,nguyen2020unified,haochen2018random,meng2019convergence,Safran2019HowGoodSGDShuffling,ying2017convergence}. 
Our work builds upon this trend by developing shuffling methods for two specific classes of minimax problems.

%%% Problem statement.
\textbf{Problem statement.}
In this paper, we study the following minimax optimization problem:
\vspace{-0.5ex}
\begin{equation}\label{eq:minimax_prob}
\min_{w \in \mathbb{R}^p}\max_{u \in \R^q}\Big\{ \Lc(w, u) := f(w) + \mathcal{H}(w, u)  - h(u) \equiv f(w) +  \tfrac{1}{n}\sum_{i=1}^n \mathcal{H}_i(w, u) - h(u) \Big\},
\end{equation}
where $f : \R^p\to\Rext$ is a proper, closed, and convex function, $\mathcal{H}_i : \R^p\times \R^q \to \R$ are smooth for all $i \in [n] := \sets{1,2,\cdots, n}$, and $h : \R^q\to \Rext$ is also a proper, closed, and convex function.
In this paper, we will focus on two classes of problems in \eqref{eq:minimax_prob}, overlapped to each other.
\vspace{-0.25ex}
\begin{compactitem}
\item[(NL)] $\mathcal{H}_i$ is nonconvex in $w$ and linear in $u$ as $\mathcal{H}_i(w, u) := \iprods{F_i(w), Ku}$ for a given function $F_i : \R^p \to \R^m$ and a matrix $K \in \R^{q\times m}$ for all $i \in [n]$ and $(w, u) \in \dom{\Lc}$.
\item[(NC)] $\mathcal{H}_i$ is nonconvex in $w$ and $\mathcal{H}_i(w, \cdot) - h(\cdot)$ is strongly concave in $u$ for all $(w, u)\in\dom{\Lc}$.
\end{compactitem}
\vspace{-0.25ex}
Although (NC) looks more general than (NL), both cases can be overlapped, but one is not a special case of the other. 
Under these two settings, our approach will rely on a \textit{bilevel optimization} approach, where the lower-level problem is to solve $\max_u\Lc(w, u)$, while the upper-level one is $\min_w\Lc(w, u)$.

%%% Challenges.
\textbf{Challenges.}
The setting (NL) is a special case of stochastic nonconvex-concave minimax problems because the objective term $\Hc(w, u) := \iprods{F(w), Ku}$ is linear in $u$. 
It is equivalent to the compositional model \eqref{eq:com_opt0} described below. 
However, if $h$ is only merely convex and not strongly convex (e.g., the indicator of a standard simplex), then $\Phi_0$ in \eqref{eq:com_opt0} becomes nonsmooth regardless of $F$'s properties.
This presents our first challenge.
A natural approach to address this issue, as discussed in Section~\ref{sec:Ass_and_Grad}, is to smooth $\Phi_0$.
The second challenge arises from the composition between the outer function $h^{*}$ and the finite sum $F(\cdot)$ in \eqref{eq:com_opt0}. 
Unlike standard finite-sum optimization, this composition prevents any direct use of existing techniques, requiring a novel approach for algorithmic development and analysis.
The third challenge involves unbiased estimators for gradients or ``hyper-gradients'' in minimax problems. 
Most existing methods rely on unbiased estimators for objective gradients, with limited work exploring biased estimators. 
While biased estimators can be used, they require variance reduction properties (see, e.g., \cite{driggs2019bias}).
The setting (NC) faces the same second and third challenges as the setting (NL). 
Additionally, when reformulating it as a minimization problem using a bilevel optimization approach \eqref{eq:upper_level_min}, constructing a shuffling estimator for the ``hyper-gradient'' $\nabla{\Phi}_0$ becomes unclear. 
This requires solving the lower-level maximization problem \eqref{eq:lower_level_max}.
Therefore, it remains an open question whether shuffling gradient-type methods can be extended to this bilevel optimization approach to address \eqref{eq:minimax_prob}.
In this paper, we address the following research question:
\begin{center}
\vspace{-1ex}
\textit{Can we efficiently develop shuffling gradient methods to solve  \eqref{eq:minimax_prob} for both  $\mathrm{(NL)}$ and $\mathrm{(NC)}$ settings?}
\vspace{-1ex}
\end{center}
Our attempt to tackle this question leads to a novel way of constructing shuffling estimators for the hyper-gradient $\nabla{\Phi}_0$ or its smoothed counterpart. 
This allows us to develop two shuffling gradient-based algorithms with rigorous theoretical  guarantees on oracle complexity, matching state-of-the-art complexity results in shuffling-type algorithms for nonconvex optimization.

%%% Related work
\textbf{Related work.}
Shuffling optimization algorithms have gained significant attention in optimization and machine communities, demonstrating advantages over standard SGDs, see, e.g., \cite{mishchenko2020random,mishchenko2022proximal,nguyen2020unified,haochen2018random,meng2019convergence,Safran2019HowGoodSGDShuffling,ying2017convergence}. 
Nevertheless, applying these techniques to minimax problems like \eqref{eq:minimax_prob} remains challenging, with limited existing literature (e.g., \cite{cho2022sgda,das2022sampling,emmanouilidis2024stochastic}).
Das \textit{et al.} in \cite{das2022sampling} explored a specific case of \eqref{eq:minimax_prob} without nonsmooth terms $f$ and $h$, assuming strong monotonicity and $L$-Lipschitz continuity of the gradient $\nabla{\Hc} := [\nabla_w{\Hc}, -\nabla_u{\Hc} ]$ of the joint objective $\Hc$.
Their algorithm simplifies to a shuffling variant of fixed-point iteration or a gradient descent-ascent scheme, not applicable to our settings.
Cho and Yun in \cite{cho2022sgda} built upon \cite{das2022sampling} by relaxing the strong monotonicity to Polyak-{\L}ojasiewicz (P{\L}) conditions.
This work is perhaps the most closely related one to our algorithm, Algorithm~\ref{alg:SGM2}, for the (NC) setting.
Note that the method in \cite{cho2022sgda} exploits Nash's equilibrium perspective with a simultaneous update, which is different from our alternative update.
Moreover, \cite{cho2022sgda} only considers the noncomposite case with $f=0$ and $h = 0$.
Though we only focus on a nonconvex-strongly-concave setting (NC), our results here can be extended to the P{\L} condition as in  \cite{cho2022sgda}.
Very recently, Konstantinos \textit{et al.} in \cite{emmanouilidis2024stochastic} introduced shuffling extragradient methods for variational inequalities, which encompass convex-concave minimax problems as a special case. 
However, this also falls outside the scope of our work due to the nonconvexity of \eqref{eq:minimax_prob} in $w$. 
Again, all the existing works in \cite{cho2022sgda,das2022sampling,emmanouilidis2024stochastic} utilize a Nash's equilibrium perspective, while ours leverages a bilevel optimization technique.
Besides, in contrast to our sampling-without-replacement approach, stochastic and randomized methods (i.e. using i.i.d. sampling strategies) have been extensively studied for minimax problems, see, e.g., \cite{TranDinh2020f,gorbunov2022stochastic,hamedani2018iteration,hsieh2020explore,loizou2021stochastic,luo2020stochastic,palaniappan2016stochastic,yang2020global,zhao2019optimal}.
A comprehensive comparison can be found, e.g., in \cite{cho2022sgda}.

%%% Contribution.
\textbf{Contribution.}
Our main contribution can be summarized as follows.
\begin{compactitem}
\item[(a)] For setting (NL), we suggest to reformulate \eqref{eq:minimax_prob} into a compositional minimization and exploit a smoothing technique to treat this reformulation.
We propose a new way of constructing shuffling estimators for the ``hyper-gradient'' $\nabla{\Phi_{\gamma}}$ (cf. \eqref{eq:smoothed_com_opt_prob}) and establish their properties.

\item[(b)] We propose a novel shuffling gradient-based algorithm (\textit{cf.} Algorithm~\ref{alg:SGM1}) to approximate an $\epsilon$-KKT point of \eqref{eq:minimax_prob} for the setting (NL).
Our method requires $\Oc(n\epsilon^{-3})$ evaluations of $F_i$ and $\nabla{F_i}$ under the strong convexity of $h$, and $\Oc(n\epsilon^{-7/2})$  evaluations of $F_i$ and $\nabla{F_i}$  without the strong convexity of $h$, for a desired accuracy $\epsilon > 0$.

\item[(c)] For setting (NC), we develop two variants of the shuffling gradient method: \textit{semi-shuffling} and \textit{full-shuffling} schemes (\textit{cf.} Algorithm~\ref{alg:SGM2}).
The semi-shuffling variant combines both gradient ascent and shuffling gradient methods to construct a new algorithm, which requires  $\Oc(n\epsilon^{-3})$ evaluations of both $\nabla_w{\Hc}_i$ and $\nabla_u{\Hc}_i$.
The full-shuffling scheme allows to perform both shuffling schemes on the maximization and the minimization alternatively, requiring either $\Oc(n\epsilon^{-3})$ or $\Oc(n\epsilon^{-4})$ evaluations of $\nabla_u{\Hc}_i$ depending on our assumptions, while maintaining $\Oc(n\epsilon^{-3})$ evaluations of $\nabla_w{\Hc}_i$ for a given desired accuracy $\epsilon > 0$.
\end{compactitem}
%%%
If a random shuffling strategy is used in our algorithms, then the oracle complexity in all the cases presented above is improved by a factor  of $\sqrt{n}$.
Our settings (NL) and (NC) of \eqref{eq:minimax_prob} are different from existing works \cite{cho2022sgda,das2022sampling,emmanouilidis2024stochastic}, as we work with general nonconvexity in $w$, and linearity or [strong] concavity in $u$, and both $f$ and $h$ are possibly nonsmooth.
Our algorithms are not reduced or similar to existing shuffling methods for optimization, but we use shuffling strategies to form estimators for the hyper-gradient $\nabla{\Phi}_0$ in \eqref{eq:exact_grad_phi}.
The oracle complexity in both settings (NL) and (NC) is similar to the ones in nonconvex optimization and in a special case of \eqref{eq:minimax_prob} from \cite{cho2022sgda} (up to a constant factor). 

%%% Paper outline.
\textbf{Paper outline.}
The rest of this paper is organized as follows.
Section~\ref{sec:Ass_and_Grad} presents our bilevel optimization approach to \eqref{eq:minimax_prob} and recalls necessary preliminary results.
Section~\ref{sec:SGM1_method} develops our shuffling algorithm to solve the setting (NL) of  \eqref{eq:minimax_prob} and establishes its convergence.
Section~\ref{sec:SGM2_method} proposes new shuffling methods to solve the setting (NC) and investigates their convergence.
Section~\ref{sec:num_experiments} presents numerical experiments, while technical proofs and supporting results are deferred to Supp. Docs.

%%% Notations.
\textbf{Notations.}
For a function $f$, we use $\dom{f}$ to denote its effective domain, and $\nabla{f}$ for its gradient or Jacobian.
If $f$ is convex, then $\nabla{f}$ denotes a subgradient, $\partial{f}$ is its subdifferential, and $\prox_f$ is its proximal operator.
We use $\Fc_t$ to denote $\sigma(w_0, w_1,\cdots, w_t)$, a $\sigma$-algebra generated by random vectors $w_0, w_1, \cdots, w_t$, $\Eb_t[\cdot] = \Eb[ \cdot \vert \Fc_t]$ is a conditional expectation, and $\Eb[\cdot]$ is the full expectation.
As usual, $\Oc(\cdot)$ denotes Big-O notation in the theory of algorithm complexity.

%%%%%%%%%%%%%%%%%%%%%%%%%%%%%%%%%%%%%%%%%%
%%% 2. Bilevel optimization approach and technical results
%%%%%%%%%%%%%%%%%%%%%%%%%%%%%%%%%%%%%%%%%%
\beforesec
\section{Bilevel Optimization Approach and Preliminary Results}\label{sec:Ass_and_Grad}
\aftersec
Our approach relies on a bilevel optimization technique \cite{dempe2002foundations} in contrast to Nash's game viewpoint \cite{Luo1996}, which treats the maximization as a lower level  and the minimization as an upper level problem.

%%%%%%%%%%%%%%%%%%%%%%%%%%%%%%%%%%%%%%%%%%%%%%%%%%%
%%% 2.1. Bilevel optimization approach.
%%%%%%%%%%%%%%%%%%%%%%%%%%%%%%%%%%%%%%%%%%%%%%%%%%%
\beforesubsec
\subsection{Bilevel optimization approach}\label{subsec:bilevel_opt_approach}
\aftersubsec
%%% (a) Lower and upper level problems.
The minimax model \eqref{eq:minimax_prob} is split into a \textit{lower-level $($i.e. a follower$)$ maximization problem} of the form:
\begin{equation}\label{eq:lower_level_max}
\begin{array}{lcl}
\Phi_0(w) & := &  {\displaystyle\max_{u\in\R^q}}\big\{ \mathcal{H}(w, u) - h(u) \equiv \frac{1}{n}\sum_{i=1}^n \mathcal{H}_i(w, u) - h(u) \big\}, \\
u_0^{*}(w) & := &  {\argmax_{u \in \R^q}} \big\{ \mathcal{H}(w, u) - h(u) \equiv \frac{1}{n}\sum_{i=1}^n \mathcal{H}_i(w, u) - h(u) \big\}.
\end{array}
\end{equation}
%%%
For $\Phi_0$ defined by \eqref{eq:lower_level_max}, then the \textit{upper-level $($i.e. the leader$)$ minimization problem} can be written as
\begin{equation}\label{eq:upper_level_min}
\Psi_0^{\star} := \min_{w \in \R^p} \Big\{ \Psi_0(w) := \Phi_0(w) + f(w) \Big\}.
\end{equation}
Clearly, this approach is sequential, and only works if $\Phi_0$ is well-defined, i.e. \eqref{eq:lower_level_max} is globally solvable. 
Hence, the concavity of $\Hc(w,\cdot) - h(\cdot)$ w.r.t. to $u$ is crucial for this approach as stated below.
However, this assumption can be relaxed to a global solvability of \eqref{eq:lower_level_max} combined with a P{\L} condition as in \cite{cho2022sgda}.

%%% Assumption A.0.
\begin{assumption}[Basic]\label{as:A0}
Problems \eqref{eq:minimax_prob} and  \eqref{eq:upper_level_min} satisfy the following assumptions for all $i \in [n]$:
\begin{compactitem}
\item[$\mathrm{(a)}$] $\Psi_0^{\star} := \inf_{w}\Psi_0(w) > -\infty$.
\item[$\mathrm{(b)}$] $\Hc_i$ is differentiable w.r.t. $(w, u) \in \dom{\Lc}$ and $\Hc_i(w,\cdot)$ is concave in $u$ for any $w$.
\item[$\mathrm{(c)}$] Both $f : \R^p \to\Rext$ and $h : \R^q \to\Rext$ are proper, closed, and convex.
\end{compactitem}
\end{assumption}
This assumption remains preliminary. 
To develop our algorithms, we will need more conditions on $\Hc_i$ and possibly on $f$ and $h$, which will be stated later.
In addition, we can work with a sublevel set
\begin{equation}\label{eq:Psi0_sublevel_set}
\Lc_{\Psi_0}(w_0) := \set{ w \in \dom{\Psi_0} : \Psi_0(w) \leq \Psi_0(w_0)}
\end{equation}
 of $\Psi_0$ for a given initial point $w_0$ from our methods.
If $u_0^{*}(w)$ is uniquely well-defined for given $w \in \Lc_{\Psi_0}(w_0)$, then by the well-known Danskin's theorem, $\Phi_0$ is differential at $w$ and its gradient is
\begin{equation}\label{eq:exact_grad_phi}
\begin{array}{ll}
\nabla{\Phi}_0(w) = \nabla_{w}\mathcal{H}(w, u_0^{*}(w)) =  \frac{1}{n}\sum_{i=1}^n \nabla_{w}\mathcal{H}_i(w, u_0^{*}(w)).
\end{array}
\end{equation}
We adopt the term ``hyper-gradient'' from bilevel optimization to name $\nabla{\Phi}_0$ in this paper.

%%%%%%%%%%%%%%%%%%%%%%%%%%%%%%%%%%%%%%%%%%%%%%%%%%%
%%% 2.2. Technical assumptions and properties of $\Phi$ for the nonconvex-linear case
%%%%%%%%%%%%%%%%%%%%%%%%%%%%%%%%%%%%%%%%%%%%%%%%%%%
\beforesubsec
\subsection{Technical assumptions and properties of $\Phi_0$ for nonconvex-linear setting (NL)}\label{subsec:nonconvex_linear_case}
\aftersubsec
%%% (a)~Composition reformulations.
\textbf{$\mathrm{(a)}$~\textit{Compositional minimization formulation}.}
If $\mathcal{H}_i(w, u) := \iprods{F_i(w), Ku}$ as in setting (NL), then \eqref{eq:minimax_prob} is equivalently reformulated into the following \textit{nonconvex compositional minimization} problem:
\begin{equation}\label{eq:com_opt0}
\hspace{-1ex}
\min_{w\in\mathbb{R}^p}\Big\{ \Psi_0(w) := f(w) + \Phi_0(w) = f(w) + h^{*}\Big(\tfrac{1}{n}\sum_{i=1}^nK^{\top}F_i(w)\Big) \Big\},
\hspace{-4ex}
\tag{CO}
\end{equation}
where $h^{*}(v) := \sup_{u}\{ \langle v, u\rangle - h(u) \}$, the Fenchel conjugate of $h$, and $\Phi_0(w) = h^{*}(K^{\top}F(w))$.
If $h$ is not strongly convex, then $h^{*}$ is convex but possibly nonsmooth.  

%%% (b)~Technical assumptions.
\textbf{$\mathrm{(b)}$~\textit{Technical assumptions}.}
To develop our algorithms, we also need the following assumptions.
%%% Assumption A.2.
\begin{assumption}\label{as:A2}
$h$ is $\mu_h$-strongly convex with $\mu_h \geq 0$, and $\mathrm{dom}(h)$ is bounded by $M_h < +\infty$.
\end{assumption}
%%% Assumption A.1.
\begin{assumption}[For $F_i$]\label{as:A1}
For setting $\mathrm{(NL)}$ with $\Hc_i(w, u) := \iprods{F_i(w), Ku}$ $(i \in [n])$, assume that
\begin{compactitem}
\item[($\mathrm{a)}$] $F_i$ is continuously differentiable, and its Jacobian $\nabla{F_i}$ is $L_{F_i}$-Lipschitz continuous.  
\item[$\mathrm{(b)}$] $F_i$ is also $M_{F_i}$-Lipschitz continuous or equivalently, its Jacobian $\nabla{F_i}$ is $M_{F_i}$-bounded.  
\item[$\mathrm{(c)}$] There exists a positive constant $\sigma_J \in (0, +\infty)$ such that
\begin{equation}\label{eq:bounded_variance}
\begin{array}{l}
\frac{1}{n}\sum_{i=1}^n\Vert \nabla{F_i}(w) - \nabla{F}(w)\Vert^2 \leq \sigma_J^2, \quad \forall w \in \dom{F}.
\end{array}
\end{equation}
\end{compactitem}
\end{assumption}
Assumption~\ref{as:A2}  allows $\mu_h = 0$ that also covers the non-strong convexity of $h$.
Assumption~\ref{as:A1} is rather standard to develop gradient-based methods for solving \eqref{eq:minimax_prob}. 
Under Assumption~\ref{as:A1}, the finite-sum $F$ is also $M_F$-Lipschitz continuous and the Jacobian $\nabla{F}$ of $F$ is also $L_F$-Lipschitz continuous with 
\begin{equation}\label{eq:MF_and_LF}
M_F := \max\sets{ M_{F_i} : i \in [n] }\quad  \text{and} \quad L_F :=  \max\sets{ L_{F_i} : i \in [n]}.
\end{equation}
Condition~\eqref{eq:bounded_variance} can be relaxed to the form $\frac{1}{n}\sum_{i=1}^n\Vert \nabla{F_i}(w) - \nabla{F}(w)\Vert^2 \leq \sigma_J^2 + \Theta_J\norms{\nabla\Phi_0(w)}^2$ for some $\Theta_J \geq 0$, where $\nabla{\Phi_0}$ is a [sub]gradient of $\Phi_0$ or $\Phi_{\gamma}$ (its smoothed approximation).
Moreover, under Assumption \ref{as:A1}, if $\mu_h > 0$, then $\nabla{h^{*}}$  is $L_{h^{*}}$-Lipschitz continuous with $L_{h^{*}} := \frac{1}{\mu_h}$.
Thus it is possible (see \cite{TranDinh2020f}) to prove that $\Phi_0$ is differentiable, and $\nabla{\Phi_0}$ is also $L_{\Phi_0}$-Lipschitz continuous with $L_{\Phi_0} := M_{h}\norms{K} L_F + \frac{M_F^2\norms{K}^2}{\mu_{h}}$ as a consequence of Lemma~\ref{le:smoothness_of_phi}  when $\gamma\downarrow 0^{+}$ in Supp. Doc.~\ref{apdx:sec:techical_results}.
 
%%% d. Smoothing techniques.
\textbf{$\mathrm{(c)}$~\textit{Smoothing technique for lower-level maximization problem \eqref{eq:lower_level_max}}.}
If $h$ is only merely convex (i.e. $\mu_h = 0$), then \eqref{eq:lower_level_max} may not be uniquely solvable, leading to the possible non-differentiability of $\Phi_0$.
Let us define the following convex  function:
\begin{equation}\label{eq:psi_func}
\phi_0(v) := \max_{u \in \R^q}\left\{ \langle v, Ku\rangle - h(u) \right\} = h^{*}(K^Tv).
\end{equation}
Then, $\Phi_0$ in \eqref{eq:lower_level_max} or \eqref{eq:com_opt0} can be written as $\Phi_0(w) = \phi_0(F(w)) = \phi_0\left( \frac{1}{n}\sum_{i=1}^nF_i(w) \right)$.
Our goal is to smooth $\phi_0$ if $h$ is not strongly convex, leading to 
\begin{equation}\label{eq:smoothed_phi0}
\left\{\begin{array}{lcl}
\phi_{\gamma}(v) & := &  {\displaystyle\max_u}\left\{ \langle v, Ku\rangle - h(u) - \gamma b(u)\right\}, \\
u^{*}_{\gamma}(v) &:= & \argmax_u \left\{ \langle v, Ku\rangle - h(u) - \gamma b(u)\right\},
\end{array}\right.
\end{equation}
where $\gamma > 0$ is a given smoothness parameter and $b : \R^q\to\R$ is a proper, closed, and $1$-strongly convex function such that $\mathrm{dom}(h)\subseteq\mathrm{dom}(b)$.
We also denote $D_b := \sup\sets{\norms{\nabla{b}(u)} : u \in \dom{h}}$.
In particular, if we choose $b(u) := \frac{1}{2}\norms{u - \bar{u}}^2$ for a fixed $\bar{u}$, then $u^{*}_{\gamma}(v) = \prox_{h/\gamma}(\bar{u} - K^Tv)$. 

Using $\phi_{\gamma}$, problem \eqref{eq:com_opt0} can be approximated by its smoothed formulation:
\begin{equation}\label{eq:smoothed_com_opt_prob}
\begin{array}{lcl}
{\displaystyle\min_{w\in\mathbb{R}^p}}\Big\{ \Psi_{\gamma}(w) := f(w) + \Phi_{\gamma}(w) = f(w) + \phi_{\gamma}(F(w)) \equiv f(w) + \phi_{\gamma} \Big(\frac{1}{n}\sum_{i=1}^nF_i(w)\Big) \Big\}.
\end{array}
\end{equation}
To develop our method, one key step is to approximate the hyper-gradient of $\Phi_{\gamma}$ in \eqref{eq:smoothed_com_opt_prob}, where
\begin{equation}\label{eq:grad_phi_gamma}
\begin{array}{lcl}
\nabla{\Phi_{\gamma}}(w) & = & \nabla{F}(w)^{\top}\nabla{\phi}_{\gamma}(F(w)) = \frac{1}{n}\sum_{i=1}^n \nabla{F_i}(w)^{\top}\nabla{\phi}_{\gamma}(F(w)).
\end{array}
\end{equation}
Then, $\nabla{\Phi}_{\gamma}$  is $L_{\Phi_{\gamma}}$-Lipschitz continuous  with $L_{\Phi_{\gamma}} := M_{h}\norms{K} L_F + \frac{M_F^2\norms{K}^2}{\mu_{h} + \gamma}$ (see Lemma~\ref{le:smoothness_of_phi}).

%%%%%%%%%%%%%%%%%%%%%%%%%%%%%%%%%%%%%%%%%%%%%%%%%
%%% 2.3. Technical assumptions and properties of $\Phi$ for the nonconvex-strongly concave case
%%%%%%%%%%%%%%%%%%%%%%%%%%%%%%%%%%%%%%%%%%%%%%%%%
\beforesubsec
\subsection{Technical assumptions and properties of $\Phi_0$ for the nonconvex-strongly-concave setting}\label{subsec:nonconvex_scvx_case}
\aftersubsec
To develop our shuffling gradient-based algorithms for solving \eqref{eq:minimax_prob} under the nonconvex-strongly-concave setting (NC), we impose the following assumptions.
%%%% Assumption A.4.
\begin{assumption}[For $\Hc_i$]\label{as:A3_SGM2}
$\mathcal{H}_i$ for all $i \in [n]$ in \eqref{eq:minimax_prob} satisfies the following conditions:
\begin{compactitem}
\item[$\mathrm{(a)}$] For any given $w$ such that $(w, u) \in \dom{\Hc}$, $\mathcal{H}_i(w, \cdot)$ is $\mu_{H}$-strongly concave w.r.t. $u$.
%%%
\item[$\mathrm{(b)}$] $\nabla{\mathcal{H}}_i$ is $(L_w, L_u)$-Lipschitz continuous, i.e. for all $(w, u), (\hat{w}, \hat{u}) \in \dom{\Hc}$:  
\begin{equation}\label{eq:Lsmooth2}
\hspace{-2ex}
\begin{array}{lcl}
\Vert \nabla{\mathcal{H}}_i(w, u) - \nabla{\mathcal{H}}_i(\hat{w}, \hat{u}) \Vert^2 \leq L_w^2\norms{w - \hat{w}}^2 + L_u^2\norms{u - \hat{u}}^2.  
\end{array}
\hspace{-2ex}
\end{equation}
%%%
\item[$\mathrm{(c)}$] There exist  two constants $\Theta_w \geq 0$ and $\sigma_{w} \geq 0$  such that for $(w, u) \in\dom{\mathcal{H}}$, we have
\begin{equation}\label{eq:bounded_variance2}
\begin{array}{lcl}
\frac{1}{n}\sum_{i=1}^n\Vert \nabla_w{\mathcal{H}}_i(w, u) - \nabla_w{\mathcal{H}}(w, u)\Vert^2 & \leq & \Theta_w \norms{  \nabla_w{\mathcal{H}}(w, u) }^2 + \sigma_{w}^2.
\end{array}
\end{equation}
There exist two constants  $\Theta_u \geq 0$ and  $\sigma_{u} \geq 0$ such that for all $(w, u) \in\dom{\mathcal{H}}$, we have
\begin{equation}\label{eq:bounded_variance2b}
\begin{array}{lcl}
\frac{1}{n}\sum_{i=1}^n\Vert \nabla_u{\mathcal{H}}_i(w, u) - \nabla_u{\mathcal{H}}(w, u)\Vert^2 & \leq &  \Theta_u \norms{ \nabla_u{\mathcal{H}}(w, u)} ^2 + \sigma_{u}^2.
\end{array}
\end{equation}
\end{compactitem}
\end{assumption}
%%%
Assumption~\ref{as:A3_SGM2}(a) makes sure that our lower-level maximization of \eqref{eq:minimax_prob} is well-defined.
Assumption~\ref{as:A3_SGM2}(b) and (c) are standard in shuffling gradient-type methods as often seen in nonconvex optimization \cite{TranDinh2020f}.

%%% Lemma 2.
\begin{lemma}[Smoothness of $\Phi_0$]\label{le:property_of_Phi0}
Under Assumptions~\ref{as:A2} and~\ref{as:A3_SGM2}, $u_0^{*}(\cdot)$ in \eqref{eq:lower_level_max} is $\kappa$-Lipschitz continuous with $\kappa := \frac{L_u}{\mu_{H} + \mu_h}$.
Moreover, $\nabla{\Phi}_0$ in \eqref{eq:exact_grad_phi} is $L_{\Phi_0}$-Lipschitz continuous with $L_{\Phi_0} := (1 + \kappa)L_w$.
\end{lemma}

%%% (b)~KKT points and stationary points.
\beforesubsec
\subsection{Approximate KKT points and approximate stationary points}\label{subsec:KKT_points}
\aftersubsec
\textbf{$\mathrm{(a)}$~\textit{Exact and approximate KKT points and stationary points}.}
A pair $(w^{\star}, u^{\star}) \in \dom{\Lc}$ is called a KKT (Karush-Kuhn-Tucker) point of \eqref{eq:minimax_prob} if
\begin{equation}\label{eq:KKT_point}
0 \in \nabla_w{\Hc}(w^{\star}, u^{\star}) +  \partial{f}(w^{\star}) \quad \text{and} \quad 0 \in  - \nabla_u{\Hc}(w^{\star}, u^{\star}) + \partial{h}(u^{\star}).
\end{equation}
Given a tolerance $\epsilon > 0$, \textbf{our goal} is to find an $\epsilon$-approximate KKT point $(\widehat{w}, \widehat{u})$ of \eqref{eq:minimax_prob} defined as
\begin{equation}\label{eq:app_KKT_point}
r_w \in \nabla_w{\Hc}(\widehat{w}, \widehat{u}) + \partial{f}(\widehat{w}),  \quad r_u \in - \nabla_u{\Hc}(\widehat{w}, \widehat{u}) + \partial{h}(\widehat{u}), \quad \text{and} \quad \Expb{\norms{[r_w, r_u]}^2} \leq \epsilon^2.
\end{equation}
A vector $w^{\star} \in \dom{\Psi_0}$ is said to be a stationary point of \eqref{eq:upper_level_min} if
\begin{equation}\label{eq:stationary_point}
0 \in \nabla{\Phi}_0(w^{\star}) + \partial{f}(w^{\star}).
\end{equation}
Since $f$ is  possibly nonsmooth, we can define a stationary point of  \eqref{eq:upper_level_min} via a gradient mapping as:
\begin{equation}\label{eq:upper_level_min_grad_mapping}
{\Gc}_{\eta}(w) := \eta^{-1}\big(w - \mathrm{prox}_{\eta f}(w - \eta \nabla{\Phi}_0(w) )\big),
\end{equation}
where $\eta > 0$ is given.
It is well-known that ${\Gc}_{\eta}(w^{\star}) = 0$ iff $w^{\star}$ is a stationary point of \eqref{eq:upper_level_min}.
Again, since we cannot exactly compute $w^{\star}$, we expect to find an $\epsilon$-stationary point $\widehat{w}_T$ of \eqref{eq:upper_level_min} such that $\Expb{\norms{\Gc_{\eta}(\widehat{w}_T)}^2 } \leq \epsilon^2$ for a given tolerance $\epsilon > 0$. 

\textbf{$\mathrm{(b)}$~\textit{Constructing an approximate stationary point and KKT point from algorithms}.}
Our algorithms below generate a sequence $\sets{\widetilde{w}_t}_{t\geq 0}^T$ such that $\frac{1}{T+1}\sum_{t=0}^T \Expb{\norms{\Gc_{\eta}(\widetilde{w}_t)}^2 } \leq \epsilon^2$.
Hence, we construct an $\epsilon$-stationary point $\widehat{w}_T$ using one of the following two options:
\begin{equation}\label{eq:approx_stationary_point}
\widehat{w}_T := \widetilde{w}_{t_{*}},  \ \ \text{where} \ 
\left\{\begin{array}{llll}
& t_{*} := \argmin\sets{ \norms{\Gc_{\eta}(\widetilde{w}_t)} : 0 \leq t \leq T},  &\text{(\textit{Option 1})}& \text{or} \\
&\text{$t_{*}$ is uniformly randomly chosen from $\sets{0, 1, \cdots, T}$} & \text{(\textit{Option 2})}.
\end{array}\right.
\end{equation}
Clearly, we have $\Expb{ \norms{\Gc_{\eta}(\widehat{w}_T)}^2 } \leq \frac{1}{T+1}\sum_{t=0}^T \Expb{\norms{\Gc_{\eta}(\widetilde{w}_t)}^2 } \leq \epsilon^2$.
We need the following result.  

%%% Lemma 3.1.
\begin{lemma}\label{le:kkt_point}
$\mathrm{(a)}$~If $(w^{\star}, u^{\star})$ is a KKT point of \eqref{eq:minimax_prob}, then $w^{\star}$ is a stationary point of \eqref{eq:upper_level_min}.
Conversely, if $w^{\star}$ is a stationary point of \eqref{eq:upper_level_min}, then $(w^{\star}, u_0^{*}(w^{\star}))$ is a KKT point of \eqref{eq:minimax_prob}.

$\mathrm{(b)}$~If $\widehat{w}_T$ is an $\epsilon$-stationary point of \eqref{eq:upper_level_min} and $\nabla{\Phi}_0$ is $L_{\Phi_0}$-Lipschitz continuous, then $(\overline{w}_T, \overline{u}_T)$ is an $\hat{\epsilon}$-KKT point of \eqref{eq:minimax_prob}, where $\overline{w}_T := \prox_{\eta f}(\widehat{w}_T - \eta\nabla{\Phi_0}(\widehat{w}_T))$, $\overline{u}_T := u_0^{*}(\overline{w}_T)$, and $\hat{\epsilon} := (1+L_{\Phi_0}\eta)\epsilon$.

$\mathrm{(c)}$~If $\widehat{w}_T$ is an $\epsilon$-stationary point of \eqref{eq:smoothed_com_opt_prob}, then $(\overline{w}_T, \overline{u}_T)$ is an $\hat{\epsilon}$-KKT point of \eqref{eq:minimax_prob}, where $\overline{w}_T := \prox_{\eta f}(\widehat{w}_T - \eta\nabla{\Phi_{\gamma}}(\widehat{w}_T))$, $\overline{u}_T := u_{\gamma}^{*}(F(\overline{w}_T))$, and $\hat{\epsilon} := \max\sets{(1 + L_{\Phi_{\gamma}}\eta)\epsilon, \gamma D_b}$.
\end{lemma}
Lemma~\ref{le:kkt_point} allows us to construct an $\hat{\epsilon}$-approximate KKT point $(\overline{w}_T, \overline{u}_T)$ of \eqref{eq:minimax_prob} from an $\epsilon$-stationary point $\widehat{w}_T$ of either \eqref{eq:upper_level_min} or its smoothed problem \eqref{eq:smoothed_com_opt_prob}, where $\hat{\epsilon} = \Oc(\max\sets{\epsilon,\gamma})$.

\beforesubsec
\subsection{Technical condition to handle the possible nonsmooth term $f$}\label{subsec:grad_op_vs_grad_mapping}
\aftersubsec
To handle the nonsmooth term $f$ of \eqref{eq:minimax_prob} in our algorithms we require one more condition as in  \cite{mishchenko2022proximal}.  

%%% Assumption 3.
\begin{assumption}\label{as:A3}
Let $\Phi_{\gamma}$ be defined by \eqref{eq:smoothed_com_opt_prob}, which reduces to $\Phi_0$ given by \eqref{eq:lower_level_max} as $\gamma\downarrow 0^{+}$, and ${\Gc}_{\eta}$ be defined by \eqref{eq:upper_level_min_grad_mapping}.
Assume that there exist two constants $\Lambda_0 \geq 1$ and $\Lambda_1 \geq 0$ such that:
\begin{equation}\label{eq:grad_mapp_bound}
\norms{\nabla{\Phi}_{\gamma}(w) }^2 \leq \Lambda_0\norms{{\Gc}_{\eta}(w)}^2 + \Lambda_1, \quad \forall w \in \dom{\Phi_0}.
\end{equation}
\end{assumption}

If $f = 0$, then ${\Gc}_{\eta}(w) \equiv \nabla{\Phi}_{\gamma}(w)$, and Assumption~\ref{as:A3} automatically holds with $\Lambda_0 = 1$ and $\Lambda_1 = 0$.
If $f \neq 0$, then it is crucial to have $\Lambda_0 \geq 1$ in \eqref{eq:grad_mapp_bound}.
Let us consider two examples to see why?
\begin{compactitem}
\item[(i)]~ If $f$ is $M_f$-Lipschitz continuous (e.g., $\ell_1$-norm), then \eqref{eq:grad_mapp_bound} also holds with $\Lambda_0 := 1 + \nu > 1$ and $\Lambda_1 :=  \frac{1 + \nu}{\nu}M_f$ for a given $\nu > 0$.
\item[(ii])~If $f = \delta_{\mathcal{W}}$, the indicator of a nonempty, closed, convex, and bounded set $\mathcal{W}$, then  Assumption~\ref{as:A3} also holds by the same reason as in Example (i) (see Supp. Doc. \ref{apdx:sec:techical_results}).
\end{compactitem}

%%%%%%%%%%%%%%%%%%%%%%
%%%%%%%%%%%%%%%%%%%%%%%%%%%%%%%%%%%%%%%%%%
%%% 2. Technical Assumptions and Gradient Estimators.
%%%%%%%%%%%%%%%%%%%%%%%%%%%%%%%%%%%%%%%%%%
\beforesec
\section{Shuffling Gradient Method for Nonconvex-Linear Minimax Problems}\label{sec:SGM1_method}
\aftersec
We first propose a new construction using shuffling techniques to approximate the true gradient $\nabla{\Phi}_{\gamma}$ in \eqref{eq:grad_phi_gamma} for any $\gamma \geq 0$.
Next, we propose our algorithm and analyze its convergence.

\beforesubsec
\subsection{The shuffling gradient estimators for $\nabla{\Phi}_{\gamma}$}\label{subsec:grad_estimator}
\aftersubsec
\textbf{Challenges.} 
To evaluate $\nabla{\Phi_{\gamma}}(w)$ in \eqref{eq:grad_phi_gamma}, we need to evaluate both $\nabla{F}(w)$ and $F(w)$ at each $w$.
However, in SGD or shuffling gradient methods, we want to approximate both quantities at each iteration.
Note that this gradient can be written in a finite-sum $\frac{1}{n}\sum_{i=1}^n\nabla{F_i}(w)^{\top}\nabla{\phi}_{\gamma}(F(w))$ (see \eqref{eq:grad_phi_gamma}), but every summand requires $\nabla{\phi}_{\gamma}(F(w))$, which involves the full evaluation of $F$.

\textbf{Our estimators.}
Let $F_{\pi^{(t)}(i)}(w_{i-1}^{(t)})$ and $\nabla{F}_{\hat{\pi}^{(t)}(i)}(w_{i-1}^{(t)})$ be the function value and  the Jacobian component  evaluated at $w_{i-1}^{(t)}$ respectively for $i \in [n]$, where $\pi^{(t)} = (\pi^{(t)}(1), \pi^{(t)}(2), \cdots, \pi^{(t)}(n))$ and $\hat{\pi}^{(t)} = (\hat{\pi}^{(t)}(1), \hat{\pi}^{(t)}(2), \cdots, \hat{\pi}^{(t)}(n))$ are two permutations of $[n] := \{1, 2, \cdots, n\}$.
We want to use these quantities to approximate the function value $F(w_0^{(t)})$ and its Jacobian $\nabla{F}(w_0^{(t)})$ of $F$ at $w_0^{(t)}$, respectively, where $w_0^{(t)}$ the iterate vector at the beginning of each epoch $t$.

For function value $F(w_0^{(t)})$, we suggest the following approximation at each \textit{inner iteration} $i \in [n]$:
\begin{equation}\label{eq:approxi_Ft}
\begin{array}{l l lcl}
\textbf{Option 1:} & \qquad\quad & F_i^{(t)} & := & \frac{1}{n} \left[ \sum_{j=1}^{i} F_{\pi^{(t)}(j)}(w_{j-1}^{(t)}) + \sum_{j=i+1}^{n}F_{\pi^{(t)}(j)}(w_0^{(t)}) \right].
\end{array}
\end{equation}
Alternative to \eqref{eq:approxi_Ft}, for all $i \in [n]$, we can simply choose another option:
\begin{equation}\label{eq:approxi_Ftb}
\begin{array}{l l lcl}
\textbf{Option 2:} & \qquad\quad &  F_i^{(t)} & := & \frac{1}{n}\sum_{j=1}^nF_j (w_0^{(t)}) = \frac{1}{n}\sum_{j=1}^nF_{\pi^{(t)}(j)} (w_0^{(t)}). \quad\qquad \ \
\end{array}
\end{equation}
For Jacobian $\nabla{F}(w^{(t)}_0)$, we suggest to use the following standard shuffling estimator for all $i \in [n]$:
\begin{equation}\label{eq:approxi_Jt}
\nabla{F}_i^{(t)} := \nabla{F_{\hat{\pi}^{(t)}(i)}}(w_{i-1}^{(t)}).  
\end{equation}
For $F_i^{(t)}$  from \eqref{eq:approxi_Ft} (or \eqref{eq:approxi_Ftb}) and for $\nabla{F}_i^{(t)}$ from \eqref{eq:approxi_Jt},  we form an approximation of $\nabla{\Phi_{\gamma}}(w_0^{(t)})$ as
\begin{equation}\label{eq:approx_grad_Phi}
\widetilde{\nabla}{\Phi_{\gamma}}(w_{i-1}^{(t)}) := (\nabla{F}_i^{(t)})^{\top}\nabla{\phi}_{\gamma}(F_i^{(t)}) \equiv (\nabla{F}_i^{(t)})^{\top}Ku^{*}_{\gamma}(F_i^{(t)}). 
\end{equation}
%%%
\textbf{Discussion.}
The estimator $F_i^{(t)}$ for $F$ requires $n-i$ more function evaluations $F_{\pi^{(t)}(j)}(w_0^{(t)})$ at each epoch $t$. 
The first option \eqref{eq:approxi_Ft} for $F$ uses $2n$ function evaluations $F_i$, while the second one in \eqref{eq:approxi_Ftb} only needs $n$ function evaluations at each epoch $t \geq 0$.
However, \eqref{eq:approxi_Ft}  uses the most updated information up to the \textit{inner iteration} $i$ compared to \eqref{eq:approxi_Ftb}, which is expected to perform better. 
The Jacobian estimator $\nabla{F}_i^{(t)}$ is standard and only uses one sample or a mini-batch at each iteration $i$.

\beforesubsec
\subsection{The shuffling gradient-type algorithm for nonconvex-linear setting (NL)}\label{subsec:alg}
\aftersubsec
We propose Algorithm~\ref{alg:SGM1}, a shuffling gradient-type method, to approximate a stationary point of \eqref{eq:smoothed_com_opt_prob}.

\begin{algorithm}[hpt!]
\caption{(Shuffling Proximal Gradient-Based Algorithm for Solving \eqref{eq:smoothed_com_opt_prob})}\label{alg:SGM1}
\begin{algorithmic}[1]
   \STATE {\bfseries Initialization:} Choose an initial point $\widetilde{w}_0\in\dom{\Phi_0}$ and a smoothness parameter $\gamma > 0$.
   \FOR{$t=1,2,\cdots,T $}
   \STATE Set $w_0^{(t)} := \widetilde{w}_{t-1}$;
   \STATE Generate two permutations  $\pi^{(t)}$ and $\hat{\pi}^{(t)}$ of $[n]$ (identically or randomly and independently)
   \FOR{$i = 1,\cdots, n$}
     \STATE Evaluate $F_i^{(t)}$ by  either  \eqref{eq:approxi_Ft} or  \eqref{eq:approxi_Ftb} using $\pi^{(t)}$, and $\nabla{F}_i^{(t)}$ by \eqref{eq:approxi_Jt} using $\hat{\pi}^{(t)}$.
     \STATE Solve \eqref{eq:smoothed_phi0} to get $u^{*}_{\gamma}(F_i^{(t)})$ and form $\widetilde{\nabla}{\Phi_{\gamma}}(w^{(t)}_{i-1} ) := (\nabla{F}_i^{(t)})^{\top}Ku^{*}_{\gamma}(F_i^{(t)})$.
    \STATE Update $w_{i}^{(t)} := w_{i-1}^{(t)} - \frac{\eta_t}{n} \widetilde{\nabla}{\Phi_{\gamma}}(w^{(t)}_{i-1})$; 
   \ENDFOR
   \STATE Compute $\widetilde{w}_t := \mathrm{prox}_{\eta_t f }(w_{n}^{(t)})$;
   \ENDFOR
\end{algorithmic}
\end{algorithm} 

\textbf{Discussion.}
First, the  cost per epoch of Algorithm~\ref{alg:SGM1} consists of either $2n$ or $n$ function evaluations $F_i$, and $n$ Jacobian evaluations $\nabla{F_i}$.
Compare to standard shuffling gradient-type methods, e.g., in \cite{nguyen2020unified}, Algorithm~\ref{alg:SGM1} has either  $n$ more evaluations of $F_i$ or the same cost.
Second, when implementing Algorithm~\ref{alg:SGM1}, we do not need to evaluate the full Jacobian $\nabla{F}_i^{(t)}$, but rather the product of  matrix $(\nabla{F}_i^{(t)})^{\top}$ and vector $\nabla{\Phi}_{\gamma}(F_i^{(t)})$ as $\widetilde{\nabla}{\Phi_{\gamma}}(w^{(t)}_{i-1} ) := (\nabla{F}_i^{(t)})^{\top}\nabla{\Phi}_{\gamma}(F_i^{(t)})$.
Evaluating this matrix-vector multiplication is much more efficient than evaluating the full Jacobian $\nabla{F}_i^{(t)}$ and $\nabla{\Phi}_{\gamma}(F_i^{(t)})$ individually. 
Third, thanks to Assumption~\ref{as:A3}, the proximal step $\widetilde{w}_t := \mathrm{prox}_{\eta_t f }(w_{n}^{(t)})$ is only required at the end of each epoch $t$.
This significantly reduces the computational cost if $\prox_{\eta_t f}$ is expensive.

%%%%%%%%%%%%%%%%%%%%%%%%%%%%%%%%%%%%%
%%% 3.3. Convergence Analysis
%%%%%%%%%%%%%%%%%%%%%%%%%%%%%%%%%%%%%
\beforesubsec
\subsection{Convergence Analysis of Algorithm~\ref{alg:SGM1} for Nonconvex-Linear Setting (NL)}\label{sec_analysis}
\aftersubsec
Now, we are ready to state the convergence result of Algorithm~\ref{alg:SGM1} in a short version: Theorem~\ref{th:main_result1_short}.
The full version of this theorem is Theorem~\ref{th:main_result1}, which can be found in Supp. Doc.~\ref{apdx:sec:SGM1_convergence}.

%%% Theorem 3.1
\begin{theorem}\label{th:main_result1_short}
Suppose that Assumptions~\ref{as:A0}, \ref{as:A2}, \ref{as:A1}, and \ref{as:A3} holds for the setting $\mathrm{(NL)}$ of \eqref{eq:minimax_prob} and $\epsilon > 0$ is a sufficiently small tolerance. 
Let $\sets{\widetilde{w}_t }$ be generated by Algorithm~\ref{alg:SGM1} after $T = \BigOs{\epsilon^{-3}}$ epochs using arbitrarily permutations $\pi^{(t)}$ and $\hat{\pi}^{(t)}$ and a learning rate $\eta_t = \eta := \BigOs{\epsilon}$ $($see Theorem~\ref{th:main_result1} in Supp. Doc.~\ref{apdx:sec:SGM1_convergence} for the exact formulas of $T$ and $\eta$$)$.
Then, we have $\frac{1}{T+1}\sum_{t=0}^{T}\norms{{\Gc}_{\eta_t}(\widetilde{w}_{t})}^2 \leq \epsilon^2$.

Alternatively, if  $\sets{\widetilde{w}_t }$ is generated by Algorithm~\ref{alg:SGM1} after $T := \BigOs{n^{-1/2}\epsilon^{-3}}$ epochs using two random and independent permutations $\pi^{(t)}$ and $\hat{\pi}^{(t)}$ and a learning rate $\eta_t = \eta := \BigOs{n^{1/2}\epsilon}$ $($see Theorem~\ref{th:main_result1} in Supp. Doc.~\ref{apdx:sec:SGM1_convergence} for the exact formulas$)$.
Then, we have $\frac{1}{T+1}\sum_{t=0}^{T} \Eb[ \norms{{\Gc}_{\eta_t}(\widetilde{w}_{t})}^2 ] \leq \epsilon^2$.
\end{theorem}

%%% Corollary 3.1.
Our first goal is to approximate a stationary point $w^{\star}$ of \eqref{eq:com_opt0} as $\Eb[ \norms{{\Gc}_{\eta}(\widehat{w}) }^2] \leq \epsilon^2$, while Algorithm~\ref{alg:SGM1} only provides an $\epsilon$-stationary of  \eqref{eq:smoothed_com_opt_prob}.
For a proper choice of $\gamma$, it is also an $\epsilon$-stationary point of \eqref{eq:upper_level_min}.

%%% Lemma 3.1.
\begin{corollary}\label{co:app_KKT_for_nonconvex_linear_case}
Let $\widehat{w}_T$ defined by \eqref{eq:approx_stationary_point} be generated from $\sets{\widetilde{w}_t}$ of Algorithm~\ref{alg:SGM1}.
Under the conditions of Theorem~\ref{th:main_result1_short} and any permutations $\pi^{(t)}$ and $\hat{\pi}^{(t)}$, the following statements hold.
\vspace{-0.5ex}
\begin{compactitem}
\item[$\mathrm{(a)}$] If $h$ is $\mu_h$-strongly convex with $\mu_h > 0$, then we can set $\gamma = 0$, and Algorithm~\ref{alg:SGM1} requires $\Oc(n\epsilon^{-3})$ evaluations of $F_i$ and $\nabla{F_i}$ to achieve an $\epsilon$-stationary $\widehat{w}_T$ of \eqref{eq:upper_level_min}.
\item[$\mathrm{(b)}$] If $h$ is only convex $($i.e. $\mu_h = 0$$)$, then we can set $\gamma := \Oc(\epsilon)$, and Algorithm~\ref{alg:SGM1} needs $\Oc(n\epsilon^{-7/2})$ evaluations of $F_i$ and $\nabla{F_i}$ to achieve an $\epsilon$-stationary $\widehat{w}_T$ of \eqref{eq:upper_level_min}.
\vspace{-0.5ex}
\end{compactitem}
If, in addition, $\pi^{(t)}$ and $\hat{\pi}^{(t)}$ are sampled uniformly at random without replacement and independently, and $\Lambda_1 = \Oc(n^{-1})$, then the numbers of  evaluations of $F_i$ and $\nabla{F_i}$ are reduced by a factor of $\sqrt{n}$.
\end{corollary}

%%%%%%%%%%%%%%%%%%%%%%%%%%%%%%%%%%%%%%%%%%%%%%%%%%%%%%%
%%% End of Section 3.
%%%%%%%%%%%%%%%%%%%%%%%%%%%%%%%%%%%%%%%%%%%%%%%%%%%%%%%
%%%%%%%%%%%%%%%%%%%%%%%%%%%%%%%%%%%%%%%%%%%%%%%%%
%%%% 5. Extension to the nonconvex-general strongly concave minimax problems.
%%%%%%%%%%%%%%%%%%%%%%%%%%%%%%%%%%%%%%%%%%%%%%%%%
\beforesec
\section{Shuffling Method for Nonconvex-Strongly Concave Minimax Problems}\label{sec:SGM2_method}
\aftersec
In this section, we develop shuffling gradient-based methods to solve \eqref{eq:minimax_prob} under the nonconvex-strongly concave setting (NC).
Since this setting does not cover the nonconvex-linear setting (NL) in Section~\ref{sec:SGM1_method} as a special case,  we need to treat it separately using different ideas and proof techniques.

%%%%%%%%%%%%%%%%%%%%%%%%%%%%%%%%%%%%%%%%%%%%%%%%%
%%% 4.2. The construction of algorithm
%%%%%%%%%%%%%%%%%%%%%%%%%%%%%%%%%%%%%%%%%%%%%%%%%
\beforesubsec
\subsection{The construction of algorithm}\label{subsec:SGM2_construction}
\aftersubsec
Unlike the linear case with $\mathcal{H}_i(w, u) = \iprods{F_i(w), Ku}$ in Section~\ref{sec:SGM1_method}, we cannot generally compute the solution $u_0^{*}(\widetilde{w}_{t-1})$ in \eqref{eq:lower_level_max} exactly for a given $\widetilde{w}_{t-1}$.
We can only approximate $u_0^{*}(\widetilde{w}_{t-1})$ by some $\widetilde{u}_t$.
This leads to another level of inexactness in an approximate ``hyper-gradient'' $\widetilde{\nabla}{\Phi_0}(w_{i-1}^{(t)})$ defined by
\begin{equation}\label{eq:approx_grad_phi}
\widetilde{\nabla}{\Phi}_0(w_{i-1}^{(t)}) := \nabla_w\mathcal{H}_{\hat{\pi}^{(t)}(i)}(w^{(t)}_{i-1}, \widetilde{u}_t). 
\end{equation}
There are different options to approximate $u_0^{*}(\widetilde{w}_{t-1})$.
We propose two options below, but other choices are possible, including accelerated gradient ascent methods and stochastic algorithms \cite{Nesterov2004,lan2020first}.

\textbf{$\mathrm{(a_1)}$~Gradient ascent scheme for the lower-level problem.}
We apply a standard gradient ascent scheme to update $\widetilde{u}_t$:
\textit{Starting from $s=0$ with $u_0^{(t)} := \widetilde{u}_{t-1}$, at each epoch $s=1,\cdots, S$, we update
\begin{equation}\label{eq:GD4lower_prob}
\begin{array}{lcl}
\widehat{u}^{(t)}_{s} &:= & \prox_{\hat{\eta}_t h}\big(  \widehat{u}^{(t)}_{s-1} + \frac{\hat{\eta}_t}{n} \sum_{i=1}^n \nabla_u\mathcal{H}_i(\widetilde{w}_{t-1}, \widehat{u}^{(t)}_{s-1}) \big),
\end{array} 
\end{equation}
for a given learning rate $\hat{\eta}_t > 0$.
Then, we finally output $\widetilde{u}_t := \widehat{u}_S^{(t)}$ to approximate $u_0^{*}(\widetilde{w}_{t-1})$.}

To make our method more flexible, we allow to perform either only \textit{one iteration} (i.e. $S=1$) or \textit{multiple iterations} (i.e. $S>1$) of \eqref{eq:GD4lower_prob}.
Each iteration $s$ requires $n$ evaluations of $\nabla_u\mathcal{H}_i$.

\textbf{$\mathrm{(a_2)}$~Shuffling gradient ascent scheme for the lower-level problem.}
We can also construct $\widetilde{u}_t$ by a \textit{shuffling gradient ascent scheme}.
Again, we allow to run either only \textit{one epoch} (i.e. $S=1$) or \textit{multiple epochs} (i.e. $S > 1$) of the shuffling algorithm to update $\widetilde{u}_t$, leading to the following scheme:
\textit{Starting from $s := 1$ with $\widehat{u}_{0}^{(t)} := \widetilde{u}_{t-1}$, at each epoch $s=1,2,\cdots, S$, having $\widehat{u}_{s-1}^{(t)}$, we generate a permutation $\pi^{(s,t)}$ of $[n]$ and run a shuffling gradient ascent scheme as 
\begin{equation}\label{eq:SGM4lower_prob}
\left\{\begin{array}{lll}
&u_0^{(s,t)} := \widehat{u}_{s-1}^{(t)}, \\ 
& \text{For $i=1,2, \cdots, n$, update}\\
&\qquad u_{i}^{(s,t)} :=  u_{i-1}^{(s,t)} + \frac{\hat{\eta}_t}{n}  \nabla_u\mathcal{H}_{\pi^{(s,t)}(i)}(\widetilde{w}_{t-1}, u_{i-1}^{(s,t)}), \\
& \widehat{u}_s^{(t)} :=   \prox_{\hat{\eta}_t h}(u_n^{(s,t)}).
\end{array}\right.
\end{equation}
At the end of the $S$-th epoch, we output $\widetilde{u}_t := \widehat{u}_{S}^{(t)}$ as an approximation to $u_0^{*}(\widetilde{w}_{t-1})$.
}
Here, we use the same learning rate $\hat{\eta}_t$ for all epochs $s \in [S]$.
Each epoch $s$ requires $n$ evaluations of $\nabla_u\mathcal{H}_i$.

\textbf{$\mathrm{(b)}$~Shuffling gradient descent scheme for the upper-level minimization problem.}
Having $\widetilde{u}_t$ from either \eqref{eq:GD4lower_prob} or \eqref{eq:SGM4lower_prob}, we run a \textit{shuffling gradient descent epoch} to update  $\widetilde{w}_t$ from $\widetilde{w}_{t-1}$ as
\begin{equation}\label{eq:SGM4upper_prob}
\left\{\begin{array}{lll}
&w_0^{(t)} := \widetilde{w}_{t-1}, \\
&\textrm{For $i=1,2,\cdots, n$, update}\\
& \qquad w_i^{(t)}   :=    w_{i-1}^{(t)} - \frac{\eta_t}{n}\widetilde{\nabla}{\Phi_0}(w_{i-1}^{(t)}) \equiv w_{i-1}^{(t)} - \frac{\eta_t}{n}\nabla_w\mathcal{H}_{\hat{\pi}^{(t)}(i)}(w^{(t)}_{i-1}, \widetilde{u}_t), \\
& \widetilde{w}_t :=  \prox_{\eta_tf}(w_n^{(t)}).
\end{array}\right.
\end{equation}
These two steps \eqref{eq:GD4lower_prob} (or \eqref{eq:SGM4lower_prob}) in $u$ and \eqref{eq:SGM4upper_prob} in $w$ are implemented alternatively for $t=1, \cdots, T$.

\textbf{$\mathrm{(c)}$~The full algorithm.} 
Combining both steps \eqref{eq:GD4lower_prob}  (or \eqref{eq:SGM4lower_prob}) and \eqref{eq:SGM4upper_prob}, we can present an \textit{alternating shuffling proximal gradient algorithm} for solving \eqref{eq:minimax_prob} as in Algorithm~\ref{alg:SGM2}.

%%% Algorithm 2.
\begin{algorithm}[hpt!]
\caption{(Alternating Shuffling Proximal Gradient Algorithm for Solving \eqref{eq:minimax_prob} under setting (NC))}\label{alg:SGM2}
\begin{algorithmic}[1]
   \STATE {\bfseries Initialization:} Choose an initial point $(\widetilde{w}_0, \widetilde{u}_0) \in\dom{\Lc}$.
   \FOR{$t=1,2,\cdots,T $}
   \STATE\label{step:SGM2_step3} Compute $\widetilde{u}_t$ using either \eqref{eq:GD4lower_prob}  or \eqref{eq:SGM4lower_prob}. 
   \STATE\label{step:SGM2_step3} Set $w_0^{(t)} := \widetilde{w}_{t-1}$ and generate a permutation  $\hat{\pi}^{(t)}$ of $[n]$.
   \FOR{$i = 1,\cdots, n$}
     \STATE\label{step:SGM2_step12} Evaluate $\widetilde{\nabla}{\Phi_0}(w_{i-1}^{(t)}) := \nabla_w{\mathcal{H}}_{\hat{\pi}^{(t)}(i)} (w_{i-1}^{(t)}, \widetilde{u}_t )$. 
    \STATE\label{step:SGM2_step13} Update $w_{i}^{(t)} := w_{i-1}^{(t)} - \frac{\eta_t}{n} \widetilde{\nabla}{\Phi_0}(w_{i-1}^{(t)})$. 
   \ENDFOR
   \STATE\label{step:SGM2_step15} Compute $\widetilde{w}_t := \mathrm{prox}_{\eta_t f }(w_{n}^{(t)})$.
   \ENDFOR
\end{algorithmic}
\end{algorithm} 
%%% Discussion.
\textbf{Discussion.}
Algorithm~\ref{alg:SGM2} has a similar form as Algorithm~\ref{alg:SGM1}, where $u_0^{*}(\widetilde{w}_{t-1})$ is approximated by $\widetilde{u}_t$.
In Algorithm~\ref{alg:SGM1}, $u_0^{*}(\widetilde{w}_{t-1})$ is approximated by $u^{*}_{\gamma}(F_i^{(t)})$. 
Moreover, Algorithm~\ref{alg:SGM1} solves the smoothed problem \eqref{eq:smoothed_com_opt_prob} of \eqref{eq:upper_level_min}, while Algorithm~\ref{alg:SGM2} directly solves \eqref{eq:upper_level_min}.
Depending on the choice of method to approximate $u_0^{*}(\widetilde{w}_{t-1})$, we obtain different variants of Algorithm~\ref{alg:SGM2}.
We have proposed two variants:
\begin{compactitem}
\item \textbf{Semi-shuffling variant:} We use \eqref{eq:GD4lower_prob} for computing $\widetilde{u}_t$ to approximate  $u_0^{*}(\widetilde{w}_{t-1})$.
\item \textbf{Full-shuffling variant:} We use \eqref{eq:SGM4lower_prob} for computing $\widetilde{u}_t$ to approximate  $u_0^{*}(\widetilde{w}_{t-1})$.
\end{compactitem}
Note that Algorithm~\ref{alg:SGM2} works in an alternative manner, where it approximates $u_0^{*}(\widetilde{w}_{t-1})$ up to a certain accuracy before updating $\widetilde{w}_t$.
This alternating update is very natural and has been widely applied to solve minimax optimization as well as bilevel optimization problems, see, e.g., \cite{arjovsky2017wasserstein,dempe2002foundations,goodfellow2014generative}.

%%%%%%%%%%%%%%%%%%%%%%%%%%%%%%%%%%%%%%%%%%%%%%%%%
%%% 4.2. Convergence analysis
%%%%%%%%%%%%%%%%%%%%%%%%%%%%%%%%%%%%%%%%%%%%%%%%%
\beforesubsec
\subsection{Convergence analysis}\label{subsec:SGM2_convergence}
\aftersubsec
Now, we state the convergence of both variants of Algorithm~\ref{alg:SGM2}: \textit{semi-shuffling} and \textit{full-shuffling} variants.
The full proof of the following theorems can be found in Supp. Doc.~\ref{apdx:sec:SGM2_convergence}.

\textbf{$\mathrm{(a)}$~\textit{Convergence of the semi-shuffling variant}.} 
Our first result is as follows.

%%% Theorem 4.1.
\begin{theorem}\label{th:SGM2_semi_shuffling_variant_convergence}
Suppose that Assumptions~\ref{as:A0}, \ref{as:A2},~\ref{as:A3_SGM2}, and ~\ref{as:A3} hold for \eqref{eq:minimax_prob}, 
and ${\Gc}_{\eta}$ is defined by \eqref{eq:upper_level_min_grad_mapping}.

Let $\sets{(\widetilde{w}_t, \widetilde{u}_t)}$ be generated by Algorithm~\ref{alg:SGM2} using  the \textbf{gradient ascent scheme \eqref{eq:GD4lower_prob}} with  $\eta := \mathcal{O}(\epsilon)$ explicitly given in Theorem~\ref{th:SGM2_full_shuffling_main_result1_supp} of Supp. Doc.~\ref{apdx:sec:SGM2_convergence}, $\hat{\eta} \in (0, \frac{2}{L_u+\mu_h}]$, $S := \Oc\big( \frac{1}{\hat{\eta}} \big( \mu_h + \frac{4L_u\mu_H}{L_u + \mu_H} \big)^{-1} \big) = \Oc(1)$, and $T := \mathcal{O}(\epsilon^{-3})$ explicitly given in Theorem~\ref{th:SGM2_full_shuffling_main_result1_supp}.
Then, we have $\frac{1}{T+1}\sum_{t=0}^T \norms{ {\Gc}_{\eta}(\widetilde{w}_{t}) }^2 \leq \epsilon^2$.

Consequently, Algorithm~\ref{alg:SGM2} requires $\mathcal{O}( n\epsilon^{-3})$ evaluations of both $\nabla_w{\Hc}_i$  and  $\nabla_u{\Hc}_i$ to achieve an $\epsilon$-stationary point $\widehat{w}_T$ of \eqref{eq:upper_level_min} computed by \eqref{eq:approx_stationary_point}.
\end{theorem}
Note that Theorem~\ref{th:SGM2_semi_shuffling_variant_convergence} holds for both $S > 1$ and $S = 1$ (i.e. we perform only one iteration of \eqref{eq:GD4lower_prob}).

%%%% (b) - Full-Shuffling Method - Multiple Epochs.
\textbf{$\mathrm{(b)}$~\textit{Convergence of the full-shuffling variant -- The case $S > 1$ with multiple epochs}.}
We state our results for two separated cases: only $\Hc_i$ is $\mu_H$-strongly convex, and only $h$ is $\mu_h$-strongly convex.

%%% Theorem 4.1.
\begin{theorem}[Strong convexity of $\Hc_i$]\label{th:SGM2_full_shuffling_main_result1}
Suppose that Assumptions~\ref{as:A0}, \ref{as:A2},~\ref{as:A3_SGM2}, and ~\ref{as:A3} hold, and $\Hc_i$ is $\mu_H$-strongly concave with $\mu_H > 0$ for $i\in [n]$, but $h$ is only merely convex.

Let $\sets{(\widetilde{w}_t, \widetilde{u}_t)}$ be generated by Algorithm~\ref{alg:SGM2} using  $S$ epochs of the \textbf{shuffling routine \eqref{eq:SGM4lower_prob}} and fixed learning rates $\eta_t = \eta := \mathcal{O}(\epsilon)$ as given in Theorem~\ref{th:SGM2_full_shuffling_main_result1_supp} of Supp. Doc.~\ref{apdx:sec:SGM2_convergence} for a given $\epsilon > 0$, $\hat{\eta}_t := \hat{\eta} = \Oc(\epsilon)$, $S  := \big\lfloor \frac{\ln(7/2)}{\mu_H \hat{\eta} } \big\rfloor$, and $T := \mathcal{O}( \epsilon^{-3})$.
Then, we have  $\frac{1}{T+1}\sum_{t=0}^T \norms{ {\Gc}_{\eta}(\widetilde{w}_{t}) }^2 \leq \epsilon^2$.

Consequently, Algorithm~\ref{alg:SGM2} requires $\mathcal{O}( n \epsilon^{-3})$ evaluations of $\nabla_w{\Hc_i}$ and $\mathcal{O}( n\epsilon^{-4})$ evaluations of $\nabla_u{\Hc}_i$  to achieve an $\epsilon$-stationary point $\widehat{w}_T$ of \eqref{eq:upper_level_min} computed by \eqref{eq:approx_stationary_point}.
\end{theorem}

%%% Theorem 4.2.
\begin{theorem}[Strong convexity of $h$]\label{th:SGM2_full_shuffling_main_result2}
Suppose that Assumptions~\ref{as:A0}, \ref{as:A2}, \ref{as:A3_SGM2}, and ~\ref{as:A3} hold for \eqref{eq:minimax_prob}, and $h$ is $\mu_h$-strongly convex with $\mu_h > 0$, but $\Hc_i$ is only merely concave for all $i\in [n]$.
Then, under the same settings as in Theorem~\ref{th:SGM2_full_shuffling_main_result1}, but with  $S := \big\lfloor\frac{\ln(7/2)}{\mu_h\hat{\eta}}\big\rfloor$, the conclusions of Theorem~\ref{th:SGM2_full_shuffling_main_result1} still hold.
\end{theorem}

%%%% (c) - Full-Shuffling Method - One Epoch.
\textbf{$\mathrm{(c)}$~\textit{Convergence of the full-shuffling variant  -- The case $S=1$ with one epoch}.}
Both Theorems~\ref{th:SGM2_full_shuffling_main_result1} and \ref{th:SGM2_full_shuffling_main_result2} require $\Oc(n\epsilon^{-4})$ evaluations of $\nabla_u{\Hc}_i$.
To improve this complexity, we need two additional assumptions but can perform only one epoch of \eqref{eq:SGM4lower_prob}, i.e. $S=1$.

%%% Assumption 3.
\begin{assumption}\label{as:A4}
Let $\hat{\Gc}_{\eta}(u) := \eta^{-1}(u - \prox_{\eta h}(u + \eta\nabla_u{\Hc}(w, u)  ) )$ be the gradient mapping of $\psi(w, \cdot) := - \Hc(w, \cdot) + h(\cdot)$.
Assume that there exist   $\hat{\Lambda}_0 \geq 1$ and $\hat{\Lambda}_1 \geq 0$ such that
\begin{equation}\label{eq:grad_mapp_bound2}
\norms{\nabla_u{\Hc}(w, u) }^2 \leq \hat{\Lambda}_0\norms{ \hat{\Gc}_{\eta}(u)}^2 + \hat{\Lambda}_1, \quad \forall (w, u) \in \dom{\Lc}.
\vspace{-0.5ex}
\end{equation}
\end{assumption}
Clearly, if $h = 0$, then $\hat{\Gc}_{\eta}(u) = -\nabla_u{\Hc}(w, u)$ and \eqref{eq:grad_mapp_bound} automatically holds for $\hat{\Lambda}_0 = 1$ and $\hat{\Lambda}_1 = 0$.
Assumption~\ref{as:A4} is similar to Assumption~\ref{as:A3}, and it is required to handle the $\prox$ operator of $h$ in \eqref{eq:SGM4lower_prob}.

%%% Assumption 4.
\begin{assumption}\label{as:A5}
For $f$ in \eqref{eq:minimax_prob}, there exists $L_f \geq 0$ such that 
\begin{equation}\label{eq:SGM2_f_smoothness}
f(y) \leq f(x) + \iprods{f'(x), y - x} + \tfrac{L_f}{2}\norms{y - x}^2, \quad \forall x, y \in \dom{f}, \ f'(x) \in \partial{f}(x).
\end{equation}
\end{assumption}
Clearly, if $f$ is $L_f$-smooth, then \eqref{eq:SGM2_f_smoothness} holds.
If $f$ is also convex, then \eqref{eq:SGM2_f_smoothness} implies that $f$ is $L_f$-smooth.

Under these additional assumptions, we have the following result.

%%% Theorem 4.1.
\begin{theorem}\label{th:SGM2_full_shuffling_main_result3}
Suppose that Assumptions~\ref{as:A0}, \ref{as:A2},~\ref{as:A3_SGM2}, ~\ref{as:A3},~ \ref{as:A4}, and \ref{as:A5} hold and ${\Gc}_{\eta}$ is defined by  \eqref{eq:upper_level_min_grad_mapping}.

Let $\sets{(\widetilde{w}_t, \widetilde{u}_t)}$ be generated by Algorithm~\ref{alg:SGM2} using \textbf{one epoch} $(S=1)$ of the \textbf{shuffling routine \eqref{eq:SGM4lower_prob}}, and fixed learning rates $\eta_t = \eta := \mathcal{O}(\epsilon)$ as in Theorem~\ref{th:SGM3_full_shuffling_convergence1} of Supp. Doc.~\ref{apdx:sec:SGM2_convergence} for a given $\epsilon > 0$, $\hat{\eta}_t := \hat{\eta} = 30\kappa^2\eta$, and $T := \mathcal{O}( \epsilon^{-3})$, where $\kappa := \frac{L_u}{\mu_H+\mu_h}$.
Then, we have $\frac{1}{T+1}\sum_{t=0}^T \norms{ {\Gc}_{\eta}(\widetilde{w}_{t}) }^2 \leq \epsilon^2$.

Consequently, Algorithm~\ref{alg:SGM2} requires $\mathcal{O}( n \epsilon^{-3})$ evaluations of both $\nabla_w{\Hc_i}$ and of $\nabla_u{\Hc}_i$  to achieve an $\epsilon$-stationary point $\widehat{w}_T$ of \eqref{eq:upper_level_min} computed by \eqref{eq:approx_stationary_point}.
\end{theorem}

Similar to Algorithm~\ref{alg:SGM1}, if $\pi^{(s,t)}$ and $\hat{\pi}^{(t)}$ are generated randomly and independently,  $\Lambda_1 = \Oc(1/n)$, and $\hat{\Lambda}_1 = \Oc(1/n)$, then our complexity stated above can be improved by a factor of $\sqrt{n}$.
Nevertheless, we omit this analysis.
Finally, we can combine each Theorem~\ref{th:SGM2_semi_shuffling_variant_convergence}, \ref{th:SGM2_full_shuffling_main_result1}, \ref{th:SGM2_full_shuffling_main_result2} or \ref{th:SGM2_full_shuffling_main_result3} and Lemma~\ref{le:kkt_point} to construct an $\hat{\epsilon}$-KKT point of \eqref{eq:minimax_prob}.
Theorem~\ref{th:SGM2_full_shuffling_main_result3} has a better complexity than Theorems~\ref{th:SGM2_full_shuffling_main_result1} and \ref{th:SGM2_full_shuffling_main_result2}, but requires stronger assumptions.
Algorithm~\ref{alg:SGM2} is also  different from the one in \cite{cho2022sgda} both in terms of algorithmic form and the underlying problem to be solved, while achieving the same oracle complexity.

%%%%%%%%%%%%%%%%%%%%%%%%%%%%%%%%%%%%%%%%%%%%%%%%%%%%%%%%%%%%%%%%%%%%%%
%%% End of Section 3.
%%%%%%%%%%%%%%%%%%%%%%%%%%%%%%%%%%%%%%%%%%%%%%%%%%%%%%%%%%%%%%%%%%%%%%

%%%%%%%%%%%%%%%%%%%%%%%%%%%%%%%%%%%%%%%%%%%%%%%%%%%%%%%
%%% 5. Numerical Experiments.
%%%%%%%%%%%%%%%%%%%%%%%%%%%%%%%%%%%%%%%%%%%%%%%%%%%%%%%
\beforesec
\section{Numerical Experiments}\label{sec:num_experiments}
\aftersec
We perform some experiments to illustrate Algorithm~\ref{alg:SGM1} and compare it with two existing and related algorithms.
Further details  and additional experiments can be found in Supp. Doc. \ref{apdx:add_experiments}.

We consider the following regularized stochastic minimax problem studied, e.g., in \cite{TranDinh2020f,shapiro2002minimax}:
\begin{equation}\label{eq:min_max_stochastic_opt}
\begin{array}{ll}
{\displaystyle\min_{w\in \R^p} \Big\{ \max_{1\leq j\leq m}} \big\{ \frac{1}{n}\sum_{i=1}^n F_{i,j}(w) \big\} + \tfrac{\lambda}{2}\norms{w}^2\Big\},
\end{array}
\end{equation}
where $ F_{i,j} : \R^p\times\Omega\to\R_{+}$ can be viewed as the loss of the $j$-th model for data point $i \in [n]$.
If we define $\phi_0(v) := \max_{1\leq j\leq m}\{v_j\}$ and $ f(w) := \frac{\lambda}{2}\norms{w}^2$, then \eqref{eq:min_max_stochastic_opt} can be reformulated into \eqref{eq:upper_level_min}.
Since $v_j \geq 0$, we have $\phi_0(v) := \max_{1\leq j\leq m}\{v_j\} = \Vert v \Vert_{\infty} = \max_{\norms{u}_1 \leq 1} \iprods{v, u}$, which is nonsmooth.
Thus we can smooth $\phi_0$ as $\phi_{\gamma}(v) := \max_{\norms{u}_1 \leq 1}\{\iprods{v, u} - (\gamma/2)\norms{u}^2\}$ using $b(u) := \frac{1}{2}\norms{u}^2$.

Here, we apply our problem \eqref{eq:min_max_stochastic_opt} to solve a model selection problem in binary classification with nonnegative nonconvex losses, see, e.g., \cite{zhao2010convex}.
Each function $F_{i,j}$ belongs to $4$ different nonconvex losses $(m=4)$: $ F_{i,1}(w, \xi) := 1 - \tanh(b_i\iprods{a_i,w})$, $ F_{i,2}(w, \xi) := \log(1 + \exp(-b_i\iprods{a_i,w})) - \log(1 + \exp(-b_i\iprods{a_i,w}-1))$,  $ F_{i,3}(w, \xi) := (1 - 1/(\exp(-b_i\iprods{a_i,w})+1))^2$, and $ F_{i,4}(w, \xi) := \log(1 + \exp(-b_i\iprods{a_i,w}))$ (see \cite{zhao2010convex} for more details), where $ (a_i, b_i)$ represents data samples.

We implement 4 algorithms: our \texttt{SGM} with 2 options, \texttt{SGD} from \cite{wang2017stochastic}, and \texttt{Prox-Linear} from \cite{zhang2020stochastic}. 
We test these algorithms on two datasets from LIBSVM \cite{CC01a}. 
We set $\lambda := 10^{-4}$ and update the smooothing parameter $\gamma_t$ as $\gamma_t := \frac{1}{2(t+1)^{1/3}}$. 
The learning rate $\eta$ for all algorithms is finely tuned from $\sets{100, 50, 10, 5, 1, 0.5, 0.1, 0.05, 0.01, 0.001, 0.0001}$, and the results are shown in Figure \ref{fig:min_max_stochastic_opt} for  \textbf{w8a} and \textbf{rcv1} datasets using $k_b=32$ blocks.
The details of this experiment is given in Supp. Doc. \ref{apdx:add_experiments}.

\begin{figure}[ht!]
\vspace{-1ex}
\centering
\includegraphics[width=0.49\textwidth]{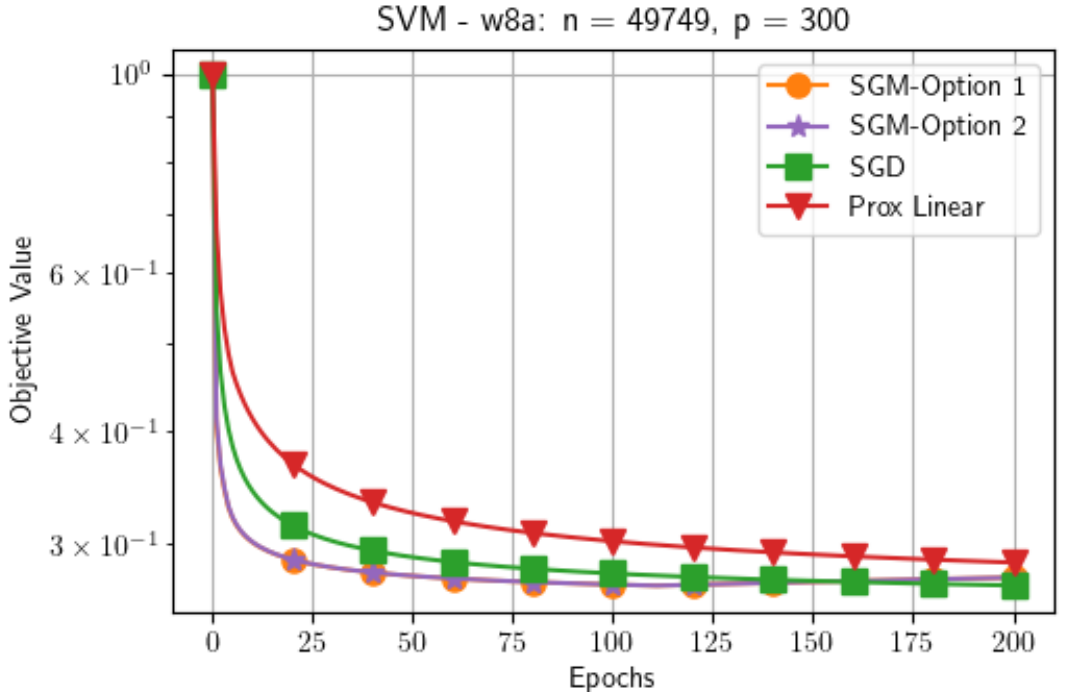}
\includegraphics[width=0.49\textwidth]{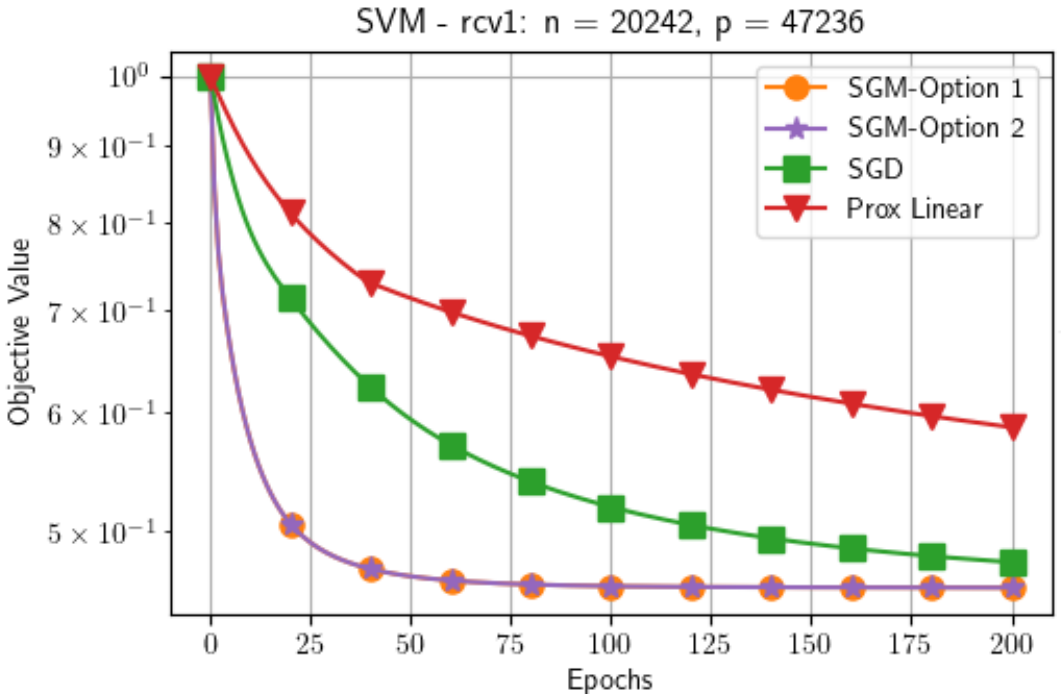}
\vspace{-0.75ex}
\caption{
The performance of 4 algorithms for solving \eqref{eq:min_max_stochastic_opt} on two datasets after 200 epochs.
}
\label{fig:min_max_stochastic_opt}
\vspace{-1ex}
\end{figure}

As shown in Figure \ref{fig:min_max_stochastic_opt}, the two variants of our \texttt{SGM} have a comparable performance with \texttt{SGD} and \texttt{Prox-Linear}, providing supportive evidence for using shuffling strategies in minimax algorithms. 

%%%%%%%%%%%%%%%%%%%%%%%%%%%%%%%%%%%%%%%%%%%%%%%%%%%%%%%
%%% 6. Conclusons.
%%%%%%%%%%%%%%%%%%%%%%%%%%%%%%%%%%%%%%%%%%%%%%%%%%%%%%%
\beforesec
\section{Conclusions}\label{sec:conclusions}
\aftersec
This work explores a bilevel optimization approach to address two prevalent classes of nonconvex-concave minimax problems. 
These problems find numerous applications in practice, including robust learning and generative AIs.
Motivated by the widespread use of shuffling strategies in implementing gradient-based methods within the machine learning community, we develop novel shuffling-based algorithms for solving these problems under standard assumptions.
The first algorithm uses a non-standard shuffling strategy and achieves the state-of-the-art oracle complexity typically observed in nonconvex optimization.  
The second algorithm is also new, flexible, and offers a promising possibility for further exploration.
Our results are expected to provide theoretical justification for incorporating shuffling strategies into minimax optimization algorithms, especially in nonconvex settings.

\begin{ack}
This work was partly supported by the National Science Foundation (NSF): NSF-RTG grant No. NSF DMS-2134107 and the Office of Naval Research (ONR), grant No. N00014-23-1-2588.
\end{ack}

%%%%%% References
%\bibliographystyle{plain}
%\input{sections/SGM_main_paper_refs.tex}
%%\bibliography{/Users/quoctd/Dropbox/E-Books/tran_bibtex_new}

%%%%%%%%%%%%%%%%%%%%%%%%%%%%%%%%%%%%%%%%%%%%%%%%%%%%%%%%%%%%
%%%%% NEURIPS CHECKLIST
%%%%%%%%%%%%%%%%%%%%%%%%%%%%%%%%%%%%%%%%%%%%%%%%%%%%%%%%%%%%
%\input{NeurIPS_checklist}

%%%%%%%%%%%%%%%%%%%%%%%%%%%%%%%%%%%%%%%%%%%%%%%%%%%%%%%%%%%%
%%%%% SUPPLEMENTARY DOCUMENTS.
%%%%%%%%%%%%%%%%%%%%%%%%%%%%%%%%%%%%%%%%%%%%%%%%%%%%%%%%%%%%
%\newpage
%\appendix
%\hrulefill
%\begin{center}
%\textbf{\Large Supplementary Document:} \vspace{0.5ex}\\
%\textbf{\LARGE Shuffling Gradient-Based Methods for\vspace{0.35ex}\\ Nonconvex-Concave Minimax Optimization}\\
%\hrulefill
%\end{center}
%\vspace{2ex}
%\setcounter{page}{1}

\appendix
%%%%%%%%%%%%%%%%%%%%%%%%%%%%%%%%%%%%%%%%%%%%%%%%%%%%%%%%%%%%
%%% General clarification.
%%%%%%%%%%%%%%%%%%%%%%%%%%%%%%%%%%%%%%%%%%%%%%%%%%%%%%%%%%%%
%Due to space limit, some results in the main text are not fully presented and clearly clarified. 
%This supplementary document provides further details of our results in the main text.
%It also provides and proves technical lemmas used in this paper, presents the full proofs of our theoretical results, and additional examples and details of our numerical experiments. 

%%%%%%%%%%%%%%%%%%%%%%%%%%%%%%%%%%%%%%%%%%%%%%%%%%%%%%%%%%%%
%%% A. Technical Results and Proofs
%%%%%%%%%%%%%%%%%%%%%%%%%%%%%%%%%%%%%%%%%%%%%%%%%%%%%%%%%%%%
\beforesec
\section{Supplementary Document: Technical Results and Proofs}\label{apdx:sec:techical_results}
\aftersec
This section gives the details of results related to minimax problem \eqref{eq:minimax_prob}, and discusses the underlying technical assumptions and the properties of related functions and quantities used in this paper.

\textbf{$\mathrm{(a)}$~Elementary facts.}
We recall the following facts, which will be used in the sequel.
\begin{compactitem}
\item[$\mathrm{[F_1]}$] If $h : \R^p \to \Rext$ is proper, closed, and $\mu_h$-strongly convex, and $\prox_{\eta h}$ is the proximal operator of $\eta h$ for any $\eta > 0$, then for any $u, \hat{u} \in\dom{h}$, we have
\begin{equation}\label{eq:prox_contraction}
\begin{array}{lcl}
\norms{\prox_{\eta h}(u) - \prox_{\eta h}(\hat{u})}^2 & \leq & \frac{1}{1 + 2\mu_h\eta}\norms{u - \hat{u}}^2.
\end{array}
\end{equation}
\item[$\mathrm{[F_2]}$] For any proper, closed, and convex function $h : \R^p\to\Rext$ and $\eta > 0$, we have  
\begin{equation*}
x - \prox_{\eta h}(x)  \in \eta \partial{h}(\prox_{\eta h}(x)).
\end{equation*}
\item[$\mathrm{[F_3]}$] Consider the lower level maximization problem \eqref{eq:lower_level_max} as
\begin{equation*} 
\begin{array}{lcl}
u^{*}_0(w) := {\displaystyle\argmax_{u\in\R^q}}\big\{ \mathcal{H}(w, u) - h(u) \equiv \frac{1}{n}\sum_{i=1}^n \mathcal{H}_i(w, u) - h(u) \big\}.
\end{array}
\end{equation*}
Then, under Assumption~\ref{as:A0}, its optimality condition can be written as
\begin{equation}\label{eq:lower_level_max_opt_cond}
\begin{array}{lcl}
\nabla_u\mathcal{H}(w, u_0^{*}(w)) \in \partial{h}(u_0^{*}(w)).
\end{array}
\end{equation}
\end{compactitem}

\textbf{$\mathrm{(b)}$~Details of Assumption~\ref{as:A3} and Assumption~\ref{as:A4}.}
Both Assumptions ~\ref{as:A3} and \ref{as:A4} look relatively technical, though they have been used in previous works such as \cite{mishchenko2022proximal}.
Both assumptions are the same, but one for $f$ and the other for $h$, and thus we only discuss Assumption~\ref{as:A3}.

Note that \cite{mishchenko2022proximal} did not provide any example to motivate Assumption~\ref{as:A3} for the case $f\neq 0$.
Assumption~\ref{as:A3} extends the one from \cite{mishchenko2022proximal} so that it holds for certain cases, including the two examples described after Assumption~\ref{as:A3}.
Here, we further elaborate these examples in detail.
\begin{compactitem}
\item[$\mathrm{(i)}$]~\underline{\textit{Example 1}}. 
If $f$ is $M_f$-Lipschitz continuous (e.g., the $\ell_1$-norm), then \eqref{eq:grad_mapp_bound} in Assumption~\ref{as:A3} also holds.
Indeed, since $f$ is $M_f$-Lipschitz continuous, it is obvious that $\partial{f}$ is $M_f$-bounded, and hence, by the fact $[F_2]$ above, we have $\norms{\prox_{\eta f}(u) - u} \leq \eta M_f$ for any $u$.
Using this inequality, and the definition of $\Gc_{\eta}$ in \eqref{eq:upper_level_min_grad_mapping}, we can easily show that
\begin{equation*}
\norms{\nabla{\Phi}_{\gamma}(w) - {\Gc}_{\eta}(w)} = \gamma^{-1}\Vert \mathrm{prox}_{\eta f}(w - \gamma \nabla{\Phi}_{\gamma}(w)) - (w - \gamma \nabla{\Phi}_{\gamma}(w)) \Vert \leq M_f.
\end{equation*}
Then, for any $\nu > 0$, by Young's inequality, we have $\norms{\nabla{\Phi}_{\gamma}(w) }^2 \leq (1 + \nu)\norms{{\Gc}_{\eta}(w)}^2 + \frac{1 + \nu}{\nu}\norms{\nabla{\Phi}_{\gamma}(w) - {\Gc}_{\eta}(w)}^2 \leq  (1 + \nu)\norms{{\Gc}_{\eta}(w)}^2 + \frac{1 + \nu}{\nu}M_f$.
Hence, Assumption~\ref{as:A3} holds for $\Lambda_0 := 1 + \nu$ and $\Lambda_1 :=  \frac{1 + \nu}{\nu}M_f$.

\item[$\mathrm{(ii)}$]~\underline{\textit{Example 2}}.
It is also easy to check that if $f = \delta_{\mathcal{W}}$, the indicator of a nonempty, closed, convex, and bounded set $\mathcal{W}$, then for any $w \in \mathcal{W}$, we also have $\Vert \prox_{\eta f}(w) - w\Vert = \Vert \mathrm{proj}_{\mathcal{W}}(w) - w\Vert  \leq 2\mathrm{diam}(\mathcal{W})$, where $\mathrm{diam}(\mathcal{W})$ is the diameter of $\mathcal{W}$.
Hence, by the same proof as in \textit{Example 1}, Assumption~\ref{as:A3} also holds.
\end{compactitem}

%%As shown in \cite[Lemma  A.1.]{TranDinh2020f}, $\phi_{\gamma}$ is $L_{\Phi_0}$-smooth with $L_{\Phi_0} := \frac{\norms{K}^2}{\mu_h + \gamma}$.
%%Moreover, since $\mathrm{dom}(h)$ is $M_h$ bounded, we have $\phi_{\gamma}$ is also $M_{\Phi_0}$-Lipschitz continuous with $M_{\Phi_0} := M_h\norms{K}$.

%%%%% c. A Technical Lemmas
\textbf{$\mathrm{(c)}$~Technical results.}
The following lemma summarizes the properties of $\phi_{\gamma}$ defined by \eqref{eq:smoothed_phi0}, which was proved in \cite{TranDinh2020f}.
It will be used in the sequel for analyzing Algorithm~\ref{alg:SGM1}.

%%% Lemma 1.
\begin{lemma}\label{le:psi_eta_pro1}
Let $\phi_0$ and $\phi_{\gamma}$ be defined by \eqref{eq:psi_func} and \eqref{eq:smoothed_phi0}, respectively.
Then, under Assumption~\ref{as:A1}:
\begin{itemize}
\item[$\mathrm{(a)}$] $\mathrm{dom}(h)$ is bounded by $M_h$ iff $\phi_{\gamma}$ is $M_{\Phi_0}$-Lipschitz continuous with $M_{\phi_0} := M_h\norms{K}$.
\item[$\mathrm{(b)}$] $\phi_{\gamma}$ is $L_{\phi_{\gamma}}$-smooth with $L_{\phi_{\gamma}} := \frac{\norms{K}^2}{\mu_h + \gamma}$ $($i.e. $\nabla{\phi}_{\gamma}$ is $L_{\phi_{\gamma}}$-Lipschitz continuous$)$.
\item[$\mathrm{(c)}$] $\phi_{\gamma}(v) \leq \phi_0(v) \leq \phi_{\gamma}(v) + \gamma B_{\phi_0}$ for any $v$, where $B_{\phi_0} := \sup\sets{b(u) : u \in\dom{h}}$.
\item[$\mathrm{(d)}$] For any $ \hat{\gamma} \geq \gamma > 0$ and $v$, we have $\phi_{\gamma}(v) \leq \phi_{\hat{\gamma}}(v) + (\hat{\gamma} - \gamma)b(u^{*}_{\gamma}(v)) \leq \phi_{\hat{\gamma}}(v) + (\hat{\gamma} - \gamma)B_{\phi_0}$.
\end{itemize}
\end{lemma}
%This lemma will be used to prove  Lemma~\ref{le:smoothness_of_phi}. The proof can be found in \cite[Lemma A.3.]{TranDinh2020f}.

%%%% d. The smoothness of \Phi_0 and Phi_{\gamma}.
\textbf{$\mathrm{(d)}$~The smoothness of $\Phi_{\gamma}$ and $\Phi_0$.}
One key step to develop our algorithms is to show that $\Phi_{\gamma}$ defined by  \eqref{eq:smoothed_com_opt_prob} and $\Phi_0$ in \eqref{eq:lower_level_max} are $L$-smooth (i.e. their gradient is Lipschitz continuous).
The following lemma shows the $L_{\Phi_{\gamma}}$-smoothness of $\Phi_{\gamma}$ defined in \eqref{eq:smoothed_com_opt_prob}, whose proof is given in  \cite[Lemma A.3.]{TranDinh2020f}.

%% Lemma 1.
\begin{lemma}[Smoothness of $\Phi_{\gamma}$]\label{le:smoothness_of_phi}
Under Assumption~\ref{as:A1}, $\nabla{\Phi_{\gamma}}$ of $\Phi_{\gamma}$ defined by \eqref{eq:grad_phi_gamma}  is $L_{\Phi_{\gamma}}$-Lipschitz continuous with $L_{\Phi_{\gamma}} := M_{h}\norms{K} L_F + \frac{M_F^2\norms{K}^2}{\mu_{h} + \gamma}$, where $\gamma \geq 0$ such that $\mu_h + \gamma > 0$.

Consequently, for any $w, \hat{w} \in \dom{\Phi_{\gamma}}$, we have
\begin{equation}\label{eq:Lsmooth_Phi_gamma}
\arraycolsep=0.2em
\begin{array}{lcl}
- \frac{L_{\Phi_{\gamma}}}{2}\norms{\hat{w} - w}^2 \leq \Phi_{\gamma}(\hat{w}) - \Phi_{\gamma}(w) - \iprods{\nabla{\Phi}_{\gamma}(w), \hat{w} - w}  \leq  \frac{L_{\Phi_{\gamma}}}{2}\norms{\hat{w} - w}^2.
\end{array}
\end{equation}
\end{lemma}

Alternatively, Lemma~\ref{le:property_of_Phi0} in the main text can be expanded in detail as follows.
%%% Lemma 2.
\begin{lemma}\label{le:property_of_Phi0_details}
Under Assumption~\ref{as:A3_SGM2}, let $u_0^{*}(\cdot)$ and $\Phi_0$ be defined by \eqref{eq:lower_level_max}.
Then, $u_0^{*}(\cdot)$ is $\kappa$-Lipschitz continuous with $\kappa := \frac{L_u}{\mu_{H} + \mu_h} > 0$, i.e.:
\begin{equation}\label{eq:u_star_smoothness}
\norms{ u_0^{*}(w) - u_0^{*}(\hat{w})} \leq \kappa \norms{w - \hat{w}}, \quad \forall w, \hat{w} \in  \dom{\Phi_0}.
\end{equation}
Moreover, $\Phi_0$ is $L_{\Phi_0}$-smooth, i.e. $\Vert \nabla{\Phi}_0(w) - \nabla{\Phi}_0(\hat{w})\Vert \leq L_{\Phi_0}\norms{w - \hat{w}}$ for all $w, \hat{w} \in  \dom{\Phi_0}$, where $L_{\Phi_0} := (1 + \kappa)L_w$.
Consequently, for all $w, \hat{w} \in \dom{\Phi_0}$, we have
\begin{equation}\label{eq:Lsmooth_Phi2}
-\frac{L_{\Phi_0}}{2} \Vert \hat{w} - w \Vert^2 \leq \Phi_0( \hat{w} ) - \Phi_0 (w) - \langle \nabla{\Phi_0}(w), \hat{w} - w \rangle  \leq  \frac{L_{\Phi_0}}{2} \Vert \hat{w} - w \Vert^2.
\end{equation}
\end{lemma}
This lemma is proven similar to the one, e.g., in \cite{Nesterov2005c}, and we omit it here.

%%%%% e. Proof of Lemma 1.
\textbf{$\mathrm{(e)}$~Proof of Lemma~\ref{le:kkt_point} -- Approximate stationary and KKT points.}
Now, we provide the proof of Lemma~\ref{le:kkt_point} in the main text.

%%% Proof of Lemma 1.
\begin{proof}[\textbf{Proof of Lemma~\ref{le:kkt_point}}]
(a)~If $(w^{\star}, u^{\star})$ is a KKT point of \eqref{eq:minimax_prob}, then 
\begin{equation*}
\arraycolsep=0.2em
\begin{array}{lcl}
0 \in \nabla_w{\Hc}(w^{\star}, u^{\star}) + \partial{f}(w^{\star}) \quad \text{and} \quad 0 \in -\nabla_u{\Hc}(w^{\star}, u^{\star}) + \partial{h}(u^{\star}).
\end{array}
\end{equation*} 
Since $\Hc(w^{\star}, \cdot) - h(\cdot)$ is concave, $0 \in -\nabla_u{\Hc}(w^{\star}, u^{\star}) + \partial{h}(u^{\star})$ implies that $u^{\star} \in \argmax_u\sets{\Hc(w^{\star}, u) - h(u)}$.
For $\Phi_0$ defined by \eqref{eq:lower_level_max}, by Danskin's theorem, we have $\nabla{\Phi}_0(w^{\star}) = \nabla_w{\Hc}(w^{\star}, u^{\star})$.
Hence, combining this relation and $0 \in \nabla_w{\Hc}(w^{\star}, u^{\star}) + \partial{f}(w^{\star})$, we have $0 \in \nabla{\Phi}_0(w^{\star}) + \partial{f}(w^{\star})$, which shows that $w^{\star}$ is a stationary point of \eqref{eq:upper_level_min}.
The converse statement is proved similarly, and we omit.

(b)~If $\widehat{w}_T$ is an $\epsilon$-stationary point of \eqref{eq:upper_level_min}, then using a shorthand $g_T := \Gc_{\eta}(\widehat{w}_T)$,  we have $\Eb[ \norms{g_T}^2 ] \leq \epsilon^2$.
From \eqref{eq:upper_level_min_grad_mapping}, we have $g_T = \eta^{-1}(\widehat{w}_T - \prox_{\eta f}(\widehat{w}_T - \eta\nabla{\Phi}_{0}(\widehat{w}_T)))$, which is equivalent to $g_T \in \nabla{\Phi}_{0}(\widehat{w}_T)  + \partial{f}(\widehat{w}_T - \eta g_T)$.
Let us define $\overline{w}_T$ as in Lemma~\ref{le:kkt_point} and $e_T$ as follows:
\begin{equation}\label{eq:kkt_point_lm1_proof1}
\left\{\begin{array}{lcl}
\overline{w}_T & := & \widehat{w}_T - \eta g_T = \prox_{\eta f}(\widehat{w}_T - \eta\nabla{\Phi}_{0}(\widehat{w}_T))), \\
e_T & := &  g_T + \nabla{\Phi}_{0}(\overline{w}_T) - \nabla{\Phi}_{0}(\widehat{w}_T).
\end{array}\right.
\end{equation} 
Then, $g_T \in \nabla{\Phi}_{0}(\widehat{w}_T)  + \partial{f}(\overline{w}_T)$ is equivalent to $e_T \in \nabla{\Phi}_{0}(\overline{w}_T)  + \partial{f}(\overline{w}_T) = \nabla_w{\Hc}(\overline{w}_T, u_0^{*}(\overline{w}_T)) + \partial{f}(\overline{w}_T)$.
On the other hand, from \eqref{eq:lower_level_max_opt_cond}, we have $0 \in -\nabla_u{\Hc}(\overline{w}_T, u_0^{*}(\overline{w}_T)) + \partial{h}(u_0^{*}(\overline{w}_T))$.
By the triangle inequality, and the $L_{\Phi_0}$-Lipschitz continuity of $\nabla{\Phi}_0$, we have
\begin{equation*}
\arraycolsep=0.2em
\begin{array}{lcl}
\norms{e_T} & \overset{\tiny \eqref{eq:kkt_point_lm1_proof1} }{\leq} &  \norms{g_T} + \norms{ \nabla{\Phi}_{0}(\overline{w}_T) - \nabla{\Phi}_{0}(\widehat{w}_T)} \\
& \leq & \norms{g_T} + L_{\Phi_0}\norms{\overline{w}_T - \widehat{w}_T} \\
& \overset{\tiny\eqref{eq:kkt_point_lm1_proof1} }{\leq} & (1 + L_{\Phi_0}\eta)\norms{g_T}.
\end{array}
\end{equation*} 
Hence, we get 
\begin{equation*}
\arraycolsep=0.2em
\begin{array}{lcl}
\Eb[ \norms{e_T}^2 ] \leq  (1 + L_{\Phi_0}\eta)^2 \Eb[\norms{g_T}^2] \leq (1 + L_{\Phi_0}\eta)^2\epsilon^2.
\end{array}
\end{equation*}
This concludes that if  $\widehat{w}_T$ is an $\epsilon$-stationary point of \eqref{eq:upper_level_min}, then $(\overline{w}_T, u_0^{*}(\overline{w}_T))$ is an $\hat{\epsilon}$-KKT point of \eqref{eq:minimax_prob} with $\hat{\epsilon} := (1 + L_{\Phi_0}\eta)\epsilon$.

(c)~Since $\overline{w}_T  :=  \prox_{\eta f}(\widehat{w}_T - \eta\nabla{\Phi}_{\gamma}(\widehat{w}_T)))$, we have $\widehat{w}_T - \overline{w}_T - \eta \nabla{\Phi}_{\gamma}(\widehat{w}_T) \in \eta \partial{f}(\overline{w}_T)$.
Using this inclusion and 
\begin{equation*}
\begin{array}{lcl}
\nabla{\Phi}_{\gamma}(\overline{w}_T)  = \nabla{F}(\overline{w}_T)^{\top}\nabla{\phi}_{\gamma}(F(\overline{w}_T)) =  \nabla{F}(\overline{w}_T)^{\top}Ku^{*}_{\gamma}(F(\overline{w}_T)) = \nabla_w{\Hc}(\overline{w}_T, u^{*}_{\gamma}(\overline{w}_T)),
\end{array}
\end{equation*} 
we can show that
\begin{equation*}
\begin{array}{lcl}
\bar{r}_w &:= &  \eta^{-1}(\widehat{w}_T - \overline{w}_T) +    \nabla{\Phi}_{\gamma}(\overline{w}_T) - \nabla{\Phi}_{\gamma}(\widehat{w}_T) \in \nabla{\Phi}_{\gamma}(\overline{w}_T) +  \partial{f}(\overline{w}_T) \\
& \equiv & \nabla_w{\Hc}(\overline{w}_T, u^{*}_{\gamma}(\overline{w}_T)) +  \partial{f}(\overline{w}_T).
\end{array}
\end{equation*} 
Since $\nabla{\Phi}_{\gamma}$ is $L_{\Phi_{\gamma}}$-Lipschitz continuous and $\Gc_{\eta}(\overline{w}_T) = \eta^{-1}(\widehat{w}_T - \overline{w}_T)$, we have 
\begin{equation*}
\begin{array}{lcl}
\norms{ \bar{r}_w } &\leq & \norms{ \Gc_{\eta}(\overline{w}_T) } + \norms{ \nabla{\Phi}_{\gamma}(\overline{w}_T) - \nabla{\Phi}_{\gamma}(\widehat{w}_T) } \leq (1 + \eta L_{\Phi_{\gamma}})\norms{\Gc_{\eta}(\overline{w}_T)}.
\end{array}
\end{equation*} 
On the other hand, since $\overline{u}_T := u^{*}_{\gamma}(F(\overline{w}_T))$, using the optimality condition of \eqref{eq:smoothed_phi0}, and noticing that $\Hc(w, u) = \iprods{F(w), Ku}$, we have
\begin{equation*}
\begin{array}{lcl}
\bar{r}_u := -\gamma\nabla{b}(\overline{u}_T) &\in &  -K^{\top} F(\overline{u}_T) + \partial{h}(\overline{u}_T) \equiv -\nabla_u{\Hc}(\overline{w}_T, \overline{u}_T) + \partial{h}(\overline{u}_T).
\end{array}
\end{equation*} 
Since $\dom{h}$ is bounded by $M_h$ by Assumption~\ref{as:A2}, we can show that $\norms{\bar{r}_u} = \gamma\norms{\nabla{b}(\overline{u}_T)} \leq \gamma D_b$, where $D_b := \sup\sets{ \norms{\nabla{b}(u)} : u \in \dom{h} }$.
Combining the above analysis and noticing that $\Eb[  \norms{ \Gc_{\eta}(\overline{w}_T) }^2 ] \leq \epsilon^2$, we can show that
\begin{equation*}
\begin{array}{lcl}
\bar{r}_w \in \nabla_w{\Hc}(\overline{w}_T, \overline{u}_T) +  \partial{f}(\overline{w}_T) \quad \text{and} \quad \bar{r}_u \in -\nabla_u{\Hc}(\overline{w}_T, \overline{u}_T) + \partial{h}(\overline{u}_T).
\end{array}
\end{equation*} 
where $\Eb[ \norms{\bar{r}_w}^2 ] \leq (1 + \eta L_{\Phi_{\gamma}})^2\epsilon^2$ and $\Eb[ \norms{\bar{r}_u}^2 ] \leq \gamma^2 D_b^2$.
This proves that $(\overline{w}_T, \overline{u}_T)$ is an $\hat{\epsilon}$-KKT of \eqref{eq:minimax_prob} with $\hat{\epsilon} := \max\sets{ (1 + \eta L_{\Phi_{\gamma}})\epsilon, \gamma D_b}$.
Clearly, we have $\hat{\epsilon} = \Oc(\max\sets{\epsilon, \gamma})$.
In particular, if we choose $\eta := \Oc(\epsilon)$ and $\gamma := \Oc(\epsilon)$, then $\hat{\epsilon} = \Oc(\epsilon)$.
\end{proof}
%%% End of proof.

%%%%%%%%%%%%%%%%%%%%%%%%%%%%%%%%%%%%%%%%%%%%%%%%%%%%%%%%%%%%
%%% End of Supplementary Documents -- Part 1.
%%%%%%%%%%%%%%%%%%%%%%%%%%%%%%%%%%%%%%%%%%%%%%%%%%%%%%%%%%%%
%%%%%%%%%%%%%%%%%%%%%%%%%%%%%%%%%%%%%%%%%%%%%%%%%%%%%%%%%%%%
%%% B. Convergence Analysis of Algorithm 1.
%%%%%%%%%%%%%%%%%%%%%%%%%%%%%%%%%%%%%%%%%%%%%%%%%%%%%%%%%%%%
\beforesec
\section{Convergence Analysis of Algorithm~\ref{alg:SGM1} -- The NL Setting}\label{apdx:sec:SGM1_convergence}
\aftersec
We first prove some key estimates for the shuffling estimator of $\nabla{\Phi}_{\gamma}(w)$.
Next, we establish the technical lemmas that will be used to prove Theorem~\ref{th:main_result1}.
Finally, we prove Theorem~\ref{th:main_result1} and Corollary~\ref{co:app_KKT_for_nonconvex_linear_case}.

%%%%%%%%%%%%%%%%%%%%%%%%%%%%%%%%%%%%
%%% B.1. Key estimates for shuffling gradient estimators
%%%%%%%%%%%%%%%%%%%%%%%%%%%%%%%%%%%%
\beforesubsec
\subsection{Properties of shuffling estimators}\label{apdx:subsec:properties}
\aftersubsec
We state the following properties of $\widetilde{\nabla}{\Phi_{\gamma}}(\cdot)$ defined by \eqref{eq:approx_grad_Phi}, which could be of independent interest.

%Let us start with deterministic and arbitrary permutations.
%In this case, we can choose $\hat{\pi}^{(t)} = \pi^{(t)}$ for all $t \geq 0$.
%\trang{since we use these derivation for the random case later, let's say they hold for any permutation?}
%Then, we prove some key properties for our estimates to develop the algorithm in the sequel. 

%% Lemma 1.
\begin{lemma}[Arbitrary permutation]\label{le:key_estimates_det}
Assume that Assumption~\ref{as:A1} holds.
Then
\begin{compactitem}
\item[$\mathrm{(a)}$] For any $i \in [n]$, the approximation $F_i^{(t)}$ defined by \eqref{eq:approxi_Ft} satisfies 
\begin{equation}\label{eq:main_est2}
\begin{array}{lcl}
\norms{ F_i^{(t)} -  F(w_0^{(t)})}^2  \leq  \frac{M_F^2}{n}\sum_{j=1}^{n}\Vert w_{j-1}^{(t)}  - w^{(t)}_0 \Vert^2.
\end{array}
\end{equation}

\item[$\mathrm{(b)}$] Let $\mathcal{T}_{[i]} :=  \Vert \frac{1}{i}\sum_{j=1}^i\widetilde{\nabla}{\Phi_{\gamma}}(w_{j-1}^{(t)}) - \nabla{\Phi_{\gamma}}(w_0^{(t)}) \Vert^2$ for $\widetilde{\nabla}\Phi_{\gamma}(w_{i-1}^{(t)})$ defined by \eqref{eq:approx_grad_Phi}. 
Then 
\begin{equation}\label{eq:grad_Phi_bound2}
\begin{array}{lcl}
\mathcal{T}_{[i]}
& \leq & \left(\frac{C_1}{n}+ \frac{2C_2L_F^2}{i}\right)\sum_{j=1}^{n}\Vert w_{j-1}^{(t)}  - w^{(t)}_0 \Vert^2 +  \frac{2n C_2 \sigma_J^2}{i} \\
\mathcal{T}_{[n]} & \leq &\frac{1}{n} \left( C_1 + 2C_2 L_{F}^2 \right)\sum_{j=1}^{n}\Vert w_{j-1}^{(t)}  - w^{(t)}_0 \Vert^2, 
\end{array}
\end{equation}
where $C_1 := \frac{2M_F^4\norms{K}^4}{(\mu_h + \gamma)^2}$ and $C_2 := 2M_h^2\norms{K}^2$.
\end{compactitem}
\end{lemma}

%% Proof of Lemma 5.
\begin{proof} 
(a)~Since $F(w_0^{(t)}) = \frac{1}{n}\sum_{j=1}^{n}F_j(w^{(t)}_0) = \frac{1}{n}\sum_{j=1}^{n}F_{\pi^{(t)}(j)}(w^{(t)}_0)$, using \textbf{Option 1} as \eqref{eq:approxi_Ft}, we have
\begin{equation*} 
\begin{array}{lcl}
\norms{ F_i^{(t)} -  F(w_0^{(t)})}^2 &= & \frac{1}{n^2} \Vert \sum_{j=1}^{i}F_{\pi^{(t)}(j)}(w_{j-1}^{(t)}) + \sum_{j=i+1}^{n}F_{\pi^{(t)}(j)}(w^{(t)}_0) - \sum_{j=1}^{n} F_{\pi^{(t)}(j)}(w^{(t)}_0)  \Vert^2 \\
&= & \frac{1}{n^2}\Vert \sum_{j=1}^{i}\big[ F_{\pi^{(t)}(j)}(w_{j-1}^{(t)})  - F_{\pi^{(t)}(j)}(w^{(t)}_0) \big]\Vert^2 \\
& \leq & \frac{i}{n^2}\sum_{j=1}^{i}\Vert F_{\pi^{(t)}(j)}(w_{j-1}^{(t)})  - F_{\pi^{(t)}(j)}(w^{(t)}_0) \Vert^2 \\
& \leq & \frac{i\cdot M_F^2}{n^2}\sum_{j=1}^{i}\Vert w_{j-1}^{(t)}  - w^{(t)}_0 \Vert^2 \\
& \leq & \frac{M_F^2}{n}\sum_{j=1}^{n}\Vert w_{j-1}^{(t)}  - w^{(t)}_0 \Vert^2,
\end{array}
\end{equation*}
which proves \eqref{eq:main_est2} due to $1 \leq i \leq n$.

Alternatively if we use the update \eqref{eq:approxi_Ftb} as  in \textbf{Option 2}, then we have $F_i^{(t)} =  F(w_0^{(t)})$ which also automatically satisfies \eqref{eq:main_est2}.

(b)~From the definition of $\nabla{\Phi_{\gamma}}(w_0^{(t)})$ in \eqref{eq:grad_phi_gamma} and of $\widetilde{\nabla}{\Phi_{\gamma}}(w_{i-1}^{(t)})$  in \eqref{eq:approx_grad_Phi}, by Young's inequality in $\myeqc{1}$ and $\myeqc{2}$, the Cauchy-Schwarz inequality in $\myeqc{2}$, and Lemma~\ref{le:psi_eta_pro1} in $\myeqc{3}$, we have
\begin{equation*}
\begin{array}{lcl}
\mathcal{T}_{[i]} &:= & \Vert \frac{1}{i}\sum_{j=1}^i\widetilde{\nabla}{\Phi_{\gamma}}(w_{j-1}^{(t)}) - \nabla{\Phi_{\gamma}}(w_0^{(t)}) \Vert^2 \\
&= & \norms{ \frac{1}{i}\sum_{j=1}^i \big[ (\nabla{F}_{j}^{(t)})^{\top}\nabla{\phi}_{\gamma}(F_{j}^{(t)}) - \nabla{F}(w_0^{(t)})^{\top} \nabla{\phi}_{\gamma}(F(w_0^{(t)})) \big] }^2 \\
&= & \norms{ \frac{1}{i}\sum_{j=1}^i \big[ (\nabla{F}_{j}^{(t)})^{\top}\nabla{\phi}_{\gamma}(F_{j}^{(t)}) - (\nabla{F}_{j}^{(t)})^{\top}\nabla{\phi}_{\gamma}(F(w_0^{(t)}) \\
&& \qquad + {~} (\nabla{F}_{j}^{(t)})^{\top}\nabla{\phi}_{\gamma}(F(w_0^{(t)})   - \nabla{F}(w_0^{(t)})^{\top} \nabla{\phi}_{\gamma}(F(w_0^{(t)})) \big] }^2 \\
& \overset{\myeqc{1}}{\leq} & 2 \Vert \frac{1}{i}\sum_{j=1}^i (\nabla{F}_{j}^{(t)})^{\top} \big[ \nabla{\phi}_{\gamma}(F_{j}^{(t)}) -  \nabla{\phi}_{\gamma}(F(w_0^{(t)}) \big] \Vert^2 \\
&& + {~}  2\Vert \frac{1}{i}\sum_{j=1}^i \big[ \nabla{F}_{j}^{(t)} - \nabla{F}(w_0^{(t)}) \big]^{\top}\nabla{\phi}_{\gamma}(F(w_0^{(t)}) \Vert^2 \\
& \overset{\myeqc{2}}{\leq} & \frac{2}{i} \sum_{j=1}^i \Vert \nabla{F}_{j}^{(t)}\Vert^2\Vert \nabla{\phi}_{\gamma}(F_{j}^{(t)}) -  \nabla{\phi}_{\gamma}(F(w_0^{(t)}) \Vert^2 \\
&& + {~}  2\Vert \nabla{\phi}_{\gamma}(F(w_0^{(t)})\Vert^2\Vert \frac{1}{i}\sum_{j=1}^i \big[ \nabla{F}_{j}^{(t)} - \nabla{F}(w_0^{(t)}) \big] \Vert^2 \\
& \overset{\myeqc{3}}{\leq} & \frac{2M_F^2\norms{K}^4}{i(\mu_h+\gamma)^2} \sum_{j=1}^i \ \Vert  F_{j}^{(t)} -  F(w_0^{(t)})\Vert^2 +  2M_h^2\norms{K}^2\Vert \frac{1}{i}\sum_{j=1}^i \big[ \nabla{F}_{j}^{(t)} - \nabla{F}(w_0^{(t)}) \big] \Vert^2.
\end{array}
\end{equation*}
Substituting \eqref{eq:main_est2} into this estimate  and noting that $C_1 = \frac{2M_F^4\norms{K}^4}{(\mu_h + \gamma)^2}$ and $C_2 = 2M_h^2\norms{K}^2$ we obtain
\begin{equation} \label{eq:main_est2a}
\hspace{-0.02ex}
\begin{array}{lcl}
\mathcal{T}_{[i]}
& \leq & \frac{2M_F^2\norms{K}^4}{i(\mu_h+\gamma)^2} \sum_{j=1}^i \frac{M_F^2}{n}\sum_{j=1}^{n}\Vert w_{j-1}^{(t)}  - w^{(t)}_0 \Vert^2 +  C_2\Vert \frac{1}{i}\sum_{j=1}^i \big[ \nabla{F}_{j}^{(t)} - \nabla{F}(w_0^{(t)}) \big] \Vert^2 \\
& \leq & C_2 \Vert \frac{1}{i}\sum_{j=1}^i \big[ \nabla{F}_j^{(t)} - \nabla{F}_{\pi^{(t)}(j)}(w_0^{(t)}) \big] + \frac{1}{i}\sum_{j=1}^i \big[ \nabla{F}_{\pi^{(t)}(j)}(w_0^{(t)})  - \nabla{F}(w_0^{(t)}) \big] \Vert^2 \\
&& + {~}  \frac{C_1}{n}\sum_{j=1}^{n}\Vert w_{j-1}^{(t)}  - w^{(t)}_0 \Vert^2 \\
& \leq & \frac{C_1}{n}\sum_{j=1}^{n}\Vert w_{j-1}^{(t)}  - w^{(t)}_0 \Vert^2 +  2C_2 \Vert \frac{1}{i}\sum_{j=1}^i \big[ \nabla{F}_{\pi^{(t)}(j)}(w_{j-1}^{(t)}) - \nabla{F}_{\pi^{(t)}(j)}(w_0^{(t)}) \big]\Vert^2 \\
&& + {~} 2C_2 \Vert \frac{1}{i}\sum_{j=1}^i \big[ \nabla{F}_{\pi^{(t)}(j)}(w_0^{(t)})  - \nabla{F}(w_0^{(t)}) \big] \Vert^2 \\
& \leq & \frac{C_1}{n}\sum_{j=1}^{n}\Vert w_{j-1}^{(t)}  - w^{(t)}_0 \Vert^2 +  2C_2  \frac{1}{i}\sum_{j=1}^i \Vert  \nabla{F}_{\pi^{(t)}(j)}(w_{j-1}^{(t)}) - \nabla{F}_{\pi^{(t)}(j)}(w_0^{(t)}) \Vert^2 \\
&& + 2C_2 \Vert \frac{1}{i}\sum_{j=1}^i \big[ \nabla{F}_{\pi^{(t)}(j)}(w_0^{(t)})  - \nabla{F}(w_0^{(t)}) \big] \Vert^2 \\
& \leq & \frac{C_1}{n}\sum_{j=1}^{n}\Vert w_{j-1}^{(t)}  - w^{(t)}_0 \Vert^2 +  2C_2  \frac{1}{i}\sum_{j=1}^n L_{F_j}^2\Vert w_{j-1}^{(t)} - w_0^{(t)} \Vert^2 \\
&& + 2C_2 \Vert \frac{1}{i}\sum_{j=1}^i \big[ \nabla{F}_{\pi^{(t)}(j)}(w_0^{(t)})  - \nabla{F}(w_0^{(t)}) \big] \Vert^2 \\
& \leq & \left( \frac{C_1}{n} + \frac{2C_2 L_{F}^2}{i} \right)\sum_{j=1}^{n}\Vert w_{j-1}^{(t)}  - w^{(t)}_0 \Vert^2 + 2C_2 \Vert \frac{1}{i}\sum_{j=1}^i \big[ \nabla{F}_{\pi^{(t)}(j)}(w_0^{(t)})  - \nabla{F}(w_0^{(t)}) \big] \Vert^2.
\end{array}
\hspace{-6ex}
\end{equation}
For $i=n$, we have 
\begin{equation*}
\begin{array}{lcl}
\mathcal{T}_{[n]}
& \leq & \left( \frac{C_1}{n} + \frac{2C_2 L_{F}^2}{n} \right)\sum_{j=1}^{n}\Vert w_{j-1}^{(t)}  - w^{(t)}_0 \Vert^2 + 2C_2 \Vert \frac{1}{n}\sum_{j=1}^n \big[ \nabla{F}_{\pi^{(t)}(j)}(w_0^{(t)})  - \nabla{F}(w_0^{(t)}) \big] \Vert^2 \\
& \leq & \frac{1}{n} \left( C_1 + 2C_2 L_{F}^2 \right)\sum_{j=1}^{n}\Vert w_{j-1}^{(t)}  - w^{(t)}_0 \Vert^2.
\end{array}
\end{equation*}
For any other index $i \in [n]$ and $i < n$, we can show that
\begin{equation*}
\begin{array}{lcl}
\mathcal{T}_{[i]}
& \leq & \left( \frac{C_1}{n} +  \frac{2C_2 L_{F}^2}{i} \right)\sum_{j=1}^{n}\Vert w_{j-1}^{(t)}  - w^{(t)}_0 \Vert^2 + 2C_2 \frac{1}{i}\sum_{j=1}^i\Vert   \nabla{F}_{\pi^{(t)}(j)}(w_0^{(t)})  - \nabla{F}(w_0^{(t)})  \Vert^2 \\
& \leq & \left( \frac{C_1}{n} +  \frac{2C_2 L_{F}^2}{i} \right)\sum_{j=1}^{n}\Vert w_{j-1}^{(t)}  - w^{(t)}_0 \Vert^2 + 2C_2 \frac{n}{i}\frac{1}{n}\sum_{j=1}^n\Vert   \nabla{F}_{\pi^{(t)}(j)}(w_0^{(t)})  - \nabla{F}(w_0^{(t)})  \Vert^2, 
\end{array}
\end{equation*}
which proves the desired estimate. 
\end{proof}
%%% End of proof.

%%% This for random permutation.
If $\pi^{(t)}$ and $\hat{\pi}^{(t)}$ are generated randomly and independently, the we have the following result.

%%% Lemma 2.2.
\begin{lemma}[Random permutation]\label{le:key_estimates_rand}
Assume that Assumption~\ref{as:A1} holds, and $\pi^{(t)} $ and $\hat{\pi}^{(t)}$ are two random permutations of $[n]$.
We recall that $\mathcal{T}_{[i]} :=  \Vert \frac{1}{i}\sum_{j=1}^i\widetilde{\nabla}{\Phi_{\gamma}}(w_{j-1}^{(t)}) - \nabla{\Phi_{\gamma}}(w_0^{(t)}) \Vert^2$. 
Then
\begin{equation}\label{eq:grad_Phi_bound2_rand}
\begin{array}{lcl}
\Eb[\mathcal{T}_{[i]}]
& \leq & \left( \frac{C_1}{n} + \frac{2C_2 L_{F}^2}{i} \right)\sum_{j=1}^{n} \Eb\big[\Vert w_{j-1}^{(t)}  - w^{(t)}_0 \Vert^2\big] + \frac{2C_2 }{i} \sigma_J^2, 
\end{array}
\end{equation}
where $C_1 := \frac{2M_F^4\norms{K}^4}{(\mu_h + \gamma)^2}$ and $C_2 := 2M_h^2\norms{K}^2$.
\end{lemma}

%%% Proof of Lemma 6.
\begin{proof} 
In this proof, we will use  \citep{mishchenko2020random}[Lemma 1]  for sampling without replacement at random.
From \eqref{eq:main_est2a} in Lemma \ref{le:key_estimates_det} we have 
\begin{equation*} 
\hspace{-1ex}
\begin{array}{lcl}
\mathcal{T}_{[i]}
& \leq & \left( \frac{C_1}{n} + \frac{2C_2 L_{F}^2}{i} \right)\sum_{j=1}^{n}\Vert w_{j-1}^{(t)}  - w^{(t)}_0 \Vert^2 + 2C_2 \Vert \frac{1}{i}\sum_{j=1}^i \big[ \nabla{F}_{\pi^{(t)}(j)}(w_0^{(t)})  - \nabla{F}(w_0^{(t)}) \big] \Vert^2.
\end{array}
\hspace{-1ex}
\end{equation*}
For each epoch $t=1, \cdots, T$, we denote by  $\mathcal{F}_t := \sigma(w_0^{(1)},\cdots,w_0^{(t)})$ as the $\sigma$-algebra generated by the iterates of our algorithm (\textit{cf.} Algorithm~\ref{alg:SGM1}) up to the beginning of the epoch $t$. 
We observe that the permutation $\pi^{(t)}$ used at time $t$ is independent of the $\sigma$-algebra $\mathcal{F}_t$. We also denote by $\E_t[\cdot] := \E [\cdot \mid \mathcal{F}_t]$ as the conditional expectation on the $\sigma$-algebra $\mathcal{F}_t$.

Taking the expectation conditioned on $\mathcal{F}_t$, we get 
\begin{equation*} 
\begin{array}{lcl}
\Eb_t[\mathcal{T}_{[i]}]
& \leq & \left( \frac{C_1}{n} + \frac{2C_2 L_{F}^2}{i} \right)\sum_{j=1}^{n} \Eb\big[\Vert w_{j-1}^{(t)}  - w^{(t)}_0 \Vert^2 \big] \\
&& + {~} 2C_2 \Eb \Big[\Vert \frac{1}{i}\sum_{j=1}^i \big[ \nabla{F}_{\pi^{(t)}(j)}(w_0^{(t)})  - \nabla{F}(w_0^{(t)}) \big] \Vert^2 \Big]. 
\end{array}
\end{equation*}
By \citep{mishchenko2020random}[Lemma 1] and Assumption~\ref{as:A2}(c), we have 
\begin{equation*} 
\begin{array}{lcl}
\Eb_t[\mathcal{T}_{[i]}]
& \leq & \left( \frac{C_1}{n} + \frac{2C_2 L_{F}^2}{i} \right)\sum_{j=1}^{n} \Eb\big[\Vert w_{j-1}^{(t)}  - w^{(t)}_0 \Vert^2 \big] + 2C_2 \frac{n-i}{i(n-1)} \sigma_J^2.
\end{array}
\end{equation*}
Taking the total expectation and noting that $n-i \leq n-1$ as $i \geq 1$, we get the desired estimate. 
\end{proof}
%%% End of proof.

%%%%%%%%%%%%%%%%%%%%%%%%%%%%%%%%%%%%
%%%% B.2. One-Iteration Analysis.
%%%%%%%%%%%%%%%%%%%%%%%%%%%%%%%%%%%%
\beforesubsec
\subsection{One-iteration analysis of Algorithm~\ref{alg:SGM1}: Key lemmas}\label{apdx:subsec:SGM1_one_iter_analysis}
\aftersubsec
The update of $w_i^{(t)}$ in Algorithm~\ref{alg:SGM1} can be written as 
\begin{equation}\label{eq:wt_i_update}
\arraycolsep=0.2em
\begin{array}{lcl}
w_{i}^{(t)} = w_0^{(t)} - \frac{\eta_t}{n}\sum_{j=1}^{i}  \widetilde{\nabla}{\Phi}_{\gamma}(w_{j-1}^{(t)}) = \widetilde{w}_{t-1} -  \frac{\eta_t}{n}\sum_{j=1}^{i}  \widetilde{\nabla}{\Phi}_{\gamma}(w_{j-1}^{(t)}),
\end{array}
\end{equation}
for $i \in [n]$, and $\widetilde{w}_t := \prox_{\eta_tf}(w_n^{(t)})$.

For simplicity of our proof, we also denote by $C_1 := \frac{2M_F^4\norms{K}^4}{(\mu_h + \gamma)^2}$ and $C_2 := 2M_h^2\norms{K}^2$.
Using the expression \eqref{eq:wt_i_update}, we can prove the following two lemmas.

%%% Lemma 8.
\begin{lemma}\label{le:ShGD_main_est1}
Let $\{w_i^{(t)} \}$ be generated by Algorithm~\ref{alg:SGM1}.
If $\left( 2 C_1  +  4C_2L_F^2\right)\eta_t^2 \leq \frac{1}{2}$, then we have  
\begin{equation}\label{eq:ShGD_main_est1}
\arraycolsep=0.2em
\begin{array}{lcl}
\Delta_t := \frac{1}{n} \sum_{i=1}^n\norms{w_{i}^{(t)} - w_0^{(t)}}^2 \leq  4\eta_t^2 \big[   \Vert \nabla{\Phi_{\gamma}}(w^{(t)}_0) \Vert^2 + 2C_2\sigma_J^2 \big].
\end{array}
\end{equation}
If $\pi^{(t)} $ and $\hat{\pi}^{(t)}$ are two random permutations of $[n] := \sets{1, 2, \cdots, n}$, then 
\begin{equation}\label{eq:ShGD_main_est1_rand}
\arraycolsep=0.2em
\begin{array}{lcl}
\widetilde{\Delta}_t := \frac{1}{n} \sum_{i=1}^n\Eb \big[ \norms{w_{i}^{(t)} - w_0^{(t)}}^2 \big] \leq  4\eta_t^2 \big[  \Eb \big[\Vert \nabla{\Phi_{\gamma}}(w^{(t)}_0) \Vert^2 \big] + \frac{ 2C_2\sigma_J^2 }{ n }  \big].
\end{array}
\end{equation}
\end{lemma}

%%% Proof of Lemma 8.
\begin{proof}
Using \eqref{eq:wt_i_update} and then \eqref{eq:grad_Phi_bound2}, we can first derive that
\begin{equation}\label{eq:ShGD_proof_1}
\begin{array}{lcl}
\Vert w_{i}^{(t)} - w_0^{(t)}\Vert^2  & = & \frac{\eta_t^2\cdot i^2}{n^2} \Vert \frac{1}{i}\sum_{j=1}^{i}\widetilde{\nabla}\Phi_{\gamma}(w_{j-1}^{(t)}) \Vert^2 \vspace{1ex} \\
& \leq & \frac{2 \eta_t^2 \cdot i^2}{n^2} \Vert \frac{1}{i}\sum_{j=1}^{i} \big[ \widetilde{\nabla}\Phi_{\gamma}(w_{j-1}^{(t)}) - \nabla{\Phi_{\gamma}}(w^{(t)}_0)\big] \Vert^2 + \frac{2 \eta_t^2 \cdot i^2}{n^2}  \Vert \nabla{\Phi_{\gamma}}(w^{(t)}_0) \Vert^2 \\
& \leq & \frac{2 \eta_t^2 \cdot i^2}{n^2} \left(\frac{C_1}{n}+ \frac{2C_2L_F^2}{i}\right)\sum_{j=1}^{n}\Vert w_{j-1}^{(t)}  - w^{(t)}_0 \Vert^2 \\
&& +  {~} \frac{2 \eta_t^2 \cdot i^2}{n^2}  \frac{2n C_2 \sigma_J^2}{i} + \frac{2 \eta_t^2 \cdot i^2}{n^2}  \Vert \nabla{\Phi_{\gamma}}(w^{(t)}_0) \Vert^2 \\
&\leq & \eta_t^2 \left( \frac{2 C_1\cdot i^2}{n^3} + \frac{2C_2L_F^2 \cdot i}{n^2} \right)\sum_{j=1}^{n}\Vert w_{j-1}^{(t)}  - w^{(t)}_0 \Vert^2 \\
&& + {~}  \frac{4 C_2 \sigma_J^2\eta_t^2 \cdot i}{n}+ \frac{2 \eta_t^2 \cdot i^2}{n^2}  \Vert \nabla{\Phi_{\gamma}}(w^{(t)}_0) \Vert^2.
\end{array}
\end{equation}
Let us denote $\Delta_t := \frac{1}{n} \sum_{i=1}^n\Vert w_{i-1}^{(t)} - w_0^{(t)}\Vert^2$.
Then, from \eqref{eq:ShGD_proof_1}, we have
\begin{equation*} 
\arraycolsep=0.1em
\begin{array}{lcl}
\Delta_t &:= & \frac{1}{ n } \sum_{i=1}^n\Vert w_{i}^{(t)} - w_0^{(t)}\Vert^2 \\
& \leq & \frac{1}{ n } \sum_{i=1}^n\Big[ \eta_t^2 \left( \frac{2 C_1\cdot i^2}{n^3} + \frac{2C_2L_F^2 \cdot i}{n^2} \right)\sum_{j=1}^{n}\Vert w_{j-1}^{(t)}  - w^{(t)}_0 \Vert^2 +  \frac{4 C_2 \sigma_J^2\eta_t^2 \cdot i}{n}+ \frac{2 \eta_t^2 \cdot i^2}{n^2}  \Vert \nabla{\Phi_{\gamma}}(w^{(t)}_0) \Vert^2 \Big]  \\
& \leq &  \eta_t^2 \left( \frac{2 C_1\cdot \sum_{i=1}^n i^2}{n^3} + \frac{2C_2L_F^2 \cdot \sum_{i=1}^n i}{n^2} \right) \frac{1}{n} \sum_{j=1}^{n}\Vert w_{j-1}^{(t)}  - w^{(t)}_0 \Vert^2 \\
&& + {~}  \frac{4 C_2 \sigma_J^2\eta_t^2 \cdot  \sum_{i=1}^n i}{ n^2 }+ \frac{2 \eta_t^2 \cdot \sum_{i=1}^n i^2}{ n^3 }  \Vert \nabla{\Phi_{\gamma}}(w^{(t)}_0) \Vert^2 \\
& \leq &  \eta_t^2\left( 2C_1   +   4C_2 L_F^2   \right) \Delta_t  + 4  C_2\sigma_J^2 \eta_t^2 + 2  \eta_t^2   \Vert \nabla{\Phi_{\gamma}}(w^{(t)}_0) \Vert^2.
\end{array}
\end{equation*}
Here, we have used $\sum_{i=1}i^2 = \frac{n(n+1)(2n+1)}{6} \leq n^3$, $\sum_{i=1}^ni = \frac{n(n+1)}{2} \leq n^2$ in the last inequality.
Under the condition $\eta_t^2\left( 2 C_1  +  4C_2L_F^2\right) \leq \frac{1}{2}$, we obtain \eqref{eq:ShGD_main_est1} from the last inequality.

If $\pi^{(t)} $ and $\hat{\pi}^{(t)}$ are two random permutations of $[n] := \{1, 2, \cdots, n\}$ using similar argument with \eqref{eq:wt_i_update} and then, with \eqref{eq:grad_Phi_bound2_rand} we have: 
\begin{equation}\label{eq:ShGD_proof_1_rand}
\begin{array}{lcl}
\Eb \big[\Vert w_{i}^{(t)} - w_0^{(t)}\Vert^2 \big] & = & \frac{\eta_t^2\cdot i^2}{n^2} \Eb \big[\Vert \frac{1}{i}\sum_{j=1}^{i}\widetilde{\nabla}\Phi_{\gamma}(w_{j-1}^{(t)}) \Vert^2 \big] \\
& \leq & \frac{2 \eta_t^2 \cdot i^2}{n^2} \Eb \big[ \Vert \frac{1}{i}\sum_{j=1}^{i} [ \widetilde{\nabla}\Phi_{\gamma}(w_{j-1}^{(t)}) - \nabla{\Phi_{\gamma}}(w^{(t)}_0) ] \Vert^2 \big] \\
&& + {~} \frac{2 \eta_t^2 \cdot i^2}{n^2}  \Eb \big[ \Vert \nabla{\Phi_{\gamma}}(w^{(t)}_0) \Vert^2 \big] \\
& \leq & \frac{2 \eta_t^2 \cdot i^2}{n^2} \Big( \frac{C_1}{n}+ \frac{2C_2L_F^2}{i} \Big) \sum_{j=1}^{n} \Eb \big[\Vert w_{j-1}^{(t)}  - w^{(t)}_0 \Vert^2 \big]  \\
&& + {~}  \frac{2 \eta_t^2 \cdot i^2}{n^2}  \frac{2C_2 \sigma_J^2}{i} + \frac{2 \eta_t^2 \cdot i^2}{n^2} \Eb \big[ \Vert \nabla{\Phi_{\gamma}}(w^{(t)}_0) \Vert^2 \big] \\
&\leq & \eta_t^2 \Big( \frac{2 C_1\cdot i^2}{n^3} + \frac{2C_2L_F^2 \cdot i}{n^2} \Big) \sum_{j=1}^{n} \Eb \big[\Vert w_{j-1}^{(t)}  - w^{(t)}_0 \Vert^2 \big]  \\
&& + {~}  \frac{4 C_2 \sigma_J^2\eta_t^2 \cdot i}{n^2}+ \frac{2 \eta_t^2 \cdot i^2}{n^2}  \Eb \big[ \Vert \nabla{\Phi_{\gamma}}(w^{(t)}_0) \Vert^2 \big].
\end{array}
\end{equation}
Let us denote $\widetilde{\Delta}_t := \frac{1}{n} \sum_{i=1}^n \Eb \big[ \Vert w_{i-1}^{(t)} - w_0^{(t)}\Vert^2 \big]$.
Then, from \eqref{eq:ShGD_proof_1}, we have
\begin{equation*} 
\begin{array}{lcl}
\widetilde{\Delta}_t &:= & \frac{1}{ n } \sum_{i=1}^n \Eb\big[\Vert w_{i}^{(t)} - w_0^{(t)}\Vert^2 \big] \\
& \leq & \frac{1}{ n } \sum_{i=1}^n\Big[ \eta_t^2 \Big( \frac{2 C_1\cdot i^2}{n^3} + \frac{2C_2L_F^2 \cdot i}{n^2} \Big) \sum_{j=1}^{n} \Eb \big[\Vert w_{j-1}^{(t)}  - w^{(t)}_0 \Vert^2 \big]  \\
&& \qquad\qquad + {~}  \frac{4 C_2 \sigma_J^2\eta_t^2 \cdot i}{n^2}+ \frac{2 \eta_t^2 \cdot i^2}{n^2}  \Eb\big[\Vert \nabla{\Phi_{\gamma}}(w^{(t)}_0) \Vert^2 \big] \Big]  \\
& \leq &  \eta_t^2 \left( \frac{2 C_1\cdot \sum_{i=1}^n i^2}{n^3} + \frac{2C_2L_F^2 \cdot \sum_{i=1}^n i}{n^2} \right) \frac{1}{n} \sum_{j=1}^{n} \Eb \big[\Vert w_{j-1}^{(t)}  - w^{(t)}_0 \Vert^2 \big]  \\
&& + {~}  \frac{4 C_2 \sigma_J^2\eta_t^2 \cdot  \sum_{i=1}^n i}{n^3}+ \frac{2 \eta_t^2 \cdot \sum_{i=1}^n i^2}{n^3} \cdot \Eb \big[\Vert \nabla{\Phi_{\gamma}}(w^{(t)}_0) \Vert^2 \big] \\
& \leq &  \eta_t^2\left( 2C_1   +   4C_2 L_F^2   \right) \widetilde{\Delta}_t   +  2 \eta_t^2  \Eb\big[ \Vert \nabla{\Phi_{\gamma}}(w^{(t)}_0) \Vert^2 \big] + \frac{ 4 C_2\sigma_J^2 \eta_t^2 }{ n }.
\end{array}
\end{equation*}
Using similar arguments as before that $\sum_{i=1}i^2 = \frac{n(n+1)(2n+1)}{6} \leq n^3$, $\sum_{i=1}^ni = \frac{n(n+1)}{2} \leq n^2$ and 
 $\eta_t^2\left( 2 C_1  +  4C_2L_F^2\right) \leq \frac{1}{2}$, we obtain \eqref{eq:ShGD_main_est1_rand} from the last inequality.
\end{proof}
%%% End of Proof.

%%% Lemma 9.
\begin{lemma}\label{la:phi_bound}
Let $\{(w_i^{(t)}, \widetilde{w}_t) \}$ be generated by Algorithm~\ref{alg:SGM1}.
Then, we have 
\begin{equation}\label{eq:ShGD_main_est2}
\begin{array}{lcl}
\Psi_{\gamma}(\widetilde{w}_t) &\leq & \Psi_{\gamma}(\widetilde{w}_{t-1})  - \frac{\eta_t(1 - 2L_{\Phi_{\gamma} }\eta_t)}{2}  \Vert {\Gc}_{\eta_t}(\widetilde{w}_{t-1})  \Vert^2 - \frac{(1-L_{\Phi_{\gamma}}\eta_t)}{2\eta_t} \norms{ \widetilde{w}_t  - \widetilde{w}_{t-1} }^2 \\
&& + {~}   \frac{L_{\Psi_{\gamma} } \cdot \eta_t }{2n}\sum_{i=1}^n\Vert w_{i-1}^{(t)} - w_0^{(t)} \Vert^2,
\end{array}
\end{equation}
where $L_{\Psi_{\gamma}} := \frac{2M_F^4\norms{K}^4 }{(\mu_h + \gamma)^2}  + 4M_h^2\norms{K}^2L_F^2$ and $L_{\Phi_{\gamma}}$ is given in Lemma~\ref{le:smoothness_of_phi}.
\end{lemma}

%%% The proof of Lemma 9.
\begin{proof}
The proof of this lemma is adopted from the proof of  \cite[Theorem~3]{mishchenko2022proximal} with some modification. 
First, we denote $\widehat{w}_t := \prox_{\eta_tf}\big(\widetilde{w}_{t-1} - \eta_t \nabla{\Phi}_{\gamma}(\widetilde{w}_{t-1})\big)$.
Then, from \eqref{eq:upper_level_min_grad_mapping}, we have  $\Gc_{\eta_t}(\widetilde{w}_{t-1}) = \frac{1}{\eta_t}(\widetilde{w}_{t-1} - \widehat{w}_t)$.
Moreover, we also have $\nabla{f}(\widehat{w}_t) := \eta_t^{-1}\big( \widetilde{w}_{t-1} - \widehat{w}_t \big) - \nabla{\Phi}_{\gamma}(\widetilde{w}_{t-1}) \in \partial{f}(\widehat{w}_t)$.

Next, by the convexity of $f$, we can easily show that
\begin{equation*}
\begin{array}{lcl}
f(\widehat{w}_t) & \leq &  f(\widetilde{w}_{t-1}) +  \iprods{\nabla{f}(\widehat{w}_t), \widehat{w}_t - \widetilde{w}_{t-1}} \\
& = &  f(\widetilde{w}_{t-1}) -  \iprods{\nabla{\Phi}_{\gamma}(\widetilde{w}_{t-1}), \widehat{w}_t - \widetilde{w}_{t-1}} -  \frac{1}{\eta_t} \norms{  \widehat{w}_t -  \widetilde{w}_{t-1} }^2.
\end{array}
\end{equation*}
Next, by the $L_{\Phi_{\gamma}}$-smoothness of $\Phi_{\gamma}$ from \eqref{eq:Lsmooth_Phi_gamma} of Lemma~\ref{le:smoothness_of_phi}, we have
\begin{equation*} 
\begin{array}{lcl}
\Phi_{\gamma}(\widehat{w}_t ) & \leq & \Phi_{\gamma}(\widetilde{w}_{t-1}) + \langle \nabla{\Phi}_{\gamma}(\widetilde{w}_{t-1}),  \widehat{w}_t  - \widetilde{w}_{t-1} \rangle + \frac{L_{\Phi_{\gamma}}}{2}\Vert \widehat{w}_t  - \widetilde{w}_{t-1} \Vert^2.
\end{array}
\end{equation*}
Adding the last two inequalities together and using $\Psi_{\gamma}(w) = f(w) + \Phi_{\gamma}(w)$ and $\widehat{w}_t -  \widetilde{w}_{t-1} = - \eta_t {\Gc}_{\eta_t}(\widetilde{w}_{t-1})$, we have 
\begin{equation}\label{eq:lm2_proof1} 
\hspace{-0.5ex}
\begin{array}{lcl}
\Psi_{\gamma}(\widehat{w}_t ) & \leq & \Psi_{\gamma}(\widetilde{w}_{t-1})  - \frac{(2 - L_{\Phi_{\gamma}}\eta_t)}{2\eta_t}   \norms{  \widehat{w}_t -  \widetilde{w}_{t-1} }^2 =  \Psi_{\gamma}(\widetilde{w}_{t-1})  - \frac{\eta_t(2 - L_{\Phi_{\gamma}}\eta_t)}{2}   \norms{  {\Gc}_{\eta_t}(\widetilde{w}_{t-1}) }^2.
\end{array}
\hspace{-2ex}
\end{equation}
Now, let us denote $ g_t := \frac{1}{n}\sum_{i=0}^{n}\widetilde{\nabla}{\Phi}_{\gamma}(w_i^{(t)})$.
Then, from the update of $w_i^{(t)}$, we have
\begin{equation*}
\arraycolsep=0.2em
\begin{array}{lcl}
g_t =  \frac{1}{\eta_t}(\widetilde{w}_{t-1}- w_n^{(t)}) =  \frac{1}{ \eta_t }( w_0^{(t)} - w_n^{(t)}).
 \end{array}
 \end{equation*}
Since $\widetilde{w}_t = \prox_{\eta_tf}(w_n^{(t)})$, we have $\nabla{f}(\widetilde{w}_t) := \eta_t^{-1} \big( w_n^{(t)} - \widetilde{w}_t \big) = - g_t - \eta_t^{-1}(\widetilde{w}_t - \widetilde{w}_{t-1}) \in \partial{f}(\widetilde{w}_t)$.
Hence, by the convexity of $f$, we have
\begin{equation*}
\begin{array}{lcl}
f(\widetilde{w}_t) & \leq &  f(\widehat{w}^{t}) +  \iprods{\nabla{f}(\widetilde{w}_t), \widetilde{w}_t - \widehat{w}^{t}}  =  f(\widehat{w}^{t}) -  \iprods{ g_t, \widetilde{w}_t - \widehat{w}^{t}}  - \frac{1}{\eta_t}\iprods{\widetilde{w}_{t} - \widetilde{w}_{t-1}, \widetilde{w}_t - \widehat{w}^{t}} \\
& = & f(\widehat{w}^{t})  - \iprods{ g_t, \widetilde{w}_t - \widehat{w}^{t}} + \frac{1}{2\eta_t}\big[ \norms{\widehat{w}^{t} - \widetilde{w}_{t-1}}^2 - \norms{\widetilde{w}_t - \widetilde{w}_{t-1}}^2 - \norms{ \widetilde{w}_t - \widehat{w}^{t}}^2 \big].
\end{array}
\end{equation*}
Again, by the $L_{\Phi_{\gamma}}$-smoothness of $\Phi_{\gamma}$ from \eqref{eq:Lsmooth_Phi_gamma} of Lemma~\ref{le:smoothness_of_phi}, we also have
\begin{equation*} 
\begin{array}{lcl}
\Phi_{\gamma}(\widetilde{w}_t ) & \leq & \Phi_{\gamma}(\widetilde{w}_{t-1}) + \langle \nabla{\Phi}_{\gamma}(\widetilde{w}_{t-1}),  \widetilde{w}_t  - \widetilde{w}_{t-1} \rangle + \frac{L_{\Phi_{\gamma}}}{2}\Vert \widetilde{w}_t  - \widetilde{w}_{t-1} \Vert^2, \\
\Phi_{\gamma}(\widetilde{w}_{t-1} ) & \leq & \Phi_{\gamma}(\widehat{w}^{t}) + \langle \nabla{\Phi}_{\gamma}(\widetilde{w}_{t-1}),  \widetilde{w}_{t-1}  - \widehat{w}^{t} \rangle + \frac{L_{\Phi_{\gamma}}}{2} \Vert \widehat{w}_t  - \widetilde{w}_{t-1} \Vert^2.
\end{array}
\end{equation*}
Adding the last three inequalities together, and using $\Psi_{\gamma}(w) = f(w) + \Phi_{\gamma}(w)$ and $\widehat{w}_t -  \widetilde{w}_{t-1} = - \eta_t {\Gc}_{\eta_t}(\widetilde{w}_{t-1})$, we have 
\begin{equation}\label{eq:lm2_proof2} 
\begin{array}{lcl}
\Psi_{\gamma}(\widetilde{w}_t ) & \leq & \Psi_{\gamma}(\widehat{w}^{t}) +  \iprods{ \nabla{\Phi}_{\gamma}(\widetilde{w}_{t-1}) - g_t,  \widetilde{w}_{t}  - \widehat{w}^{t} } - \frac{(1-L_{\Phi_{\gamma}}\eta_t)}{2\eta_t} \norms{ \widetilde{w}_t  - \widetilde{w}_{t-1} }^2 \\
&& + {~}  \frac{(1 + L_{\Phi_{\gamma}}\eta_t)}{2\eta_t}  \Vert \widehat{w}_t  - \widetilde{w}_{t-1} \Vert^2  - \frac{1}{2\eta_t} \norms{ \widetilde{w}_t - \widehat{w}^{t}}^2 \\
& \leq & \Psi_{\gamma}(\widehat{w}^{t}) +  \frac{\eta_t}{2}\norms{ \nabla{\Phi}_{\gamma}(\widetilde{w}_{t-1}) - g_t }^2 - \frac{(1-L_{\Phi_{\gamma}}\eta_t)}{2\eta_t} \norms{ \widetilde{w}_t  - \widetilde{w}_{t-1} }^2 \\
&& + {~}  \frac{\eta_t(1 + L_{\Phi_{\gamma}}\eta_t)}{2}  \Vert {\Gc}_{\eta_t}(\widetilde{w}_{t-1})  \Vert^2,
\end{array}
\end{equation}
where we have used Young's inequality in the last line as $\iprods{ \nabla{\Phi}_{\gamma}(\widetilde{w}_{t-1}) - g_t,  \widetilde{w}_{t}  - \widehat{w}^{t} } \leq \frac{\eta_t}{2}\norms{\nabla{\Phi}_{\gamma}(\widetilde{w}_{t-1}) - g_t}^2 + \frac{1}{2\eta_t}\norms{\widetilde{w}_{t}  - \widehat{w}^{t}}^2$.

Summing up \eqref{eq:lm2_proof1} and \eqref{eq:lm2_proof2}, we get
\begin{equation}\label{eq:lm2_proof4} 
\begin{array}{lcl}
\Psi_{\gamma}(\widetilde{w}_t ) & \leq & \Psi_{\gamma}(\widetilde{w}_{t-1})   - \frac{(1-L_{\Phi_{\gamma}}\eta_t)}{2\eta_t} \norms{ \widetilde{w}_t  - \widetilde{w}_{t-1} }^2 -  \frac{\eta_t(1 - 2L_{\Phi_{\gamma}}\eta_t)}{2}  \Vert {\Gc}_{\eta_t}(\widetilde{w}_{t-1})  \Vert^2 \\
&&  + {~}  \frac{\eta_t}{2}\norms{ \nabla{\Phi}_{\gamma}(\widetilde{w}_{t-1}) - g_t }^2.
\end{array}
\end{equation}
Using \eqref{eq:grad_Phi_bound2} with $ g_t =  \frac{1}{n}\sum_{i=0}^{n}\widetilde{\nabla}{\Phi}_{\gamma}(w_i^{(t)})$,  we arrive at
\begin{equation*} 
\begin{array}{lcl}
\Psi_{\gamma}(\widetilde{w}_t ) & \leq & \Psi_{\gamma}(\widetilde{w}_{t-1})   - \frac{(1-L_{\Phi_{\gamma}}\eta_t)}{2\eta_t} \norms{ \widetilde{w}_t  - \widetilde{w}_{t-1} }^2 -  \frac{\eta_t(1 - 2L_{\Phi_{\gamma}}\eta_t)}{2}  \Vert {\Gc}_{\eta_t}(\widetilde{w}_{t-1})  \Vert^2  +  \frac{\eta_t}{2} \cdot \mathcal{T}_{[n]} \\
& \overset{\tiny\eqref{eq:grad_Phi_bound2}}{\leq} & \Psi_{\gamma}(\widetilde{w}_{t-1})   - \frac{(1-L_{\Phi_{\gamma}}\eta_t)}{2\eta_t} \norms{ \widetilde{w}_t  - \widetilde{w}_{t-1} }^2 -  \frac{\eta_t(1 - 2L_{\Phi_{\gamma}}\eta_t)}{2}  \Vert {\Gc}_{\eta_t}(\widetilde{w}_{t-1})  \Vert^2  \\ 
&& + {~}  \frac{\eta_t}{2} \frac{1}{n} \left( C_1 + 2C_2 L_{F}^2 \right)\sum_{j=1}^{n}\Vert w_{j-1}^{(t)}  - w^{(t)}_0 \Vert^2, 
\end{array}
\end{equation*}
which is \eqref{eq:ShGD_main_est2}, where $L_{\Psi_{\gamma} } := C_1 + 2C_2 L_{F}^2  = \frac{2M_F^4\norms{K}^4 }{(\mu_h+\gamma)^2}  + 4M_{h}^2\norms{K}^2L_F^2$.
\end{proof}
%%% End of proof.

%%%% C.2. The proof of Theorem~1 and Lemma 7.
\beforesubsec
\subsection{The proof of Theorem~\ref{th:main_result1} and Corollary~\ref{co:app_KKT_for_nonconvex_linear_case} for Algorithm~\ref{alg:SGM1}}\label{apdx:subsec:SGM1_proof_of_main_results}
\aftersubsec
%%%%%%
Let us recall that $C_1 := \frac{2M_F^4\norms{K}^4}{(\mu_h + \gamma)^2}$, $C_2 := 2M_h^2\norms{K}^2$, and $L_{\Phi_{\gamma}} := M_{h}\norms{K} L_F + \frac{M_F^2\norms{K}^2}{\mu_{h} + \gamma}$ from Lemma~\ref{le:smoothness_of_phi}.
To prove Theorem~\ref{th:main_result1}, we will need the following lemma.

%%% Lemma 3.1.
\begin{lemma}\label{le:main_result1_detail}
Let $\{w_i^{(t)} \}$ be generated by Algorithm~\ref{alg:SGM1} using arbitrarily permutations $\pi^{(t)} = \hat{\pi}^{(t)}$, and $\eta_t = \eta > 0$ such that $\left( 2 C_1  +  4C_2L_F^2\right)\eta^2 \leq \frac{1}{2}$ and  $4L_{\Phi_{\gamma}}\eta +  8L_{\Psi}\Lambda_0\eta^2 \leq 1$.
Then 
\begin{equation}\label{eq:main_result1_detail}
\begin{array}{lcl}
\frac{1}{T+1}\sum_{t=0}^{T}\norms{{\Gc}_{\eta}(\widetilde{w}_{t})}^2 & \leq & \frac{4}{T\eta}\left[ \Psi_{\gamma}(\widetilde{w}^0) - \Psi^{\star}_{\gamma}\right] +  8L_{\Psi}(2C_2\sigma_J^2 + \Lambda_1) \cdot \eta^2.
\end{array}
\end{equation}
Alternatively, if $\pi^{(t)}$ and $\hat{\pi}^{(t)}$ are random permutations and generated independently, then, with a similar condition on $\eta$ as above, we have
\begin{equation}\label{eq:main_result1_rand_detail}
\begin{array}{lcl}
\frac{1}{T+1}\sum_{t=0}^{T}\Eb \left[\norms{{\Gc}_{\eta}(\widetilde{w}_{t})}^2 \right] & \leq & \frac{4}{T\eta}\left[ \Psi_{\gamma}(\widetilde{w}^0) - \Psi^{\star}_{\gamma}\right] +  8L_{\Psi}(2C_2\frac{\sigma_J^2}{n} + \Lambda_1) \cdot \eta^2.
\end{array}
\end{equation}
\end{lemma}

%%% Proof of Theorem 1.
\begin{proof}
From \eqref{eq:ShGD_main_est2}, and note that $L_{\Phi_0}\eta_t \leq 1$, we obtain
\begin{equation*} 
\arraycolsep=0.1em
\begin{array}{lcl}
\Psi_{\gamma}(\widetilde{w}_t) &\leq & \Psi_{\gamma}(\widetilde{w}_{t-1})  - \frac{\eta_t(1 - 2L_{\Phi_{\gamma}}\eta_t)}{2}  \Vert {\Gc}_{\eta_t}(\widetilde{w}_{t-1})  \Vert^2 +   \frac{L_{\Psi} \cdot \eta_t }{2n}\sum_{i=1}^n\Vert w_{i-1}^{(t)} - w_0^{(t)} \Vert^2.
\end{array}
\end{equation*}
Using \eqref{eq:ShGD_main_est1} with the condition $\left( 2 C_1  +  4C_2L_F^2\right)\eta_t^2 \leq \frac{1}{2}$ and \eqref{eq:grad_mapp_bound} of Assumption~\ref{as:A3}, and $w_0^{(t)} = \widetilde{w}_{t-1}$, we have
\begin{equation*} 
\begin{array}{lcl}
\frac{1}{n}\sum_{i=1}^n\norms{w_{i}^{(t)} - w_0^{(t)}}^2 & \leq &  4\eta_t^2 \big[   \Vert \nabla{\Phi_{\gamma}}(w^{(t)}_0) \Vert^2 + 2C_2\sigma_J^2 \big] \\
& \overset{\tiny\eqref{eq:grad_mapp_bound}}{\leq}  & 4\eta_t^2 \left[   \Lambda_0 \Vert {\Gc}_{\eta_t}(\widetilde{w}_{t-1}) \Vert^2 +  2C_2\sigma_J^2 +  \Lambda_1  \right].
\end{array}
\end{equation*}
Combining the two estimates, we obtain
\begin{equation*} 
\begin{array}{lcl}
\Psi_{\gamma}(\widetilde{w}_t) & \leq &  \Psi_{\gamma}(\widetilde{w}_{t-1})  - \frac{\eta_t(1 - 2L_{\Phi_{\gamma}}\eta_t)}{2}  \Vert {\Gc}_{\eta_t}(\widetilde{w}_{t-1})  \Vert^2  + \frac{L_{\Psi} \eta_t}{2} \cdot 4\eta_t^2 \left[    \Lambda_0 \Vert {\Gc}_{\eta_t}(\widetilde{w}_{t-1}) \Vert^2 +  2C_2\sigma_J^2 +  \Lambda_1  \right] \\
& = &  \Psi_{\gamma}(\widetilde{w}_{t-1})  - \frac{\eta_t}{2}\left(1 - 2L_{\Phi_{\gamma}}\eta_t -  4L_{\Psi}\Lambda_0\eta_t^2\right)\Vert  {\Gc}_{\eta_t}(\widetilde{w}_{t-1})  \Vert^2  + 2L_{\Psi}(2C_2\sigma_J^2 + \Lambda_1 \big) \cdot \eta_t^3  \\
& \leq & \Psi_{\gamma}(\widetilde{w}_{t-1})   - \frac{\eta_t}{4} \Vert  {\Gc}_{\eta_t}(\widetilde{w}_{t-1}) \Vert^2  + 2L_{\Psi}(2C_2\sigma_J^2 + \Lambda_1 \big) \cdot \eta_t^3,
\end{array}
\end{equation*}
provided that $4L_{\Phi_{\gamma}}\eta_t +  8L_{\Psi}\Lambda_0\eta_t^2 \leq 1$.
Following the same proof as in \cite[Theorem 3]{nguyen2020unified}, we obtain our bound in \eqref{eq:main_result1_detail}.

For the randomized bound, we take expectation and obtain
\begin{equation*} 
\begin{array}{lcl}
\Eb \left[\Psi_{\gamma}(\widetilde{w}_t)\right] &\leq & \Eb \left[\Psi_{\gamma}(\widetilde{w}_{t-1}) \right] - \frac{\eta_t(1 - 2L_{\Phi_{\gamma}}\eta_t)}{2} \Eb \left[ \Vert {\Gc}_{\eta_t}(\widetilde{w}_{t-1})  \Vert^2 \right] +   \frac{L_{\Psi} \cdot \eta_t }{2n}\sum_{i=1}^n\Eb \big[ \Vert w_{i-1}^{(t)} - w_0^{(t)} \Vert^2 \big].
\end{array}
\end{equation*}
Using \eqref{eq:ShGD_main_est1_rand} with similar argument as the deterministic case, we have
\begin{equation*} 
\begin{array}{lcl}
\frac{1}{n}\sum_{i=1}^n\Eb \big[ \Vert w_{i-1}^{(t)} - w_0^{(t)} \Vert^2 \big] & \leq &  4\eta_t^2 \Big[   \Eb \big[ \Vert \nabla{\Phi_{\gamma}}(w^{(t)}_0) \Vert^2 \big] + 2C_2\frac{\sigma_J^2}{n} \Big] \\
& \overset{\tiny\eqref{eq:grad_mapp_bound}}{\leq} &  4\eta_t^2 \left[   \Lambda_0 \Eb \big[ \Vert {\Gc}_{\eta_t}(\widetilde{w}_{t-1}) \Vert^2 \big] +  2C_2\frac{\sigma_J^2}{n}  +  \Lambda_1  \right].
\end{array}
\end{equation*}
Combining the last two estimates, we get
\begin{equation*} 
\begin{array}{lcl}
\Eb \left[\Psi_{\gamma}(\widetilde{w}_t)\right] &\leq & \Eb \left[\Psi_{\gamma}(\widetilde{w}_{t-1}) \right] - \frac{\eta_t(1 - 2L_{\Phi_{\gamma} }\eta_t)}{2} \Eb \left[ \Vert {\Gc}_{\eta_t}(\widetilde{w}_{t-1})  \Vert^2 \right]  \\
&& + {~} \frac{L_{\Psi} \eta_t}{2} \cdot 4\eta_t^2 \left[   \Lambda_0 \Eb \left[\Vert {\Gc}_{\eta_t}(\widetilde{w}_{t-1}) \Vert^2 \right] +  2C_2\frac{\sigma_J^2}{n}  +  \Lambda_1  \right] \\
& = &  \Eb \left[\Psi_{\gamma}(\widetilde{w}_{t-1}) \right] - \frac{\eta_t}{2}\left(1 - 2L_{\Phi_{\gamma}}\eta_t -  4L_{\Psi}\Lambda_0\eta_t^2\right)\Eb \left[ \Vert {\Gc}_{\eta_t}(\widetilde{w}_{t-1})  \Vert^2 \right]  \\
&& + {~} 2L_{\Psi}(2C_2\frac{\sigma_J^2}{n} + \Lambda_1 \big) \cdot \eta_t^3  \\
& \leq & \Eb \left[\Psi_{\gamma}(\widetilde{w}_{t-1}) \right]  - \frac{\eta_t}{4} \Eb \left[ \Vert {\Gc}_{\eta_t}(\widetilde{w}_{t-1})  \Vert^2 \right] + 2L_{\Psi}(2C_2\frac{\sigma_J^2}{n} + \Lambda_1 \big) \cdot \eta_t^3,
\end{array}
\end{equation*}
provided that $4L_{\Phi_{\gamma} }\eta_t +  8L_{\Psi}\Lambda_0\eta_t^2 \leq 1$.
Follow the same proof as in \cite[Theorem 3]{nguyen2020unified}, we can easily get \eqref{eq:main_result1_rand_detail}.
\end{proof}
%%% End of proof.

The following theorem, Theorem~\ref{th:main_result1} is the full version of Theorem~\ref{th:main_result1_short} in the main text, where the learning rate $\eta$ and the number of epochs $T$ are given explicitly.

%%% Theorem 3.1
\begin{theorem}\label{th:main_result1}
Suppose that Assumptions~\ref{as:A0}, \ref{as:A2}, \ref{as:A1}, and \ref{as:A3} holds for the setting (NL) of \eqref{eq:minimax_prob} and 
\begin{equation}\label{eq:main_result1_Q_gamma_quantity} 
\begin{array}{lcl}
Q_{\gamma} :=   \frac{M_F^2\norms{K}^2 }{\mu_h + \gamma} + M_h L_F\norms{K}.
\end{array}
\end{equation}
Let $\sets{\widetilde{w}_t }$ be generated by Algorithm~\ref{alg:SGM1} after $T$ epochs using arbitrarily deterministic permutations $\pi^{(t)}$ and $\hat{\pi}^{(t)}$ and a learning rate $\eta_t = \eta > 0$ such that
\begin{equation}\label{eq:main_result1_lr1}
\begin{array}{lcl}
\eta :=  \frac{\epsilon}{\sqrt{2Q_{\gamma}(4M_h^2\norms{K}^2\sigma_J^2 + \Lambda_1)}} \quad \text{and} \quad 
T :=  \Big\lfloor \frac{16 \sqrt{Q_{\gamma}(4M_h^2\norms{K}^2\sigma_J^2 + \Lambda_1)} \cdot [ \Psi_0(\widetilde{w}_0) - \Psi_0^{\star} + \gamma B_{\phi_0} ] }{\epsilon^3}\Big\rfloor,
\end{array}
\end{equation}
for a given sufficiently small tolerance $\epsilon > 0$ such that $\eta \leq \frac{1}{8Q_{\gamma}}$.
Then, we have 
\begin{equation*}
\begin{array}{ll}
\frac{1}{T+1}\sum_{t=0}^{T}\norms{{\Gc}_{\eta_t}(\widetilde{w}_{t})}^2 \leq \epsilon^2.
\end{array}
\end{equation*}
%%%
Alternatively, if  $\sets{\widetilde{w}_t }$ is generated by Algorithm~\ref{alg:SGM1} after $T$ epochs using two random and independent permutations $\pi^{(t)}$ and $\hat{\pi}^{(t)}$ and a learning rate $\eta_t = \eta > 0$ such that
\begin{equation}\label{eq:main_result1_lr2}
\begin{array}{lcl}
\eta :=  \frac{\sqrt{n}\epsilon}{\sqrt{2Q_{\gamma}(4M_h^2\norms{K}^2\sigma_J^2 + n\Lambda_1)}} \quad \text{and} \quad 
T :=  \Big\lfloor \frac{16 \sqrt{Q_{\gamma}(4M_h^2\norms{K}^2\sigma_J^2 + n \Lambda_1)} \cdot [ \Psi_0(\widetilde{w}_0) - \Psi_0^{\star} + \gamma B_{\phi_0} ] }{\sqrt{n} \epsilon^3}\Big\rfloor,
\end{array}
\end{equation}
for a given sufficiently small tolerance $\epsilon > 0$ such that $\eta \leq \frac{1}{8Q_{\gamma}}$.
Then, we have 
\begin{equation*}
\begin{array}{ll}
\frac{1}{T+1}\sum_{t=0}^{T} \Eb[ \norms{{\Gc}_{\eta_t}(\widetilde{w}_{t})}^2 ] \leq \epsilon^2.
\end{array}
\end{equation*}
\end{theorem}

\begin{proof}[\textbf{Proof of Theorem~\ref{th:main_result1}}]
Recall that $C_1 := \frac{2M_F^4\norms{K}^4}{(\mu_h + \gamma)^2}$ and $C_2 := 2M_h^2\norms{K}^2$, $L_{\Phi_{\gamma}} := M_{h}\norms{K} L_F + \frac{M_F^2\norms{K}^2}{\mu_{h} + \gamma}$,  and $L_{\Psi} := \frac{2M_F^4\norms{K}^4 }{(\mu_h + \gamma)^2}  + 4M_{h}^2\norms{K}^2L_F^2$.
Let us denote by $Q_{\gamma} := \frac{M_F^2\norms{K}^2 }{\mu_h + \gamma} + M_h L_F\norms{K} $ as in Theorem~\ref{th:main_result1}.

In this case, the first conditions $\left( 2 C_1  +  4C_2L_F^2\right)\eta^2 \leq \frac{1}{2}$ and  $4L_{\Phi_{\gamma}}\eta +  8L_{\Psi}\Lambda_0\eta^2 \leq 1$ of Lemma~\ref{le:main_result1_detail} respectively reduce to
\begin{equation*}
\begin{array}{lll}
&  \frac{M_F^4\norms{K}^4 + 2(\mu_h + \gamma)^2 M_h^2\norms{K}^2L_F^2 }{ (\mu_h + \gamma)^2} \cdot \eta^2 \leq \frac{1}{8} \quad \text{and} \quad \\
&   \frac{ M_F^2\norms{K}^2 + (\mu_h + \gamma) M_{h} \norms{ K } L_F }{\mu_h + \gamma } \cdot \eta +  \frac{4( M_F^4\norms{K}^4 + 2(\mu_h + \gamma)^2 M_h^2\norms{K}^2L_F^2 ) }{(\mu_h + \gamma)^2}  \cdot \eta^2 \leq \frac{1}{4}.
\end{array}
\end{equation*}
Since 
\begin{equation*}
\begin{array}{lcl}
2\big( M_F^2\norms{K}^2 + (\mu_h+\gamma) M_{h}\norms{K} L_F \big)^2 & = &  2M_F^4\norms{K}^4 + 2(\mu_h+\gamma)^2 M_{h}^2 \norms{K}^2 L_F^2 \\
&&  + {~} 4(\mu_h+\gamma) M_{h}\norms{ K}^3 L_FM_F^2 \\
& \geq &  M_F^4\norms{K}^4 + 2(\mu_h+\gamma)^2 M_{h}^2 \norms{K}^2 L_F^2, 
\end{array}
\end{equation*}
the last two conditions hold if $0 <  \eta \leq \frac{\mu_h + \gamma}{8(M_F^2\norms{K}^2 + (\mu_h + \gamma) M_h L_F\norms{K})} = \frac{1}{8Q_{\gamma}}$.
Moreover, we also have $L_{\Psi} \leq 2Q_{\gamma}$.

Now, from \eqref{eq:main_result1_detail}, to guarantee $\frac{1}{T+1}\sum_{t=0}^{T}\norms{{\Gc}_{\eta}(\widetilde{w}_{t})}^2 \leq \epsilon^2$, we impose
\begin{equation*} 
\begin{array}{lcl}
\frac{4}{T\eta}\left[ \Psi_{\gamma}(\widetilde{w}^0) - \Psi^{\star}_{\gamma}\right] +  8L_{\Psi}(2C_2\sigma_J^2 + \Lambda_1) \cdot \eta^2 \leq \epsilon^2.
\end{array}
\end{equation*}
Since $0 <  \eta \leq \frac{1}{8Q_{\gamma}}$ and $L_{\Psi} \leq 2Q_{\gamma}$, we can choose $\eta := \frac{1}{2} \min\Big\{ \frac{1}{4Q_{\gamma}}, \frac{\epsilon}{\sqrt{Q_{\gamma}(4M_h^2\norms{K}^2\sigma_J^2 + \Lambda_1)}} \Big\}$.
Hence, the last inequality holds if 
\begin{equation*} 
\begin{array}{lcl}
T \geq 16 \cdot \max\Big\{ \frac{ \sqrt{Q_{\gamma}(4M_h^2\norms{K}^2\sigma_J^2 + \Lambda_1)}}{\epsilon^3}, \frac{4Q_{\gamma}}{\epsilon^2} \Big\} \cdot \left[ \Psi_{\gamma}(\widetilde{w}^0) - \Psi^{\star}_{\gamma}\right] .
\end{array}
\end{equation*}
By Lemma~\ref{le:psi_eta_pro1}(c), we can easily show that  $\Psi_{\gamma}(w) \leq \Psi_0(w) \leq \Psi_{\gamma}(w) + \gamma B_{\phi_0}$ for any $w$, where $B_{\phi_0} := \sup\sets{b(u) : u \in\dom{h}}$.
Hence, we have $\Psi_{\gamma}(\widetilde{w}^0) - \Psi^{\star}_{\gamma} \leq \Psi_0(\widetilde{w}_0) - \Psi_0^{\star} + \gamma B_{\phi_0}$.
Using this condition, we obtain
\begin{equation*} 
\begin{array}{lcl}
T := \Big\lfloor 16 \cdot  \max\Big\{ \frac{\sqrt{Q_{\gamma}(4M_h^2\norms{K}^2\sigma_J^2 + \Lambda_1)}}{\epsilon^3}, \frac{4 Q_{\gamma}}{\epsilon^2} \Big\} \cdot \left[ \Psi_0(\widetilde{w}_0) - \Psi_0^{\star} + \gamma B_{\phi_0} \right] \Big\rfloor.
\end{array}
\end{equation*}
If we choose $\epsilon$ sufficiently small such that the $0 < \epsilon \leq \frac{\sqrt{Q_{\gamma}(4M_h^2\norms{K}^2\sigma_J^2 + \Lambda_1)}}{4Q_{\gamma}}$, then
\begin{equation*}
\begin{array}{lcl}
\eta :=  \frac{\epsilon}{\sqrt{2Q_{\gamma}(4M_h^2\norms{K}^2\sigma_J^2 + \Lambda_1)}} \quad \text{and} \quad 
T :=  \Big\lfloor \frac{16 \sqrt{Q_{\gamma}(4M_h^2\norms{K}^2\sigma_J^2 + \Lambda_1)} \cdot [ \Psi_0(\widetilde{w}_0) - \Psi_0^{\star} + \gamma B_{\phi_0} ] }{\epsilon^3}\Big\rfloor,
\end{array}
\end{equation*}
as shown in \eqref{eq:main_result1_lr1} of Theorem~\ref{th:main_result1}.

If a random shuffling strategy is used, then to guarantee $\frac{1}{T+1}\sum_{t=0}^{T} \Eb[ \norms{{\Gc}_{\eta}(\widetilde{w}_{t})}^2 ] \leq \epsilon^2$, from \eqref{eq:main_result1_rand_detail}, we can impose the following condition 
\begin{equation*} 
\begin{array}{lcl}
\frac{4}{T\eta}\left[ \Psi_{\gamma}(\widetilde{w}^0) - \Psi^{\star}_{\gamma}\right] +  8L_{\Psi}(2C_2\frac{\sigma_J^2}{n} + \Lambda_1) \cdot \eta^2 \leq \epsilon^2.
\end{array}
\end{equation*}
Reasoning the same way as above, we can choose  $\eta := \frac{1}{2} \min\Big\{ \frac{1}{4Q_{\gamma}}, \frac{\sqrt{n} \epsilon}{\sqrt{Q_{\gamma}(4M_h^2\norms{K}^2\sigma_J^2 + n \Lambda_1)}} \Big\}$.
This leads to the choice of $T$ as
\begin{equation*} 
\begin{array}{lcl}
T := \Big\lfloor 16 \cdot \max\Big\{ \frac{ \sqrt{Q_{\gamma}(4M_h^2\norms{K}^2\sigma_J^2 + n \Lambda_1)}}{\sqrt{n} \epsilon^3}, \frac{4Q_{\gamma}}{\epsilon^2} \Big\} \cdot \left[ \Psi_0(\widetilde{w}_0) - \Psi_0^{\star} + \gamma B_{\phi_0} \right] \Big\rfloor.
\end{array}
\end{equation*}
If we choose $\epsilon$ sufficiently small such that the $0 < \epsilon \leq \frac{\sqrt{Q_{\gamma}(4M_h^2\norms{K}^2\sigma_J^2 + n\Lambda_1)}}{4Q_{\gamma}\sqrt{n}}$, then
\begin{equation*}
\begin{array}{lcl}
\eta :=  \frac{\sqrt{n}\epsilon}{\sqrt{2Q_{\gamma}(4M_h^2\norms{K}^2\sigma_J^2 + n\Lambda_1)}} \quad \text{and} \quad 
T :=  \Big\lfloor \frac{16 \sqrt{Q_{\gamma}(4M_h^2\norms{K}^2\sigma_J^2 + n \Lambda_1)} \cdot [ \Psi_0(\widetilde{w}_0) - \Psi_0^{\star} + \gamma B_{\phi_0} ] }{\sqrt{n} \epsilon^3}\Big\rfloor,
\end{array}
\end{equation*}
as shown in \eqref{eq:main_result1_lr2} of Theorem~\ref{th:main_result1}.
\end{proof}
%%% End of proof.

%%% Proof of Lemma 7.
\begin{proof}[\textbf{Proof of Corollary~\ref{co:app_KKT_for_nonconvex_linear_case}}]
(a)~If $h$ is $\mu_h$-strongly convex with $\mu_h > 0$, then we can set $\gamma = 0$, i.e. without using smoothing technique.
Then, we have $Q_{\gamma} := \frac{M_F^2\norms{K}^2}{\mu_h + \gamma} +  M_h L_F\norms{K}$ reduces to $Q_0 := \frac{ M_F^2\norms{K}^2 }{\mu_h} +  M_h L_F\norms{K}$.

If arbitrary permutations $\pi^{(t)}$ and $\hat{\pi}^{(t)}$ are used, then $T$ from  \eqref{eq:main_result1_lr1} reduces to 
\begin{equation*}
\begin{array}{lcl}
T :=  \Big\lfloor \frac{16 \sqrt{Q_0(4M_h^2\norms{K}^2\sigma_J^2 + \Lambda_1)} \cdot [ \Psi_0(\widetilde{w}_0) - \Psi_0^{\star} ] }{\epsilon^3}\Big\rfloor.
\end{array}
\end{equation*}
Note that, each epoch $t \in [T]$ requires either $2n$ (for \textbf{Option 1}) or $n$ (for \textbf{Option 2}) evaluations of $F_i$ and $n$ evaluations of $\nabla{F}_i$.
Hence, Algorithm~\ref{alg:SGM1} requires $\Oc(n\epsilon^{-3})$ evaluations of $F_i$ and $\Oc(n\epsilon^{-3})$ evaluations of $\nabla{F_i}$ to achieve an $\epsilon$-stationary point of \eqref{eq:upper_level_min}.

Alternatively, if $\pi^{(t)}$ and $\hat{\pi}^{(t)}$ are random and independent permutations, then $T$ from  \eqref{eq:main_result1_lr2} reduces to 
\begin{equation*}
\begin{array}{lcl}
T :=  \Big\lfloor \frac{16 \sqrt{Q_0(4M_h^2\norms{K}^2\sigma_J^2 + n \Lambda_1)} \cdot [ \Psi_0(\widetilde{w}_0) - \Psi_0^{\star} ] }{\sqrt{n} \epsilon^3}\Big\rfloor.
\end{array}
\end{equation*}
Clearly, if $\Lambda_1 = \frac{\Gamma}{n}$ for some constant $\Gamma > 0$, then plugging this $\Lambda_1$ into the right-hand side of $T$ above, we can conclude that Algorithm~\ref{alg:SGM1} requires $\Oc(\sqrt{n}\epsilon^{-3})$ evaluations of $F_i$ and $\Oc(\sqrt{n}\epsilon^{-3})$ evaluations of $\nabla{F_i}$ to achieve an $\epsilon$-stationary point of \eqref{eq:upper_level_min}.

(b)~If $h$ is only merely convex, i.e. $\mu_h = 0$, then we have $Q_{\gamma} = \frac{M_F^2\norms{K}^2 }{\gamma} +  M_h L_F\norms{K} = \Oc(\gamma^{-1})$.
Moreover, to obtain an $\epsilon$-stationary point of \eqref{eq:upper_level_min} from a stationary point of its smoothed problem \eqref{eq:smoothed_com_opt_prob}, with a similar proof as of Lemma~\ref{le:kkt_point}, we need to choose $\gamma := \epsilon$.
In this case, we get $Q_{\epsilon} = \Oc(\epsilon^{-1})$.

If arbitrary permutations $\pi^{(t)}$ and $\hat{\pi}^{(t)}$ are used, then $T$ from  \eqref{eq:main_result1_lr1} reduces to 
\begin{equation*}
\begin{array}{lcl}
T :=  \Big\lfloor \frac{16 \sqrt{Q_{\epsilon}(4M_h^2\norms{K}^2\sigma_J^2 + \Lambda_1)} \cdot [ \Psi_0(\widetilde{w}_0) - \Psi_0^{\star} + \epsilon B_{\phi_0} ] }{\epsilon^3}\Big\rfloor = \Oc\Big(\frac{1}{\epsilon^{7/2}} \Big).
\end{array}
\end{equation*}
Hence, Algorithm~\ref{alg:SGM1} requires $\Oc(n\epsilon^{-7/2})$ evaluations of $F_i$ and $\Oc(n\epsilon^{-7/2})$ evaluations of $\nabla{F_i}$ to achieve an $\epsilon$-stationary point of \eqref{eq:upper_level_min}.

Alternatively, if $\pi^{(t)}$ and $\hat{\pi}^{(t)}$ are random and independent permutations, then $T$ from  \eqref{eq:main_result1_lr2} reduces to 
\begin{equation*}
\begin{array}{lcl}
T :=  \Big\lfloor \frac{16 \sqrt{Q_{\epsilon}(4M_h^2\norms{K}^2\sigma_J^2 + n \Lambda_1)} \cdot [ \Psi_0(\widetilde{w}_0) - \Psi_0^{\star} + \epsilon B_{\phi_0} ] }{\sqrt{n} \epsilon^3}\Big\rfloor.
\end{array}
\end{equation*}
Clearly, if  $\Lambda_1 = \frac{\Gamma}{n}$ for some constant $\Gamma > 0$, then plugging this $\Lambda_1$ into the right-hand side of $T$ above, we can conclude that  Algorithm~\ref{alg:SGM1} requires $\Oc(\sqrt{n}\epsilon^{-7/2})$ evaluations of $F_i$ and $\Oc(\sqrt{n}\epsilon^{-7/2})$ evaluations of $\nabla{F_i}$ to achieve an $\epsilon$-stationary point of \eqref{eq:upper_level_min}.
\end{proof}
%%% End of Proof.

\begin{remark}\label{re:total_prox_evaluations}
We note that since each epoch $t$ of Algorithm~\ref{alg:SGM1} requires one evaluation of $\prox_{\eta_ff}$, the total number of $\prox_{\eta_tf}$ evaluations is $T$.
\end{remark}

%%%%%%%%%%%%%%%%%%%%%%%%%%%%%%%%%%%%%%%%%%%%%%%%%%%%%%%%%%%%
%%% End of Supplementary Documents -- Part 2.
%%%%%%%%%%%%%%%%%%%%%%%%%%%%%%%%%%%%%%%%%%%%%%%%%%%%%%%%%%%%
%%%%%%%%%%%%%%%%%%%%%%%%%%%%%%%%%%%%%%%%%%%%%%%%%%
%%% C. Convergence Analysis of Algorithm 2.
%%%%%%%%%%%%%%%%%%%%%%%%%%%%%%%%%%%%%%%%%%%%%%%%%%
\beforesec
\section{Convergence Analysis of Algorithm~\ref{alg:SGM2} -- The NC Setting}\label{apdx:sec:SGM2_convergence}
\aftersec
In this section, we present the full convergence analysis of Algorithm~\ref{alg:SGM2} for both the \textbf{semi-shuffling}  and the \textbf{full-shuffling} variants.

For our notational convenience, we introduce the following function:
\begin{equation}\label{eq:SGM3_psi_func}
\arraycolsep=0.2em
\begin{array}{lcl}
\psi(w, u) & := & - \Hc(w, u) + h(u).
\end{array}
\end{equation}
By Assumption~\ref{as:A3_SGM2}, $\psi(w, \cdot)$ is $\mu_{\psi}$-strongly convex with the strong convexity parameter $\mu_{\psi} := \mu_h + \mu_H > 0$ for any $w$ such that $(w, u) \in \dom{\Lc}$.
Moreover, the Lipschitz constant  $\kappa$ of $u_0^{*}(\cdot)$ in Lemma~\ref{le:property_of_Phi0} becomes $\kappa := \frac{L_u}{\mu_h + \mu_H} = \frac{L_u}{\mu_{\psi}} > 0$.

Furthermore, $\Phi_0$ and $\Psi_0$ defined by \eqref{eq:lower_level_max} and \eqref{eq:upper_level_min}, respectively  can be expressed as 
\begin{equation}\label{eq:SGM3_psi_property1}
\arraycolsep=0.2em
\begin{array}{lcl}
\Phi_0(w) & := & {\displaystyle\max_{u \in \R^q}} \big\{ \Hc(w, u) - h(u) \big\} = - {\displaystyle \min_{u \in \R^q}} \psi(w, u), \\
\Psi_0(w) & := & f(w) + \Phi_0(w) = f(w) + \Hc(w, u_0^{*}(w)) - h(u_0^{*}(w)),
\end{array}
\end{equation}
where $u_0^{*}(w) := \argmin_{u \in\R^q}\psi(w, u)$ is computed by \eqref{eq:lower_level_max}.

\beforesubsec
\subsection{One-epoch analysis: Key lemmas}\label{apdx:subsec:convergence_of_SGM2}
\aftersubsec
We separate the technical lemmas for two variants: the \textit{semi-shuffling variant} using \eqref{eq:GD4lower_prob}, and the \textit{full-shuffling variant} using \eqref{eq:SGM4lower_prob} into two subsections, respectively.

\textbf{(a)~Key bound for the gradient-ascent scheme \eqref{eq:GD4lower_prob}.}
If we apply \eqref{eq:GD4lower_prob} to approximate $u_0^{*}(\widetilde{w}_{t-1})$, then we have the following result.

%%% Lemma 3.
\begin{lemma}\label{le:SGM2_bound_u1_dist1}
Suppose that Assumption~\ref{as:A3_SGM2} holds.
Let $\sets{\widehat{u}^{(t)}_s}$ be updated by \eqref{eq:GD4lower_prob} such that $0 < \hat{\eta}_t \leq \frac{2}{L_u + \mu_H}$.
Then, we have 
\begin{equation}\label{SGM2_bound_u1_dist1}
\begin{array}{lcl}
\norms{\widetilde{u}_t  - u_0^{*}(\widetilde{w}_{t-1})}^2 \leq \frac{1}{(1 + 2\mu_h\hat{\eta}_t)^S}\big(1 - \frac{2L_u\mu_H\hat{\eta}_t}{L_u + \mu_H}\big)^{S} \norms{\widetilde{u}_{t-1}  - u_0^{*}(\widetilde{w}_{t-1})}^2.
\end{array}
\end{equation}
\end{lemma}

%%% Proof of Lemma~
\begin{proof}
The proof of Lemma~\ref{le:SGM2_bound_u1_dist1} is certainly classical and not new. It can be found in the literature, including \cite{Nesterov2004}.
However, it may be inconvenient to find a unified proof for the strong convexity of $\Hc_i$ and $h$ altogether.
Therefore, we present it here for completeness.

For simplicity of our presentation, we denote $\varphi(u) := -\Hc(\widetilde{w}_{t-1}, u) = - \frac{1}{n}\sum_{i=1}^n\Hc_i(\widetilde{w}_{t-1}, u)$ and $u^{*}_t := u_0^{*}(\widetilde{w}_{t-1})$ computed by \eqref{eq:lower_level_max}.

By Assumption~\ref{as:A3_SGM2}, $\varphi$ is $\mu_H$-strongly convex and $L_u$-smooth.
The scheme \eqref{eq:GD4lower_prob} is exactly a proximal gradient method to solve $\min_u\sets{Q(u) := \varphi(u) + h(u)}$, where $h$ is also $\mu_h$-strongly convex.
Moreover, by the definition of $\varphi$ and of $u^{*}_t$, and  \eqref{eq:GD4lower_prob},  it is obvious to show that 
\begin{equation*}
\left\{\begin{array}{lcl}
u^{*}_t & = & \prox_{\hat{\eta}_th}\big(u^{*}_t - \hat{\eta}_t\nabla{\varphi}(u^{*}_t)\big), \\
\widehat{u}^{(t)}_{s} & = & \prox_{\hat{\eta}_th}\big( \widehat{u}^{(t)}_{s-1} - \hat{\eta}_t\nabla{\varphi}(\widehat{u}^{(t)}_{s-1}) \big).
\end{array}\right.
\end{equation*}
Hence, by \eqref{eq:prox_contraction} from Fact [$F_1$], we have
\begin{equation*}
\begin{array}{lcl}
\norms{\widehat{u}^{(t)}_{s} - u^{*}_t}^2 & = & \norms{ \prox_{\hat{\eta}_th}\big( \widehat{u}^{(t)}_{s-1} - \hat{\eta}_t\nabla{\varphi}(\widehat{u}^{(t)}_{s-1}) \big) - \prox_{\hat{\eta}_th}\big(u^{*}_t - \hat{\eta}_t\nabla{\varphi}(u^{*}_t)\big)}^2 \\
& \leq & \frac{1}{1 + 2\mu_h\hat{\eta}_t}\norms{\widehat{u}^{(t)}_{s-1} - u^{*}_t - \hat{\eta}_t[\nabla{\varphi}(\widehat{u}^{(t)}_{s-1}) - \nabla{\varphi}(u^{*}_t)] }^2.
\end{array}
\end{equation*}
Expanding the right-hand side of the last estimate, we get
\begin{equation*}
\begin{array}{lcl}
\norms{\widehat{u}^{(t)}_{s-1} - u^{*}_t - \hat{\eta}_t[\nabla{\varphi}(\widehat{u}^{(t)}_{s-1}) - \nabla{\varphi}(u^{*}_t)] }^2 &= & \norms{\widehat{u}^{(t)}_{s-1} - u^{*}_t}^2 + \hat{\eta}_t^2\norms{\nabla{\varphi}(\widehat{u}^{(t)}_{s-1}) - \nabla{\varphi}(u^{*}_t)}^2 \\
&& - {~} 2\hat{\eta}_t\iprods{ \nabla{\varphi}(\widehat{u}^{(t)}_{s-1}) - \nabla{\varphi}(u^{*}_t), \widehat{u}^{(t)}_{s-1} - u^{*}_t}.
\end{array}
\end{equation*}
Using  \cite[Theorem 2.1.12]{Nesterov2004}, we can show that
\begin{equation*}
\begin{array}{lcl}
\iprods{ \nabla{\varphi}(\widehat{u}^{(t)}_{s-1}) - \nabla{\varphi}(u^{*}_t), \widehat{u}^{(t)}_{s-1} - u^{*}_t} &\geq & \frac{L_u\mu_H}{L_u+\mu_H}\norms{\widehat{u}^{(t)}_{s-1} - u^{*}_t}^2 + \frac{1}{L_u + \mu_H}\norms{\nabla{\varphi}(\widehat{u}^{(t)}_{s-1}) - \nabla{\varphi}(u^{*}_t)}^2.
\end{array}
\end{equation*}
Combining the last three inequalities, we obtain
\begin{equation*}
\begin{array}{lcl}
\norms{\widehat{u}^{(t)}_{s} - u^{*}_t}^2 &\leq & \frac{1}{1 + 2\mu_h\hat{\eta}_t}\big(1 - \frac{2L_u\mu_H\hat{\eta}_t}{L_u + \mu_H}\big)\norms{\widehat{u}^{(t)}_{s-1} - u^{*}_t}^2 \\
&& - {~} \frac{\hat{\eta}_t}{1+2\mu_h\hat{\eta}_t}\big(\frac{2}{L_u+\mu_H} - \hat{\eta}_t\big) \norms{\nabla{\varphi}(\widehat{u}^{(t)}_{s-1}) - \nabla{\varphi}(u^{*}_t)}^2.
\end{array}
\end{equation*}
Therefore, if $0 < \hat{\eta}_t \leq \frac{2}{L_u + \mu_H}$, then the last inequality reduces to
\begin{equation*}
\begin{array}{lcl}
\norms{\widehat{u}^{(t)}_{s} - u^{*}_t}^2 &\leq & \frac{1}{1 + 2\mu_h\hat{\eta}_t}\big(1 - \frac{2L_u\mu_H\hat{\eta}_t}{L_u + \mu_H}\big)\norms{\widehat{u}^{(t)}_{s-1} - u^{*}_t}^2.
\end{array}
\end{equation*}
By induction, and noting that $\widehat{u}^{(t)}_{0} := \widetilde{u}_{t-1}$ and $\widetilde{u}_t := \widehat{u}^{(t)}_{S}$, this inequality implies \eqref{SGM2_bound_u1_dist1}.
\end{proof}
%%% End of proof.

\textbf{(b)~Key bound for the shuffling gradient-ascent scheme \eqref{eq:SGM4lower_prob}.}
Alternatively, if the \textit{full-shuffling variant} \eqref{eq:SGM4lower_prob} is used in Algorithm~\ref{alg:SGM2}, then we can bound $\norms{\widetilde{u}_t - u_0^{*}(\widetilde{w}_{t-1})}^2$ for \eqref{eq:SGM4lower_prob} as follows.

First, let us define $u_0^{s*} := u_0^{*}(\widetilde{w}_{t-1})$ and for all $i \in [n]$: 
\begin{equation}\label{eq:SGM2_ui^star_quantity}
\begin{array}{lcl}
u_i^{s*} & := & u_0^{*}(\widetilde{w}_{t-1}) + \frac{\hat{\eta}_t}{n} \sum_{j=1}^i\nabla_u\mathcal{H}_{\pi^{(s,t)}(j)}(\widetilde{w}_{t-1}, u_0^{*}(\widetilde{w}_{t-1})). 
\end{array}
\end{equation}
Here, $\nabla_u{\Hc}_i$ is the partial derivative (or the gradient) of $\Hc_i$ w.r.t. $u$.

Next, we prove the following lemma.

%%% Lemma 11.
\begin{lemma}\label{le:SGM2_bound_ut_dist2}
Suppose that Assumption~\ref{as:A3_SGM2} holds, and $u_i^{s*}$ is defined by \eqref{eq:SGM2_ui^star_quantity} for all $i = 0, \cdots, n$.
Then
\begin{equation}\label{eq:SGM2_bound_ut_dist2}
\begin{array}{lcl}
\norms{u^{s*}_i - u_0^{*}(\widetilde{w}_{t-1})}^2   & \leq & \frac{ 2  \hat{\eta}_t^2 \cdot i }{ n }   \cdot \big(  \Theta_u  \norms{\nabla{\Phi_0}(\widetilde{w}_{t-1})}^2  +  \sigma_{u}^2 \big) +  \frac{ 2\hat{\eta}_t^2 \cdot i^2 }{ n^2 } \cdot  \norms{\nabla{\Phi_0}(\widetilde{w}_{t-1})}^2 \\
& \leq & 2 \hat{\eta}_t^2\big[ ( \Theta_u + 1)  \norms{\nabla{\Phi_0}(\widetilde{w}_{t-1})}^2 + \sigma_{u}^2 \big].
\end{array}
\end{equation}
\end{lemma}

%%% Proof of Lemma 11.
\begin{proof}
For simplicity, we denote $u^{*}_t := u_0^{*}(\widetilde{w}_{t-1})$.
For $i=0$, we obviously have $\norms{u_0^{s*} - u^{*}_t}^2 = 0$, showing that \eqref{eq:SGM2_bound_ut_dist2} trivially holds.

Next, for $i \in [n]$, using $u_i^{s*}$ from \eqref{eq:SGM2_ui^star_quantity} and Young's inequality twice in $\myeqc{1}$ and $\myeqc{2}$, we can derive that
\begin{equation*} 
\arraycolsep=0.1em
\begin{array}{lcl}
\norms{u_i^{s*} - u^{*}_t }^2 & = &  \frac{\hat{\eta}_t^2}{n^2} \Vert \sum_{j=1}^i\nabla_u\mathcal{H}_{ \pi^{(s)}(j)}(\widetilde{w}_{t-1}, u^{*}_t ) \Vert^2 \\
& \overset{\myeqc{1}}{ \leq } &  \frac{ 2\hat{\eta}_t^2 }{ n^2 }  \cdot i^2 \cdot \Vert \frac{1}{i} \sum_{j=1}^i \big[ \nabla_u\mathcal{H}_{ \pi^{(s)}(j)}(\widetilde{w}_{t-1}, u^{*}_t ) -  \nabla_u\mathcal{H}(\widetilde{w}_{t-1}, u^{*}_t ) \big] \Vert^2 \\
&& + {~}   \frac{ 2\hat{\eta}_t^2 }{ n^2 }  \cdot i^2  \Vert   \nabla_u\mathcal{H}(\widetilde{w}_{t-1}, u^{*}_t) \Vert^2 \\
& \overset{\myeqc{2}}{ \leq } &   \frac{ 2 i \hat{\eta}_t^2  }{ n^2 } \sum_{j=1}^i \big\Vert \nabla_u\mathcal{H}_{\pi^{(s)}(j)}(\widetilde{w}_{t-1},  u^{*}_t ) -  \nabla_u\mathcal{H}(\widetilde{w}_{t-1}, u^{*}_t ) \Vert^2  +    \frac{ 2i^2 \hat{\eta}_t^2 }{ n^2 } \norms{ \nabla_u\mathcal{H}(\widetilde{w}_{t-1}, u^{*}_t)  }^2 \\
& \overset{\tiny\eqref{eq:exact_grad_phi}}{ \leq } &  \frac{ 2 i \hat{\eta}_t^2 }{ n^2 }  \sum_{j=1}^n \big\Vert \nabla_u\mathcal{H}_{\pi^{(s)}(j)}(\widetilde{w}_{t-1}, u^{*}_t ) -  \nabla_u\mathcal{H}(\widetilde{w}_{t-1}, u^{*}_t ) \Vert^2 +  \frac{ 2 i^2 \hat{\eta}_t^2 }{ n^2 } \norms{\nabla{\Phi_0}(\widetilde{w}_{t-1})}^2.
\end{array}
\end{equation*}
By \eqref{eq:bounded_variance2} from Assumption~\ref{as:A3_SGM2} and \eqref{eq:exact_grad_phi}, we have
\begin{equation*}
\begin{array}{lcl}
\frac{1}{n} \sum_{j=1}^n \big\Vert \nabla_u\mathcal{H}_{\pi^{(s)}(j)}(\widetilde{w}_{t-1}, u^{*}_t ) -  \nabla_u\mathcal{H}(\widetilde{w}_{t-1}, u^{*}_t ) \Vert^2  & \overset{\tiny\eqref{eq:bounded_variance2}}{\leq} & \Theta_u \norms{ \nabla_u\mathcal{H}(\widetilde{w}_{t-1}, u^{*}_t )}^2 + \sigma_{u}^2 \\
& \overset{\tiny\eqref{eq:exact_grad_phi}}{ = } & \Theta_u \norms{\nabla{\Phi_0}(\widetilde{w}_{t-1})}^2 + \sigma_{u}^2. 
\end{array}
\end{equation*}
Combining the last two inequalities, and noting that $0 \leq i \leq n$, we obtain \eqref{eq:SGM2_bound_ut_dist2}.
\end{proof}
%%% End of Proof.

Finally, we can prove the necessary bound for $\norms{\widetilde{u}_t - u_0^{*}(\widetilde{w}_{t-1})}^2$.
For simplicity of our proof, let us denote $g^{s,t}_{i-1}(\cdot) := -\mathcal{H}_{\pi^{(s)}(i)}(\widetilde{w}_{t-1}, \cdot)$ and again $u^{*}_t := u_0^{*}(\widetilde{w}_{t-1})$.
By Assumption~\ref{as:A3_SGM2}(a) and (b), it is clear that $g^{s,t}_{i-1}(\cdot)$ is $\mu_H$-strongly convex and $L_u$-smooth.
Let us consider the following the Bregman distance constructed from $g_{i-1}^{s,t}$:
\begin{equation}\label{eq:SGM2_key_bounds_for_ut2_breg_dist}
D^{s,t}_{i-1}(u, \hat{u}) = g^{s,t}_{i-1}(u) - g^{s,t}_{i-1}(\hat{u}) - \iprods{ \nabla_u{g^{s,t}_{i-1}}(\hat{u}), u - \hat{u} }.
\end{equation}
The following lemma is adapted from Theorems 2 and 3 in \cite{mishchenko2022proximal} with some modification.

%%% Lemma 3.1.
\begin{lemma}\label{le:SGM2_key_bounds_for_ut2}
Suppose that Assumption~\ref{as:A3_SGM2} holds.
Let $u_i^{s*}$ be defined by \eqref{eq:SGM2_ui^star_quantity}, $\sets{u_i^{(s,t)}}$ be updated by \eqref{eq:SGM4lower_prob} at the $s$-th epoch for all $i \in [n]$, and $D^{s,t}_{i-1}$ be defined by \eqref{eq:SGM2_key_bounds_for_ut2_breg_dist}.
Then, it holds that
\begin{equation}\label{eq:SGM2_key_est_for_ut2}
\hspace{-0.0ex}
\begin{array}{lcl}
\norms{u_i^{(s,t)} - u_i^{s*}}^2 & \leq & \big(1 - \frac{ \mu_H\hat{\eta}_t }{ n }  \big)\norms{u^{(s,t)}_{i-1} - u_{i-1}^{s*}}^2  + \frac{ 2  L_u  \hat{\eta}_t^3}{ n }   \big[  (\Theta_u + 1) \norms{ \nabla{\Phi_0}(\widetilde{w}_{t-1} ) }^2 + \sigma_{u}^2 \big] \\
&& - {~}  \frac{ 2\hat{\eta}_t }{ n } \big(1 -  \frac{ L_u\hat{\eta}_t }{ n }  \big) D^{s,t}_{i-1}( u_{i-1}^{(s,t)}, u_0^{*}(\widetilde{w}_{t-1}) ).
\end{array}
\hspace{-2.0ex}
\end{equation}
Consequently, at each epoch $s$, the following bound holds:
\begin{equation}\label{eq:SGM2_key_est_for_ut2b}
\hspace{-1.0ex}
\begin{array}{lcl}
\norms{\widehat{u}_{s}^{(t)}  - u_0^{*}(\widetilde{w}_{t-1})}^2  & \leq & \frac{1}{1 + 2 \mu_h\hat{\eta}_t } \big(1 - \frac{\mu_H\hat{\eta}_t }{ n } \big)^n  \norms{\widehat{u}_{s-1}^{(t)} - u_0^{*}(\widetilde{w}_{t-1}) }^2 \\
&& + {~} \frac{2 L_u \cdot \hat{\eta}_t^3 }{  n( 1 + 2 \mu_h\hat{\eta}_t) }  \big[ \sum_{j=0}^{n-1}  \big(1 -  \frac{ \mu_H\hat{\eta}_t }{ n }   \big)^{j} \big] \big[ (\Theta_u + 1 ) \norms{\nabla{\Phi_0}(\widetilde{w}_{t-1} )}^2  + \sigma_{u}^2 \big].
\end{array}
\hspace{-5.0ex}
\end{equation}
If we update \eqref{eq:SGM4lower_prob} by $S$ epochs starting from $\widehat{u}_0^{(t)}  := \widetilde{u}_{t-1}$ and output $\widetilde{u}_t := \widehat{u}_S^{(t)} $, then
\begin{equation}\label{eq:SGM2_key_est_for_ut2c}
\hspace{-0.0ex}
\begin{array}{lcl}
\norms{\widetilde{u}_t - u_0^{*}(\widetilde{w}_{t-1})}^2 
& \leq & \frac{1}{(1 + 2 \mu_h\hat{\eta}_t)^{S} } \big(1 - \frac{ \mu_H\hat{\eta}_t }{ n } \big)^{nS} \norms{\widetilde{u}_{t-1} - u_0^{*}(\widetilde{w}_{t-1}) }^2 \\
&& + {~} \frac{ 2 L_u}{ n } C_S \hat{\eta}_t^3  \cdot \big[ (\Theta_u + 1 ) \norms{\nabla{\Phi_0}(\widetilde{w}_{t-1} )}^2  + \sigma_{u}^2 \big],
\end{array}
\hspace{-3.0ex}
\end{equation}
where $C_S :=  \big[ \sum_{j=0}^{n-1}  \frac{ 1 }{(1 + 2\mu_h\hat{\eta}_t)} \big( 1 - \frac{ \mu_H\hat{\eta}_t }{ n }  \big)^j  \big] \sum_{s=0}^{S-1} \frac{1}{(1 + 2 \mu_h\hat{\eta}_t)^s} \big(1 - \frac{ \mu_H\hat{\eta}_t }{ n }  \big)^{ns}$.
\end{lemma}

%%% Proof of Lemma 3.1.
\begin{proof}
By \eqref{eq:SGM4lower_prob}, using the definition of $g_{i-1}^{s,t}(\cdot)$ above, and $u_i^{s*}$ defined by \eqref{eq:SGM2_ui^star_quantity}, we have 
\begin{equation*}
\arraycolsep=0.2em
\begin{array}{lcl}
u_i^{(s,t)} & = & u_0^{(s,t)} +  \frac{ \hat{\eta}_t }{ n}  \sum_{j=1}^i\nabla_u\mathcal{H}_{\pi^{(s)}(j)}(\widetilde{w}_{t-1}, u_{j-1}^{(s,t)}) = u_0^{(s,t)} -  \frac{ \hat{\eta}_t }{ n} \sum_{j=1}^i \nabla_u{g^{s,t}_{j-1}}(u_{j-1}^{(s,t)}) \\
& = & u_{i - 1} ^{(s,t)} - \frac{\hat{\eta}_t }{ n } \nabla_u{g^{s,t}_{i-1}}(u_{i-1}^{(s,t)}) \\
u_i^{s*} &= & u^{*}_0(\widetilde{w}_{t-1}) -  \frac{\hat{\eta}_t }{ n } \sum_{j=1}^{i} \nabla_u{g^{s,t}_{j-1}}(u_t^{*})  =  u_{i-1}^{s*} - \frac{\hat{\eta}_t }{ n } \nabla_u{g^{s,t}_{i-1}}(u_t^{*}).
\end{array}
\end{equation*}
Using these expressions, for any $i \in [n]$, we can show that
\begin{equation}\label{eq:lm31_proof1}
\begin{array}{lcl}
\norms{u_i^{(s,t)} - u_i^{s*}}^2 &= & \norms{u^{(s,t)}_{i-1} - u_{i-1}^{s*}}^2 - \frac{ 2\hat{\eta}_t }{ n}    \iprods{ \nabla_u{g^{s,t}_{i-1}}(u_{i-1}^{(s,t)}) - \nabla_u{g^{s,t}_{i-1}}(u^{*}_t),  u_{i-1}^{(s,t)} - u_{i-1}^{s*}} \\
&& + {~}  \frac{ \hat{\eta}_t^2 }{ n^2 } \Vert \nabla_u{g^{s,t}_{i-1}}(u_{i-1}^{(s,t)}) - \nabla_u{g^{s,t}_{i-1}}(u^{*}_t) \Vert^2.
\end{array}
\end{equation}
By the $L$-smoothness condition \eqref{eq:Lsmooth2} of $\mathcal{H}_i$ from Assumption~\ref{as:A3_SGM2}, we have 
\begin{equation}\label{eq:lm31_proof2}
\begin{array}{lcl}
\Vert \nabla_u{g^{s,t}_{i-1}}(u_{i-1}^{(s,t)}) - \nabla_u{g^{s,t}_{i-1}}(u^{*}_t) \Vert^2 & \leq & 2L_u D^{s,t}_{i-1}( u_{i-1}^{(s,t)}, u^{*}_t).
\end{array}
\end{equation}
By the well-known three-point identity, see, e.g., \cite{Chen1994}, we have
\begin{equation*}
\begin{array}{lcl}
\iprods{ \nabla_u{g^{s,t}_{i-1}}(u_{i-1}^{(s,t)}) - \nabla_u{g^{s,t}_{i-1}}(u^{*}_t),  u_{i-1}^{(s,t)} - u_{i-1}^{s*}} & = &  D^{s,t}_{i-1}( u_{i-1}^{(s,t)}, u^{s*}_{i-1}) + D^{s,t}_{i-1}( u_{i-1}^{(s,t)}, u^{*}_t) \\
&& - {~} D^{s,t}_{i-1}( u_{i-1}^{s*}, u^{*}_t).
\end{array}
\end{equation*}
Substituting this inequality and \eqref{eq:lm31_proof2} into \eqref{eq:lm31_proof1}, we can show that
\begin{equation*}
\begin{array}{lcl}
\norms{u_i^{(s.t)} - u_i^{s*}}^2 & \leq & \norms{u^{(s,t)}_{i-1} - u_{i-1}^{s*}}^2 - \frac{ 2\hat{\eta}_t }{ n }  D^{s,t}_{i-1}( u_{i-1}^{(s,t)}, u^{s*}_{i-1}) + \frac{ 2\hat{\eta}_t  }{ n }  D^{s,t}_{i-1}( u_{i-1}^{s*}, u^{*}_t) \\
&& - {~} \frac{ 2\hat{\eta}_t }{ n }  \big( 1 - \frac{ L_u\hat{\eta}_t }{ n } \big) D^{s,t}_{i-1}( u_{i-1}^{(s,t)}, u^{*}_t).
\end{array}
\end{equation*}
On the one hand, by the $\mu_H$-strong-convexity of $g_{i-1}^{s,t}$, we have $D^{s,t}_{i-1}( u_{i-1}^{(s,t)}, u^{s*}_{i-1})  \geq \frac{\mu_H}{2}\norms{u^{(s,t)}_{i-1} - u_{i-1}^{s*}}^2$.
On the other hand, by the $L_u$-smoothness of $g_{i-1}^{s,t}$, we also have $D^{s,t}_{i-1}( u_{i-1}^{s*}, u^{*}_t) \leq \frac{L_u}{2}\norms{u_{i-1}^{s*} -  u^{*}_t}^2$.
Using these bounds into the last inequality, we can show that
\begin{equation*}
\begin{array}{lcl}
\norms{u_i^{(s,t)} - u_i^{s*}}^2 & \leq & \big( 1 -  \frac{ \mu_H\hat{\eta}_t }{ n }  \big)\norms{u^{(s,t)}_{i-1} - u_{i-1}^{s*}}^2  +  \frac{ L_u\hat{\eta}_t }{ n }   \norms{u_{i-1}^{s*} -  u^{*}_t}^2 \\
&&  - {~}  \frac{ 2\hat{\eta}_t }{ n }  \big( 1 -   \frac{ L_u\hat{\eta}_t }{ n } \big) D^{s,t}_{i-1}( u_{i-1}^{(s,t)}, u^{*}_t).
\end{array}
\end{equation*}
Combining this inequality and \eqref{eq:SGM2_bound_ut_dist2}, we obtain \eqref{eq:SGM2_key_est_for_ut2}.

%%%%%
Next, since $1 - \frac{ L_u\hat{\eta}_t }{ n }  \geq 0$ and $D^{s,t}_{i-1}( u_{i-1}^{(s,t)}, u^{*}_t) \geq 0$, we obtain from \eqref{eq:SGM2_key_est_for_ut2} that
\begin{equation*} 
\begin{array}{lcl}
\norms{u_i^{(s,t)} - u_i^{s*}}^2 & \leq & \big( 1 - \frac{ \mu_H\hat{\eta}_t }{ n }  \big) \norms{u^{(s,t)}_{i-1} - u_{i-1}^{s*}}^2  + 2n^2 L_u  \big[ ( \Theta_u + 1) \norms{ \nabla{\Phi_0}(\widetilde{w}_{t-1}) }^2  + \sigma_{u}^2 \big] \cdot \hat{\eta}_t^3.
\end{array}
\end{equation*}
By induction, rolling this inequality from $i=1$ to $n$, we have
\begin{equation}\label{eq:lm31_proof5} 
\begin{array}{lcl}
\norms{u_n^{(s,t)} - u_n^{s*}}^2 & \leq & \big(1 - \frac{ \mu_H\hat{\eta}_t }{ n }  \big)^n \norms{u^{(s,t)}_{0} - u_{0}^{s*}}^2  \\
&& + {~} \frac{ 2  L_u }{ n } [ ( \Theta_u + 1)  \norms{ \nabla{\Phi_0}(\widetilde{w}_{t-1}) }^2  + \sigma_{u}^2  ] \cdot \hat{\eta}_t^3 \cdot \sum_{j=0}^{n-1}  \big(1 - \frac{ \mu_H\hat{\eta}_t}{ n }  \big)^{j}.
\end{array}
\end{equation}
Next, from \eqref{eq:SGM2_ui^star_quantity} and \eqref{eq:lower_level_max_opt_cond}, it is not hard to show that $u^{*}_t = u_0^{*}(\widetilde{w}_{t-1}) = \prox_{\hat{\eta}_th}(u_n^{s*})$.
Furthermore, by the second line of \eqref{eq:SGM4lower_prob}, we also have $\widehat{u}_{s}^{(t)} = \prox_{ \hat{\eta}_th}(u_n^{(s,t)})$.
Since $h$ is $\mu_h$-strongly convex, by \eqref{eq:prox_contraction} from Fact [$F_1$], we can show that
\begin{equation*} 
\begin{array}{lcl}
\norms{\widehat{u}_{s}^{(t)} - u_0^{*}(\widetilde{w}_{t-1})}^2 = \norms{ \prox_{ \hat{\eta}_th}(u_n^{(s,t)}) -  \prox_{ \hat{\eta}_th}(u_n^{s*}) }^2 \leq \frac{1}{1 + 2\mu_h\hat{\eta}_t}\norms{u_n^{(s,t)} - u_n^{s*}}^2.
\end{array}
\end{equation*}
Using this inequality, $u_0^{(s,t)} = \widehat{u}^{(t)}_{s-1}$, and $u_0^{s*} = u_t^{*} = u_0^{*}(\widetilde{w}_{t-1})$, it follows from \eqref{eq:lm31_proof5} that
\begin{equation*} 
\begin{array}{lcl}
\norms{\widehat{u}_{s}^{(t)}  - u_0^{*}(\widetilde{w}_{t-1})}^2 & \leq & \frac{1}{1 + 2 \mu_h\hat{\eta}_t} \big(1 - \frac{ \mu_H\hat{\eta}_t}{ n } \big)^n \norms{\widehat{u}_{s-1}^{(t)} - u_0^{*}(\widetilde{w}_{t-1}) }^2  \\
&& + {~} \frac{2 L_u \hat{\eta}_t^3 }{n (1 + 2\mu_h\hat{\eta}_t)} \cdot  \big[ \sum_{j=0}^{n-1}  \big(1 - \frac{ \mu_H\hat{\eta}_t }{ n } \big)^{j} \big] \cdot \big[ ( \Theta_u + 1)  \norms{ \nabla{\Phi_0}(\widetilde{w}_{t-1}) }^2  + \sigma_{u}^2 \big],
\end{array}
\end{equation*}
which proves \eqref{eq:SGM2_key_est_for_ut2b}.

Next, rolling \eqref{eq:SGM2_key_est_for_ut2b} from $s=1$ to $S$, we have 
\begin{equation*} 
\begin{array}{lcl}
\norms{\widehat{u}_{s}^{(t)}  - u_0^{*}(\widetilde{w}_{t-1})}^2 & \leq & \frac{1}{(1 + 2\mu_h\hat{\eta}_t)^{S}}\big(1 - \frac{\mu_H\hat{\eta}_t}{n } \big)^{nS} \norms{\widehat{u}_0^{(t)}  - u_0^{*}(\widetilde{w}_{t-1}) }^2  \\
&& + {~} \frac{ 2 L_u}{ n}   C_S \hat{\eta}_t^3  \cdot \big[ ( \Theta_u + 1)  \norms{ \nabla{\Phi_0}(\widetilde{w}_{t-1}) }^2  + \sigma_{u}^2 \big],
\end{array}
\end{equation*}
where $C_S :=  \big[ \sum_{j=0}^{n-1}  \frac{1}{(1 + 2\mu_h\hat{\eta}_t)} \big(1 - \frac{ \mu_H\hat{\eta}_t}{ n }  \big)^j  \big] \sum_{s=0}^{S-1} \frac{1}{(1 + 2 \mu_h\hat{\eta}_t)^s} \big(1 - \frac{ \mu_H\hat{\eta}_t}{ n } \big)^{ns} $.
Substituting $\widehat{u}_0^{(t)}  := \widetilde{u}_{t-1}$ and $\widetilde{u}_t := \widehat{u}_S^{(t)} $ into the last inequality, it proves  \eqref{eq:SGM2_key_est_for_ut2c}.
\end{proof}
%%% End of Proof.

%%%%%++++++++++++++++++++++++++++++++++++++++++++++++++++++++++++++++++++
%%%% (c) Key bounds for the shuffling gradient descent scheme  on w.
%%%%%++++++++++++++++++++++++++++++++++++++++++++++++++++++++++++++++++++
\textbf{(c)~Key bounds for the shuffling gradient descent scheme \eqref{eq:SGM4upper_prob}.}
We define the following quantity:
\begin{equation}\label{eq:SGM2_gt_quantity}
\arraycolsep=0.2em
\begin{array}{lcl}
g_t := \frac{1}{n}\sum_{j=1}^n \nabla_w{\Hc}_{\hat{\pi}^{(t)}(j)}(w^{(t)}_{j-1}, \widetilde{u}_t).
\end{array}
\end{equation}
From the update of $w_i^{(t)}$ in \eqref{eq:SGM4upper_prob}, for any $i \in [n]$, we have  
\begin{equation}\label{eq:SGM2_w_update_upto_i}
\arraycolsep=0.2em
\begin{array}{lcl}
w_{i}^{(t)} & = & w_{0}^{(t)} - \frac{\eta_t}{n} \sum_{j=1}^{i} \nabla_w{\mathcal{H}}_{\hat{\pi}^{(t)}(j)} (w_{j-1}^{(t)}, \widetilde{u}_t ). 
\end{array}
\end{equation}
Then, it is obvious that $w_n^{(t)} = w_0^{(t)} - \eta_tg_t$.

First, we bound  $\Delta_t := \frac{1}{ n } \sum_{i=0}^{n-1}\norms{w_i^{(t)} - w_0^{(t)}}^2$ for  \eqref{eq:SGM4upper_prob} to handle the upper-level problem \eqref{eq:upper_level_min}.

%%% Lemma 3.2.
\begin{lemma}\label{le:SGM2_key_bounds_for_wt_2}
Suppose that Assumption~\ref{as:A3_SGM2} holds.
Let $\sets{w_i^{(t)}}$ be generated by \eqref{eq:SGM4upper_prob} such that $w_0^{(t)} := \widetilde{w}_{t-1}$.
Then, if we choose $\eta_t > 0$ such that $1 - 3L_w^2\eta_t^2  \geq  0$, then 
\begin{equation}\label{eq:SGM2_key_bounds_for_wt_2}
\begin{array}{lcl}
\Delta_t := \frac{1}{n} \sum_{i=0}^{n-1}\norms{w_i^{(t)} - w_0^{(t)}}^2  & \leq &  2 (3 \Theta_w + 1) \eta_t^2 \norms{ \nabla{\Phi_0}(\widetilde{w}_{t-1}) }^2 +  6 \eta_t^2 \sigma_{w}^2 \vspace{0.5ex} \\
&& + {~} 4 L_u^2\eta_t^2 \norms{ \widetilde{u}_t - u_0^{*}(\widetilde{w}_{t-1})}^2.
\end{array}
\end{equation}
Let $g_t$ be defined by \eqref{eq:SGM2_gt_quantity} and $\Phi_0$ be defined by \eqref{eq:lower_level_max}.
Then, we have
\begin{equation}\label{eq:SGM2_key_bounds_for_gt_2}
\begin{array}{lcl}
\norms{ g_t - \nabla_w{\Hc}(\widetilde{w}_{t-1}, \widetilde{u}_t)}^2 & \leq & \frac{L_w^2}{n}\sum_{i=0}^{n-1}\norms{w_i^{(t)} - w_0^{(t)}}^2 \equiv  L_w^2 \Delta_t, \vspace{0.5ex} \\
\norms{\nabla_w{\Hc}(\widetilde{w}_{t-1}, \widetilde{u}_t) - \nabla{\Phi}_0(\widetilde{w}_{t-1}) }^2 & \leq & L_u^2\norms{ \widetilde{u}_t - u_0^{*}(\widetilde{w}_{t-1}) }^2, \vspace{0.5ex}\\
\norms{ g_t -   \nabla{\Phi}_0(\widetilde{w}_{t-1})  }^2 & \leq &  L_w^2 \Delta_t + L_u^2\norms{ \widetilde{u}_t - u_0^{*}(\widetilde{w}_{t-1}) }^2.
\end{array}
\end{equation}
\end{lemma}

%%% Proof of Lemma 3.1.
\begin{proof}
Utilizing \eqref{eq:SGM2_w_update_upto_i} and Young's inequality in $\myeqc{1}$ and $\myeqc{2}$ below, we can show that
\begin{equation*}
\begin{array}{lcl}
\norms{w_{i}^{(t)} - w_{0}^{(t)} }^2 & \overset{\tiny\eqref{eq:SGM2_w_update_upto_i}}{=} & \frac{i^2 \cdot \eta_t^2}{n^2} \Vert \frac{1}{ i  } \sum_{j=1}^{i} \nabla_w{\mathcal{H}}_{\hat{\pi}^{(t)}(j)} (w_{j-1}^{(t)}, \widetilde{u}_t ) \Vert^2 \\
& \overset{\myeqc{1}}{\leq} & \frac{3 i^2 \cdot \eta_t^2}{n^2} \big\Vert \frac{1}{i} \sum_{j=1}^{i} \big[ \nabla_w{\mathcal{H}}_{\hat{\pi}^{(t)}(j)} (w_{0}^{(t)}, \widetilde{u}_t )  - \nabla{\Phi_0}(\widetilde{w}_{t-1}) \big] \big\Vert^2 
+  \frac{3 i^2 \cdot \eta_t^2}{n^2} \norms{ \nabla{\Phi_0}(\widetilde{w}_{t-1})  }^2 \\
&& + {~}  \frac{3 i^2 \cdot \eta_t^2}{n^2} \big\Vert \frac{1}{i} \sum_{j=1}^{i} \big[ \nabla_w{\mathcal{H}}_{\hat{\pi}^{(t)}(j)} (w_{j-1}^{(t)}, \widetilde{u}_t )  - \nabla_w{\mathcal{H}}_{\hat{\pi}^{(t)}(j)} (w_{0}^{(t)}, \widetilde{u}_t ) \big] \big\Vert^2 \\
& \overset{\myeqc{2}}{\leq} & \frac{3 i^2 \cdot \eta_t^2}{n^2} \big\Vert \frac{1}{i} \sum_{j=1}^{i} \big[ \nabla_w{\mathcal{H}}_{\hat{\pi}^{(t)}(j)} (w_{0}^{(t)}, \widetilde{u}_t )  - \nabla{\Phi_0}(\widetilde{w}_{t-1}) \big] \big\Vert^2 
+  \frac{3 i^2 \cdot \eta_t^2}{n^2} \norms{ \nabla{\Phi_0}(\widetilde{w}_{t-1})  }^2 \\
&& + {~}  \frac{3 i \cdot \eta_t^2}{n^2}  \sum_{j=1}^{i} \big\Vert \nabla_w{\mathcal{H}}_{\hat{\pi}^{(t)}(j)} (w_{j-1}^{(t)}, \widetilde{u}_t )  - \nabla_w{\mathcal{H}}_{\hat{\pi}^{(t)}(j)} (w_{0}^{(t)}, \widetilde{u}_t )  \big\Vert^2.
\end{array}
\end{equation*}
Let us denote $\Delta_t := \frac{1}{ n } \sum_{j=0}^{n-1}\norms{w_j^{(t)} - w_0^{(t)}}^2 = \frac{1}{ n } \sum_{j=0}^{n-1}\norms{w_j^{(t)} - \widetilde{w}_{t-1} }^2$.
Then, by \eqref{eq:Lsmooth2} of Assumption~\ref{as:A3_SGM2}, we have
\begin{equation*}
\begin{array}{lcl}
\frac{1}{n} \sum_{j=1}^{i} \big\Vert \nabla_w{\mathcal{H}}_{\hat{\pi}^{(t)}(j)} (w_{j-1}^{(t)}, \widetilde{u}_t )  - \nabla_w{\mathcal{H}}_{\hat{\pi}^{(t)}(j)} (w_{0}^{(t)}, \widetilde{u}_t )  \big\Vert^2 & \overset{\tiny\eqref{eq:Lsmooth2}}{ \leq } & \frac{ L_{w}^2 }{ n } \sum_{j=1}^i\norms{w^{(t)}_{j-1} - w_0^{(t)}}^2 \leq L_w^2\Delta_t.
\end{array}
\end{equation*}
Next, by Young's inequality again in $\myeqc{1}$, $w_0^{(t)} = \widetilde{w}_{t-1}$, and \eqref{eq:Lsmooth2} and \eqref{eq:bounded_variance2} from Assumption~\ref{as:A3_SGM2}, and the fact that $\nabla_w\mathcal{H}(\widetilde{w}_{t-1}, u_0^{*}(\widetilde{w}_{t-1})) = \nabla{\Phi_0}(\widetilde{w}_{t-1})$ from \eqref{eq:exact_grad_phi}, we can show that
\begin{equation*}
\arraycolsep=-0.1em
\begin{array}{lcl}
\mathcal{T}_{[2]} &:= & \Vert \frac{1}{i} \sum_{j=1}^i \big[ \nabla_w{\mathcal{H}}_{\hat{\pi}^{(t)}(j)} (w_{0}^{(t)}, \widetilde{u}_t )  - \nabla{\Phi_0}(\widetilde{w}_{t-1}) \big] \big\Vert^2 \\
& \overset{\tiny\textcircled{1}}{\leq} & \frac{2}{i} \sum_{j=1}^i \big\Vert \nabla_w{\mathcal{H}}_{\hat{\pi}^{(t)}(j)} (w_{0}^{(t)}, \widetilde{u}_t )  -   \nabla_w\mathcal{H}_{\hat{\pi}^{(t)}(j)}(w^{(t)}_{0}, u_0^{*}(w^{(t)}_{0}) )  \big\Vert^2 \\
&& + {~} \frac{2}{i}   \sum_{j=1}^i \big\Vert \nabla_w{\mathcal{H}}_{\hat{\pi}^{(t)}(j)} (w_{0}^{(t)}, u_0^{*}(w^{(t)}_{0}) )  -   \nabla_w{\mathcal{H}} (w_{0}^{(t)}, u_0^{*}(w^{(t)}_{0}) ) \big\Vert^2 \\
& \overset{\tiny\eqref{eq:Lsmooth2}}{ \leq } &  \frac{2 }{i}  \sum_{j=1}^n \Vert \nabla_w{\mathcal{H}}_i ( \widetilde{w}_{t-1} , u_0^{*}(\widetilde{w}_{t-1}) )  -   \nabla_w{\mathcal{H}} (\widetilde{w}_{t-1}, u_0^{*}(\widetilde{w}_{t-1}) ) \Vert^2 \\
&& + {~}  \frac{2L_u^2}{i} \sum_{j=1}^i \norms{ \widetilde{u}_t - u_0^{*}(\widetilde{w}_{t-1})}^2 \\
& \overset{\tiny\eqref{eq:bounded_variance2}, \eqref{eq:exact_grad_phi}}{ \leq } &  2L_u^2 \norms{ \widetilde{u}_t - u_0^{*}(\widetilde{w}_{t-1})}^2 +  \frac{2 n }{i} \big[ \Theta_w \Vert  \nabla{\Phi_0}(\widetilde{w}_{t-1}) \Vert^2 + \sigma_{w}^2 \big].
\end{array}
\end{equation*}
Combining three inequalities above, we arrive at
\begin{equation*}
\begin{array}{lcl}
\norms{w_{i}^{(t)} - w_{0}^{(t)} }^2 & \leq & \frac{6 i^2 \cdot L_u^2 \eta_t^2 }{n^2} \norms{ \widetilde{u}_t - u_0^{*}(\widetilde{w}_{t-1})}^2 +  \frac{ 6 i \cdot \eta_t^2}{n}  \big[ \Theta_w \Vert  \nabla{\Phi_0}(\widetilde{w}_{t-1}) \Vert^2 + \sigma_{w}^2 \big] \\
&& +  {~} \frac{3 i^2 \cdot \eta_t^2}{n^2} \norms{ \nabla{\Phi_0}(\widetilde{w}_{t-1})  }^2 + \frac{ 3 i \cdot L_w^2 \eta_t^2}{n } \Delta_t \\
& = & \frac{ 6i^2 \cdot L_u^2 \eta_t^2 }{n^2} \norms{ \widetilde{u}_t - u_0^{*}(\widetilde{w}_{t-1})}^2 + \frac{3i \eta_t^2}{n^2}(2n \Theta_w + i) \norms{ \nabla{\Phi_0}(\widetilde{w}_{t-1}) }^2 \\
&& + {~} \frac{6 i \cdot \eta_t^2}{n} \sigma_{w}^2  + \frac{ 3 i \cdot L_w^2 \eta_t^2}{n } \Delta_t.
\end{array}
\end{equation*}
Averaging this inequality from $i=0$ to $n-1$, we get
\begin{equation*}
\begin{array}{lcl}
\Delta_t &:= & \frac{1}{n} \sum_{i=0}^{n-1} \norms{w_{i}^{(t)} - w_{0}^{(t)} }^2 \\
& \leq & \frac{1}{n} \sum_{i=0}^{n-1} \big[ \frac{ 6i^2 \cdot L_u^2 \eta_t^2}{n^2} \norms{ \widetilde{u}_t - u_0^{*}(\widetilde{w}_{t-1})}^2 + \frac{3i \eta_t^2}{n^2}(2n \Theta_w + i) \norms{ \nabla{\Phi_0}(\widetilde{w}_{t-1}) }^2 \big] \\
&& + {~} \frac{1}{n} \sum_{i=0}^{n-1} \big[ \frac{ 6 i \cdot \eta_t^2}{ n } \sigma_{w}^2  + \frac{3 i \cdot L_w^2 \eta_t^2}{n} \Delta_t \big] \\
& \leq & 2  L_u^2\eta_t^2 \norms{ \widetilde{u}_t - u_0^{*}(\widetilde{w}_{t-1})}^2 + (3 \Theta_w + 1) \eta_t^2 \norms{ \nabla{\Phi_0}(\widetilde{w}_{t-1}) }^2 + 3 \eta_t^2 \sigma_{w}^2 + \frac{3 L_w^2\eta_t^2}{2} \Delta_t.
\end{array}
\end{equation*}
Here, we have used the facts that $\sum_{i=0}^{n-1}i = \frac{n(n-1)}{2} \leq \frac{n^2}{2}$ and $\sum_{i=0}^{n-1}i^2 = \frac{n(n-1)(2n-1)}{6} \leq \frac{n^3}{3}$.
Rearranging the last inequality, we obtain \eqref{eq:SGM2_key_bounds_for_wt_2}.

Finally, to prove \eqref{eq:SGM2_key_bounds_for_gt_2}, we proceed as follows.
Using \eqref{eq:SGM2_gt_quantity} and \eqref{eq:Lsmooth2} from Assumption~\ref{as:A3_SGM2}, we have 
\begin{equation*} 
\arraycolsep=0.05em
\begin{array}{lcl}
\norms{ g_t -   \nabla_w{\Hc}(\widetilde{w}_{t-1}, \widetilde{u}_t)  }^2 & \overset{\tiny\eqref{eq:SGM2_gt_quantity} }{ = } & \big\Vert \frac{1}{n} \sum_{j=1}^{n} \big[ \nabla_w{\mathcal{H}}_{\hat{\pi}^{(t)}(j)} (w_{j-1}^{(t)}, \widetilde{u}_t ) - \nabla_w{\mathcal{H}}_{\hat{\pi}^{(t)}(j)}(\widetilde{w}_{t-1}, \widetilde{u}_t ) \big] \big\Vert^2 \\
&\leq & \frac{1}{n} \sum_{j = 1}^{n} \big\Vert \nabla_w{\mathcal{H}}_{\hat{\pi}^{(t)}(j)} (w_{j-1}^{(t)}, \widetilde{u}_t ) - \nabla_w{\mathcal{H}}_{\hat{\pi}^{(t)}(j)}( w_0^{(t)},  \widetilde{u}_t ) \big\Vert^2 \\
& \overset{\tiny\eqref{eq:Lsmooth2}}{\leq} & \frac{L_w^2}{n}\sum_{j=1}^n\norms{w_{j-1}^{(t)} - w^{(t)}_0 }^2,
\end{array}
\end{equation*}
which proves the first line of \eqref{eq:SGM2_key_bounds_for_gt_2}.

We also note that $\nabla{\Phi}_0(\widetilde{w}_{t-1}) = \sum_{j=1}^n \nabla_w{\Hc}_{\hat{\pi}^{(t)}(j)}(\widetilde{w}_{t-1}, u_0^{*}(\widetilde{w}_{t-1}))$ due to \eqref{eq:exact_grad_phi}.
Using this expression, and \eqref{eq:Lsmooth2} from Assumption~\ref{as:A3_SGM2}, we can show that
\begin{equation*} 
\arraycolsep=0.05em
\begin{array}{lcl}
\norms{\nabla_w{\Hc}(\widetilde{w}_{t-1}, \widetilde{u}_t) - \nabla{\Phi}_0(\widetilde{w}_{t-1}) }^2 & = &  \big\Vert \frac{1}{n} \sum_{i=1}^{n} \big[ \nabla_w{\mathcal{H}}_i(\widetilde{w}_{t-1}, \widetilde{u}_t ) -  \nabla_w{\Hc}_i(\widetilde{w}_{t-1}, u_0^{*}(\widetilde{w}_{t-1})) \big] \big\Vert^2 \\
&\leq & \frac{1}{n} \sum_{i=1}^{n} \big\Vert \nabla_w{\mathcal{H}}_i(\widetilde{w}_{t-1}, \widetilde{u}_t ) -  \nabla_w{\Hc}_i(\widetilde{w}_{t-1}, u_0^{*}(\widetilde{w}_{t-1})) \big\Vert^2 \\
& \overset{\tiny\eqref{eq:Lsmooth2}}{\leq} & \frac{L_u^2}{n}\sum_{i=1}^n\norms{ \widetilde{u}_t - u_0^{*}(\widetilde{w}_{t-1}) }^2 = L_u^2 \norms{ \widetilde{u}_t - u_0^{*}(\widetilde{w}_{t-1})}^2,
\end{array}
\end{equation*}
which proves the second line of  \eqref{eq:SGM2_key_bounds_for_gt_2}.

Similarly, combining  \eqref{eq:SGM2_gt_quantity}, \eqref{eq:exact_grad_phi}, and \eqref{eq:Lsmooth2} from Assumption~\ref{as:A3_SGM2}, we can show that
\begin{equation*} 
\arraycolsep=0.05em
\begin{array}{lcl}
\norms{ g_t -   \nabla{\Phi}_0(\widetilde{w}_{t-1})  }^2 & \overset{\tiny\eqref{eq:SGM2_gt_quantity}, \eqref{eq:exact_grad_phi} }{ = } & \big\Vert \frac{1}{n} \sum_{j=1}^{n} \big[ \nabla_w{\mathcal{H}}_{\hat{\pi}^{(t)}(j)} (w_{j-1}^{(t)}, \widetilde{u}_t ) - \nabla_w{\mathcal{H}}_{\hat{\pi}^{(t)}(j)}(\widetilde{w}_{t-1}, u_0^{*}(\widetilde{w}_{t-1}) ) \big] \big\Vert^2 \\
&\leq & \frac{1}{n} \sum_{j=1}^{n} \big\Vert \nabla_w{\mathcal{H}}_{\hat{\pi}^{(t)}(j)} (w_{j-1}^{(t)}, \widetilde{u}_t ) - \nabla_w{\mathcal{H}}_{\hat{\pi}^{(t)}(j)}(\widetilde{w}_{t-1},  u_0^{*}(\widetilde{w}_{t-1}) ) \big\Vert^2 \\
& \overset{\tiny\eqref{eq:Lsmooth2}}{\leq} & \frac{1}{n}\sum_{j=1}^n \big[ L_w^2 \norms{w_{j-1}^{(t)} - \widetilde{w}_{t-1}}^2 + L_u^2 \norms{ \widetilde{u}_t - u_0^{*}(\widetilde{w}_{t-1}) }^2 \big],
\end{array}
\end{equation*}
which proves the third line of \eqref{eq:SGM2_key_bounds_for_gt_2}.
\end{proof}
%%% End of Lemma 3.1.

%%% Lemma 4.2.
\begin{lemma}\label{le:SGM2_key_bounds_for_Phi_2}
Suppose that Assumption~\ref{as:A3_SGM2} holds.
Let $\sets{w_i^{(t)}}$ be generated by \eqref{eq:SGM4upper_prob}, $g_t$ be defined by \eqref{eq:SGM2_gt_quantity}, $\Psi$ be defined by \eqref{eq:upper_level_min}, and ${\Gc}_{\eta}$ be defined by \eqref{eq:upper_level_min_grad_mapping}.
Then, we have 
\begin{equation}\label{eq:SGM2_key_bounds_for_Phi_2}
\begin{array}{lcl}
\Psi_0(\widetilde{w}_t ) & \leq & \Psi_0(\widetilde{w}_{t-1})   - \frac{(1-L_{\Phi_0}\eta_t)}{2\eta_t} \norms{ \widetilde{w}_t  - \widetilde{w}_{t-1} }^2 -  \frac{\eta_t(1 - 2L_{\Phi_0}\eta_t)}{2}  \Vert {\Gc}_{\eta_t}(\widetilde{w}_{t-1})  \Vert^2  \\
&& + {~}  \frac{\eta_t}{2}\norms{ g_t - \nabla{\Phi_0}(\widetilde{w}_{t-1}) }^2.
\end{array}
\end{equation}
\end{lemma}

%%% The proof of Lemma 4.2.
\begin{proof}
Let us denote  $\widehat{w}_t :=  \prox_{\eta_tf}\big(\widetilde{w}_{t-1} - \eta_t \nabla{\Phi_0}(\widetilde{w}_{t-1})\big)$.
Then, from  \eqref{eq:upper_level_min_grad_mapping}, we can easily show that $\Gc_{\eta_t}(\widetilde{w}_{t-1}) = \frac{1}{\eta_t}(\widetilde{w}_{t-1} - \widehat{w}_t)$.
Therefore, we have $\nabla{f}(\widehat{w}_t) := \eta_t^{-1}\big( \widetilde{w}_{t-1} - \widehat{w}_t \big) - \nabla{\Phi_0}(\widetilde{w}_{t-1}) \in \partial{f}(\widehat{w}_t)$.
By the convexity of $f$, we have
\begin{equation*}
\begin{array}{lcl}
f(\widehat{w}_t) & \leq &  f(\widetilde{w}_{t-1}) +  \iprods{\nabla{f}(\widehat{w}_t), \widehat{w}_t - \widetilde{w}_{t-1}} \\
& = &  f(\widetilde{w}_{t-1}) -  \iprods{\nabla{\Phi_0}(\widetilde{w}_{t-1}), \widehat{w}_t - \widetilde{w}_{t-1}} -  \frac{1}{\eta_t} \norms{  \widehat{w}_t -  \widetilde{w}_{t-1} }^2.
\end{array}
\end{equation*}
Next, by the $L_{\Phi_0}$-smoothness of $\Phi$ from \eqref{eq:Lsmooth_Phi2}, we have
\begin{equation*} 
\begin{array}{lcl}
\Phi_0(\widehat{w}_t ) & \leq & \Phi_0(\widetilde{w}_{t-1}) + \langle \nabla{\Phi_0}(\widetilde{w}_{t-1}),  \widehat{w}_t  - \widetilde{w}_{t-1} \rangle + \frac{L_{\Phi_0}}{2}\Vert \widehat{w}_t  - \widetilde{w}_{t-1} \Vert^2.
\end{array}
\end{equation*}
Adding the last two inequalities together and using $\Psi_0(w) = f(w) + \Phi_0(w)$ from \eqref{eq:upper_level_min} and $\widehat{w}_t -  \widetilde{w}_{t-1} = - \eta_t {\Gc}_{\eta_t}(\widetilde{w}_{t-1})$, we can derive
\begin{equation}\label{eq:SGM2_lm2_proof1} 
\hspace{-0.0ex}
\begin{array}{lcl}
\Psi_0(\widehat{w}_t ) & \leq & \Psi_0(\widetilde{w}_{t-1})  - \frac{(2 - L_{\Phi_0}\eta_t)}{2\eta_t}   \norms{  \widehat{w}_t -  \widetilde{w}_{t-1} }^2 =  \Psi_0(\widetilde{w}_{t-1})  - \frac{\eta_t(2 - L_{\Phi_0}\eta_t)}{2}   \norms{  {\Gc}_{\eta_t}(\widetilde{w}_{t-1}) }^2.
\end{array}
\hspace{-2ex}
\end{equation}
Now, from \eqref{eq:SGM2_w_update_upto_i}, we have
\begin{equation}\label{eq:SGM2_lm2_gt}
\begin{array}{lcl}
g_t := \frac{1}{\eta_t}(\widetilde{w}_{t-1}- w_n^{(t)}) =  \frac{1}{ \eta_t }( w_0^{(t)} - w_n^{(t)}) = \frac{1}{n} \sum_{j=1}^{n} \nabla_w{\mathcal{H}}_{\hat{\pi}^{(t)}(j)} (w_{j-1}^{(t)}, \widetilde{u}_t ).
\end{array}
\end{equation}
Since $\widetilde{w}_t = \prox_{\eta_tf}(w_n^{(t)})$ from the second line of \eqref{eq:SGM4upper_prob}, we get $\nabla{f}(\widetilde{w}_t) := \eta_t^{-1} \big( w_n^{(t)} - \widetilde{w}_t \big) = -g_t - \eta_t^{-1}(\widetilde{w}_t - \widetilde{w}_{t-1}) \in \partial{f}(\widetilde{w}_t)$.
Hence, again by the convexity of $f$, we can deduce that 
\begin{equation*}
\begin{array}{lcl}
f(\widetilde{w}_t) & \leq &  f(\widehat{w}^{t}) +  \iprods{\nabla{f}(\widetilde{w}_t), \widetilde{w}_t - \widehat{w}^{t}}  =  f(\widehat{w}^{t}) -  \iprods{g_t, \widetilde{w}_t - \widehat{w}^{t}}  - \frac{1}{\eta_t}\iprods{\widetilde{w}_{t} - \widetilde{w}_{t-1}, \widetilde{w}_t - \widehat{w}^{t}} \\
& = & f(\widehat{w}^{t})  - \iprods{g_t, \widetilde{w}_t - \widehat{w}^{t}} - \frac{1}{2\eta_t}\big[ \norms{\widetilde{w}_t - \widetilde{w}_{t-1}}^2 + \norms{ \widetilde{w}_t - \widehat{w}^{t}}^2 -  \norms{\widehat{w}^{t} - \widetilde{w}_{t-1}}^2 \big].
\end{array}
\end{equation*}
Again, by the $L_{\Phi_0}$-smoothness of $\Phi$ from \eqref{eq:Lsmooth_Phi2}, we also have
\begin{equation*} 
\begin{array}{lcl}
\Phi_0(\widetilde{w}_t ) & \leq & \Phi_0(\widetilde{w}_{t-1}) + \langle \nabla{\Phi_0}(\widetilde{w}_{t-1}),  \widetilde{w}_t  - \widetilde{w}_{t-1} \rangle + \frac{L_{\Phi_0}}{2}\Vert \widetilde{w}_t  - \widetilde{w}_{t-1} \Vert^2, \\
\Phi_0(\widetilde{w}_{t-1} ) & \leq & \Phi_0(\widehat{w}^{t}) + \langle \nabla{\Phi_0}(\widetilde{w}_{t-1}),  \widetilde{w}_{t-1}  - \widehat{w}^{t} \rangle + \frac{L_{\Phi_0}}{2} \Vert \widehat{w}_t  - \widetilde{w}_{t-1} \Vert^2.
\end{array}
\end{equation*}
Adding the last three inequalities together, and using $\Psi_0(w) = f(w) + \Phi_0(w)$ from \eqref{eq:upper_level_min} and $\widehat{w}_t -  \widetilde{w}_{t-1} = - \eta_t {\Gc}_{\eta_t}(\widetilde{w}_{t-1})$, we can prove that
\begin{equation}\label{eq:SGM2_lm2_proof2} 
\begin{array}{lcl}
\Psi_0(\widetilde{w}_t ) & \leq & \Psi_0(\widehat{w}^{t}) +  \iprods{ \nabla{\Phi_0}(\widetilde{w}_{t-1}) - g_t,  \widetilde{w}_{t}  - \widehat{w}^{t} } - \frac{(1-L_{\Phi_0}\eta_t)}{2\eta_t} \norms{ \widetilde{w}_t  - \widetilde{w}_{t-1} }^2 \\
&& + {~}  \frac{(1 + L_{\Phi_0}\eta_t)}{2\eta_t}  \Vert \widehat{w}_t  - \widetilde{w}_{t-1} \Vert^2  - \frac{1}{2\eta_t} \norms{ \widetilde{w}_t - \widehat{w}^{t}}^2 \\
& \overset{\myeqc{1}}{\leq} & \Psi_0(\widehat{w}^{t}) +  \frac{\eta_t}{2}\norms{ \nabla{\Phi_0}(\widetilde{w}_{t-1}) - g_t }^2 - \frac{(1-L_{\Phi_0}\eta_t)}{2\eta_t} \norms{ \widetilde{w}_t  - \widetilde{w}_{t-1} }^2  \\
&& + {~}  \frac{\eta_t(1 + L_{\Phi_0}\eta_t)}{2}  \Vert {\Gc}_{\eta_t}(\widetilde{w}_{t-1})  \Vert^2,
\end{array}
\end{equation}
where we have used Young's inequality in the last line $\myeqc{1}$ as $\iprods{ \nabla{\Phi_0}(\widetilde{w}_{t-1}) - g_t,  \widetilde{w}_{t}  - \widehat{w}^{t} } \leq \frac{\eta_t}{2}\norms{g_t - \nabla{\Phi_0}(\widetilde{w}_{t-1}) }^2 + \frac{1}{2\eta_t}\norms{\widetilde{w}_{t}  - \widehat{w}^{t}}^2$.

Finally, summing up \eqref{eq:SGM2_lm2_proof1} and \eqref{eq:SGM2_lm2_proof2} we arrive at
\begin{equation*} 
\begin{array}{lcl}
\Psi_0(\widetilde{w}_t ) & \leq & \Psi_0(\widetilde{w}_{t-1})   - \frac{(1-L_{\Phi_0}\eta_t)}{2\eta_t} \norms{ \widetilde{w}_t  - \widetilde{w}_{t-1} }^2 -  \frac{\eta_t(1 - 2L_{\Phi_0}\eta_t)}{2}  \Vert {\Gc}_{\eta_t}(\widetilde{w}_{t-1})  \Vert^2  \\
&& + {~}  \frac{\eta_t}{2}\norms{ g_t - \nabla{\Phi_0}(\widetilde{w}_{t-1}) }^2,
\end{array}
\end{equation*}
which proves \eqref{eq:SGM2_key_bounds_for_Phi_2}.
\end{proof}
%%% End of proof.

%%%%%%%%%%%%%%%%%%%%%%%%%%%%%%%%%%%%%%%%
%%% C.2. Convergence of the semi-shuffling variant of Algorithm 2.
%%%%%%%%%%%%%%%%%%%%%%%%%%%%%%%%%%%%%%%%
\beforesubsec
\subsection{Convergence of the semi-shuffling variant of Algorithm~\ref{alg:SGM2}}\label{apdx:subsec:SGM2_full_shuffling_convergence}
\aftersubsec
We now prove the convergence of the semi-shuffling variant of Algorithm~\ref{alg:SGM2} using \eqref{eq:GD4lower_prob}.

%%% Lemma 3.4.
\begin{lemma}\label{le:SGM2_keybound_for_semi_shuffling}
Suppose that Assumptions~\ref{as:A3_SGM2} and ~\ref{as:A3} hold for \eqref{eq:minimax_prob}.
Let $\Psi$ be defined by \eqref{eq:upper_level_min} and ${\Gc}_{\eta}$ be defined by \eqref{eq:upper_level_min_grad_mapping}.
Let $\sets{(\widetilde{w}_t, \widetilde{u}_t)}$ be generated by the \textbf{semi-shuffling variant} of Algorithm~\ref{alg:SGM2} using \eqref{eq:GD4lower_prob}.
For a fixed $\omega > 0$, suppose that we choose $\eta_t$ and $\hat{\eta}_t$ such that $1 - 3L_w^2\eta_t^2 \geq 0$ and $0 < \hat{\eta}_t \leq \frac{2}{L_u + \mu_H}$, and the following conditions hold:
\begin{equation}\label{eq:SGM2_semi_shuffling_para_cond1}
\left\{ \begin{array}{lll}
&2L_{\Phi_0}\eta_t + 2\omega L_u^2\kappa^2 \eta_t^2 \leq 1, \\
& \frac{1}{(1 + 2\mu_h\hat{\eta}_t)^S}\big(1 - \frac{2L_u\mu_H\hat{\eta}_t}{L_u + \mu_H}\big)^S \big( 1 + \omega +  \omega^2 L_u^2\kappa^2 \eta_t^2 + 2L_w^2 \eta_t^2 \big) \leq \omega.
\end{array}\right.
\end{equation}
Then, the following bound holds:
\begin{equation}\label{eq:SGM2_semi_shuffling_key_bound1}
\hspace{-0.0ex}
\begin{array}{lcl}
\Psi_0(\widetilde{w}_t)  & + &  \frac{\omega L_u^2\eta_t}{2} \norms{\widetilde{u}_t  - u_0^{*}(\widetilde{w}_{t})}^2  \leq   \Psi_0(\widetilde{w}_{t-1}) +  \frac{\omega L_u^2\eta_t}{2} \norms{\widetilde{u}_{t-1}  - u_0^{*}(\widetilde{w}_{t-1})}^2 \\
&& - {~} \frac{\eta_t B_t}{2} \norms{\Gc_{\eta_t}(\widetilde{w}_{t-1})}^2 +  \big[ 3L_w^2\sigma_{w}^2 + L_w^2 (3 \Theta_w + 1)\Lambda_1 \big] \cdot \eta_t^3,
\end{array}
\hspace{-3.0ex}
\end{equation}
where $B_t := 1 - 2L_{\Phi_0}\eta_t - 2L_w^2 (3 \Theta_w + 1)\Lambda_0 \eta_t^2$.
\end{lemma}

%%% Proof of Lemma 3.4.
\begin{proof}
First, combining \eqref{eq:SGM2_key_bounds_for_Phi_2} and the last line of \eqref{eq:SGM2_key_bounds_for_gt_2}, we can derive
\begin{equation}\label{eq:SGM2_key_bounds_for_Phi_2c} 
\arraycolsep=0.2em
\begin{array}{lcl}
\Psi_0(\widetilde{w}_t ) & \leq & \Psi_0(\widetilde{w}_{t-1})   -  \frac{\eta_t(1 - 2L_{\Phi_0}\eta_t)}{2}  \Vert {\Gc}_{\eta_t}(\widetilde{w}_{t-1})  \Vert^2   - \frac{(1-L_{\Phi_0}\eta_t)}{2\eta_t} \norms{ \widetilde{w}_t  - \widetilde{w}_{t-1} }^2 \\
&& + {~}  \frac{L_w^2\eta_t}{2n} \sum_{j=1}^n \norms{w_{j-1}^{(t)} - \widetilde{w}_{t-1} }^2 + \frac{L_u^2\eta_t}{2}\norms{ \widetilde{u}_t - u_0^{*}(\widetilde{w}_{t-1} )}^2.
\end{array}
\end{equation}
Next, substituting \eqref{eq:SGM2_key_bounds_for_wt_2} into \eqref{eq:SGM2_key_bounds_for_Phi_2c}, we can show that
\begin{equation*}
\arraycolsep=0.2em 
\begin{array}{lcl}
\Psi_0(\widetilde{w}_t) &\leq & 
%\Psi_0(\widetilde{w}_{t-1})  - \frac{\eta_t(1 - 2L_{\Phi_0}\eta_t)}{2}  \Vert {\Gc}_{\eta_t}(\widetilde{w}_{t-1})  \Vert^2 - \frac{(1-L_{\Phi_0}\eta_t)}{2\eta_t} \norms{ \widetilde{w}_t  - \widetilde{w}_{t-1} }^2 \\
%&& + {~}  \frac{L_w^2 \cdot \eta_t }{2n} \big[ 4 n L_u^2\eta_t^2 \norms{ \widetilde{u}_t - u_0^{*}(\widetilde{w}_{t-1})}^2 + 2 n(3 \Theta_w + 1) \eta_t^2 \norms{ \nabla{\Phi_0}(\widetilde{w}_{t-1}) }^2 + 6n \eta_t^2 \sigma_{w}^2 \big] \\
%&& + {~} \frac{L_u^2 \eta_t}{2} \norms{\widetilde{u}_t - u_0^{*}(\widetilde{w}_{t-1})}^2 \\
%& = &  
\Psi_0(\widetilde{w}_{t-1})  - \frac{\eta_t(1 - 2L_{\Phi_0}\eta_t)}{2}  \Vert {\Gc}_{\eta_t}(\widetilde{w}_{t-1})  \Vert^2 - \frac{(1-L_{\Phi_0}\eta_t)}{2\eta_t} \norms{ \widetilde{w}_t  - \widetilde{w}_{t-1} }^2  + 3L_w^2 \eta_t^3 \sigma_{w}^2 \\
&& + {~} \frac{L_u^2 \eta_t}{2}\big( 1 + 4L_w^2 \eta_t^2 \big) \norms{\widetilde{u}_t - u_0^{*}(\widetilde{w}_{t-1})}^2 +  L_w^2 (3 \Theta_w + 1) \eta_t^3 \norms{ \nabla{\Phi_0}(\widetilde{w}_{t-1}) }^2.
\end{array}
\end{equation*}
By \eqref{eq:u_star_smoothness} and Young's inequality in $\myeqc{1}$, for any $s_t > 0$, we have
\begin{equation*} 
\arraycolsep=0.2em
\begin{array}{lcl}
\norms{\widetilde{u}_t  - u_0^{*}(\widetilde{w}_{t})}^2 & \overset{\myeqc{1}}{\leq} & (1+s_t)\norms{\widetilde{u}_t  - u_0^{*}(\widetilde{w}_{t-1})}^2 + \frac{(1+s_t)}{s_t}\norms{ u_0^{*}(\widetilde{w}_{t})   - u_0^{*}(\widetilde{w}_{t-1})}^2\\
& \overset{\tiny\eqref{eq:u_star_smoothness}}{\leq} & (1+s_t)\norms{\widetilde{u}_t  - u_0^{*}(\widetilde{w}_{t-1})}^2 + \frac{(1+s_t)\kappa^2}{s_t}\norms{ \widetilde{w}_{t}    -  \widetilde{w}_{t-1}}^2.
\end{array}
\end{equation*}
Multiplying this inequality by $\frac{ \omega  L_u^2\eta_t}{2}$ for some $\omega  > 0$ and adding the result to the last estimate yields
\begin{equation*} 
\arraycolsep=0.2em
\begin{array}{lcl}
\Tc_{[1]} &:= & \Psi_0(\widetilde{w}_t) + \frac{\omega L_u^2\eta_t}{2} \norms{\widetilde{u}_t  - u_0^{*}(\widetilde{w}_{t})}^2 \\
&\leq & \Psi_0(\widetilde{w}_{t-1}) + \frac{L_u^2\eta_t}{2} \big[  1 + \omega(1+s_t) + 4L_w^2 \eta_t^2 \big]  \norms{\widetilde{u}_t  - u_0^{*}(\widetilde{w}_{t-1})}^2 \\ 
&& - {~} \frac{\eta_t(1 - 2L_{\Phi_0}\eta_t)}{2}  \Vert {\Gc}_{\eta_t}(\widetilde{w}_{t-1})  \Vert^2 - \big[ \frac{1-L_{\Phi_0}\eta_t}{2\eta_t} - \frac{\omega L_u^2\kappa^2 \eta_t(1+s_t)}{2s_t} \big] \norms{ \widetilde{w}_t  - \widetilde{w}_{t-1} }^2 \\
&& + {~}  L_w^2 (3 \Theta_w + 1) \eta_t^3 \norms{ \nabla{\Phi_0}(\widetilde{w}_{t-1}) }^2 + 3L_w^2 \eta_t^3 \sigma_{w}^2 \\
& \overset{\tiny\eqref{SGM2_bound_u1_dist1}}{ \leq } & \Psi_0(\widetilde{w}_{t-1}) + \frac{L_u^2\eta_t}{2(1 + 2\mu_h\hat{\eta}_t)^S}\big(1 - \frac{2L_u\mu_H\hat{\eta}_t}{L_u + \mu_H}\big)^S \big[  1 + \omega(1 + s_t) + 4L_w^2 \eta_t^2 \big] \norms{\widetilde{u}_{t-1}  - u_0^{*}(\widetilde{w}_{t-1})}^2 \\ 
&& - {~} \frac{\eta_t(1 - 2L_{\Phi_0}\eta_t)}{2}  \Vert {\Gc}_{\eta_t}(\widetilde{w}_{t-1})  \Vert^2 - \big[ \frac{1-L_{\Phi_0}\eta_t}{2\eta_t} - \frac{\omega L_u^2\kappa^2 \eta_t(1+s_t)}{2s_t} \big] \norms{ \widetilde{w}_t  - \widetilde{w}_{t-1} }^2 \\
&& + {~}  L_w^2 (3 \Theta_w + 1) \eta_t^3 \norms{ \nabla{\Phi_0}(\widetilde{w}_{t-1}) }^2 + 3L_w^2 \eta_t^3 \sigma_{w}^2.
\end{array}
\end{equation*}
We need to choose the parameters $\eta_t$, $\hat{\eta}_t$, and $s_t$ such that
\begin{equation*} 
\left\{ \begin{array}{ll}
& \frac{1}{(1 + 2\mu_h\hat{\eta}_t)^S }\big(1 - \frac{2L_u\mu_H\hat{\eta}_t}{L_u + \mu_H}\big)^S \big[ 1 + \omega(1+s_t) + 4L_w^2 \eta_t^2 \big] \leq \omega,  \\
&  \frac{1-L_{\Phi_0}\eta_t}{\eta_t} - \frac{\omega L_u^2\kappa^2 \eta_t(1+s_t)}{s_t} \geq 0.
\end{array}\right.
\end{equation*}
The second condition leads to $\frac{ 1 - L_{\Phi_0}\eta_t - \omega L_u^2\kappa^2 \eta_t^2}{\omega L_u^2\kappa^2 \eta_t^2} \geq \frac{ 1}{s_t}$, or equivalently $0 < s_t \leq  \frac{\omega L_u^2\kappa^2 \eta_t^2}{ 1 - L_{\Phi_0}\eta_t - \omega L_u^2\kappa^2 \eta_t^2}$.
If $2L_{\Phi_0}\eta_t + 2\omega L_u^2\kappa^2 \eta_t^2 \leq 1$ as stated in the first line of \eqref{eq:SGM2_semi_shuffling_para_cond1}, then we can choose $s_t := 2\omega L_u^2\kappa^2 \eta_t^2$.
In this case, the second condition is satisfied, while the first condition becomes
\begin{equation*} 
\begin{array}{ll}
& \frac{1}{(1 + 2\mu_h\hat{\eta}_t)^S }\big(1 - \frac{2L_u\mu_H\hat{\eta}_t}{L_u + \mu_H}\big)^S \big( 1 + \omega + 2\omega^2 L_u^2\kappa^2 \eta_t^2 + 4L_w^2 \eta_t^2 \big) \leq \omega,
\end{array}
\end{equation*}
which is exactly the second condition of \eqref{eq:SGM2_semi_shuffling_para_cond1}.

By \eqref{eq:grad_mapp_bound} from Assumption~\ref{as:A3}, we have
\begin{equation*} 
\begin{array}{lcl}
\Tc_{[1]} &:= & \Psi_0(\widetilde{w}_t) + \frac{\omega L_u^2\eta_t}{2} \norms{\widetilde{u}_t  - u_0^{*}(\widetilde{w}_{t})}^2 \\
& \leq &  \Psi_0(\widetilde{w}_{t-1}) +  \frac{ \omega L_u^2\eta_t}{2} \norms{\widetilde{u}_{t-1}  - u_0^{*}(\widetilde{w}_{t-1})}^2  -  \frac{\eta_t(1 - 2L_{\Phi_0}\eta_t)}{2}  \norms{  \Gc_{\eta_t}(\widetilde{w}_{t-1}) }^2 \\
&& + {~}  L_w^2 (3 \Theta_w + 1)\Lambda_0 \eta_t^3  \norms{  \Gc_{\eta_t}(\widetilde{w}_{t-1}) }^2 +  \big[ 3L_w^2\sigma_{w}^2 + L_w^2 (3 \Theta_w + 1)\Lambda_1 \big] \eta_t^3.
\end{array}
\end{equation*}
Rearranging this inequality, we  prove \eqref{eq:SGM2_semi_shuffling_key_bound1}.
\end{proof}
%%% End of proof.

The following theorem, Theorem~\ref{th:SGM2_semi_shuffling_variant_convergence_supp}, is the full version of Theorem~\ref{th:SGM2_semi_shuffling_variant_convergence} in the main text, where the learning rates $\eta_t$ and  $\hat{\eta}_t$, and the numbers of epochs $S$ and $T$ are given explicitly.

%%% Theorem 4.1.
\begin{theorem}\label{th:SGM2_semi_shuffling_variant_convergence_supp}
Suppose that Assumptions~\ref{as:A0},~\ref{as:A2},~\ref{as:A3_SGM2}, and ~\ref{as:A3} hold for \eqref{eq:minimax_prob}.
Let $\Psi_0$ be defined by \eqref{eq:upper_level_min}, and ${\Gc}_{\eta}$ be defined by \eqref{eq:upper_level_min_grad_mapping}.
Let $C_0$ and $C_w$ be two constants given as follows:
\begin{equation}\label{eq:SGM2_semi_shuffling_constants}
\begin{array}{lcl}
C_0 &:= & 2 \Lambda_0 L_w^2 (3 \Theta_w + 1)\quad \text{and} \quad C_w :=   L_w^2 (3 \Theta_w + 1)\Lambda_1 + 3L_w^2\sigma_{w}^2.
\end{array}
\end{equation}
Let $\sets{(\widetilde{w}_t, \widetilde{u}_t)}$ be generated by Algorithm~\ref{alg:SGM2} using  the \textbf{gradient ascent scheme \eqref{eq:GD4lower_prob}}, and fixed learning rates $\eta_t := \eta > 0$ and $\hat{\eta}_t := \hat{\eta} \in \big(0, \frac{2}{L_u + \mu_H}\big] $ such that for a fixed $\omega > 0$:
\begin{equation}\label{eq:SGM2_semi_shuffling_para_cond1_s}
\begin{array}{lll}
S := \big\lfloor  \frac{ M_{\omega}(\eta) }{2\hat{\eta} }\big( \mu_h + \frac{4\mu_HL_u}{L_u + \mu_H}\big)^{-1} \big\rfloor \quad \text{and} \quad  0 < \eta \leq \min\left\{ \frac{1}{2\sqrt{C_0}},  \frac{1}{4L_{\Phi_0}},  \frac{1}{2\omega L_u\kappa} \right\},
\end{array}
\end{equation}
where $M_{\omega}(\eta) := \frac{1}{\omega} + \big( \omega L_u^2\kappa^2 +  \frac{2L_w^2}{\omega} \big)\eta^2$.
Then, the following estimate holds:
\begin{equation}\label{eq:SGM2_semi_shuffling_bound1}
\begin{array}{lcl}
\frac{1}{T+1}\sum_{t=0}^T \norms{ {\Gc}_{\eta}(\widetilde{w}_{t}) }^2 & \leq & \frac{4 \big[ 2( \Psi_0(\widetilde{w}_0) - \Psi_0^{\star}) + \omega L_u^2\eta\norms{\widetilde{u}^0 - u_0^{*}(\widetilde{w}_0)}^2 \big] }{\eta (T+1)} + 8C_w \eta^2.
\end{array}
\end{equation}
For a given $\epsilon > 0$, if we choose $\eta := \frac{s\epsilon}{4\sqrt{C_w}}$ for a fixed $s \in (0, 1)$ satisfying \eqref{eq:SGM2_semi_shuffling_para_cond1}, $\hat{\eta} \in \big(0,  \frac{2}{L_u+\mu_h}\big]$, and $T := \mathcal{O}\big( \frac{1}{\epsilon^3} \big)$, then $\frac{1}{T+1}\sum_{t=0}^T \norms{ {\Gc}_{\eta}(\widetilde{w}_{t}) }^2 \leq \epsilon^2$.

%%%
For a given $\hat{\eta} \in \big(0,  \frac{2}{L_u+\mu_h}\big]$, we denote $B_0 := \hat{\eta}\big(\mu_h + \frac{4\mu_HL_u}{L_u + \mu_H}\big)$.
If we choose $\omega := \frac{1}{B_0}$ and 
\begin{equation}\label{eq:SGM2_semi_shuffling_para_cond1_one_epoch} 
\begin{array}{lll}
0 < \eta \leq \min\left\{ \frac{1}{2\sqrt{C_0}},  \frac{1}{4L_{\Phi_0}}, \frac{B_0}{\sqrt{L_u^2\kappa^2  + 2B_0^2L_w^2}} \right\},
\end{array}
\end{equation}
then we have $S = 1$, i.e. we only need to perform one iteration of  the \textbf{gradient ascent scheme \eqref{eq:GD4lower_prob}}.

%%%
Consequently, Algorithm~\ref{alg:SGM2} requires $\mathcal{O}\big( \frac{n}{\epsilon^3} \big)$ evaluations of $\nabla_w{\Hc}_i$  and  of $\nabla_u{\Hc}_i$, and $\Oc(\epsilon^{-3})$ evaluations of $\prox_{\eta_tf}$ and of $\prox_{\hat{\eta}_th}$ to achieve an $\epsilon$-stationary point $\widehat{w}_T$ of \eqref{eq:minimax_prob} computed by \eqref{eq:approx_stationary_point}.
\end{theorem}

%%% Proof of Theorem 4.1.
\begin{proof}[\textbf{Proof of Theorem~\ref{th:SGM2_semi_shuffling_variant_convergence_supp}}]
Let us choose $\eta_t := \eta$ such that $\eta$ satisfies \eqref{eq:SGM2_semi_shuffling_para_cond1_s}.
Then, it is obvious to verify that $1 - 3L_w^2\eta_t^2 \geq 0$ and $2L_{\Phi_0}\eta_t + 2\omega L_u^2\kappa^2 \eta_t^2 \leq 1$.
Moreover,  we have $\eta_t = \eta \leq \frac{1}{4L_{\Phi_0}}$, $\eta_t = \eta\leq \frac{1}{2L_{w}} \leq \frac{1}{\sqrt{3}L_w}$, and $\eta_t = \eta \leq \frac{1}{2\omega L_u\kappa}$.
Using these bounds, we can further lower bound $B_t := 1 - 2L_{\Phi_0}\eta_t - 2L_w^2 (3 \Theta_w + 1)\Lambda_0 \eta_t^2$ from Lemma~\ref{le:SGM2_keybound_for_semi_shuffling} as
\begin{equation*}
\begin{array}{lcl}
B_t  & \geq & \frac{1}{2} - 2\Lambda_0 L_w^2 (3 \Theta_w + 1)  \eta^2.
\end{array}
\end{equation*}
Now, we need to choose $0 < \eta \leq \frac{1}{2\sqrt{2 \Lambda_0 L_w^2 (3 \Theta_w + 1) }} = \frac{1}{2\sqrt{C_0}}$ so that $B_t \geq \frac{1}{4}$, where $C_0$ is given in \eqref{eq:SGM2_semi_shuffling_constants}.
Moreover, the second condition of \eqref{eq:SGM2_semi_shuffling_para_cond1} holds if 
\begin{equation*}
\begin{array}{lclcl}
& & \frac{1}{(1 + \mu_h\hat{\eta} )^S}\big(1 - \frac{2L_u\mu_H\hat{\eta}}{L_u+\mu_H}\big)^{S} & \leq &  \frac{\omega}{1 + \omega +  (\omega^2 L_u^2\kappa^2 + 2L_w^2) \eta^2 },  \vspace{0.5ex} \\
& \Leftrightarrow & S \ln(1 + \mu_h\hat{\eta} ) - S\ln\big(1 - \frac{2L_u\mu_H\hat{\eta}}{L_u+\mu_H}\big) & \geq &  \ln\big( 1 + \frac{1}{\omega} +  (\omega L_u^2\kappa^2 + \frac{2L_w^2}{\omega}) \eta^2 \big).
\end{array}
\end{equation*}
Using the elementary facts $-\ln(1 - \tau ) \geq \tau$ and $\tau  \geq \ln(1 + \tau) \geq \frac{e}{2}$ for all $\tau \in (0, 1/2]$, we can show that the last inequality holds if 
\begin{equation*}
\begin{array}{lcl}
S \hat{\eta} \big(  \frac{\mu_h}{2} + \frac{2L_u\mu_H}{L_u+\mu_H} \big) \geq \frac{1}{\omega} +  (\omega L_u^2\kappa^2 + \frac{2L_w^2}{\omega}) \eta^2.
\end{array}
\end{equation*}
Simplifying this condition, we get
\begin{equation*}
\begin{array}{lcl}
S \geq \frac{ M_{\omega}(\eta) }{2\hat{\eta} }\big( \mu_h + \frac{4\mu_HL_u}{L_u + \mu_H}\big)^{-1}, \quad \text{where} \quad M_{\omega}(\eta) := \frac{1}{\omega} + \big( \omega L_u^2\kappa^2 +  \frac{2L_w^2}{\omega} \big)\eta^2.
\end{array}
\end{equation*}
Clearly, this leads to the choice of $S$ as in \eqref{eq:SGM2_semi_shuffling_para_cond1_s}.

Let us define $C_w := 3L_w^2\sigma_{w}^2 + L_w^2 (3 \Theta_w + 1)\Lambda_1$ as in \eqref{eq:SGM2_semi_shuffling_constants}.
Then, \eqref{eq:SGM2_semi_shuffling_key_bound1} reduces to
\begin{equation*} 
\begin{array}{lcl}
\Psi_0(\widetilde{w}_t)  +  \frac{\omega L_u^2\eta}{2} \norms{\widetilde{u}_t  - u_0^{*}(\widetilde{w}_{t})}^2  & \leq &   \Psi_0(\widetilde{w}_{t-1}) +  \frac{\omega L_u^2\eta}{2} \norms{\widetilde{u}_{t-1}  - u_0^{*}(\widetilde{w}_{t-1})}^2 \\
&& - {~} \frac{\eta}{8} \norms{\Gc_{\eta}(\widetilde{w}_{t-1})}^2 +  C_w \cdot \eta^3.
\end{array}
\end{equation*}
Subtracting $\Psi_0^{\star}$ from both sides of this inequality, and averaging the result from $t=0$ to $t = T$, and noting that $\Psi_0(\widetilde{w}_T) - \Psi_0^{\star} \geq 0$, we obtain \eqref{eq:SGM2_semi_shuffling_bound1}.

Without loss of generality, let us choose $\omega := 1$.
We also choose $\hat{\eta} \in \big(0,  \frac{2}{L_u + \mu_h}\big]$.
Then, we have $M_{\omega} = 1 + (L_u^2\kappa^2 +  2L_w^2)\eta^2 \leq 2$.
Moreover, from \eqref{eq:SGM2_semi_shuffling_para_cond1_s}, we also have 
\begin{equation*} 
\begin{array}{lcl}
S = \Big\lfloor \frac{ M_{\omega}(\eta) }{2\hat{\eta} }\big[\mu_h + \frac{4\mu_HL_u}{L_u + \mu_H}\big]^{-1}\Big\rfloor  \leq \bar{S} := \Big\lfloor \frac{ 1 }{\hat{\eta} }\big[\mu_h + \frac{4 \mu_HL_u}{L_u + \mu_H}\big]^{-1} \Big\rfloor = \Oc(1).
\end{array}
\end{equation*}
To achieve an $\epsilon$-stationary point of \eqref{eq:upper_level_min}, from \eqref{eq:SGM2_semi_shuffling_bound1} we need to impose the following condition:
\begin{equation*}
\begin{array}{lcl}
\frac{2[\Psi_0(\widetilde{w}_0) - \Psi_0^{\star}] }{\eta (T+1)} + \frac{L_u^2\norms{\widetilde{u}^0 - u_0^{*}(\widetilde{w}_0)}^2}{T+1} + 4C_w \eta^2 \leq \frac{\epsilon^2}{4}.
\end{array}
\end{equation*}
If we choose $\eta := \frac{s \epsilon }{4\sqrt{C_w}}$ for some $s \in (0, 1)$ satisfying \eqref{eq:SGM2_semi_shuffling_para_cond1_s}, then the last inequality leads to
\begin{equation*}
\begin{array}{lcl}
T  \geq \bar{T} := \Big\lfloor \frac{32\sqrt{C_w}[\Psi_0(\widetilde{w}_0) - \Psi_0^{\star}]}{s(1-s^2)\epsilon^3} + \frac{4L_u^2\norms{\widetilde{u}^0 - u_0^{*}(\widetilde{w}_0)}^2}{(1-s^2)\epsilon^2} \Big\rfloor = \Oc\big(\frac{1}{\epsilon^3}\big).
\end{array}
\end{equation*}
Therefore, we can choose $T := \bar{T} = \Oc(\epsilon^{-3})$.
Since each iteration $t$, we run $S$ epochs of the shuffling scheme \eqref{eq:GD4lower_prob}, the total number of evaluations of $\nabla_u{\Hc}_i$ is $\Tc_u := T\times S\times n$.
However, since $1 \leq S \leq \bar{S} = \Oc(1)$, we get $\Tc_u := \Oc(n\epsilon^{-3})$.
The total number of evaluations of $\nabla_w{\Hc}_i$ is $\Tc_w := Tn = \Oc(n\epsilon^{-3})$ as stated.

Since each epoch $t$, Algorithm~\ref{alg:SGM2} requires one evaluation of $\prox_{\eta_tf}$, and $S$ evaluations of $\prox_{\hat{\eta}_th}$, but since $S = \Oc(1)$, the total number of  $\prox_{\eta_tf}$ evaluations is $T = \Oc(\epsilon^{-3})$, while the total number of $\prox_{\hat{\eta}_th}$ evaluations is $TS = \Oc(\epsilon^{-3})$.
Overall, Algorithm~\ref{alg:SGM2} needs $\Oc(\epsilon^{-3})$ evaluations of both $\prox_{\eta_tf}$ and $\prox_{\hat{\eta}_th}$.

Finally, to perform only one iteration  the \textbf{gradient ascent scheme \eqref{eq:GD4lower_prob}} at each epoch $t$, we need to choose $\omega$ such that
\begin{equation*}
\begin{array}{lcl}
\frac{ M_{\omega}(\eta) }{2\hat{\eta} }\big(\mu_h + \frac{4\mu_HL_u}{L_u + \mu_H}\big)^{-1} \leq S = 1.
\end{array}
\end{equation*}
This condition leads to 
\begin{equation*}
\begin{array}{lcl}
M_{\omega}(\eta) = \frac{1}{\omega} + \big( \omega L_u^2\kappa^2 +  \frac{2L_w^2}{\omega} \big)\eta^2 \leq 2\hat{\eta}\big(\mu_h + \frac{4\mu_HL_u}{L_u + \mu_H}\big).
\end{array}
\end{equation*}
For a given $\hat{\eta} \in \big(0,  \frac{2}{L_u + \mu_h}\big]$, let us choose $\frac{1}{\omega} := \hat{\eta}\big(\mu_h + \frac{4\mu_HL_u}{L_u + \mu_H}\big) := B_0$.
Then, the last condition becomes $(L_u^2\kappa^2  + 2B_0^2L_w^2) \eta^2 \leq B_0^2$, or equivalently $\eta \leq \frac{B_0}{\sqrt{L_u^2\kappa^2  + 2B_0^2L_w^2}}$.
Combining this condition and \eqref{eq:SGM2_semi_shuffling_para_cond1_s}, we get \eqref{eq:SGM2_semi_shuffling_para_cond1_one_epoch}, i.e.:
\begin{equation*} 
\begin{array}{lll}
0 < \eta \leq \min\left\{ \frac{1}{2\sqrt{C_0}},  \frac{1}{4L_{\Phi_0}},   \frac{B_0}{\sqrt{L_u^2\kappa^2  + 2B_0^2L_w^2}} \right\}, \quad \text{where} \quad B_0 := \hat{\eta}\big(\mu_h + \frac{4\mu_HL_u}{L_u + \mu_H}\big).
\end{array}
\end{equation*}
Thus we have $S = 1$, i.e. we need to perform only one iteration of  \eqref{eq:GD4lower_prob} per epoch $t$.
\end{proof}
%%% End of proof.

%%%%%%%%%%%%%%%%%%%%%%%%%%%%%%%%%%%%%%%%
%%% C.3. Convergence of the full-shuffling variant of Algorithm 2.
%%%%%%%%%%%%%%%%%%%%%%%%%%%%%%%%%%%%%%%%
\beforesubsec
\subsection{Convergence of the full-shuffling variant of Algorithm~\ref{alg:SGM2} -- The case $S > 1$}\label{apdx:subsec:SGM2_full_shuffling_convergence}
\aftersubsec
We can combine the results above to obtain the following lemma.

%%% Lemma 3.4.
\begin{lemma}\label{le:SGM2_keybound_for_full_shuffling}
Suppose that Assumptions~\ref{as:A3_SGM2} and ~\ref{as:A3} hold for \eqref{eq:minimax_prob}, $\Psi$ be defined by \eqref{eq:upper_level_min}, and ${\Gc}_{\eta}$ be defined by \eqref{eq:upper_level_min_grad_mapping}.
Let $\sets{(\widetilde{w}_t, \widetilde{u}_t)}$ be generated by the full-shuffling variant of Algorithm~\ref{alg:SGM2} using \eqref{eq:SGM4lower_prob}.
For a fixed $\omega > 0$, assume that $\eta_t$ and $\hat{\eta}_t$ are chosen such that $1 - 3L_w^2\eta_t^2 \geq 0$ and 
\begin{equation}\label{eq:SGM2_full_shuffling_para_cond1}
\left\{ \begin{array}{lll}
&2L_{\Phi_0}\eta_t + 2\omega L_u^2\kappa^2 \eta_t^2 \leq 1, \\
& \frac{1}{(1 + 2\mu_h\hat{\eta}_t)^S}\big(1 - \frac{\mu_H\hat{\eta}_t}{n} \big)^{nS} \big( 1 + \omega + 2\omega^2 L_u^2\kappa^2 \eta_t^2 + 4L_w^2 \eta_t^2 \big) \leq \omega.
\end{array}\right.
\end{equation}
Then, the following bound holds:
\begin{equation}\label{eq:SGM2_full_shuffling_key_bound1}
\hspace{-0.0ex}
\begin{array}{lcl}
\Psi_0(\widetilde{w}_t)  & + &  \frac{\omega L_u^2\eta_t}{2} \norms{\widetilde{u}_t  - u_0^{*}(\widetilde{w}_{t})}^2  \leq   \Psi_0(\widetilde{w}_{t-1}) +  \frac{ \omega L_u^2\eta_t}{2} \norms{\widetilde{u}_{t-1}  - u_0^{*}(\widetilde{w}_{t-1})}^2 \\
&& - {~} \frac{\eta_t B_t}{2} \norms{\Gc_{\eta_t}(\widetilde{w}_{t-1})}^2 +  \big[ 3L_w^2\sigma_{w}^2 + L_w^2 (3 \Theta_w + 1)\Lambda_1 \big] \cdot \eta_t^3  \\
&& + {~} \frac{  L_u^3}{n}(  1 +  \omega +  2\omega^2 L_u^2\kappa^2 \eta_t^2 + 4L_w^2 \eta_t^2 ) \big[ \Lambda_1 (\Theta_u + 1 ) + \sigma_u^2 \big] \cdot C_S \eta_t \hat{\eta}_t^3,
\end{array}
\hspace{-3.0ex}
\end{equation}
where $B_t := 1 - 2L_{\Phi_0}\eta_t - 2L_w^2 (3 \Theta_w + 1)\Lambda_0 \eta_t^2 - \frac{2L_u^3}{n}(  1 +  \omega + 2\omega^2 L_u^2\kappa^2 \eta_t^2 + 4L_w^2 \eta_t^2) \hat{\eta}_t^3 \cdot C_S \cdot \Lambda_0(\Theta_u + 1 )$ for $C_S$ given in Lemma~\ref{le:SGM2_key_bounds_for_ut2}.
\end{lemma}

%%% Proof of Theorem 4.1.
\begin{proof}
First, combining \eqref{eq:SGM2_key_bounds_for_Phi_2c} and \eqref{eq:SGM2_key_bounds_for_wt_2}, we get
\begin{equation*} 
\begin{array}{lcl}
\Psi_0(\widetilde{w}_t) &\leq & \Psi_0(\widetilde{w}_{t-1})  - \frac{\eta_t(1 - 2L_{\Phi_0}\eta_t)}{2}  \Vert {\Gc}_{\eta_t}(\widetilde{w}_{t-1})  \Vert^2 - \frac{(1-L_{\Phi_0}\eta_t)}{2\eta_t} \norms{ \widetilde{w}_t  - \widetilde{w}_{t-1} }^2  + 3L_w^2 \sigma_{w}^2  \eta_t^3 \\
&& + {~} \frac{L_u^2 \eta_t}{2}\big( 1 + 4L_w^2 \eta_t^2 \big) \norms{\widetilde{u}_t - u_0^{*}(\widetilde{w}_{t-1})}^2 +  L_w^2 (3 \Theta_w + 1) \eta_t^3 \norms{ \nabla{\Phi_0}(\widetilde{w}_{t-1}) }^2.
\end{array}
\end{equation*}
By \eqref{eq:u_star_smoothness} and Young's inequality in $\myeqc{1}$, for any $s_t > 0$, we have
\begin{equation*} 
\begin{array}{lcl}
\norms{\widetilde{u}_t  - u_0^{*}(\widetilde{w}_{t})}^2 & \overset{\myeqc{1}}{\leq} & (1+s_t)\norms{\widetilde{u}_t  - u_0^{*}(\widetilde{w}_{t-1})}^2 + \frac{(1+s_t)}{s_t}\norms{ u_0^{*}(\widetilde{w}_{t})   - u_0^{*}(\widetilde{w}_{t-1})}^2\\
& \overset{\tiny\eqref{eq:u_star_smoothness}}{\leq} & (1+s_t)\norms{\widetilde{u}_t  - u_0^{*}(\widetilde{w}_{t-1})}^2 + \frac{(1+s_t)\kappa^2}{s_t}\norms{ \widetilde{w}_{t}    -  \widetilde{w}_{t-1}}^2.
\end{array}
\end{equation*}
Multiplying this inequality by $\frac{\omega L_u^2\eta_t}{2}$ for some $\omega > 0$ and adding the result to the last estimate, we can show that
\begin{equation*} 
\begin{array}{lcl}
\Tc_{[1]} &:= & \Psi_0(\widetilde{w}_t) + \frac{\omega L_u^2\eta_t}{2} \norms{\widetilde{u}_t  - u_0^{*}(\widetilde{w}_{t})}^2 \\
&\leq & \Psi_0(\widetilde{w}_{t-1}) + \frac{L_u^2\eta_t}{2} \big[ 1 + \omega(1+s_t) + 4L_w^2 \eta_t^2 \big) \norms{\widetilde{u}_t  - u_0^{*}(\widetilde{w}_{t-1})}^2 \\ 
&& - {~} \frac{\eta_t(1 - 2L_{\Phi_0}\eta_t)}{2}  \Vert {\Gc}_{\eta_t}(\widetilde{w}_{t-1})  \Vert^2 - \big[ \frac{1-L_{\Phi_0}\eta_t}{2\eta_t} - \frac{\omega L_u^2\kappa^2 \eta_t(1+s_t)}{2s_t} \big] \norms{ \widetilde{w}_t  - \widetilde{w}_{t-1} }^2 \\
&& + {~}  L_w^2 (3 \Theta_w + 1) \eta_t^3 \norms{ \nabla{\Phi_0}(\widetilde{w}_{t-1}) }^2 + 3L_w^2 \eta_t^3 \sigma_{w}^2 \\
& \overset{\tiny\eqref{eq:SGM2_key_est_for_ut2c}}{ \leq } & \Psi_0(\widetilde{w}_{t-1}) + \frac{L_u^2\eta_t }{2(1 + 2\mu_h\hat{\eta}_t)^S} \big(1 - \frac{\mu_H\hat{\eta}_t}{n} \big)^{nS} \big[  1 + \omega (1+s_t) + 4L_w^2 \eta_t^2 \big] \norms{\widetilde{u}_{t-1}  - u_0^{*}(\widetilde{w}_{t-1})}^2 \\ 
&& - {~} \frac{\eta_t(1 - 2L_{\Phi_0}\eta_t)}{2}  \Vert {\Gc}_{\eta_t}(\widetilde{w}_{t-1})  \Vert^2 - \big[ \frac{1-L_{\Phi_0}\eta_t}{2\eta_t} - \frac{\omega L_u^2\kappa^2 \eta_t(1+s_t)}{2s_t} \big] \norms{ \widetilde{w}_t  - \widetilde{w}_{t-1} }^2 \\
&& + {~}  L_w^2 (3 \Theta_w + 1) \eta_t^3 \norms{ \nabla{\Phi_0}(\widetilde{w}_{t-1}) }^2 + 3L_w^2 \eta_t^3 \sigma_{w}^2 \\
&& + {~}  \frac{L_u^3}{n}\big[  1 + \omega(1 + s_t) + 4L_w^2 \eta_t^2 \big] \eta_t \hat{\eta}_t^3 \cdot C_S \cdot \big[ (\Theta_u + 1 ) \norms{\nabla{\Phi_0}(\widetilde{w}_{t-1} )}^2  + \sigma_{u}^2 \big] .
\end{array}
\end{equation*}
We need to choose the parameters $\eta_t$, $\hat{\eta}_t$, and $s_t$ such that
\begin{equation*} 
\begin{array}{ll}
& \frac{1}{(1 + 2\mu_h\hat{\eta}_t)^S }\big(1 - \frac{\mu_H\hat{\eta}_t}{n} \big)^{Sn} \big[ 1 + \omega(1 + s_t) + 4L_w^2 \eta_t^2 \big] \leq \omega \\
&  \frac{1-L_{\Phi_0}\eta_t}{2\eta_t} - \frac{\omega L_u^2\kappa^2 \eta_t(1+s_t)}{2s_t} \geq 0.
\end{array}
\end{equation*}
The second one leads to $\frac{ 1 - L_{\Phi_0}\eta_t - \omega L_u^2\kappa^2 \eta_t^2}{\omega L_u^2\kappa^2 \eta_t^2} \geq \frac{ 1}{s_t}$, or equivalently $0 < s_t \leq  \frac{\omega L_u^2\kappa^2 \eta_t^2}{ 1 - L_{\Phi_0}\eta_t - \omega L_u^2\kappa^2 \eta_t^2}$.
If $2L_{\Phi_0}\eta_t + 2\omega L_u^2\kappa^2 \eta_t^2 \leq 1$ as stated in the first line of \eqref{eq:SGM2_full_shuffling_para_cond1}, then we can choose $s_t := 2\omega L_u^2\kappa^2 \eta_t^2$.
In this case, the second condition above holds, and the first condition becomes
\begin{equation*} 
\begin{array}{ll}
& \frac{1}{(1 + 2\mu_h\hat{\eta}_t)^S}\big(1 - \frac{\mu_H\hat{\eta}_t}{n} \big)^{nS} \big( 1 + \omega + 2\omega^2 L_u^2\kappa^2 \eta_t^2 + 4L_w^2 \eta_t^2 \big) \leq \omega,
\end{array}
\end{equation*}
which is exactly the second line of \eqref{eq:SGM2_full_shuffling_para_cond1}.

By \eqref{eq:grad_mapp_bound} from Assumption~\ref{as:A3}, we have
\begin{equation*} 
\begin{array}{lcl}
\Tc_{[1]} &:= & \Psi_0(\widetilde{w}_t) + \frac{\omega L_u^2\eta_t}{2} \norms{\widetilde{u}_t  - u^{*}(\widetilde{w}_{t})}^2 \\
& \leq &  \Psi_0(\widetilde{w}_{t-1}) +  \frac{\omega L_u^2\eta_t}{2} \norms{\widetilde{u}_{t-1}  - u^{*}(\widetilde{w}_{t-1})}^2  -  \frac{\eta_t(1 - 2L_{\Phi_0}\eta_t)}{2}  \norms{  \Gc_{\eta_t}(\widetilde{w}_{t-1}) }^2 \\
&& + {~}  L_w^2 (3 \Theta_w + 1)\Lambda_0 \eta_t^3  \norms{  \Gc_{\eta_t}(\widetilde{w}_{t-1}) }^2 +  \big[ 3L_w^2\sigma_{w}^2 + L_w^2 (3 \Theta_w + 1)\Lambda_1 \big] \eta_t^3  \\
&& + {~}  \frac{L_u^3}{n}(  1 + \omega +  2\omega^2 L_u^2\kappa^2 \eta_t^2 + 4L_w^2 \eta_t^2) \eta_t \hat{\eta}_t^3 \cdot C_S \cdot \Lambda_0(\Theta_u + 1 ) \norms{  \Gc_{\eta_t}(\widetilde{w}_{t-1}) }^2 \\
&& + {~} \frac{L_u^3}{n} (  1 + \omega +  2\omega^2 L_u^2\kappa^2 \eta_t^2 + 4L_w^2 \eta_t^2) \eta_t \hat{\eta}_t^3 \cdot C_S \big[ \Lambda_1 (\Theta_u + 1 ) + \sigma_u^2 \big]. 
\end{array}
\end{equation*}
Rearranging this inequality, we  prove \eqref{eq:SGM2_full_shuffling_key_bound1}.
\end{proof}
%%% End of proof.

The following theorem, Theorem~\ref{th:SGM2_full_shuffling_main_result1_supp}, is the full version of Theorem~\ref{th:SGM2_full_shuffling_main_result1} in the main text, where the learning rates $\eta_t$ and  $\hat{\eta}_t$, and the numbers of epochs $S$ and $T$ are given explicitly.

%%% Theorem 4.1.
\begin{theorem}[Strong convexity of $\Hc_i$]\label{th:SGM2_full_shuffling_main_result1_supp}
Suppose that Assumptions~\ref{as:A0},~\ref{as:A2},~\ref{as:A3_SGM2}, and ~\ref{as:A3} hold for \eqref{eq:minimax_prob}, and $\Hc_i$ is $\mu_H$-strongly concave with $\mu_H > 0$ for all $i\in [n]$, but $h$ is only merely convex.
Let $\Psi_0$ be defined by \eqref{eq:upper_level_min}, and ${\Gc}_{\eta}$ be defined by \eqref{eq:upper_level_min_grad_mapping}.
We define $C_w$ and $C_u$ respectively as
\begin{equation}\label{eq:SGM2_full_shuffling_constants}
\begin{array}{lcl}
C_w & := &   L_w^2\big[ (3 \Theta_w + 1)\Lambda_1 + 3\sigma_{w}^2 \big] \quad\text{and} \quad  C_u :=  \frac{7L_u^3}{2\mu_H} \big[ \Lambda_1 (\Theta_u + 1 ) + \sigma_u^2 \big].
\end{array}
\end{equation}
Let $\sets{(\widetilde{w}_t, \widetilde{u}_t)}$ be generated by Algorithm~\ref{alg:SGM2} using  $S$ epochs of \textbf{shuffling routine \eqref{eq:SGM4lower_prob}}, and fixed learning rates $\eta_t := \eta > 0$ and $\hat{\eta}_t := \hat{\eta}$ such that
\begin{equation}\label{eq:SGM2_full_shuffling_para_cond1_s}
\begin{array}{lll}
S  := \big\lfloor \frac{\ln(7/2)}{\mu_H \hat{\eta} } \big\rfloor, \qquad  0 < \eta \leq \bar{\eta},  \quad \text{and}\quad 0 < \hat{\eta} \leq \bar{\hat{\eta}},
\end{array}
\end{equation}
where 
\begin{equation}\label{eq:SGM2_full_shuffling_para_cond1_s2b}
\begin{array}{lll}
\bar{\eta} := \min\Big\{ \frac{1}{4L_w\sqrt{2\Lambda_0(\Theta_w+1)}}, \frac{1}{4L_{\Phi_0}}, \frac{1}{2L_{w}},    \frac{1}{2L_u\kappa} \Big\} \quad \text{and} \quad \bar{\hat{\eta}} := \frac{\sqrt{\mu_H}}{2\sqrt{14 \Lambda_0 L_u^3 (\Theta_u + 1 )} }.
\end{array}
\end{equation}

Then, the following bounds hold:
\begin{equation}\label{eq:SGM2_full_shuffling_bound1}
\begin{array}{lcl}
\frac{1}{T+1}\sum_{t=0}^T \norms{ {\Gc}_{\eta}(\widetilde{w}_{t}) }^2 & \leq & \frac{4 [ 2\Psi_0(\widetilde{w}_0) - 2\Psi_0^{\star} + L_u^2\eta\norms{\widetilde{u}^0 - u^{*}(\widetilde{w}_0)}^2]}{\eta (T+1)} + 8(C_w + C_u) \eta^2.
\end{array}
\end{equation}
For a given $\epsilon > 0$, if we choose both $\eta := \Oc(\epsilon)$ and $\hat{\eta} := \Oc(\epsilon)$ satisfying \eqref{eq:SGM2_full_shuffling_para_cond1} and $T := \mathcal{O}\big( \frac{1}{\epsilon^3} \big)$, then $\frac{1}{T+1}\sum_{t=0}^T \norms{ {\Gc}_{\eta}(\widetilde{w}_{t}) }^2 \leq \epsilon^2$.

Consequently, Algorithm~\ref{alg:SGM2} requires $\mathcal{O}\big( \frac{n}{\epsilon^4} \big)$ evaluations of $\nabla_u{\Hc_i}$ and $\mathcal{O}\big( \frac{n}{\epsilon^3} \big)$ evaluations of $\nabla_w{\Hc}_i$  to achieve an $\epsilon$-stationary point $\widehat{w}_T$ of \eqref{eq:minimax_prob} computed by \eqref{eq:approx_stationary_point}.
This algorithm also requires $\Oc(\epsilon^{-3})$ evaluations of $\prox_{\eta_tf}$ and $\Oc(\epsilon^{-4})$ evaluations of $\prox_{\hat{\eta}_th}$.
\end{theorem}

%%% Proof of Theorem 4.1.
\begin{proof}[\textbf{Proof of Theorem~\ref{th:SGM2_full_shuffling_main_result1_supp}}]
Since $\mu_H > 0$ and $\mu_h = 0$, for $C_S$ given in Lemma~\ref{le:SGM2_key_bounds_for_ut2}, it reduces to 
\begin{equation*}
\begin{array}{lcl}
C_S & := &  \big[ \sum_{j=0}^{n-1}  \frac{ 1 }{(1 + 2\mu_h\hat{\eta}_t)} \big( 1 -  \frac{\mu_H\hat{\eta}_t}{ n}  \big)^j  \big] \cdot \big[ \sum_{s=0}^{S-1} \frac{1}{(1 + 2 \mu_h\hat{\eta}_t)^s} \big(1 - \frac{\mu_H\hat{\eta}_t}{ n} \big)^{ns} \big]  \\
& = & \sum_{s=0}^{S-1}\sum_{j=0}^{n-1}\big(1 - \frac{\mu_H\hat{\eta}_t}{n} \big)^{ns+j} \\
& \leq & \frac{n}{\mu_H\hat{\eta}_t}.
\end{array}
\end{equation*}
In this case, we can lower bound $B_t$ from Lemma~\ref{le:SGM2_keybound_for_full_shuffling} as
\begin{equation*}
\begin{array}{lcl}
B_t  & := &  1 - 2L_{\Phi_0}\eta_t - 2L_w^2 (3 \Theta_w + 1)\Lambda_0 \eta_t^2  \\
&& - {~} \frac{2 L_u^3}{n}(  1 + \omega + 2\omega^2 L_u^2\kappa^2 \eta_t^2 + 4L_w^2 \eta_t^2) \hat{\eta}_t^3 \cdot C_S \cdot \Lambda_0 (\Theta_u + 1 ) \\
& \geq & 1 - 2L_{\Phi_0}\eta_t - 2L_w^2 (3 \Theta_w + 1)\Lambda_0 \eta_t^2 - \frac{2L_u^3}{\mu_H}(  1 + \omega + 2\omega^2 L_u^2\kappa^2 \eta_t^2 + 4L_w^2 \eta_t^2)  \Lambda_0(\Theta_u + 1 ) \cdot \hat{\eta}_t^2.
\end{array}
\end{equation*}
Since $\eta_t := \eta \in (0, \bar{\eta}]$   for $\bar{\eta}$  satisfying \eqref{eq:SGM2_full_shuffling_para_cond1_s2b}, we have $\eta \leq \frac{1}{4L_{\Phi_0}}$, $\eta\leq \frac{1}{2L_{w}}$, and $\eta \leq \frac{1}{2L_u\kappa}$.
Moreover, we choose $\omega := 1$ and  $\hat{\eta}_t := \hat{\eta} \in (0, \bar{\hat{\eta}}]$.
Hence, we can further lower bound $B_t$ as
\begin{equation}\label{eq:SGM2_full_shuffling_bound1_proof1}
\begin{array}{lcl}
B_t  & \geq & \frac{1}{2} - 2 \Lambda_0 L_w^2 (3 \Theta_w + 1)\eta^2 - \frac{2  \Lambda_0 L_u^3(1 + \omega + \omega M_{\omega}\eta^2)}{\mu_H}(\Theta_u + 1 ) \hat{\eta}^2 \\
& = & \frac{1}{2} - 2 \Lambda_0 L_w^2 (3 \Theta_w + 1)\eta^2 - \frac{2  \Lambda_0 L_u^3(2 +  M_1 \eta^2)}{\mu_H}(\Theta_u + 1 ) \hat{\eta}^2,
\end{array}
\end{equation}
where $M_{\omega} := 2\omega L_u^2\kappa^2 + \frac{4L_w^2}{\omega} = M_1 = 2L_u^2\kappa^2 + 4L_w^2$.

We can see that the second condition of \eqref{eq:SGM2_full_shuffling_para_cond1} holds if 
\begin{equation*}
\begin{array}{lll}
\quad & \big(1 - \frac{\mu_H\hat{\eta} }{n} \big)^{nS} \big( 2  + M_1 \eta^2 \big) \leq 1,\\
\Leftrightarrow \quad & -nS \ln \big(1 - \frac{\mu_H\hat{\eta} }{n} \big) \geq \ln\big( 2 + M_1 \eta^2 \big).
\end{array}
\end{equation*}
Since $\eta \leq \frac{1}{4L_{\Phi_0}}$ and $\eta\leq \frac{1}{2L_{w}}$, we have $M_1\eta^2 = (2L_u^2\kappa^2 + 4L_w^2)\eta^2 \leq \frac{3}{2}$.
Using this relation, and  $-\ln(1 - \tau) \geq \tau$  for $\tau \in (0, 1)$, the last condition holds if
\begin{equation*}
\begin{array}{lcl}
\mu_H\hat{\eta} S \geq \ln(7/2) \quad  \Leftrightarrow \quad S \geq \frac{\ln(7/2)}{\mu_H\hat{\eta}}.
\end{array}
\end{equation*}
Hence, we can choose $S := \lfloor   \frac{\ln(7/2)}{\mu_H \hat{\eta}} \rfloor$ as stated in \eqref{eq:SGM2_full_shuffling_para_cond1_s}.

The condition \eqref{eq:SGM2_full_shuffling_bound1_proof1} holds if
\begin{equation*} 
\begin{array}{lcl}
B_t  & \geq & \frac{1}{2} - 2 \Lambda_0 L_w^2 (3 \Theta_w + 1)\eta^2 - \frac{7 \Lambda_0 L_u^3 (\Theta_u + 1 ) }{\mu_H} \hat{\eta}^2.
\end{array}
\end{equation*}
This condition shows that if we choose
\begin{equation*} 
\begin{array}{lcl}
0 < \eta  \leq   \frac{1}{4L_w\sqrt{2\Lambda_0(\Theta_w+1)}} \quad \text{and} \quad 0 < \hat{\eta} \leq \bar{\hat{\eta}} := \frac{\sqrt{\mu_H}}{2\sqrt{14 \Lambda_0 L_u^3 (\Theta_u + 1 )} },
\end{array}
\end{equation*}
then, from \eqref{eq:SGM2_full_shuffling_bound1_proof1}, we have $B_t \geq \frac{1}{4}$.
Due to \eqref{eq:SGM2_full_shuffling_para_cond1_s}, both conditions here are satisfied. 

Next, let us define $C_w$ and $C_u$ as in \eqref{eq:SGM2_full_shuffling_constants}, respectively, i.e.:
\begin{equation*}
\begin{array}{lcl}
C_w & := &  L_w^2 \big[  (3 \Theta_w + 1)\Lambda_1 + 3\sigma_{w}^2 \big], \quad \text{and} \quad C_u :=  \frac{7L_u^3}{\mu_H} \big[ \Lambda_1 (\Theta_u + 1 ) + \sigma_u^2 \big].
\end{array}
\end{equation*}
In this case, \eqref{eq:SGM2_full_shuffling_key_bound1} reduces to
\begin{equation*} 
\begin{array}{lcl}
\Psi_0(\widetilde{w}_t)  +  \frac{L_u^2\eta}{2} \norms{\widetilde{u}_t  - u_0^{*}(\widetilde{w}_{t})}^2  & \leq &   \Psi_0(\widetilde{w}_{t-1}) +  \frac{L_u^2\eta}{2} \norms{\widetilde{u}_{t-1}  - u_0^{*}(\widetilde{w}_{t-1})}^2 \\
&& - {~} \frac{\eta}{8} \norms{\Gc_{\eta}(\widetilde{w}_{t-1})}^2 +   C_w \eta^3  + C_u  \hat{\eta}^2 \eta.
\end{array}
\end{equation*}
Subtracting $\Psi_0^{\star}$ from both sides of this inequality, and averaging the result from $t=0$ to $T$, and noting that $\Psi_0(\widetilde{w}_T) - \Psi_0^{\star} \geq 0$, we obtain \eqref{eq:SGM2_full_shuffling_bound1}.

To achieve an $\epsilon$-stationary point of \eqref{eq:upper_level_min}, from \eqref{eq:SGM2_full_shuffling_bound1}, we need to impose the following condition:
\begin{equation*}
\begin{array}{lcl}
\frac{2 [ \Psi_0(\widetilde{w}_0) - \Psi_0^{\star}]}{\eta (T+1)} + \frac{L_u^2\norms{\widetilde{u}^0 - u_0^{*}(\widetilde{w}_0)}^2}{T+1} + 2C_w\eta^2 + 2C_u \hat{\eta}^2 \leq \frac{\epsilon^2}{4}.
\end{array}
\end{equation*}
Since other terms are constant, if we choose $\eta := \Oc(\epsilon)$ and $\hat{\eta} := \Oc(\epsilon)$ such that they still satisfies \eqref{eq:SGM2_full_shuffling_para_cond1_s}, then we can choose $T := \Oc(\epsilon^{-3})$ to guarantee the last condition.
Since each iteration $t$, we run $S$ epochs of the shuffling scheme \eqref{eq:SGM4lower_prob}, the total number of evaluations of $\nabla_u{\Hc}_i$ is $\Tc_u := T\times S\times n$.
However, we have $S = \lfloor \frac{\ln(7/2)}{\mu_H \hat{\eta} } \rfloor = \Oc(\epsilon^{-1})$, we get $\Tc_u := \Oc(n\epsilon^{-4})$.
The total number of evaluations of $\nabla_w{\Hc}_i$ is $\Tc_w := Tn = \Oc(n\epsilon^{-3})$ as stated.

Finally, since each epoch $t$, Algorithm~\ref{alg:SGM2} requires one evaluation of $\prox_{\eta_tf}$, and $S$ evaluations of $\prox_{\hat{\eta}_th}$, but since $S = \Oc(\epsilon^{-1})$, the total number of  $\prox_{\eta_tf}$ evaluations is $T = \Oc(\epsilon^{-3})$, while the total number of $\prox_{\hat{\eta}_th}$ evaluations is $TS = \Oc(\epsilon^{-4})$.
Overall, Algorithm~\ref{alg:SGM2} needs $\Oc(\epsilon^{-3})$ evaluations of  $\prox_{\eta_tf}$ and $\Oc(\epsilon^{-4})$ evaluations of $\prox_{\hat{\eta}_th}$.
\end{proof}
%%% End of proof.

%%% Proof of Theorem 4.1.
\begin{proof}[\textbf{Proof of Theorem~\ref{th:SGM2_full_shuffling_main_result2}}]
Since $\mu_h > 0$ and $\mu_H = 0$, for $C_S$ given by Lemma~\ref{le:SGM2_key_bounds_for_ut2}, it reduces to 
\begin{equation*}
\begin{array}{lcl}
C_S & := &  \big[ \sum_{j=0}^{n-1}  \frac{ 1 }{(1 + 2 \mu_h\hat{\eta}_t)}  \big(1 - \frac{\mu_H\hat{\eta}_t}{n}  \big)^j \big]  \cdot \big[ \sum_{s=0}^{S-1} \frac{1}{(1 + 2 \mu_h\hat{\eta}_t)^s} \big(1 - \frac{\mu_H\hat{\eta}_t }{ n }\big)^{ns} \big] \\
& = &   \sum_{s=0}^{S-1} \frac{1}{(1 + 2\mu_h\hat{\eta}_t)^{s+1}} \leq \frac{1}{2\mu_h\hat{\eta}_t}.
\end{array}
\end{equation*}
In this case, we can lower bound $B_t$ from Lemma~\ref{le:SGM2_keybound_for_full_shuffling} as in Theorem~\ref{th:SGM2_full_shuffling_main_result1}, i.e.:
\begin{equation*}
\begin{array}{lcl}
B_t   & \geq & 1 - 2L_{\Phi_0}\eta_t - 2L_w^2 (3 \Theta_w + 1)\Lambda_0 \eta_t^2 - \frac{2 L_u^3}{ n \mu_H}(  1 +  L_u^2\kappa^2 \eta_t^2 + 2L_w^2 \eta_t^2)  \Lambda_0 (\Theta_u + 1 ) \cdot \hat{\eta}_t^2.
\end{array}
\end{equation*}
We need to choose $\eta$ as in Theorem~\ref{th:SGM2_full_shuffling_main_result1}.
Since $\eta_t = \eta \in (0, \bar{\eta}]$, we have $\eta \leq \frac{1}{4L_{\Phi_0}}$, $\eta\leq \frac{1}{2L_{w}}$, and $\eta \leq \frac{1}{2L_u\kappa}$.
Alternatively, we also have  $\hat{\eta}_t :=  \hat{\eta} \in (0, \bar{\hat{\eta}}]$.
Therefore, we can further lower bound $B_t$ as $B_t \geq \frac{1}{4}$ as in Theorem~\ref{th:SGM2_full_shuffling_main_result1}.

Now, the second condition of \eqref{eq:SGM2_full_shuffling_para_cond1} holds if 
\begin{equation*} 
\begin{array}{lcl}
(1 + 2\mu_h \hat{\eta} \big)^S \geq 2 + (2 L_u^2\kappa^2 + 4L_w^2) \eta^2.
\end{array}
\end{equation*}
Using the fact that $\ln(1 + e) \geq \frac{e}{2}$ for $e \in (0, 1/2)$ and $(2 L_u^2\kappa^2 + 4L_w^2) \eta^2 \leq \frac{3}{2}$,  the last inequality holds if $\mu_h \hat{\eta} S \geq \ln(7/2)$.
Hence, we can choose $S := \big\lfloor \frac{\ln(7/2)}{\mu_h \hat{\eta}} \big\rfloor$ as stated in Theorem~\ref{th:SGM2_full_shuffling_main_result2}.
The remaining proof follows from Theorem~\ref{th:SGM2_full_shuffling_main_result1}.
\end{proof}
%%% End of proof.

%%%%%%%%%%%%%%%%%%%%%%%%%%%%%%%%%%%%%%%%%%%%%%%%%%%%%%%%%%%%
%%% End of Supplementary Documents.
%%%%%%%%%%%%%%%%%%%%%%%%%%%%%%%%%%%%%%%%%%%%%%%%%%%%%%%%%%%%
%%%%%%%%%%%%%%%%%%%%%%%%%%%%%%%%%%%%%%%%%%%%%%%%%%
%%% C. Convergence Analysis of Algorithm 2.
%%%%%%%%%%%%%%%%%%%%%%%%%%%%%%%%%%%%%%%%%%%%%%%%%%
\beforesubsec
\subsection{Convergence of the full-shuffling variant of Algorithm~\ref{alg:SGM2} -- The case $S=1$}\label{apdx:subsec:SGM2_convergence}
\aftersubsec
In this subsection, we analyze the convergence of Algorithm~\ref{alg:SGM2} using only one epoch of the \textbf{\textit{shuffling gradient ascent}} scheme \eqref{eq:SGM4lower_prob}.
In this case, by dropping the superscript $s$, the scheme \eqref{eq:SGM4lower_prob} can be simplified as follows:
\begin{equation}\label{eq:SGM3_SGM4lower_prob}
\left\{\begin{array}{lll}
&u_0^{(t)} := \widetilde{u}_{t-1}, \\ 
& \textrm{For $i=1,2, \cdots, n$, update}\\
&\qquad u_{i}^{(t)} :=  u_{i-1}^{(t)} + \frac{\hat{\eta}_t}{n}  \nabla_u\mathcal{H}_{\pi^{(t)}(i)}(\widetilde{w}_{t-1}, u_{i-1}^{(t)}), \\
& \widetilde{u}_t :=   \prox_{\hat{\eta}_t h}(u_n^{(t)}),
\end{array}\right.
\end{equation}
where $\hat{\eta}_t > 0$ is a given learning rate.

We divide our analysis into different tasks as follows.

%%%% ++++++++++++++++++++++++++++++++++++++++++++++++++++++++++++++++++++++++
%%%% C.1. Potential function and properties
%%%% ++++++++++++++++++++++++++++++++++++++++++++++++++++++++++++++++++++++++
\beforesubsubsec
\subsubsection{Potential function and a technical lemma}\label{apdx:subsubsec:convergence_of_SGM3_p1}
\aftersubsubsec
One key step of our analysis is to construct an appropriate potential function.
We exploit the ideas from \cite{cho2022sgda} to construct this function as follows.

For $\Psi_0$ defined by \eqref{eq:upper_level_min} and $\Lc$ given in \eqref{eq:minimax_prob}, we consider the following \textbf{potential function}:
\begin{equation}\label{eq:SGM3_potential_func}
\arraycolsep=0.2em
\begin{array}{lcl}
\Vc_{\lambda}(w, u) := \lambda \big[ \Psi_0(w) - \Psi_0^{\star} \big] + \Psi_0(w) - \Lc(w, u), 
\end{array}
\end{equation}
where $\lambda > 0$ is a given parameter determined later.
Since $\Psi_0(w) \geq  \Psi_0^{\star} := \inf_{w}\Psi_0(w)$ and $\Psi_0(w) = \sup_{u}\Lc(w, u) \geq \Lc(w, u)$, it is obvious to see that $\Vc_{\lambda}(w, u) \geq 0$ for all $(w, u) \in \dom{\Lc}$.

Similar to \eqref{eq:upper_level_min_grad_mapping}, we consider the following gradient mappings for both \eqref{eq:lower_level_max} and \eqref{eq:upper_level_min}, respectively:
\begin{equation}\label{eq:SGM3_gradient_mappings}
\arraycolsep=0.2em
\begin{array}{lcllcl}
\hat{\Gc}_{\hat{\eta}_t }( \widetilde{u}_{t-1} ) := \frac{1}{\hat{\eta}_t }\big( \widetilde{u}_{t-1} -  \widehat{u}_t \big), \qquad & \text{where} & \quad \widehat{u}_t := \prox_{\hat{\eta}_t h}\big( \widetilde{u}_{t-1} + \hat{\eta}_t \nabla_u\Hc(\widetilde{w}_{t-1},  \widetilde{u}_{t-1} ) \big), \\
\Gc_{ \eta_t }(\widetilde{w}_{t-1} ) := \frac{1}{\eta_t}\big( \widetilde{w}_{t-1} -  \widehat{w}_t \big), \qquad & \text{where} & \quad \widehat{w}_t := \prox_{\eta_t f }\big( \widetilde{w}_{t-1} - \eta_t \nabla{\Phi}_0(\widetilde{w}_{t-1})   \big).
\end{array}
\end{equation}
We need the following result.

%%% Lemma 18.
\begin{lemma}\label{le:SGM3_key_properties_of_psi}
Let $\widehat{u}_t$ and $\hat{\Gc}_{\hat{\eta}_t }( \widetilde{u}_{t-1} )$ be defined by \eqref{eq:SGM3_gradient_mappings}, and $\psi$ be defined by \eqref{eq:SGM3_psi_func}.
Then, we have
\begin{equation}\label{eq:SGM3_key_properties_of_psi_p1}
\arraycolsep=0.2em
\begin{array}{lcl}
2\mu_{\psi} \norms{\widehat{u}_t - u_0^{*}(\widetilde{w}_{t-1} )}^2 & \leq &  \psi( \widetilde{w}_{t-1}, \widehat{u}_t ) - \psi( \widetilde{w}_{t-1}, u_0^{*}( \widetilde{w}_{t-1} )).
\end{array}
\end{equation}
If  $0 < \hat{\eta}_t \leq \frac{2}{L_u + \mu_H}$, then we have
\begin{equation}\label{eq:SGM3_key_properties_of_psi_p2}
\arraycolsep=0.2em
\begin{array}{lcl}
2\mu_{\psi}[ \psi( \widetilde{w}_{t-1}, \widehat{u}_t ) - \psi( \widetilde{w}_{t-1}, u_0^{*}( \widetilde{w}_{t-1} )) ] \leq   \left(1 - \frac{2L_u\mu_H\hat{\eta}_t}{L_u + \mu_H}\right) \norms{ \hat{\Gc}_{\hat{\eta}_t }( \widetilde{u}_{t-1} )  }^2.
\end{array}
\end{equation}
\end{lemma}

%%% Proof of Lemma 18.
\begin{proof} 
Since $\psi(\widetilde{w}_{t-1}, \cdot)$ is $\mu_{\psi}$-strongly convex and $u_0^{*}(\widetilde{w}_{t-1}) := \argmin_u \psi(\widetilde{w}_{t-1}, u)$, for $\widehat{u}_t$ given in \eqref{eq:SGM3_gradient_mappings},  we easily obtain \eqref{eq:SGM3_key_properties_of_psi_p1}.

Next, using again the $\mu_{\psi}$-strong convexity of  $\psi(\widetilde{w}_{t-1}, \cdot)$ and $u_0^{*}(\widetilde{w}_{t-1}) := \argmin_u \psi(\widetilde{w}_{t-1}, u)$, for $\widehat{u}_t$ given in \eqref{eq:SGM3_gradient_mappings},  by  \cite[Theorem 2.1.10]{Nesterov2004}, we have 
\begin{equation}\label{eq:SGM3_key_properties_of_psi_proof1}
\arraycolsep=0.2em
\begin{array}{lcl}
2\mu_{\psi}[ \psi( \widetilde{w}_{t-1}, \widehat{u}_t ) - \psi( \widetilde{w}_{t-1}, u_0^{*}( \widetilde{w}_{t-1} )) ] \leq \norms{\nabla_u{\psi}( \widetilde{w}_{t-1}, \widehat{u}_t )}^2.
\end{array}
\end{equation}
where $\nabla_u{\psi}( \widetilde{w}_{t-1}, \widehat{u}_t ) = -\nabla_u{\Hc}( \widetilde{w}_{t-1}, \widehat{u}_t ) + \nabla{h}( \widehat{u}_t ) \in \partial{\psi}(\widehat{w}_t) := -\nabla_u{\Hc}(\widetilde{w}_{t-1}, \widehat{u}_t) + \partial{h}( \widehat{u}_t )$.

Now, for $\widehat{u}_t$ defined by \eqref{eq:SGM3_gradient_mappings}, we have $\frac{1}{\hat{\eta}_t}(\widetilde{u}_{t-1} - \widehat{u}_t) + \nabla_u\Hc(\widetilde{w}_{t-1},  \widetilde{u}_{t-1} ) \in \partial{h}(\widehat{u}_t)$, leading to 
\begin{equation}\label{eq:SGM3_key_properties_of_psi_proof2}
\arraycolsep=0.2em
\begin{array}{lcl}
\nabla_u{\psi}(\widetilde{w}_{t-1}, \widehat{u}_t) := \frac{1}{\hat{\eta}_t}(\widetilde{u}_{t-1} - \widehat{u}_t) + \nabla_u\Hc(\widetilde{w}_{t-1},  \widetilde{u}_{t-1} ) - \nabla_u\Hc(\widetilde{w}_{t-1},  \widehat{u}_t )  \in \partial{\psi}(\widehat{u}_t).
\end{array}
\end{equation}
Since $-\Hc(\widetilde{w}_{t-1},  \cdot )$ is  $L_u$-smooth and  $\mu_H$-strongly convex, by \cite[Theorem 2.1.12]{Nesterov2004}, we have
\begin{equation}\label{eq:SGM3_key_properties_of_psi_proof3}
\arraycolsep=0.2em
\begin{array}{ll}
& -\iprods{\nabla_u\Hc(\widetilde{w}_{t-1},  \widetilde{u}_{t-1} ) - \nabla_u\Hc(\widetilde{w}_{t-1},  \widehat{u}_t ), \widetilde{u}_{t-1} - \widehat{u}_t }  \geq    \frac{L_u\mu_H}{L_u + \mu_H}\norms{ \widetilde{u}_{t-1} - \widehat{u}_t }^2 \\
& \qquad\qquad + {~} \frac{1}{L_u + \mu_H}\norms{\nabla_u\Hc(\widetilde{w}_{t-1},  \widetilde{u}_{t-1} ) - \nabla_u\Hc(\widetilde{w}_{t-1},  \widehat{u}_t )}^2.
\end{array}
\end{equation}
Utilizing \eqref{eq:SGM3_key_properties_of_psi_proof2} and \eqref{eq:SGM3_key_properties_of_psi_proof3}, we can show that
\begin{equation*} 
\arraycolsep=0.2em
\begin{array}{lcl}
\norms{ \nabla_u{\psi}(\widetilde{w}_{t-1}, \widehat{u}_t)}^2 &= & \frac{1}{\hat{\eta}_t^2}\norms{ \widetilde{u}_{t-1} - \widehat{u}_t  }^2 + \frac{2}{\hat{\eta}_t} \iprods{\nabla_u\Hc(\widetilde{w}_{t-1},  \widetilde{u}_{t-1} ) - \nabla_u\Hc(\widetilde{w}_{t-1},  \widehat{u}_t ), \widetilde{u}_{t-1} - \widehat{u}_t } \\
&& + {~} \norms{ \nabla_u\Hc(\widetilde{w}_{t-1},  \widetilde{u}_{t-1} ) - \nabla_u\Hc(\widetilde{w}_{t-1},  \widehat{u}_t )}^2 \\
& \leq &  \frac{1}{\hat{\eta}_t^2}\left(1 - \frac{2L_u\mu_H \hat{\eta}_t }{L_u + \mu_H}\right) \norms{ \widetilde{u}_{t-1} - \widehat{u}_t  }^2 \\
&& - {~} \left( \frac{2}{\hat{\eta}_t(L_u + \mu_H) }  - 1\right) \norms{ \nabla_u\Hc(\widetilde{w}_{t-1},  \widetilde{u}_{t-1} ) - \nabla_u\Hc(\widetilde{w}_{t-1},  \widehat{u}_t )}^2.
\end{array}
\end{equation*}
Substituting this inequality into \eqref{eq:SGM3_key_properties_of_psi_proof1} and noting that $0 < \hat{\eta}_t \leq \frac{2}{L_u + \mu_H}$ and $\hat{\Gc}_{\hat{\eta}_t }( \widetilde{u}_{t-1} ) := \frac{1}{\hat{\eta}_t }\big( \widetilde{u}_{t-1} -  \widehat{u}_t \big)$, we obtain \eqref{eq:SGM3_key_properties_of_psi_p2}.
\end{proof}
\beforesubsubsec
\subsubsection{A key bound for the shuffling gradient descent scheme \eqref{eq:SGM4upper_prob}}\label{apdx:subsubsec:SGM3_GD_scheme}
\aftersubsubsec
The following lemma bounds the difference $\Lc(\widetilde{w}_{t-1}, \widetilde{u}_t ) -  \Lc( \widetilde{w}_t, \widetilde{u}_t)$.

%%% Lemma 4.2.
\begin{lemma}\label{le:SGM3_key_bounds_for_Phi_2b}
Suppose that Assumptions~\ref{as:A3_SGM2} and \ref{as:A5} hold.
Let $\Lc$ be defined by \eqref{eq:minimax_prob} and $g_t$ be defined by \eqref{eq:SGM2_gt_quantity}.
Then, we have 
\begin{equation}\label{eq:SGM3_key_bounds_for_Phi_2b}
\begin{array}{lcl}
\Lc(\widetilde{w}_{t-1}, \widetilde{u}_t ) & \leq &  \Lc( \widetilde{w}_t, \widetilde{u}_t) + \frac{\eta_t}{2}\norms{ g_t - \nabla_w{\Hc}(\widetilde{w}_{t-1}, \widetilde{u}_t ) }^2 +   \frac{ 3 + (L_f+L_w)\eta_t }{2\eta_t}   \norms{  \widetilde{w}_t -  \widetilde{w}_{t-1} }^2.
\end{array}
\end{equation}
\end{lemma}

%%% The proof of Lemma 4.2.
\begin{proof}
From  \eqref{eq:SGM2_gt_quantity} and \eqref{eq:SGM2_w_update_upto_i}, we have
\begin{equation*}
\begin{array}{lcl}
g_t & := & \frac{1}{n} \sum_{j=1}^{n} \nabla_w{\mathcal{H}}_{\hat{\pi}^{(t)}(j)} (w_{j-1}^{(t)}, \widetilde{u}_t ) \overset{\tiny \eqref{eq:SGM2_w_update_upto_i} }{ = }  \frac{1}{\eta_t}(\widetilde{w}_{t-1}- w_n^{(t)}) =  \frac{1}{ \eta_t }( w_0^{(t)} - w_n^{(t)}).
\end{array}
\end{equation*}
Since $\widetilde{w}_t = \prox_{\eta_tf}(w_n^{(t)})$ from the second line of \eqref{eq:SGM4upper_prob}, we have $f'(\widetilde{w}_t) := \eta_t^{-1} \big( w_n^{(t)} - \widetilde{w}_t \big) = -g_t - \eta_t^{-1}(\widetilde{w}_t - \widetilde{w}_{t-1}) \in \partial{f}(\widetilde{w}_t)$.
Hence, by \eqref{eq:SGM2_f_smoothness} from Assumption~\ref{as:A5}, we have
\begin{equation*}
\begin{array}{lcl}
f( \widetilde{w}_{t-1} ) & \leq & f(\widetilde{w}_t)  +  \iprods{ f'(\widetilde{w}_t), \widetilde{w}_{t-1} - \widetilde{w}_t } + \frac{L_f}{2}\norms{ \widetilde{w}_t - \widetilde{w}_{t-1} }^2  \\
& = &  f(\widetilde{w}_t ) +  \iprods{g_t, \widetilde{w}_t - \widetilde{w}_{t-1} }  + \frac{2 + L_f\eta_t}{2\eta_t}\norms{\widetilde{w}_t - \widetilde{w}_{t-1} }^2.
\end{array}
\end{equation*}
Next, by the $L$-smoothness of $\Hc$ from \eqref{eq:Lsmooth2} of Assumption~\ref{as:A3_SGM2}, we also have
\begin{equation*} 
\begin{array}{lcl}
\Hc(\widetilde{w}_{t-1}, \widetilde{u}_t ) & \leq & \Hc( \widetilde{w}_t, \widetilde{u}_t) - \iprods{ \nabla_w{\Hc}(\widetilde{w}_{t-1}, \widetilde{u}_t ),  \widetilde{w}_t  - \widetilde{w}_{t-1} }  + \frac{L_w }{2} \norms{ \widetilde{w}_t  - \widetilde{w}_{t-1} }^2.
\end{array}
\end{equation*}
Summing up the last two inequalities and using $\Lc(w, \widetilde{u}_t) = f(w) + \Hc(w, \widetilde{u}_t) - h(\widetilde{u}_t)$ from \eqref{eq:minimax_prob} and Young's inequality in $\myeqc{1}$, we can show that 
\begin{equation*}
\hspace{-0.0ex}
\begin{array}{lcl}
 \Lc(\widetilde{w}_{t-1}, \widetilde{u}_t) & \leq & \Lc( \widetilde{w}_t, \widetilde{u}_t) +  \iprods{g_t - \nabla_w{\Hc}(\widetilde{w}_{t-1}, \widetilde{u}_t ) , \widetilde{w}_t - \widetilde{w}_{t-1} }   + \frac{ 2 + (L_f+L_w)\eta_t }{2\eta_t}   \norms{  \widetilde{w}_t -  \widetilde{w}_{t-1} }^2 \\
 & \overset{\tiny\myeqc{1}}{\leq } &   \Lc( \widetilde{w}_t, \widetilde{u}_t) + \frac{\eta_t}{2}\norms{ g_t - \nabla_w{\Hc}(\widetilde{w}_{t-1}, \widetilde{u}_t ) }^2 +   \frac{ 3 + (L_f+L_w)\eta_t }{2\eta_t}   \norms{  \widetilde{w}_t -  \widetilde{w}_{t-1} }^2,
\end{array}
\hspace{-2ex}
\end{equation*}
which proves \eqref{eq:SGM3_key_bounds_for_Phi_2b}.
\end{proof}
%%% End of proof.

%%%% ++++++++++++++++++++++++++++++++++++++++++++++++++++++++++++++++++++++++
%%%% C.4. Key bound for the shuffling gradient-ascent scheme.
%%%% ++++++++++++++++++++++++++++++++++++++++++++++++++++++++++++++++++++++++
\beforesubsubsec
\subsubsection{Key bounds for the shuffling gradient ascent scheme \eqref{eq:SGM4lower_prob}}\label{apdx:subsubsec:SGM3_SGA_scheme}
\aftersubsubsec
We also derive necessary bounds to analyze Algorithm~\ref{alg:SGM2} using the simplified version \eqref{eq:SGM3_SGM4lower_prob} of the \textbf{\textit{shuffling gradient ascent}} scheme \eqref{eq:SGM4lower_prob}.
Let us define
\begin{equation}\label{eq:SGM3_vt_quantity}
\arraycolsep=0.2em
\begin{array}{lcl}
v_t := \frac{1}{n}\sum_{ j = 1}^n\nabla_u{\Hc}_{\pi^{(t)}(j) }( \widetilde{w}_{t-1}, u_{j-1}^{(t)}).
\end{array}
\end{equation}
We establish the following two lemmas.

%%% Lemma 3.2.
\begin{lemma}\label{le:SGM3_key_bounds_for_ut_2}
Suppose that Assumption~\ref{as:A3_SGM2} holds for \eqref{eq:minimax_prob}.
Let $\sets{u_i^{(t)}}_{i=1}^n$ be generated by the simplified version \eqref{eq:SGM3_SGM4lower_prob} of \eqref{eq:SGM4lower_prob}.
Then, if we choose $\hat{\eta}_t > 0$ such that $1 - 3L_u^2 \hat{\eta}_t^2  \geq  0$, then 
\begin{equation}\label{eq:SGM3_key_bounds_for_ut_2}
\begin{array}{lcl}
\hat{\Delta}_t := \frac{1}{n} \sum_{i=0}^{n-1}\norms{u_i^{(t)} - u_0^{(t)}}^2 &\leq &  (3 \Theta_u + 2) \hat{\eta}_t^2 \cdot \norms{  \nabla_u{\Hc}(\widetilde{w}_{t-1}, \widetilde{u}_{t-1})  }^2 + 3 \hat{\eta}_t^2  \sigma_u^2.
\end{array}
\end{equation}
Let $v_t$ be defined by \eqref{eq:SGM3_vt_quantity}.
Then, we also have
\begin{equation}\label{eq:SGM3_key_bounds_for_ut_2b}
\begin{array}{lcl}
\norms{ v_t -   \nabla_u{\Hc}(\widetilde{w}_{t-1}, \widetilde{u}_{t-1}) }^2 &\leq & \frac{L_u^2}{n}\sum_{i=1}^n\norms{u_{i-1}^{(t)} - \widetilde{u}_{t-1}}^2 = L_u^2 \hat{\Delta}_t.
\end{array}
\end{equation}
\end{lemma}

%%% Proof of Lemma 3.1.
\begin{proof}
First, for simplicity of notation, we denote $\nabla_u{\Hc}_{t-1} := \nabla_u{\Hc}(\widetilde{w}_{t-1}, \widetilde{u}_{t-1})$.
Then, from the update of $u_i^{(t)}$ in \eqref{eq:SGM3_SGM4lower_prob}, for any $i \in [n]$, we have  
\begin{equation}\label{eq:ut_update_upto_i}
u_{i}^{(t)} = u_{0}^{(t)} + \frac{\hat{\eta}_t}{n} \sum_{j=1}^{i} \nabla_u{\mathcal{H}}_{ \pi^{(t)}(j)} (\widetilde{w}_{t-1}, u_{j-1}^{(t)} ). 
\end{equation}
Using this expression and  Young's inequality in $\myeqc{1}$ and $\myeqc{2}$ below, we can show that
\begin{equation*}
\begin{array}{lcl}
\norms{u_{i}^{(t)} - u_{0}^{(t)} }^2 & \overset{\tiny\eqref{eq:ut_update_upto_i}}{=} & \frac{i^2 \cdot \hat{\eta}_t^2}{n^2} \Vert \frac{1}{ i  } \sum_{j=1}^{i} \nabla_u{\mathcal{H}}_{ \pi^{(t)}(j)} (\widetilde{w}_{t-1}, u_{j-1}^{(t)} ) \Vert^2 \\
& \overset{\myeqc{1}}{\leq} & \frac{3 i^2 \cdot \hat{\eta}_t^2}{n^2} \big\Vert \frac{1}{i} \sum_{j=1}^{i} \big[ \nabla_u{\mathcal{H}}_{\pi^{(t)}(j)} (\widetilde{w}_{t-1}, u_{0}^{(t)}  )  - \nabla_u{\Hc}_{t-1} \big] \big\Vert^2 
+  \frac{3 i^2 \cdot \hat{\eta}_t^2}{n^2} \norms{ \nabla_u{\Hc}_{t-1} }^2 \\
&& + {~}  \frac{3 i^2 \cdot \hat{\eta}_t^2}{n^2} \big\Vert \frac{1}{i} \sum_{j=1}^{i} \big[ \nabla_u{\mathcal{H}}_{ \pi^{(t)}(j)} (\widetilde{w}_{t-1}, u_{j-1}^{(t)})  - \nabla_u{\mathcal{H}}_{ \pi^{(t)}(j)} (\widetilde{w}_{t-1}, u_{0}^{(t)} ) \big] \big\Vert^2 \\
& \overset{\myeqc{2}}{\leq} & \frac{3 i^2 \cdot \hat{\eta}_t^2}{n^2} \big\Vert \frac{1}{i} \sum_{j=1}^{i} \big[ \nabla_u{\mathcal{H}}_{ \pi^{(t)}(j)} (\widetilde{w}_{t-1}, u_{0}^{(t)} )  - \nabla_u{\Hc}_{t-1}  \big] \big\Vert^2 
+  \frac{3 i^2 \cdot \hat{\eta}_t^2}{n^2} \norms{ \nabla_u{\Hc}_{t-1}   }^2 \\
&& + {~}  \frac{3 i \cdot \hat{\eta}_t^2}{n^2}  \sum_{j=1}^{i} \big\Vert \nabla_u{\mathcal{H}}_{ \pi^{(t)}(j)} (\widetilde{w}_{t-1}, u_{j-1}^{(t)} )  - \nabla_u{\mathcal{H}}_{ \pi^{(t)}(j)} (\widetilde{w}_{t-1}, u_{0}^{(t)} )  \big\Vert^2.
\end{array}
\end{equation*}
Now, we denote $\hat{\Delta}_t := \frac{1}{n} \sum_{j=0}^{n-1}\norms{u_j^{(t)} - u_0^{(t)}}^2 = \frac{1}{n} \sum_{j=0}^{n-1}\norms{u_j^{(t)} - \widetilde{u}_{t-1} }^2$.
Then, by \eqref{eq:Lsmooth2} from Assumption~\ref{as:A3_SGM2}, we have
\begin{equation*}
\begin{array}{lcl}
\frac{1}{n}\sum_{j=1}^{i} \big\Vert \nabla_u{\mathcal{H}}_{ \pi^{(t)}(j)} (\widetilde{w}_{t-1}, u_{j-1}^{(t)} )  - \nabla_u{\mathcal{H}}_{ \pi^{(t)}(j)} (\widetilde{w}_{t-1}, u_{0}^{(t)} )  \big\Vert^2 & \overset{\tiny\eqref{eq:Lsmooth2}}{ \leq } & \frac{L_u^2}{n} \sum_{j=1}^i\norms{ u^{(t)}_{j-1} - u_0^{(t)}}^2 \\
& \leq & L_u^2 \hat{\Delta}_t.
\end{array}
\end{equation*}
Next, by Young's inequality again in $\myeqc{1}$, $u_0^{(t)} = \widetilde{u}_{t-1}$, and \eqref{eq:bounded_variance2b} from Assumption~\ref{as:A3_SGM2}, and the fact that $\nabla_u{\Hc}_{t-1} := \nabla_u{\Hc}(\widetilde{w}_{t-1}, \widetilde{u}_{t-1})$, we can show that
\begin{equation*}
\arraycolsep=0.2em
\begin{array}{lcl}
\mathcal{T}_{[2]} &:= & \Vert \frac{1}{i} \sum_{j=1}^i \big[ \nabla_u{\mathcal{H}}_{  \pi^{(t)}(j)} (\widetilde{w}_{t-1}, u_{0}^{(t)}  )  -  \nabla_u{\Hc}_{t-1}  \big] \big\Vert^2 \\
& \overset{\tiny\textcircled{1}}{\leq} & \frac{1}{i} \sum_{j=1}^i \big\Vert \nabla_u{\mathcal{H}}_{ \pi^{(t)}(j)} ( \widetilde{w}_{t-1}, \widetilde{u}_{t-1}  )  -   \nabla_u{\Hc}(\widetilde{w}_{t-1}, \widetilde{u}_{t-1}) )  \big\Vert^2 \\
& \overset{\tiny\eqref{eq:bounded_variance2b} }{ \leq } &    \frac{n }{i} \big[ \Theta_u \Vert  \nabla_u{\Hc}(\widetilde{w}_{t-1}, \widetilde{u}_{t-1})   \Vert^2 + \sigma_u^2 \big].
\end{array}
\end{equation*}
Combining the last three inequalities above, we can show that
\begin{equation*}
\begin{array}{lcl}
\norms{u_{i}^{(t)} - u_{0}^{(t)} }^2 & \leq & \frac{ 3 i \cdot \hat{\eta}_t^2}{n}  \big[ \Theta_u \Vert   \nabla_u{\Hc}(\widetilde{w}_{t-1}, \widetilde{u}_{t-1})  \Vert^2 + \sigma_u^2 \big] \\
&& +  {~} \frac{3 i^2 \cdot \hat{\eta}_t^2}{n^2} \norms{  \nabla_u{\Hc}(\widetilde{w}_{t-1}, \widetilde{u}_{t-1})   }^2 + \frac{ 3 i \cdot L_u^2 \hat{\eta}_t^2}{n } \hat{\Delta}_t \\
& = & \frac{3i \hat{\eta}_t^2}{n^2}( n \Theta_u + i) \norms{   \nabla_u{\Hc}(\widetilde{w}_{t-1}, \widetilde{u}_{t-1})  }^2 +  \frac{3 i \cdot \hat{\eta}_t^2}{n} \sigma_u^2  + \frac{ 3 i \cdot L_u^2 \hat{\eta}_t^2}{n } \hat{\Delta}_t.
\end{array}
\end{equation*}
Averaging this inequality from $i=0$ to $ i = n-1$, we get
\begin{equation*}
\begin{array}{lcl}
\hat{\Delta}_t &:= & \frac{1}{n} \sum_{i=0}^{n-1} \norms{ u_{i}^{(t)} - u_{0}^{(t)} }^2 \\
& \leq & \frac{1}{n} \sum_{i=0}^{n-1} \big[  \frac{3i \hat{\eta}_t^2}{n^2}( n \Theta_u + i) \norms{   \nabla_u{\Hc}(\widetilde{w}_{t-1}, \widetilde{u}_{t-1})  }^2 + \frac{ 3 i \cdot \hat{\eta}_t^2}{ n } \sigma_u^2  + \frac{3 i \cdot L_u^2 \hat{\eta}_t^2}{n} \hat{\Delta}_t \big] \\
& \leq &  \frac{ (3 \Theta_u + 2) \hat{\eta}_t^2 }{2} \cdot \norms{  \nabla_u{\Hc}(\widetilde{w}_{t-1}, \widetilde{u}_{t-1})  }^2 + \frac{ 3 \hat{\eta}_t^2}{2} \sigma_u^2 + \frac{3 L_u^2 \hat{\eta}_t^2}{2} \hat{\Delta}_t.
\end{array}
\end{equation*}
Here, we have used the facts that $\sum_{i=0}^{n-1}i = \frac{n(n-1)}{2} \leq \frac{n^2}{2}$ and $\sum_{i=0}^{n-1}i^2 = \frac{n(n-1)(2n-1)}{6} \leq \frac{n^3}{3}$.
Rearranging the last inequality and noting that $1 - \frac{3 L_u^2 \hat{\eta}_t^2}{2} \geq \frac{1}{2}$, we obtain \eqref{eq:SGM3_key_bounds_for_ut_2}.

Finally, using \eqref{eq:SGM3_vt_quantity} and \eqref{eq:Lsmooth2} from Assumption~\ref{as:A3_SGM2}, we have 
\begin{equation*} 
\arraycolsep=0.05em
\begin{array}{lcl}
\norms{ v_t -   \nabla_u{\Hc}(\widetilde{w}_{t-1}, \widetilde{u}_{t-1})  }^2 & \overset{\tiny\eqref{eq:SGM3_vt_quantity} }{ = } & \big\Vert \frac{1}{n} \sum_{i=1}^{n} \big[ \nabla_u {\mathcal{H}}_{ \pi^{(t)}(i)} (\widetilde{w}_{t-1}, u_{i-1}^{(t)} ) -  \nabla_u{\Hc}_{\pi^{(t)}(i) }(\widetilde{w}_{t-1}, \widetilde{u}_{t-1}) \big] \big\Vert^2 \\
&\leq & \frac{1}{n} \sum_{i=1}^{n} \big\Vert \nabla_u{\mathcal{H}}_{ \pi^{(t)}(i)} ( \widetilde{w}_{t-1}, u_{i-1}^{(t)}  ) - \nabla_u {\mathcal{H}}_{ \pi^{(t)}(i)}(\widetilde{w}_{t-1},  \widetilde{u}_{t-1} ) \big\Vert^2 \\
& \overset{\tiny\eqref{eq:Lsmooth2}}{\leq} & \frac{L_u^2}{n}\sum_{i=1}^n\norms{ u_{i-1}^{(t)} - \widetilde{u}_{t-1}}^2,
\end{array}
\end{equation*}
which proves \eqref{eq:SGM3_key_bounds_for_ut_2b}.
\end{proof}
%%% End of Lemma 3.1.

%We also need the following result.

%%% Lemma 3.
\begin{lemma}\label{le:GA_ut_shuffling}
Suppose that Assumption~\ref{as:A3_SGM2} holds for \eqref{eq:minimax_prob}.
Let $\sets{u_i^{(t)}}_{i=1}^n$ be generated by the simplified version \eqref{eq:SGM3_SGM4lower_prob} of \eqref{eq:SGM4lower_prob}, $\psi$ be defined by \eqref{eq:SGM3_psi_func}, $v_t$ be defined by \eqref{eq:SGM3_vt_quantity}, and $\hat{\Gc}_{\hat{\eta}}$ be defined as in \eqref{eq:SGM3_gradient_mappings}.
Then 
\begin{equation}\label{eq:GA_ut_shuffling_bound1}
\begin{array}{lcl}
\psi(\widetilde{w}_{t-1},  \widetilde{u}_t ) & \leq & \psi(\widetilde{w}_{t-1}, \widetilde{u}_{t-1})  -  \frac{ \hat{\eta}_t [ 1 + (\mu_h + \mu_H) \hat{\eta}_t - L_u\hat{\eta}_t ] }{2}  \Vert \hat{\Gc}_{ \hat{\eta}_t}(\widetilde{u}_{t-1})  \Vert^2  \\
&& - {~}  \frac{(1-L_u \hat{\eta}_t)}{2\hat{\eta}_t} \norms{ \widetilde{u}_t  - \widetilde{u}_{t-1} }^2  +  \frac{\hat{\eta}_t}{2(1 + \mu_h\hat{\eta}_t)}\norms{ v_t  - \nabla_u{\Hc}(\widetilde{w}_{t-1}, \widetilde{u}_{t-1}) }^2.
\end{array}
\end{equation}
\end{lemma}

%%% Proof of Lemma 3.
\begin{proof}
Let us denote  $\widehat{u}_t :=  \prox_{\hat{\eta}_th}\big(\widetilde{u}_{t-1} + \hat{\eta}_t \nabla_u{\Hc}(\widetilde{w}_{t-1}, \widetilde{u}_{t-1})\big)$ as in \eqref{eq:SGM3_gradient_mappings}.
Then, from \eqref{eq:SGM3_gradient_mappings}, we have $\hat{\eta}_t {\hat{\Gc}}_{\hat{\eta}_t}(\widetilde{u}_{t-1}) = \widetilde{u}_{t-1} - \widehat{u}_t$.
Moreover, we can show that $\nabla{h}(\widehat{u}_t) := \frac{1}{\hat{\eta}_t }\big( \widetilde{u}_{t-1} - \widehat{u}_t \big)  +  \nabla{\Hc}(\widetilde{w}_{t-1}, \widetilde{u}_{t-1}) \in \partial{h}(\widehat{u}_t)$.

By the $\mu_h$-strong convexity of $h$, we have
\begin{equation*}
\begin{array}{lcl}
h(\widehat{u}_t) & \leq &  h(\widetilde{u}_{t-1}) +  \iprods{\nabla{h}(\widehat{u}_t), \widehat{u}_t - \widetilde{u}_{t-1}} - \frac{\mu_h}{2}\norms{ \widehat{u}_t - \widetilde{u}_{t-1} }^2 \\
& = &  h(\widetilde{w}_{t-1}) +  \iprods{\nabla{\Hc}(\widetilde{w}_{t-1}, \widetilde{u}_{t-1}), \widehat{u}_t - \widetilde{u}_{t-1}} -  \frac{2 + \mu_h\hat{\eta}_t }{ 2\hat{\eta}_t}   \norms{  \widehat{u}_t -  \widetilde{u}_{t-1} }^2.
\end{array}
\end{equation*}
Next, by the $L_u$-smoothness of $\Hc(\widetilde{w}_{t-1}, \cdot)$ from Assumption~\ref{as:A3_SGM2}, we have
\begin{equation*} 
\begin{array}{lcl}
-\Hc(\widetilde{w}_{t-1}, \widehat{u}_t ) & \leq &-\Hc(\widetilde{w}_{t-1}, \widetilde{u}_{t-1}) - \iprods{ \nabla_u{\Hc}(\widetilde{w}_{t-1}, \widetilde{u}_{t-1}),  \widehat{u}_t  - \widetilde{u}_{t-1} } + \frac{L_u }{2}\norms{ \widehat{u}_t  - \widetilde{u}_{t-1} }^2.
\end{array}
\end{equation*}
Summing up the last two inequalities and using both $\psi(\widetilde{w}_{t-1}, u) := - \Hc(\widetilde{w}_{t-1}, u) + h(u)$ from \eqref{eq:SGM3_psi_func}  and $\widehat{u}_t -  \widetilde{u}_{t-1} = - \hat{\eta}_t {\hat{\Gc}}_{ \hat{\eta}_t}(\widetilde{u}_{t-1})$, we can show that
\begin{equation}\label{eq:GA_lm2_proof1} 
\hspace{-0.0ex}
\begin{array}{lcl}
\psi(\widetilde{w}_{t-1}, \widehat{u}_t ) & \leq & \psi(\widetilde{w}_{t-1}, \widetilde{u}_{t-1})  - \frac{(2 + \mu_h\hat{\eta}_t - L_u\hat{\eta}_t)}{2 \hat{\eta}_t}   \norms{  \widehat{u}_t -  \widetilde{u}_{t-1} }^2 \\
& = &  \psi(\widetilde{w}_{t-1}, \widetilde{u}_{t-1})  - \frac{\hat{\eta}_t(2 + \mu_h\hat{\eta}_t - L_u \hat{\eta}_t)}{2}   \norms{  \hat{\Gc}_{\hat{\eta}_t}(\widetilde{u}_{t-1}) }^2.
\end{array}
\hspace{-2ex}
\end{equation}
Next, from \eqref{eq:SGM3_vt_quantity} and  \eqref{eq:SGM3_SGM4lower_prob}, one can derive that
\begin{equation}\label{eq:GA_lm2_ht}
\begin{array}{lcl}
v_t & \overset{\tiny \eqref{eq:SGM3_vt_quantity} }{ := } & \frac{1}{n} \sum_{j=1}^{n} \nabla_u{\mathcal{H}}_{\pi^{(t)}(j)} (\widetilde{w}_{t-1}, u_{j-1}^{(t)}) \overset{\tiny \eqref{eq:SGM3_SGM4lower_prob} }{=}  \frac{1}{\hat{\eta}_t}(u_n^{(t)} - \widetilde{u}_{t-1}) =  \frac{1}{ \hat{\eta}_t }( u_n^{(t)} - u_0^{(t)}).
\end{array}
\end{equation}
Since $\widetilde{u}_t = \prox_{\hat{\eta}_th}(u_n^{(t)})$ from \eqref{eq:SGM3_SGM4lower_prob}, we have $\nabla{h}(\widetilde{u}_t) := \frac{1}{\hat{\eta}_t } \big( u_n^{(t)} - \widetilde{u}_t \big) = v_t - \frac{1}{\hat{\eta}_t }(\widetilde{u}_t - \widetilde{u}_{t-1}) \in \partial{h}(\widetilde{u}_t)$.
Hence, again by the $\mu_h$-strong convexity of $h$, we have
\begin{equation*}
\begin{array}{lcl}
h(\widetilde{u}_t) & \leq &  h(\widehat{u}_t) +  \iprods{\nabla{h}(\widetilde{u}_t), \widetilde{u}_t - \widehat{u}_t }  - \frac{\mu_h}{2}\norms{ \widetilde{u}_t - \widehat{u}_t }^2 \\
& = &  h(\widehat{u}_t ) + \iprods{v_t, \widetilde{u}_t - \widehat{u}_t }  - \frac{1}{ \hat{\eta}_t}\iprods{\widetilde{u}_t - \widetilde{u}_{t-1}, \widetilde{u}_t - \widehat{u}_t} -  \frac{\mu_h}{2}\norms{ \widetilde{u}_t - \widehat{u}_t }^2  \\
& = & h(\widehat{u}_t)  + \iprods{v_t, \widetilde{u}_t - \widehat{u}_t } - \frac{1}{2 \hat{\eta}_t} \norms{\widetilde{u}_t - \widetilde{u}_{t-1}}^2 - \frac{(1 + \mu_h\hat{\eta}_t)}{2 \hat{\eta}_t} \norms{ \widetilde{u}_t - \widehat{u}_t }^2 +  \frac{1}{2 \hat{\eta}_t}\norms{\widehat{u}_t  - \widetilde{u}_{t-1}}^2.
\end{array}
\end{equation*}
Again, by the $L_u$-smoothness and $\mu_H$-strong concavity of $\Hc(\widetilde{w}_{t-1}, \cdot)$ from Assumption~\ref{as:A3_SGM2}, we have
\begin{equation*} 
\begin{array}{lcl}
-\Hc(\widetilde{w}_{t-1}, \widetilde{u}_t ) & \leq & -\Hc(\widetilde{w}_{t-1}, \widetilde{u}_{t-1}) - \iprods{ \nabla_u{\Hc}(\widetilde{w}_{t-1}, \widetilde{u}_{t-1}),  \widetilde{u}_t  - \widetilde{u}_{t-1} } + \frac{L_u}{2}\Vert \widetilde{u}_t  - \widetilde{u}_{t-1} \Vert^2, \\
-\Hc(\widetilde{w}_{t-1}, \widetilde{u}_{t-1} ) & \leq & -\Hc(\widetilde{w}_{t-1}, \widehat{u}_t) - \iprods{ \nabla_u{\Hc}(\widetilde{w}_{t-1}, \widetilde{u}_{t-1}),  \widetilde{u}_{t-1}  - \widehat{u}_t } - \frac{\mu_H}{2}\norms{\widehat{u}_t  - \widetilde{u}_{t-1} }^2.
\end{array}
\end{equation*}
Adding the last three inequalities together, and using $\psi(\widetilde{w}_{t-1}, u) = h(u) - \Hc(\widetilde{w}_{t-1}, u)$ from \eqref{eq:SGM3_psi_func} and $\widehat{u}_t -  \widetilde{u}_{t-1} = - \hat{\eta}_t \hat{\Gc}_{\hat{\eta}_t}(\widetilde{u}_{t-1})$, we have 
\begin{equation}\label{eq:GA_lm2_proof2} 
\hspace{-0ex}
\begin{array}{lcl}
\psi(\widetilde{w}_{t-1}, \widetilde{u}_t ) & \leq & \psi(\widetilde{w}_{t-1}, \widehat{u}_t) -  \iprods{ \nabla_u{\Hc}(\widetilde{w}_{t-1}, \widetilde{u}_{t-1}) - v_t,  \widetilde{u}_t  - \widehat{u}_t } - \frac{(1-L_u\hat{\eta}_t)}{2 \hat{\eta}_t} \norms{ \widetilde{u}_t  - \widetilde{u}_{t-1} }^2 \\
&& + {~}  \frac{1 - \mu_H\hat{\eta}_t}{2\hat{\eta}_t}  \Vert \widehat{u}_t  - \widetilde{u}_{t-1} \Vert^2  - \frac{(1 + \mu_h\hat{\eta}_t)}{2\eta_t} \norms{ \widetilde{u}_t - \widehat{u}_t }^2 \\
& \overset{\myeqc{1}}{\leq} & \varphi_t(\widehat{u}_t) +  \frac{\hat{\eta}_t}{2(1 + \mu_h\hat{\eta}_t)}\norms{ \nabla_u{\Hc}(\widetilde{w}_{t-1}, \widetilde{u}_{t-1}) - v_t }^2 - \frac{(1-L_u\hat{\eta}_t)}{2 \hat{\eta}_t} \norms{ \widetilde{u}_t  - \widetilde{u}_{t-1} }^2  \\
&& + {~}  \frac{ \hat{\eta}_t(1 - \mu_H\hat{\eta}_t) }{2}  \Vert \hat{\Gc}_{\hat{\eta}_t}(\widetilde{u}_{t-1})  \Vert^2,
\end{array}
\hspace{-2ex}
\end{equation}
where we have used Young's inequality in the last line $\myeqc{1}$ as $ \iprods{ \nabla_u{\Hc}(\widetilde{w}_{t-1}, \widetilde{u}_{t-1}) - v_t,  \widetilde{u}_t  - \widehat{u}_t } \leq \frac{\hat{\eta}_t}{2(1+\mu_h\hat{\eta}_t)}\norms{ \nabla_u{\Hc}(\widetilde{w}_{t-1}, \widetilde{u}_{t-1}) - v_t  }^2 + \frac{1 + \mu_h\hat{\eta}_t}{2\hat{\eta}_t}\norms{\widetilde{u}_t  - \widehat{u}_t }^2$.

Finally, summing up \eqref{eq:GA_lm2_proof1} and \eqref{eq:GA_lm2_proof2}, we get
\begin{equation*} 
\begin{array}{lcl}
\psi(\widetilde{w}_{t-1},  \widetilde{u}_t ) & \leq & \psi(\widetilde{w}_{t-1},  \widetilde{u}_{t-1})   - \frac{(1-L_u\hat{\eta}_t)}{2 \hat{\eta}_t} \norms{ \widetilde{u}_t  - \widetilde{u}_{t-1} }^2  
-  \frac{\hat{\eta}_t [ 1 + (\mu_h + \mu_H)\hat{\eta}_t - L_u \hat{\eta}_t ]}{2}  \Vert \hat{\Gc}_{\hat{\eta}_t}(\widetilde{u}_{t-1})  \Vert^2  \\
&& + {~}  \frac{\hat{\eta}_t}{2(1 + \mu_h\hat{\eta}_t)}\norms{ v_t  - \nabla_u{\Hc}(\widetilde{w}_{t-1}, \widetilde{u}_{t-1})  }^2,
\end{array}
\end{equation*}
which proves \eqref{eq:GA_ut_shuffling_bound1}.
\end{proof}
%%% End of Proof.

%%%% ++++++++++++++++++++++++++++++++++++++++++++++++++++++++++++++++++++++++
%%%% C.5. Convergence Analysis of The Full-Shuffling Variant of Algorithm 2.
%%%% ++++++++++++++++++++++++++++++++++++++++++++++++++++++++++++++++++++++++
\beforesubsubsec
\subsubsection{Convergence analysis of the full-shuffling variant of Algorithm~\ref{alg:SGM2} -- The case $S=1$}\label{apdx:subsubsec:SGM3_full_shuffling_variant}
\aftersubsubsec
To analyze the convergence of the full-shuffling variant of Algorithm~\ref{alg:SGM2}, we need the following lemma.

%%% Lemma 5.
\begin{lemma}\label{le:SGM3_bounding_V}
Let $\Vc_{\lambda}$ be defined by \eqref{eq:SGM3_potential_func}, $\Vc_t := \Vc_{\lambda}( \widetilde{w}_t, \widetilde{u}_t )$, and $g_t$ and $v_t$ be defined by \eqref{eq:SGM2_gt_quantity} and \eqref{eq:SGM3_vt_quantity}, respectively.
Suppose that $f$ satisfies \eqref{eq:SGM2_f_smoothness} of Assumption~\ref{as:A5}.
Then, the following bound holds: 
\begin{equation}\label{eq:SGM3_bounding_V} 
\hspace{-0.25ex}
\begin{array}{lcl}
\Vc_t - \Vc_{t-1} & \leq & - \frac{\lambda - 2 - [(1+\lambda)L_{\Phi_0}+ L_w + L_f ] \eta_t }{2\eta_t} \norms{ \widetilde{w}_t  - \widetilde{w}_{t-1} }^2 -  \frac{(1-L_u \hat{\eta}_t)}{2\hat{\eta}_t} \norms{ \widetilde{u}_t  - \widetilde{u}_{t-1} }^2 \\
&& - {~} \frac{ (\lambda + 1) \eta_t(1 - 2L_{\Phi_0}\eta_t)}{2}  \Vert {\Gc}_{\eta_t}(\widetilde{w}_{t-1})  \Vert^2   -  \frac{ \hat{\eta}_t(1 + (\mu_h + \mu_H)\hat{\eta}_t - L_u\hat{\eta}_t)}{2}  \Vert \hat{\Gc}_{ \hat{\eta}_t}(\widetilde{u}_{t-1})  \Vert^2 \\
&& + {~}  \frac{(\lambda + 1) \eta_t}{2}\norms{ g_t - \nabla{\Phi_0}(\widetilde{w}_{t-1}) }^2 +  \frac{\eta_t}{2}\norms{ g_t - \nabla_w{\Hc}(\widetilde{w}_{t-1}, \widetilde{u}_t ) }^2  \\
&& + {~}  \frac{\hat{\eta}_t}{2(1 + \mu_h\hat{\eta}_t)}\norms{ v_t  - \nabla_u{\Hc}(\widetilde{w}_{t-1}, \widetilde{u}_{t-1}) }^2.
\end{array}
\hspace{-4ex}
\end{equation}
\end{lemma}

%%% The proof of Lemma 5.
\begin{proof}
From \eqref{eq:SGM3_potential_func}, if we denote $\Vc_t := \Vc_{\lambda}(\widetilde{w}_t, \widetilde{u}_t)$, then we have
\begin{equation}\label{eq:SGM3_bound_V_proof1} 
\begin{array}{lcl}
\Vc_t - \Vc_{t-1} &= & (\lambda + 1) \big[ \Psi_0(\widetilde{w}_t) - \Psi_0(\widetilde{w}_{t-1}) \big] + \Lc(\widetilde{w}_{t-1}, \widetilde{u}_{t-1}) - \Lc(\widetilde{w}_t, \widetilde{u}_t) \\
& = & (\lambda + 1) \big[ \Psi_0(\widetilde{w}_t) - \Psi_0(\widetilde{w}_{t-1}) \big] + \Lc(\widetilde{w}_{t-1}, \widetilde{u}_{t-1}) - \Lc(\widetilde{w}_{t-1}, \widetilde{u}_t ) \\
&& + {~} \Lc(\widetilde{w}_{t-1}, \widetilde{u}_t ) - \Lc(\widetilde{w}_t, \widetilde{u}_t) \\
& = & (\lambda + 1)  \big[ \Psi_0(\widetilde{w}_t) - \Psi_0(\widetilde{w}_{t-1}) \big]   +  \Lc(\widetilde{w}_{t-1}, \widetilde{u}_t ) - \Lc(\widetilde{w}_t, \widetilde{u}_t) \\
&& + {~} \psi(\widetilde{w}_{t-1}, \widetilde{u}_t ) - \psi(\widetilde{w}_{t-1}, \widetilde{u}_{t-1}).
\end{array}
\end{equation}
Next, from \eqref{eq:SGM2_key_bounds_for_Phi_2}, we have
\begin{equation}\label{eq:SGM3_bound_V_proof2}
\begin{array}{lcl}
\Psi_0(\widetilde{w}_t ) - \Psi_0(\widetilde{w}_{t-1})  &\leq &  -  \frac{\eta_t(1 - 2L_{\Phi_0}\eta_t)}{2}  \Vert {\Gc}_{\eta_t}(\widetilde{w}_{t-1})  \Vert^2  - \frac{(1-L_{\Phi_0}\eta_t)}{2\eta_t} \norms{ \widetilde{w}_t  - \widetilde{w}_{t-1} }^2  \\
&& + {~}  \frac{\eta_t}{2}\norms{ g_t - \nabla{\Phi_0}(\widetilde{w}_{t-1}) }^2.
\end{array}
\end{equation}
From \eqref{eq:SGM3_key_bounds_for_Phi_2b}, we also have
\begin{equation}\label{eq:SGM3_bound_V_proof3} 
\begin{array}{lcl}
\Lc(\widetilde{w}_{t-1}, \widetilde{u}_t ) -  \Lc( \widetilde{w}_t, \widetilde{u}_t) & \leq &  \frac{\eta_t}{2}\norms{ g_t - \nabla_w{\Hc}(\widetilde{w}_{t-1}, \widetilde{u}_t ) }^2 +   \frac{ 3 + (L_f+L_w)\eta_t }{2\eta_t}   \norms{  \widetilde{w}_t -  \widetilde{w}_{t-1} }^2.
\end{array}
\end{equation}
From  \eqref{eq:GA_ut_shuffling_bound1}, we can rewrite it as
\begin{equation}\label{eq:SGM3_bound_V_proof4}
\begin{array}{lcl}
\psi(\widetilde{w}_{t-1},  \widetilde{u}_t ) - \psi(\widetilde{w}_{t-1}, \widetilde{u}_{t-1}) & \leq &  -  \frac{ \hat{\eta}_t [ 1 + (\mu_h + \mu_H) \hat{\eta}_t - L_u\hat{\eta}_t ] }{2}  \Vert \hat{\Gc}_{ \hat{\eta}_t}(\widetilde{u}_{t-1})  \Vert^2  \\
&& + {~}   \frac{\hat{\eta}_t}{2(1 + \mu_h\hat{\eta}_t)}\norms{ v_t  - \nabla_u{\Hc}(\widetilde{w}_{t-1}, \widetilde{u}_{t-1}) }^2 \vspace{0.5ex} \\
&& - {~} \frac{(1-L_u \hat{\eta}_t)}{2\hat{\eta}_t} \norms{ \widetilde{u}_t  - \widetilde{u}_{t-1} }^2.
\end{array}
\end{equation}
Substituting \eqref{eq:SGM3_bound_V_proof2}, \eqref{eq:SGM3_bound_V_proof3}, and \eqref{eq:SGM3_bound_V_proof4} into \eqref{eq:SGM3_bound_V_proof1}, we can derive that
%\begin{equation*} 
%\begin{array}{lcl}
%\Vc_t - \Vc_{t-1} & \leq &
%-  \frac{(\lambda + 1) \eta_t(1 - 2L_{\Phi_0}\eta_t)}{2}  \Vert {\Gc}_{\eta_t}(\widetilde{w}_{t-1})  \Vert^2  - \frac{(\lambda + 1) (1-L_{\Phi_0}\eta_t)}{2\eta_t} \norms{ \widetilde{w}_t  - \widetilde{w}_{t-1} }^2  \\
%&& + {~}  \frac{(\lambda + 1) \eta_t}{2}\norms{ g_t - \nabla{\Phi_0}(\widetilde{w}_{t-1}) }^2 +  \frac{\eta_t}{2}\norms{ g_t - \nabla_w{\Hc}(\widetilde{w}_{t-1}, \widetilde{u}_t ) }^2  \\
%&& + {~}  \frac{ 3 + (L_f+L_w)\eta_t }{2\eta_t}   \norms{  \widetilde{w}_t -  \widetilde{w}_{t-1} }^2 +  \frac{\hat{\eta}_t}{2(1 + \mu_h\hat{\eta}_t)}\norms{ v_t  - \nabla_u{\Hc}(\widetilde{w}_{t-1}, \widetilde{u}_{t-1}) }^2 \\
%&& - {~}  \frac{ \hat{\eta}_t [ 1 + (\mu_h + \mu_H) \hat{\eta}_t - L_u\hat{\eta}_t ] }{2}  \Vert \hat{\Gc}_{ \hat{\eta}_t}(\widetilde{u}_{t-1})  \Vert^2  - \frac{(1-L_u \hat{\eta}_t)}{2\hat{\eta}_t} \norms{ \widetilde{u}_t  - \widetilde{u}_{t-1} }^2.
%\end{array}
%\end{equation*}
%Simplifying this expression, we get 
\begin{equation*} 
\begin{array}{lcl}
\Vc_t - \Vc_{t-1} & \leq & - \frac{\lambda - 2 - [(1+\lambda)L_{\Phi_0}+ L_w + L_f ] \eta_t }{2\eta_t} \norms{ \widetilde{w}_t  - \widetilde{w}_{t-1} }^2 -  \frac{(1-L_u \hat{\eta}_t)}{2\hat{\eta}_t} \norms{ \widetilde{u}_t  - \widetilde{u}_{t-1} }^2 \\
&& - {~} \frac{ (\lambda + 1) \eta_t(1 - 2L_{\Phi_0}\eta_t)}{2}  \Vert {\Gc}_{\eta_t}(\widetilde{w}_{t-1})  \Vert^2   -  \frac{ \hat{\eta}_t(1 + (\mu_h + \mu_H)\hat{\eta}_t - L_u\hat{\eta}_t)}{2}  \Vert \hat{\Gc}_{ \hat{\eta}_t}(\widetilde{u}_{t-1})  \Vert^2 \\
&& + {~}  \frac{(\lambda + 1) \eta_t}{2}\norms{ g_t - \nabla{\Phi_0}(\widetilde{w}_{t-1}) }^2 +  \frac{\eta_t}{2}\norms{ g_t - \nabla_w{\Hc}(\widetilde{w}_{t-1}, \widetilde{u}_t ) }^2  \\
&& + {~}  \frac{\hat{\eta}_t}{2(1 + \mu_h\hat{\eta}_t)}\norms{ v_t  - \nabla_u{\Hc}(\widetilde{w}_{t-1}, \widetilde{u}_{t-1}) }^2, 
\end{array}
\end{equation*}
which proves \eqref{eq:SGM3_bounding_V}.
\end{proof}
%%% End of Proof

Next, we further upper bound \eqref{eq:SGM3_bounding_V} from Lemma~\ref{le:SGM3_bounding_V} as follows.

%%% Lemma~6.
\begin{lemma}\label{le:SGM3_full_shuffling_V_bound1}
Under the same condition as in Lemma~\ref{le:SGM3_bounding_V}, $1 - 3L_w^2\eta_t^2 \geq 0$, $1 - 3L_u^2\hat{\eta}_t^2 \geq 0$, and $\hat{\eta}_t \leq \frac{2}{L_u + \mu_H}$, we have
\begin{equation}\label{eq:SGM3_full_shuffling_V_bound2} 
\arraycolsep=0.1em
\begin{array}{lcl}
\Vc_t - \Vc_{t-1} & \leq &  - \frac{C_0}{2\eta_t } \norms{ \widetilde{w}_t  - \widetilde{w}_{t-1} }^2 -  \frac{1 - (L_u + 3C_1) \hat{\eta}_t}{2\hat{\eta}_t} \norms{ \widetilde{u}_t  - \widetilde{u}_{t-1} }^2  - \frac{C_3 \eta_t}{2}  \Vert {\Gc}_{\eta_t}(\widetilde{w}_{t-1})  \Vert^2  \\
&&   - {~}  \big( \mu_{\psi} C_2 \hat{\eta}_t  -   \frac{3C_1 \eta_t}{4\mu_{\psi}} \big) \big[ \psi( \widetilde{w}_{t-1}, \widehat{u}_t ) - \psi( \widetilde{w}_{t-1}, u_0^{*}( \widetilde{w}_{t-1} )) \big] +  C_4 \eta_t^3 + C_5 \hat{\eta}_t^3,
\end{array}
\end{equation}
where $C_i$ for $i=0,1,\cdots, 5$ are respectively given as follows:
\begin{equation}\label{eq:SGM3_full_shuffling_V_bound2_constants} 
\arraycolsep=0.2em
\left\{\begin{array}{lcl}
C_0 & := &  \lambda - 2 - [(1+\lambda)L_{\Phi_0}+ L_w + L_f ] \eta_t, \\
C_1 & := &   L_u^2[ \lambda + 1 + 4(\lambda + 2) L_w^2 \eta_t^2 ], \\
C_2 & := & 1 - (L_u - \mu_{\psi})\hat{\eta}_t - \frac{ \hat{\Lambda}_0 L_u^2(3 \Theta_u + 2) \hat{\eta}_t^2}{ 1 + \mu_h\hat{\eta}_t} - 3C_1\eta_t\hat{\eta}_t, \\
C_3 &:= &  (\lambda + 1)(1 - 2L_{\Phi_0}\eta_t) - 2\Lambda_0(\lambda + 2) L_w^2(3 \Theta_w + 1) \eta_t^2, \\
C_4 & := &  (\lambda+2)L_w^2\big[ 3\sigma_w^2 +  \Lambda_1 (3 \Theta_w + 1) \big], \\
C_5 &:= & \frac{ \hat{\Lambda}_1 L_u^2(3 \Theta_u + 2)  + 3L_u^2   \sigma_u^2 }{2(1 + \mu_h\hat{\eta}_t)}.
\end{array}\right.
\end{equation}
\end{lemma}

%%% Proof of Lemma 6.
\begin{proof}
First, since $1 - 3L_w^2\eta_t^2 \geq 0$, combining the last line of \eqref{eq:SGM2_key_bounds_for_gt_2} and \eqref{eq:SGM2_key_bounds_for_wt_2} of Lemma~\ref{le:SGM2_key_bounds_for_wt_2}, we can show that
\begin{equation}\label{eq:SGM3_full_shuffling_V_bound1_proof1}
\hspace{-0.25ex}
\begin{array}{lcl}
\norms{ g_t - \nabla{\Phi_0}(\widetilde{w}_{t-1}) }^2 & \overset{ \tiny\eqref{eq:SGM2_key_bounds_for_gt_2} }{ \leq } & \frac{L_w^2}{n}\sum_{j=1}^n \norms{w_{i-1}^{(t)} - w_0^{(t)}}^2 + L_u^2\norms{ \widetilde{u}_t - u_0^{*}(\widetilde{w}_{t-1}) }^2  \\
& \overset{ \tiny\eqref{eq:SGM2_key_bounds_for_wt_2}}{ \leq } &   L_u^2\big( 4 L_w^2\eta_t^2  + 1 \big) \norms{ \widetilde{u}_t - u_0^{*}(\widetilde{w}_{t-1})}^2 +  6 L_w^2\sigma_w^2 \eta_t^2 \\
&& + {~} 2 L_w^2(3 \Theta_w + 1) \eta_t^2 \norms{ \nabla{\Phi_0}(\widetilde{w}_{t-1}) }^2.
\end{array}
\hspace{-3ex}
\end{equation}
Similarly, from the first line of \eqref{eq:SGM2_key_bounds_for_gt_2} and \eqref{eq:SGM2_key_bounds_for_wt_2}, we also have
\begin{equation}\label{eq:SGM3_full_shuffling_V_bound1_proof1b}
\begin{array}{lcl}
\norms{ g_t -  \nabla_w{\Hc}(\widetilde{w}_{t-1}, \widetilde{u}_t)  }^2  &\leq &  4 L_w^2 L_u^2\eta_t^2 \norms{ \widetilde{u}_t - u_0^{*}(\widetilde{w}_{t-1})}^2  +  6 L_w^2\sigma_w^2 \eta_t^2 \\
&& + {~} 2 L_w^2(3 \Theta_w + 1) \eta_t^2 \norms{ \nabla{\Phi_0}(\widetilde{w}_{t-1}) }^2.
\end{array}
\end{equation}
Next, since $1 - 3L_u^2\hat{\eta}_t \geq 0$, combining \eqref{eq:SGM3_key_bounds_for_ut_2} and \eqref{eq:SGM3_key_bounds_for_ut_2b} of Lemma~\ref{le:SGM3_key_bounds_for_ut_2}, and \eqref{eq:grad_mapp_bound2} from Assumption~\ref{as:A4}, we have
\begin{equation}\label{eq:SGM3_full_shuffling_V_bound1_proof2}
\begin{array}{lcl}
\norms{ v_t -   \nabla_u{\Hc}(\widetilde{w}_{t-1}, \widetilde{u}_{t-1}) }^2 &\leq & L_u^2(3 \Theta_u + 2) \hat{\eta}_t^2  \norms{  \nabla_u{\Hc}(\widetilde{w}_{t-1}, \widetilde{u}_{t-1})  }^2 + 3L_u^2   \sigma_u^2 \hat{\eta}_t^2 \\
& \leq & \hat{\Lambda}_0 L_u^2(3 \Theta_u + 2) \hat{\eta}_t^2  \norms{  \hat{\Gc}_{\hat{\eta}_t }( \widetilde{u}_{t-1})  }^2 \\
&& + {~} \big(\hat{\Lambda}_1 L_u^2(3 \Theta_u + 2)  + 3L_u^2   \sigma_u^2 \big) \hat{\eta}_t^2.
\end{array}
\end{equation}
Substituting \eqref{eq:SGM3_full_shuffling_V_bound1_proof1}, \eqref{eq:SGM3_full_shuffling_V_bound1_proof1b}, and \eqref{eq:SGM3_full_shuffling_V_bound1_proof2} into \eqref{eq:SGM3_bounding_V}, and noting that $\mu_{\psi} := \mu_h + \mu_H > 0$, we obtain
%\begin{equation*} 
%\hspace{-0.25ex}
%\begin{array}{lcl}
%\Vc_t - \Vc_{t-1} & \leq & - \frac{\lambda - 2 - [(1+\lambda)L_{\Phi_0}+ L_w + L_f ] \eta_t }{2\eta_t} \norms{ \widetilde{w}_t  - \widetilde{w}_{t-1} }^2 -  \frac{(1-L_u \hat{\eta}_t)}{2\hat{\eta}_t} \norms{ \widetilde{u}_t  - \widetilde{u}_{t-1} }^2 \\
%&& - {~} \frac{ (\lambda + 1) \eta_t(1 - 2L_{\Phi_0}\eta_t)}{2}  \Vert {\Gc}_{\eta_t}(\widetilde{w}_{t-1})  \Vert^2   -  \frac{ \hat{\eta}_t(1 + (\mu_h + \mu_H)\hat{\eta}_t - L_u\hat{\eta}_t)}{2}  \Vert \hat{\Gc}_{ \hat{\eta}_t}(\widetilde{u}_{t-1})  \Vert^2 \\
%&& + {~}  \frac{L_u^2 \eta_t}{2}\big[ 4(\lambda + 2) L_w^2\eta_t^2  + \lambda + 1 \big]  \norms{ \widetilde{u}_t - u_0^{*}(\widetilde{w}_{t-1})}^2 +  3(\lambda + 2) L_w^2\sigma_w^2 \eta_t^3 \\
%&& + {~} (\lambda + 2) L_w^2(3 \Theta_w + 1) \eta_t^3 \norms{ \nabla{\Phi_0}(\widetilde{w}_{t-1}) }^2 \\
%&& + {~} \frac{ \hat{\Lambda}_0 L_u^2(3 \Theta_u + 2) \hat{\eta}_t^3}{2(1 + \mu_h\hat{\eta}_t)}  \norms{  \hat{\Gc}_{\hat{\eta}_t }( \widetilde{u}_{t-1})  }^2  + \frac{(\hat{\Lambda}_1 L_u^2(3 \Theta_u + 2)  + 3L_u^2   \sigma_u^2) \hat{\eta}_t^3}{2(1 + \mu_h\hat{\eta}_t)}. 
%\end{array}
%\hspace{-4ex}
%\end{equation*}
%Simplifying the last estimate, we get %and using $\norms{\nabla_w{\Hc}(\widetilde{w}_{t-1}, \widetilde{u}_t) - \nabla{\Phi}_0( \widetilde{w}_{t-1} ) }^2 \leq  L_u^2 \norms{\widetilde{u}_t - u_0^{*}(\widetilde{w}_{t-1} ) }^2$ from the second line of  \eqref{eq:SGM2_key_bounds_for_gt_2}, we get
\begin{equation}\label{eq:SGM3_full_shuffling_V_bound1_proof3}
\hspace{-0.25ex}
\begin{array}{lcl}
\Vc_t - \Vc_{t-1} & \leq & - \frac{\lambda - 2 - [(1+\lambda)L_{\Phi_0}+ L_w + L_f ] \eta_t }{2\eta_t} \norms{ \widetilde{w}_t  - \widetilde{w}_{t-1} }^2 -  \frac{(1-L_u \hat{\eta}_t)}{2\hat{\eta}_t} \norms{ \widetilde{u}_t  - \widetilde{u}_{t-1} }^2 \\
&& - {~} \frac{ (\lambda + 1) (1 - 2L_{\Phi_0}\eta_t) \eta_t }{2}  \Vert {\Gc}_{\eta_t}(\widetilde{w}_{t-1})  \Vert^2 +  (\lambda + 2) L_w^2(3 \Theta_w + 1) \eta_t^3 \norms{ \nabla{\Phi_0}(\widetilde{w}_{t-1}) }^2 \\
&&   - {~}  \frac{ \hat{\eta}_t}{2} \big[1 - (L_u - \mu_{\psi} )\hat{\eta}_t - \frac{ \hat{\Lambda}_0 L_u^2(3 \Theta_u + 2) \hat{\eta}_t^2}{ 1 + \mu_h\hat{\eta}_t} \big] \Vert \hat{\Gc}_{ \hat{\eta}_t}(\widetilde{u}_{t-1})  \Vert^2 \\
&& + {~}  \frac{L_u^2\eta_t}{2} \big[ 4(\lambda + 2) L_w^2 \eta_t^2  + \lambda+1 \big]  \norms{ \widetilde{u}_t - u_0^{*}(\widetilde{w}_{t-1})}^2 \\
&& + {~}  3(\lambda + 2) L_w^2\sigma_w^2 \eta_t^3  + \frac{ [ \hat{\Lambda}_1 L_u^2(3 \Theta_u + 2)  + 3L_u^2   \sigma_u^2]  \hat{\eta}_t^3}{2(1 + \mu_h\hat{\eta}_t)}. 
\end{array}
\hspace{-4ex}
\end{equation}
Next, by Young's inequality, \eqref{eq:SGM3_gradient_mappings}, and  \eqref{eq:SGM3_key_properties_of_psi_p1}, we can show that
\begin{equation*} 
\begin{array}{lcl}
\norms{\widetilde{u}_t - u_0^{*}(\widetilde{w}_{t-1} ) }^2 &\leq & 3\norms{\widetilde{u}_t - \widetilde{u}_{t-1}}^2 + 3\norms{\widetilde{u}_{t-1} - \widehat{u}_t }^2 + 3\norms{\widehat{u}_t - u_0^{*}(\widetilde{w}_{t-1} ) }^2\\
& \overset{\tiny\eqref{eq:SGM3_gradient_mappings},  \eqref{eq:SGM3_key_properties_of_psi_p1} }{\leq} & 3\norms{\widetilde{u}_t - \widetilde{u}_{t-1}}^2 + 3\hat{\eta}_t^2 \norms{\hat{\Gc}_{\hat{\eta}_t}(\widetilde{u}_{t-1} ) }^2 \\
&& + {~} \frac{3}{2\mu_{\psi}} \big[ \psi( \widetilde{w}_{t-1}, \widehat{u}_t ) - \psi( \widetilde{w}_{t-1}, u_0^{*}( \widetilde{w}_{t-1} )) \big].
\end{array}
\end{equation*}
Substituting the last inequality and $\norms{ \nabla{\Phi_0}(\widetilde{w}_{t-1}) }^2  \leq  \Lambda_0\norms{ {\Gc}_{\eta_t}(\widetilde{w}_{t-1}) }^2 + \Lambda_1$ from \eqref{eq:grad_mapp_bound} of Assumption~\ref{as:A3} into \eqref{eq:SGM3_full_shuffling_V_bound1_proof3}, we can derive that
\begin{equation}\label{eq:SGM3_full_shuffling_V_bound1_proof5}
\begin{array}{lcl}
\Vc_t - \Vc_{t-1} & \leq & - \frac{\lambda - 2 - [(1+\lambda)L_{\Phi_0}+ L_w + L_f ] \eta_t }{2\eta_t} \norms{ \widetilde{w}_t  - \widetilde{w}_{t-1} }^2 -  \frac{1 - (L_u + 3C_1\eta_t) \hat{\eta}_t}{2\hat{\eta}_t} \norms{ \widetilde{u}_t  - \widetilde{u}_{t-1} }^2 \\
&& - {~} \frac{\eta_t}{2} \big[  (\lambda + 1)(1 - 2L_{\Phi_0}\eta_t) - 2\Lambda_0(\lambda + 2) L_w^2(3 \Theta_w + 1) \eta_t^2 \big]  \Vert {\Gc}_{\eta_t}(\widetilde{w}_{t-1})  \Vert^2  \\
&&   - {~}  \frac{ C_2 \hat{\eta}_t}{2} \Vert \hat{\Gc}_{ \hat{\eta}_t}(\widetilde{u}_{t-1})  \Vert^2  +  \frac{3C_1 \eta_t}{4 \mu_{\psi}} \big[ \psi( \widetilde{w}_{t-1}, \widehat{u}_t ) - \psi( \widetilde{w}_{t-1}, u_0^{*}( \widetilde{w}_{t-1} )) \big] \\
&& + {~}  \big[ 3(\lambda + 2) L_w^2\sigma_w^2 +  \Lambda_1(\lambda + 2) L_w^2(3 \Theta_w + 1) \big] \eta_t^3 \\
&& + {~} \frac{\big[ \hat{\Lambda}_1 L_u^2(3 \Theta_u + 2)  + 3L_u^2   \sigma_u^2 \big] \hat{\eta}_t^3}{2(1 + \mu_h\hat{\eta}_t)}, 
\end{array}
\end{equation}
where $C_1 := L_u^2\big[ 4(\lambda + 2) L_w^2 \eta_t^2  + \lambda + 1 \big]$ and $C_2 := 1 - (L_u - \mu_h - \mu_H)\hat{\eta}_t - \frac{ \hat{\Lambda}_0 L_u^2(3 \Theta_u + 2) \hat{\eta}_t^2}{ 1 + \mu_h\hat{\eta}_t} - 3C_1\eta_t\hat{\eta}_t$.

Finally, from \eqref{eq:SGM3_key_properties_of_psi_p2}, since $\hat{\eta}_t \leq \frac{2}{L_u + \mu_H}$, we also have
\begin{equation*}%\label{le:SGM3_full_shuffling_V_bound1_proof4c} 
\arraycolsep=0.2em
\begin{array}{lcl}
-  \norms{ \hat{\Gc}_{\hat{\eta}_t }( \widetilde{u}_{t-1} )  }^2 \leq - \left(1 - \frac{2L_u\mu_H \hat{\eta}_t}{L_u + \mu_H}\right) \norms{ \hat{\Gc}_{\hat{\eta}_t }( \widetilde{u}_{t-1} )  }^2 \leq - 2\mu_{\psi}[ \psi( \widetilde{w}_{t-1}, \widehat{u}_t ) - \psi( \widetilde{w}_{t-1}, u_0^{*}( \widetilde{w}_{t-1} ))].
\end{array}
\end{equation*}
Substituting this inequality into \eqref{eq:SGM3_full_shuffling_V_bound1_proof5}, we arrive at
\begin{equation*} 
\hspace{-0.25ex}
\begin{array}{lcl}
\Vc_t - \Vc_{t-1} & \leq & - \frac{\lambda - 2 - [(1+\lambda)L_{\Phi_0}+ L_w + L_f ] \eta_t }{2\eta_t} \norms{ \widetilde{w}_t  - \widetilde{w}_{t-1} }^2 -  \frac{1 - (L_u + 3C_1\eta_t) \hat{\eta}_t}{2\hat{\eta}_t} \norms{ \widetilde{u}_t  - \widetilde{u}_{t-1} }^2 \\
&& - {~} \frac{\eta_t}{2} \big[  (\lambda + 1)(1 - 2L_{\Phi_0}\eta_t) - 2\Lambda_0( \lambda + 2) L_w^2(3 \Theta_w + 1) \eta_t^2 \big]  \Vert {\Gc}_{\eta_t}(\widetilde{w}_{t-1})  \Vert^2  \\
&&   - {~}  \big( \mu_{\psi} C_2 \hat{\eta}_t  -   \frac{3C_1 \eta_t}{4\mu_{\psi}} \big) \big[ \psi( \widetilde{w}_{t-1}, \widehat{u}_t ) - \psi( \widetilde{w}_{t-1}, u_0^{*}( \widetilde{w}_{t-1} )) \big] \\
&& + {~}  (\lambda + 2)L_w^2 \big[ 3 \sigma_w^2 +  \Lambda_1 (3 \Theta_w + 1) \big] \eta_t^3 + \frac{ \big[ \hat{\Lambda}_1 L_u^2(3 \Theta_u + 2)  + 3L_u^2   \sigma_u^2 \big] \hat{\eta}_t^3}{2(1 + \mu_h\hat{\eta}_t)}, 
\end{array}
\hspace{-4ex}
\end{equation*}
which is exactly \eqref{eq:SGM3_full_shuffling_V_bound2}.
\end{proof}
%%% End of Proof.
 
Now, we are ready to prove the convergence of Algorithm~\ref{alg:SGM2} using only \textbf{one epoch} (i.e. $S=1$) of the \textbf{shuffling routine \eqref{eq:SGM4lower_prob}}.
The following theorem is the full version of Theorem~\ref{th:SGM2_full_shuffling_main_result3} in the main text.

%%% Theorem 3.
\begin{theorem}\label{th:SGM3_full_shuffling_convergence1}
Suppose that Assumptions~\ref{as:A0},~\ref{as:A2},~\ref{as:A3_SGM2}, ~\ref{as:A3},~\ref{as:A4}, and~\ref{as:A5} hold for \eqref{eq:minimax_prob} under the $($NC$)$ setting.
Let $\Psi_0$ be defined by \eqref{eq:upper_level_min}, and ${\Gc}_{\eta}$ be defined by \eqref{eq:upper_level_min_grad_mapping}.
Let us denote  $C_w$ and $C_u$ respectively by 
\begin{equation}\label{eq:SGM3_full_shuffling_convergence1_const}
\begin{array}{lcl}
C_w & := &  5L_w^2\big[  \Lambda_1 (3 \Theta_w + 1) + 3\sigma_w^2 \big] \quad \text{and} \quad
C_u   :=    \frac{L_u^2}{2} \big[ \hat{\Lambda}_1(3 \Theta_u + 2)  + 3\sigma_u^2 \big].
\end{array}
\end{equation}
Let $\sets{(\widetilde{w}_t, \widetilde{u}_t)}$ be generated by Algorithm~\ref{alg:SGM2} using  only \textbf{one epoch} $($i.e. $S=1$$)$ of the \textbf{shuffling routine \eqref{eq:SGM4lower_prob}}, and fixed learning rates $\eta_t := \eta \in (0, \bar{\eta}]$ and $\hat{\eta}_t := \hat{\eta} := 15 \kappa^2\eta$, where
\begin{equation*} 
\begin{array}{lll}
\bar{\eta} := \min\Big\{  \frac{1}{60\kappa^2 L_u}, \   \frac{1}{\sqrt{ 10 \Lambda_0 L_w^2(3 \Theta_w + 1)} }, \   \frac{2\sqrt{L_u}}{\kappa\sqrt{15(4L_u^2 + \mu_{\psi}^2)}}, \  \frac{\sqrt{L_u} }{15 \kappa  \sqrt{ 2L_u^3\hat{\Lambda}_0(3 \Theta_u + 2)  + \mu_{\psi}^2 } }, \ \frac{1}{4L_{\Phi_0}+ L_w + L_f}  \Big\}.
\end{array}
\end{equation*}
Then, the following bound holds:
\begin{equation}\label{SGM3_full_shuffling_convergence1_bound}
\begin{array}{lcl}
\frac{1}{T+1}\sum_{t=0}^T \norms{ {\Gc}_{\eta}(\widetilde{w}_{t}) }^2 & \leq & \frac{24 [ \Psi_0(\widetilde{w}_0) - \Psi_0^{\star}]  + 8[\Psi_0(\widetilde{w}_0) - \Lc(\widetilde{w}_0, \widetilde{u}_0)] }{\eta (T+1)} + 8C_w\eta^2 + \frac{8C_u \hat{\eta}^3 }{\eta} 
\end{array}
\end{equation}
For a given $\epsilon > 0$, if $\eta := \mathcal{O}(\epsilon) \in  (0, \bar{\eta}]$ and $T := \mathcal{O}\big( \frac{1}{\epsilon^3} \big)$, then $\frac{1}{T+1}\sum_{t=0}^T \norms{ {\Gc}_{\eta}(\widetilde{w}_{t}) }^2 \leq \epsilon^2$.

Consequently, Algorithm~\ref{alg:SGM2} requires $\mathcal{O}\big( \frac{n}{\epsilon^3} \big)$ evaluations of both $\nabla_w{\Hc_i}$ and $\nabla_u{\Hc}_i$, and  $\Oc(\epsilon^{-3})$ evaluations of $\prox_{\eta_tf}$ and $\prox_{\hat{\eta}_th}$  to achieve an $\epsilon$-stationary point $\widehat{w}_T$ of \eqref{eq:minimax_prob} computed by \eqref{eq:approx_stationary_point}.
\end{theorem}

%%% Proof of Theorem 4.
\begin{proof}
Let us choose $\lambda := 3$, $\eta_t := \eta > 0$ and $\hat{\eta}_t := \hat{\eta} > 0$.
First, we need to guarantee that $1 - 3L_w^2\eta^2 \geq 0$ and $1 - 3L_u^2\hat{\eta}_t \geq 0$ in Theorem~\ref{le:SGM3_full_shuffling_V_bound1}.
Suppose that $\eta \leq \frac{1}{4L_{\Phi_0}+ L_w + L_f}$.
Then, since $L_{\Phi_0} = (1+\kappa)L_w \geq L_w$, we have $\eta \leq \frac{1}{5L_w}$, which obviously guarantees that $1 - 3L_w^2\eta^2 \geq 0$.

Moreover, for $C_i$ for $i=0,\cdots, 5$ defined by \eqref{eq:SGM3_full_shuffling_V_bound2_constants}, we can show that
\begin{equation*} 
\left\{\begin{array}{lcl}
C_0 & := & 1  - (4L_{\Phi_0}+ L_w + L_f) \eta \geq 0, \\
C_1 & := &  L_u^2(4  + 20 L_w^2 \eta^2) \leq 5L_u^2, \\
C_2 & := & 1 - (L_u - \mu_{\psi})\hat{\eta} - \frac{ \hat{\Lambda}_0 L_u^2(3 \Theta_u + 2) \hat{\eta}^2}{ 1 + \mu_h\hat{\eta} } - 6C_1\eta \hat{\eta}^2 \geq 1 - L_u\hat{\eta} -  \hat{\Lambda}_0L_u^2(3 \Theta_u + 2)\hat{\eta}^2 - 30L_u^2 \eta\hat{\eta}^2, \\
C_3 &:= & 4  - 8L_{\Phi_0}\eta  - 10 \Lambda_0 L_w^2(3 \Theta_w + 1) \eta^2 \geq 2\big[ 1  -  5 \Lambda_0 L_w^2(3 \Theta_w + 1) \eta^2 \big], \\
C_4 & := &  5L_w^2\big[  \Lambda_1 (3 \Theta_w + 1) + 3\sigma_w^2 \big] = C_w, \\
C_5 &:= & \frac{\hat{\Lambda}_1 L_u^2(3 \Theta_u + 2)  + 3L_u^2   \sigma_u^2 }{2(1 + \mu_h\hat{\eta})} \leq C_u.
\end{array}\right.
\end{equation*}
Now, suppose that 
\begin{equation}\label{th:SGM3_full_shuffling_convergence1_proof1} 
\left\{\begin{array}{lll}
& L_u\hat{\eta} \leq \frac{1}{4}, \quad \hat{\Lambda}_0L_u^2(3 \Theta_u + 2)\hat{\eta}^2 + 30L_u^2 \eta\hat{\eta}^2 \leq \frac{1}{2}, \quad (L_u + 15L_u^2\eta )\hat{\eta} \leq 1,  \\
& 4C_2 \mu_{\psi}^2\hat{\eta} \geq 3C_1\eta, \quad  \text{and} \quad 5 \Lambda_0 L_w^2(3 \Theta_w + 1) \eta^2 \leq \frac{1}{2},
\end{array}\right.
\end{equation}
then we can easily show that $C_2 \geq \frac{1}{4}$, $C_3 \geq 1$, $1 - (L_u + 3C_1\eta)\hat{\eta} \geq 0$, and $ \mu_{\psi} C_2 \hat{\eta}  -   \frac{3C_1 \eta}{4\mu_{\psi}} \geq 0$.

In this case, \eqref{eq:SGM3_full_shuffling_V_bound2} reduces to
\begin{equation}\label{th:SGM3_full_shuffling_convergence1_proof3}
\begin{array}{lcl}
\Vc_t  & \leq & \Vc_{t-1}  - \frac{\eta}{8}  \Vert {\Gc}_{\eta_t}(\widetilde{w}_{t-1})  \Vert^2 + C_w \eta^3 + C_u \hat{\eta}^3.
\end{array}
\end{equation}
By induction, we obtain \eqref{SGM3_full_shuffling_convergence1_bound} from \eqref{th:SGM3_full_shuffling_convergence1_proof3} and $\Vc_0 := 3[\Psi_0(\widetilde{w}_0) - \Psi^{\star}_0] + \Psi_0(\widetilde{w}_0) - \Lc(\widetilde{w}_0, \widetilde{u}_0)$.

From \eqref{th:SGM3_full_shuffling_convergence1_proof1}, let us choose $\hat{\eta} = \frac{15 L_u^2}{\mu_{\psi}^2}\eta = 15\kappa^2\eta$ with $\kappa := \frac{L_u}{\mu_{\psi}}$.
Then, we can verify the five conditions of \eqref{th:SGM3_full_shuffling_convergence1_proof1} as follows.
\begin{itemize}
\item We have $4C_2\mu_{\psi}^2\hat{\eta} \geq \mu_{\psi}^2\hat{\eta} = 15L_u^2\eta \geq 3C_1\eta$, which satisfies the fourth condition of \eqref{th:SGM3_full_shuffling_convergence1_proof1}.
\item If $\eta \leq \frac{1}{60L_u\kappa^2}$, then the condition $L_u\hat{\eta} \leq \frac{1}{4}$ in \eqref{th:SGM3_full_shuffling_convergence1_proof1} holds. 
This condition also guarantees $1 - 3L_u^2\hat{\eta}^2 \geq 0$.
\item If $\eta \leq \frac{1}{\sqrt{10 \Lambda_0 L_w^2(3 \Theta_w + 1)}}$, then the last condition $5 \Lambda_0 L_w^2(3 \Theta_w + 1) \eta^2 \leq \frac{1}{2}$ of \eqref{th:SGM3_full_shuffling_convergence1_proof1} holds.
\item If $\eta \leq \frac{2\sqrt{L_u}}{\kappa\sqrt{15(4L_u^2 + \mu_{\psi}^2)}}$, then the condition $(L_u + 15L_u^2\eta )\hat{\eta} \leq 1$ of \eqref{th:SGM3_full_shuffling_convergence1_proof1} holds.
\item Finally, if  $\eta \leq \frac{\sqrt{L_u} }{15 \kappa  \sqrt{ 2L_u^3\hat{\Lambda}_0(3 \Theta_u + 2)  + \mu_{\psi}^2 } }$, then the second condition $\hat{\Lambda}_0L_u^2(3 \Theta_u + 2)\hat{\eta}^2 + 30L_u^2 \eta\hat{\eta}^2 \leq \frac{1}{2}$  of \eqref{th:SGM3_full_shuffling_convergence1_proof1} also holds.
\end{itemize}
Overall, we can conclude that if we choose $\eta \in (0, \bar{\eta}]$ as in Theorem~\ref{th:SGM3_full_shuffling_convergence1}, where
\begin{equation*} 
\begin{array}{lll}
\bar{\eta} := \min\Big\{  \frac{1}{60\kappa^2 L_u}, \   \frac{1}{\sqrt{ 10 \Lambda_0 L_w^2(3 \Theta_w + 1)} }, \   \frac{2\sqrt{L_u}}{\kappa\sqrt{15(4L_u^2 + \mu_{\psi}^2)}}, \  \frac{\sqrt{L_u} }{15 \kappa  \sqrt{ 2L_u^3\hat{\Lambda}_0(3 \Theta_u + 2)  + \mu_{\psi}^2 } }, \ \frac{1}{4L_{\Phi_0}+ L_w + L_f}  \Big\},
\end{array}
\end{equation*}
then all the conditions in \eqref{th:SGM3_full_shuffling_convergence1_proof1} are satisfied. 
In addition, since $\mu_H \leq L_u$, we have $L_u + \mu_H \leq 2L_u$.
Thus the condition $\eta \leq \frac{1}{60\kappa^2 L_u}$ implies  $\hat{\eta} \leq \frac{2}{L_u + \mu_H}$ due to $\hat{\eta} = 15\kappa^2\eta$.

Finally, to achieve $\frac{1}{T+1}\sum_{t=0}^T \norms{ {\Gc}_{\eta}(\widetilde{w}_{t}) }^2 \leq\epsilon^2$, we impose
\begin{equation*} 
\begin{array}{lcl}
\frac{24 [ \Psi_0(\widetilde{w}_0) - \Psi_0^{\star}]  + 8[\Psi_0(\widetilde{w}_0) - \Lc(\widetilde{w}_0, \widetilde{u}_0)] }{\eta (T+1)} + 8\big( C_w  + 15^3 \kappa^6 C_u  \big)\eta^2 \leq \epsilon^2. 
\end{array}
\end{equation*}
 If we choose $\eta := \Oc(\epsilon) \in (0, \bar{\eta}]$ sufficiently small, and $T := \Oc(\epsilon^{-3})$, then the last condition holds.

At each epoch $t$, Algorithm~\ref{alg:SGM2} requires $n$ evaluations of both $\nabla_w{\Hc}_i$ and $\nabla_u{\Hc}_i$.
Therefore, the total evaluation of $\nabla_w{\Hc}_i$ and $\nabla_u{\Hc}_i$ is $\Tc_e := n T = \BigO{n\epsilon^{-3}}$.
Similarly, since each epoch $t$, Algorithm~\ref{alg:SGM2} requires one evaluation of $\prox_{\eta_tf}$, and one evaluation of $\prox_{\hat{\eta}_th}$, the total number of  both $\prox_{\eta_tf}$ and $\prox_{\eta_tf}$ evaluations is $T = \Oc(\epsilon^{-3})$.
 \end{proof}
%%% End of Proof.

%%%%%%%%%%%%%%%%%%%%%%%%%%%%%%%%%%%%%%%%%%%%%%%%%%%%%%%%%%%%
%%% End of Supplementary Documents -- Part 3.
%%%%%%%%%%%%%%%%%%%%%%%%%%%%%%%%%%%%%%%%%%%%%%%%%%%%%%%%%%%%
%%%++++++++++++++++++++++++++++++++++++++++++++++++++++++++++++++++++++++++++++++
%%% D. Additional Numerical Experiments.
%%%++++++++++++++++++++++++++++++++++++++++++++++++++++++++++++++++++++++++++++++
\beforesec
\section{Details and Additional Results of Numerical Experiments}\label{apdx:add_experiments}
\aftersec
This section provides the details of our experiments in Section~\ref{sec:num_experiments} and also adds more experiments to illustrate our algorithms and compares them with two other methods.
All the algorithms we experiment in this paper are implemented in Python and are run on a MacBook Pro.  2.8GHz Quad-Core Intel Core I7, 16Gb Memory.

\beforesubsec
\subsection{Details of Numerical Experiments in Section~\ref{sec:num_experiments}}
\aftersubsec
We have abbreviated Algorithm~\ref{alg:SGM1} by \texttt{SGM} in Figure~\ref{fig:min_max_stochastic_opt}.
Since we have two options to construct estimator $F_i^{(t)}$ for $F(\widetilde{w}_{t-1})$, we name \texttt{SGM-Option 1} for  Algorithm~\ref{alg:SGM1} using  \eqref{eq:approxi_Ft}, and \texttt{SGM-Option 2} for Algorithm~\ref{alg:SGM1} using \eqref{eq:approxi_Ftb}.

\textbf{Implementation details and competitors.}
Since $\phi_0(v) =  \max_{\norms{u}_1 \leq 1}\{\iprods{v, u}\}$ in our model \eqref{eq:min_max_stochastic_opt} is nonsmooth, we have implemented two other algorithms,  \texttt{SGD} in \cite{wang2017stochastic} -- a variant of the stochastic gradient method for compositional minimization, and \texttt{Prox-Linear} in \cite{zhang2020stochastic} -- a type of the Gauss-Newton method with variance-reduction using large mini-batches for compositional minimization. 
Since \texttt{SGD} only works for smooth $\phi_0$, we have smoothed it as in our method, and utilized the estimator and  algorithm from \cite{wang2017stochastic}, but also updated the smoothness parameter as in our method.
Here, we only compare the performance of all algorithms in terms of epochs (i.e. the number of data passes) and ignore their computational time since \texttt{Prox-Linear} becomes slower if $p$ is getting large.
This is due to its expensive subproblem of evaluating the prox-linear operator.

To compare with \texttt{SGD} and \texttt{Prox-Linear}, we only use Algorithm~\ref{alg:SGM1} since both \texttt{SGD} and  \texttt{Prox-Linear} are designed to solve compositional minimization problems of the form \eqref{eq:com_opt0}.
However,  \texttt{Prox-Linear} requires to solve a nonsmooth convex subproblem to evaluate the prox-linear operator. 
Therefore, we have implemented a first-order primal-dual scheme in \cite{Chambolle2011} to evaluate this operator, which we believe that it is an efficient method.

\textbf{Parameter selection.}
To boost the performance of all algorithms, we implement mini-batch variants of these methods instead of a single sample variant.
Our batch size $b$ is computed by $b := \lfloor{\frac{n}{k_b}}\rfloor$, where $n$ is the number of data points and $k_b$ is the number of blocks.
In our experiments, we have also varied the number of blocks $k_b$ to observe the performance of these algorithms.
Since we want to obtain good performance, instead of using their theoretical learning rates, we have carefully tuned the learning rate $\eta$ of all algorithms in a given set of candidates $\sets{100,50, 10, 5, 1, 0.5, 0.1, 0.05, 0.01, 0.001, 0.0001}$. 
We find $\eta = 5$ (i.e. $\eta_t = 10^{-4}$) for \textbf{w8a}  and $\eta = 100$ (i.e. $\eta_t = 5\times 10^{-5}$) for \textbf{rcv1} which work well for our method.
We also update the smoothness parameter $\gamma$ as $\gamma  := \frac{1}{2(t+1)^{1/3}}$ w.r.t. to the epoch counter $t$ instead of fixing it at a small value.
For \textbf{w8a}, we find $\eta = 0.05$ as a good learning rate for both \texttt{SGD} and \texttt{Prox-Linear}.
For \textbf{rcv1}, we get $\eta = 0.5$ for both algorithms.
All experiments are run up to 200 epochs.

\textbf{The convergence of gradient mapping norm.}
Figure~\ref{fig:min_max_stochastic_opt} only reveals the objective values of \eqref{eq:min_max_stochastic_opt} against the number of epochs.
Figure~\ref{fig:min_max_stochastic_opt_grad_norm} below shows the absolute norm of the gradient mapping $\Vert \mathcal{G}_{\eta}(\widetilde{w}_t)\Vert$ for this experiment.

\begin{figure}[ht!]
%\vspace{-1ex}
\centering
\includegraphics[width=0.49\textwidth]{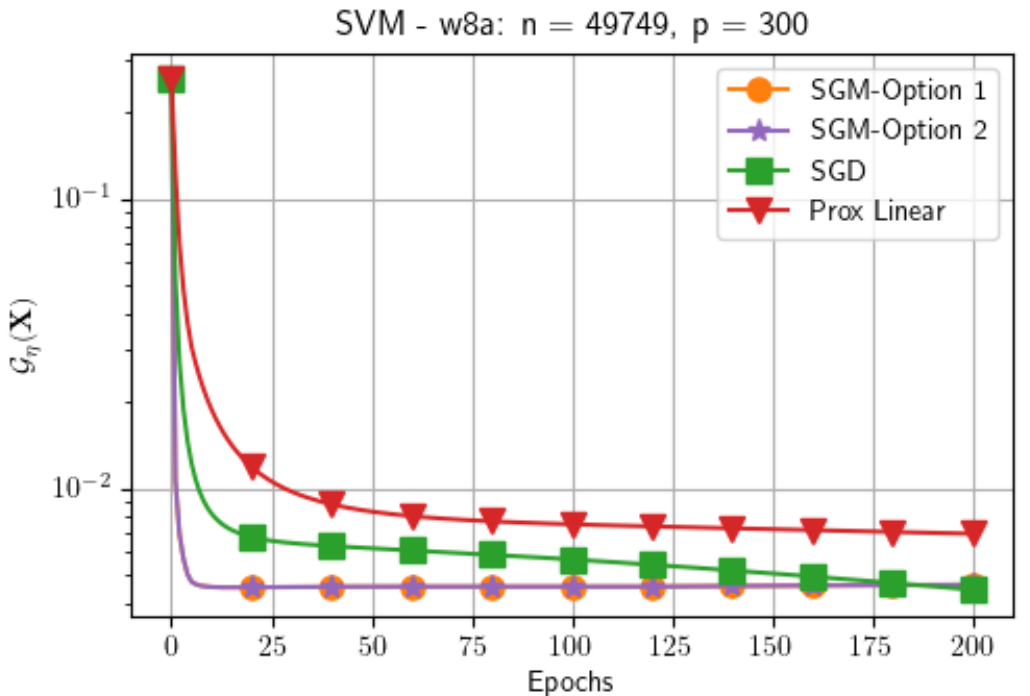}
\includegraphics[width=0.49\textwidth]{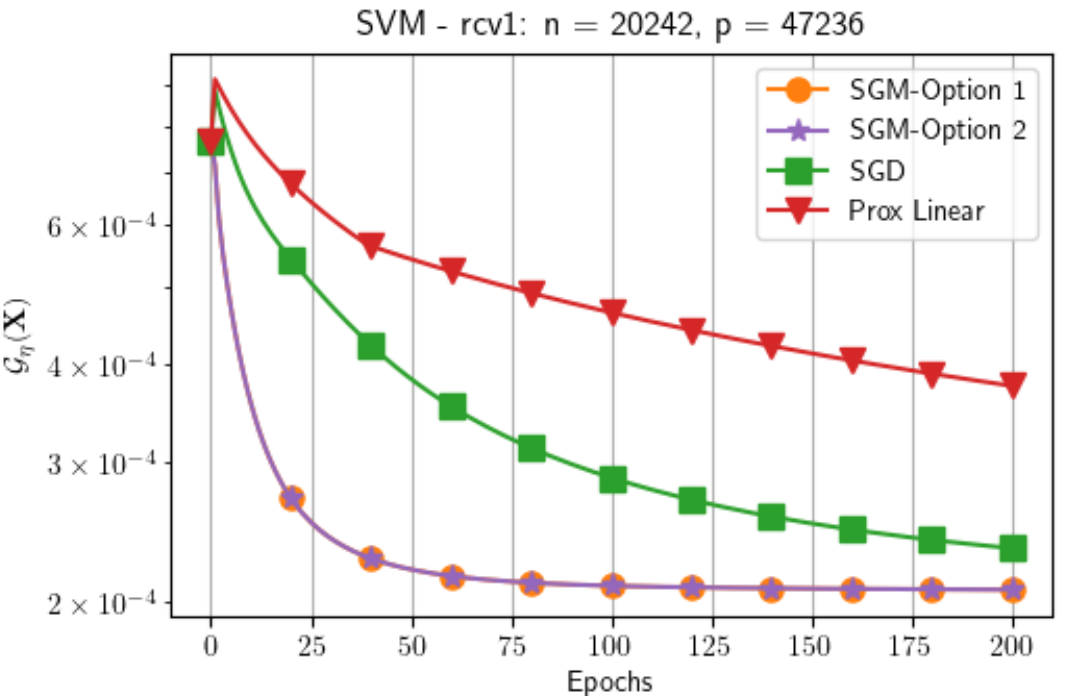}
\vspace{-0.75ex}
\caption{
The performance of 4 algorithms for solving \eqref{eq:min_max_stochastic_opt} in terms of gradient mapping norm.
}
\label{fig:min_max_stochastic_opt_grad_norm}
\vspace{-1ex}
\end{figure}

It seems that both options,  \texttt{SGM-Option 1} and  \texttt{SGM-Option 2} are almost identical for this test. 
For \textbf{w8a}, our methods look like having comparable performance with both  \texttt{SGD} and \texttt{Prox-Linear}, just slightly better.
For \textbf{rcv1}, our methods reach a better approximate solution earlier than   \texttt{SGD}, but after more than 200 epochs,   \texttt{SGD} tends to  approach a similar accuracy level.
 \texttt{Prox-Linear} has a significantly worse performance than ours and  \texttt{SGD} in this particular experiment. 

\beforesubsec
\subsection{Additional Experiments}
\aftersubsec
We provide additional experiments to test our algorithms and compare them with \texttt{SGD} and \texttt{Prox-Linear} as in  Section~\ref{sec:num_experiments}.

\textbf{The effect of mini-batch size.}
Our first test is to verify if the mini-batch size $b$ actually affects the performance of these algorithms.
We use the same datasets and the same parameters as in Section~\ref{sec:num_experiments}, but reduce $b$ by increasing $k_b$ from $32$ to $64$ blocks.
Figure \ref{fig:min_max_stochastic_opt_w8_n64} reveals the performance of 4 algorithms on two datasets  with $k_b = 64$: \textbf{w8a} corresponding to $b = 777$ and \textbf{rcv1} corresponding to $b=316$.

\begin{figure}[ht!]
\vspace{-0.5ex}
\centering
\includegraphics[width=1\textwidth]{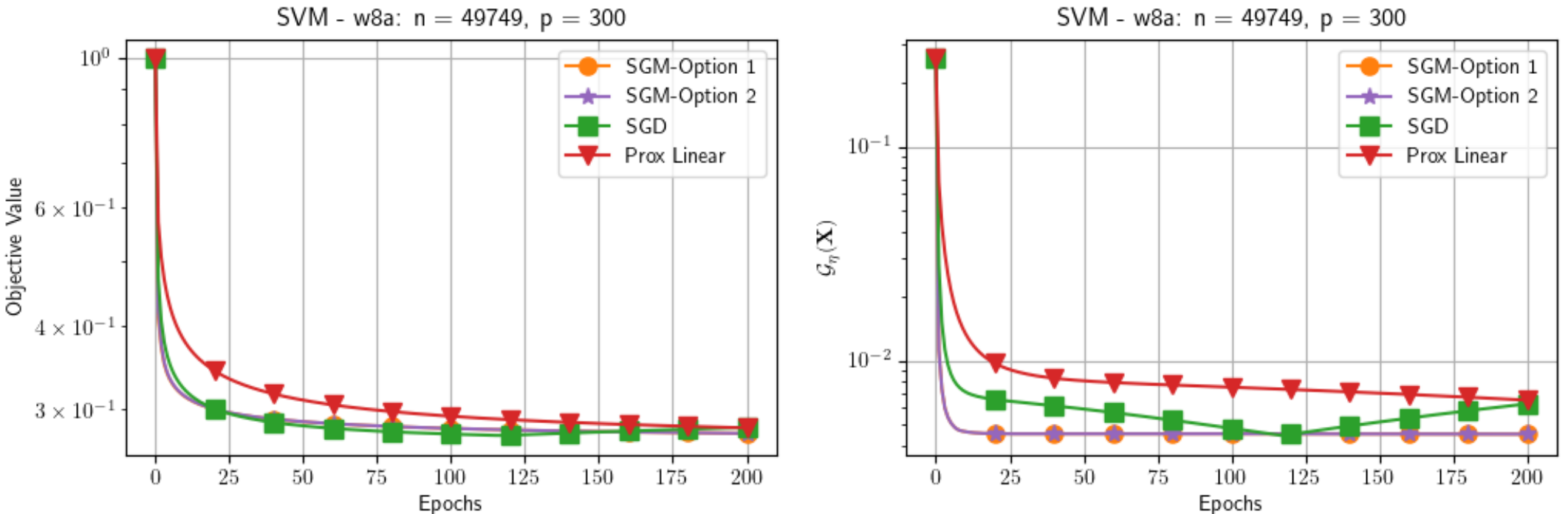}
\includegraphics[width=1\textwidth]{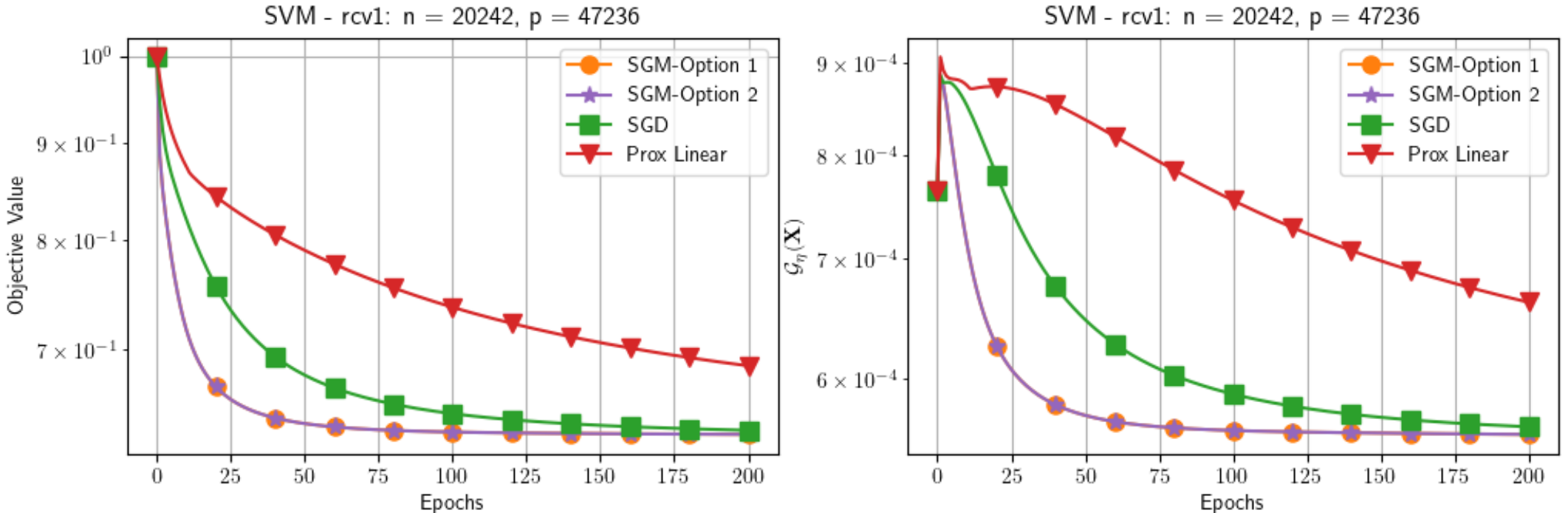}
\vspace{-0.75ex}
\caption{
The performance of 4 algorithms on two different datasets with $k_b = 64$.
}
\label{fig:min_max_stochastic_opt_w8_n64}
\vspace{-1ex}
\end{figure}

 With this choice of mini-batches, our algorithms still have a similar performance as  \texttt{SGD}, while  \texttt{Prox-Linear} does not really improve its performance, and slightly gets worse.
Note that \texttt{Prox-Linear} requires a large mini-batch to achieve a variance reduce, and decreasing this mini-batch size indeed affects its performance.

\textbf{Different learning rates.}
Now, let us test our algorithms using different learning rates, we only focus on \textbf{Option 2} as both options show similar results in our tests.
For \textbf{w8a}, we choose $4$ different learning rates $\eta = 0.5, 2.5, 5.0$, and $7.5$, while maintaining $k_b = 64$.
For \textbf{rcv1}, we also choose $4$ different learning rates $\eta = 25, 50, 100$, and $125$.
The results of this experiment are plotted in Figure~\ref{fig:min_max_stochastic_opt_lr} for both \textbf{w8a} and \textbf{rcv1} datasets.

\begin{figure}[ht!]
%\vspace{-1ex}
\centering
\includegraphics[width=1\textwidth]{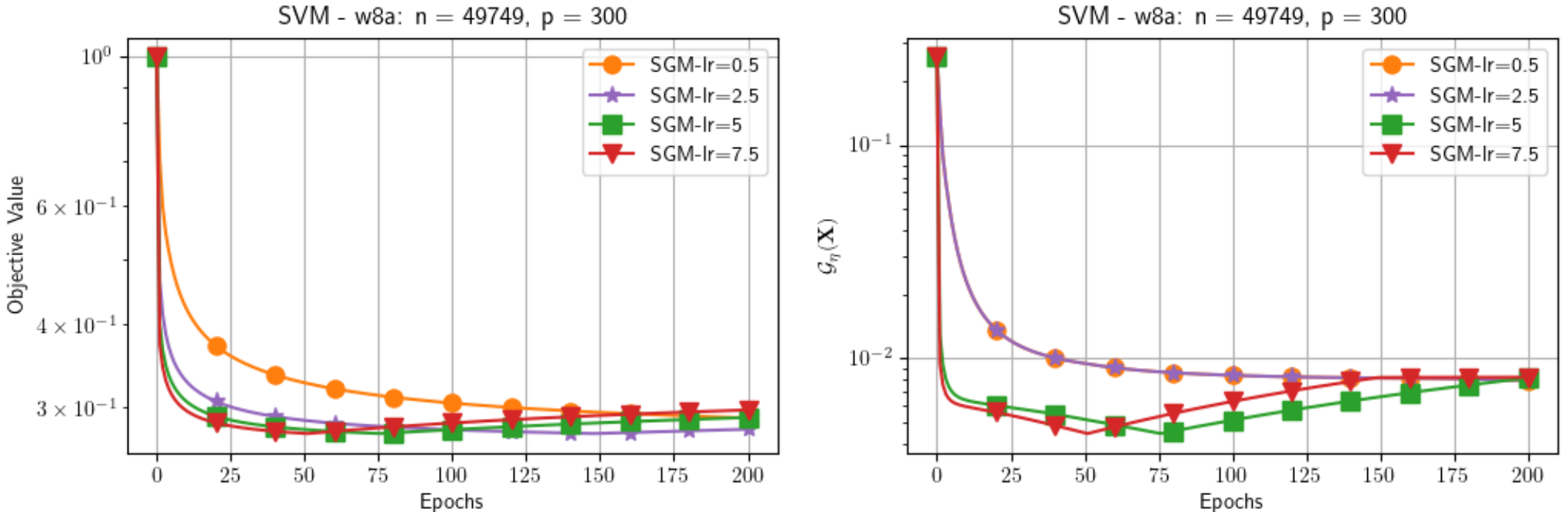}
\includegraphics[width=1\textwidth]{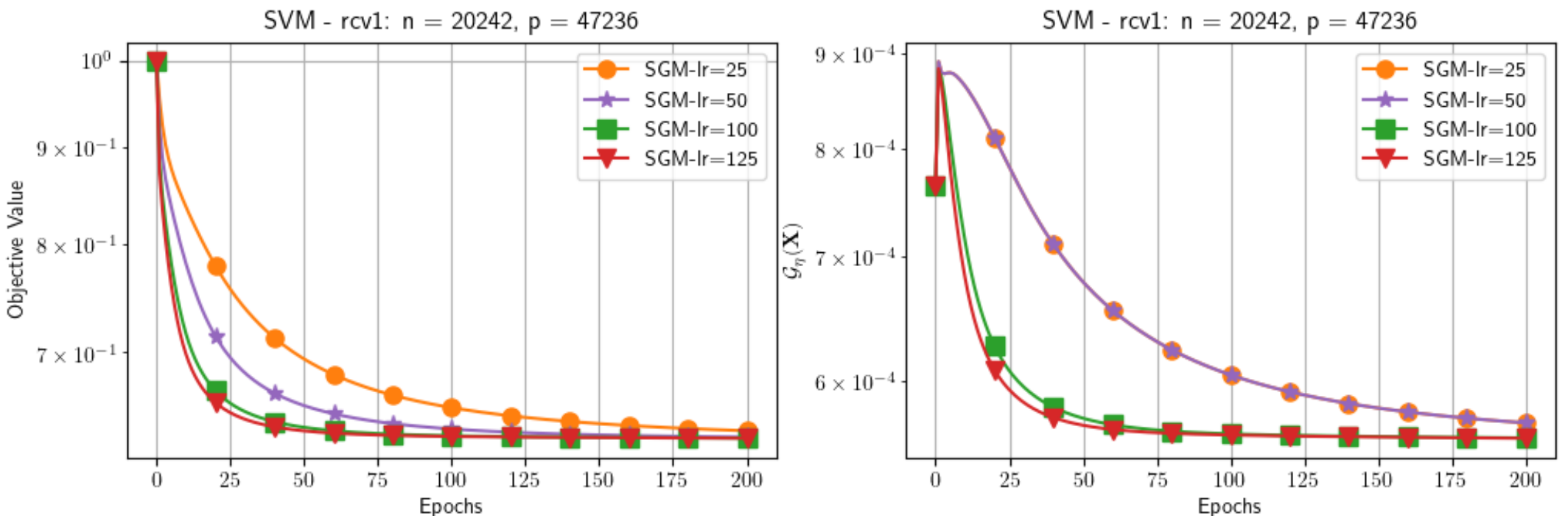}
\vspace{-0.75ex}
\caption{
The performance of Algorithm~\ref{alg:SGM1} with 4 different learning rates $\eta$ and $k_b=64$ on 2 datasets.
}
\label{fig:min_max_stochastic_opt_lr}
\vspace{-3ex}
\end{figure}

As we can see from Figure~\ref{fig:min_max_stochastic_opt_lr} that
\begin{compactitem}
\item For \textbf{w8a}, our method starts diverging when $\eta = 7.5$, while still works well for smaller learning rates.
For $\eta = 0.25$, it indeed has a slow progress in early iterations as often seen in SGD.

\item For \textbf{rcv1}, we also observe similar behaviors as in \textbf{w8a}, but with larger learning rates than $\eta = 125$.
\end{compactitem}

\textbf{Large dataset.}
We have also run our algorithms and their competitors on a bigger dataset from LIBSVM: \textbf{url} with $n = 2,396,130$ and $p = 3,231,951$.
Here, we use a learning rate $\eta = 1$ for our methods, which corresponds to $\eta_t = 4.2\times 10^{-7}$.
For \texttt{SGD}, we use a learning rate $\eta = 0.01$ and for \texttt{Prox-Linear}, we use a learning rate $\eta = 0.01$ after tuning both methods.
We also set $k_b = 64$ for all algorithms.
The results of this experiment are reported in Figure~\ref{fig:min_max_stochastic_opt_url}.

\begin{figure}[ht!]
%\vspace{-1ex}
\centering
\includegraphics[width=1\textwidth]{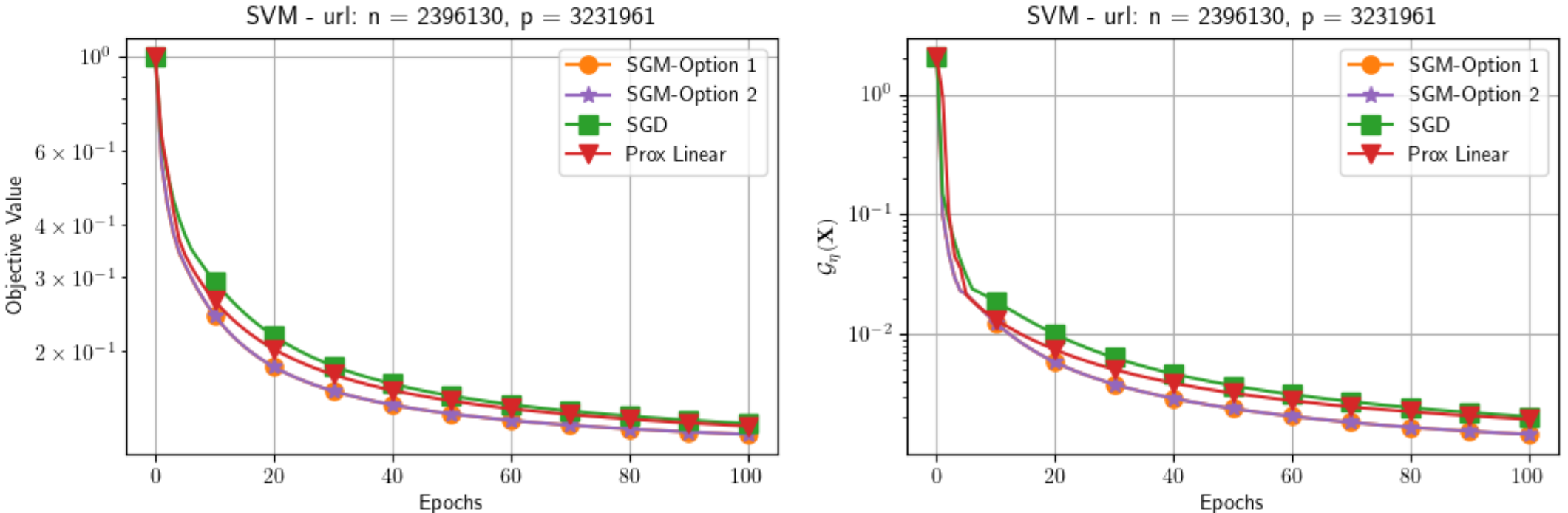}
\vspace{-0.75ex}
\caption{
The performance of 4 algorithms on a large dataset: \textbf{url}.
}
\label{fig:min_max_stochastic_opt_url}
\vspace{-1ex}
\end{figure}

As we can see from Figure~\ref{fig:min_max_stochastic_opt_url}, our methods have a comparable performance with their competitors.
All algorithms have similar behavior in terms of convergence.

%%%%%%%%%%%%%%%%%%%%%%%%%%%%%%%%%%%%%%%%%%%%%%%%%%%%%%%%%%%%
%%% End of Supplementary Documents -- Part 4.
%%%%%%%%%%%%%%%%%%%%%%%%%%%%%%%%%%%%%%%%%%%%%%%%%%%%%%%%%%%%
%
%%%%%%% References
\bibliographystyle{plain}

%\bibliography{/Users/quoctd/Dropbox/E-Books/tran_bibtex_new}

\end{document}